\documentclass[12pt,reqno,a4paper,preprint]{elsarticle}
\usepackage{graphicx}
\usepackage{color}
\usepackage{multirow}
\usepackage{amsfonts}
\usepackage{mathrsfs}
\usepackage{amssymb}
\usepackage{amsmath}
\usepackage{amsthm}
\usepackage{verbatim}
\usepackage{subcaption}
\allowdisplaybreaks
\usepackage[capitalise]{cleveref}
\crefname{subsection}{subsection}{subsections}
\usepackage[toc]{appendix}

\newcommand{\I}{{\mathcal{I}}}
\newcommand{\bcm}{\operatorname{B}}
\newcommand{\cB}{\mathcal{B}}

\newcommand{\C}{\mathbb{C}}    
\newcommand{\N}{\mathbb{N}}    
\newcommand{\NN}{\mathbb{N}_0} 
\newcommand{\R}{\mathbb{R}}    
\newcommand{\Z}{\mathbb{Z}}    
\newcommand{\PL}{\mathbb{P}}   
\newcommand{\pp}{\mathsf{p}}
\newcommand{\td}{\boldsymbol{\delta}}  

\newcommand{\lp}[1]{l_{#1}(\mathbb{Z})}
\newcommand{\lrs}[3]{(l_{#1}(\mathbb{Z}))^{#2\times #3}}

\newcommand{\setsp}{\;:\;}     
\newcommand{\fs}{\operatorname{fsupp}}  
\newcommand{\supp}{\operatorname{supp}}

\newcommand{\tp}{\mathsf{T}}  
\newcommand{\bp}{ \begin{proof} }
	\newcommand{\ep}{\hfill  \end{proof} }
\newcommand{\be}{ \begin{equation} }
\newcommand{\ee}{ \end{equation} }
\newcommand{\imply}{ \Longrightarrow }

\newcommand{\wh}{\widehat}

\renewcommand{\le}{\leqslant}
\renewcommand{\ge}{\geqslant}
\newcommand{\bs}{\backslash}
\newcommand{\ol}{\overline}
\newcommand{\la}{\langle}
\newcommand{\ra}{\rangle}
\newcommand{\Lp}[1]{L_{#1}(\mathbb{R})}
\newcommand{\AS}{\operatorname{\mathsf{AS}}} 
\newcommand{\HH}[1]{{H^{#1}(\mathbb{R})}}          
\newcommand{\sr}{\operatorname{sr}}  
\newcommand{\sm}{\operatorname{sm}}  
\newcommand{\vmo}{\operatorname{vm}} 
\newcommand{\vgu}{\upsilon} 
\newcommand{\sd}{\mathcal{S}}  
\newcommand{\si}{\mathtt{S}}            
\newcommand{\diag}{\operatorname{diag}}
\newcommand{\gl}{\lambda}
\newcommand{\gep}{\varepsilon}
\newcommand{\eps}{\epsilon}

\newtheorem{lemma}{Lemma}
\newtheorem{prop}[lemma]{Proposition}
\newtheorem{cor}[lemma]{Corollary}
\newtheorem{theorem}[lemma]{Theorem}

\newtheorem{example}{Example}

\newtheorem{algorithm}{Algorithm}

\definecolor{red}{rgb}{1,0,0}
\definecolor{blue}{rgb}{0,0,1}
\definecolor{green}{rgb}{0.2,0.6,0.15}
\definecolor{bluegreen}{rgb}{0,0.5,0.5}
\definecolor{violet}{rgb}{0.5,0,0.5}
\definecolor{ZurichBlue}{rgb}{.35,.35,.73}

\numberwithin{equation}{section}
\numberwithin{lemma}{section}
\numberwithin{figure}{section}
\numberwithin{table}{section}
\numberwithin{example}{section}

\begin{document}

\title{Wavelets on Intervals Derived from Arbitrary Compactly Supported Biorthogonal Multiwavelets}

\author[rvt1]{Bin Han\fnref{fn1}\corref{cor}}
\ead{bhan@ualberta.ca, http://www.ualberta.ca/$\sim$bhan}
\address[rvt1]{Department of Mathematical and Statistical Sciences,
University of Alberta, Edmonton, Alberta T6G 2G1, Canada.}
\cortext[cor]{Corresponding author.}
\fntext[fn1]{Research supported in part by Natural Sciences and Engineering Research Council (NSERC) of Canada under grant RGPIN-2019-04276 and Alberta Innovates}

\author[rvt1]{Michelle Michelle\fnref{fn1}}
\ead{mmichell@ualberta.ca}

\makeatletter \@addtoreset{equation}{section} \makeatother

\begin{abstract}
Orthogonal and biorthogonal (multi)wavelets on the real line have been extensively studied and employed in applications with success.
On the other hand, a lot of problems in applications such as images and solutions of differential equations are defined on bounded intervals or domains. Therefore, it is important in both theory and application to construct all possible wavelets on intervals with some desired properties from (bi)orthogonal (multi)wavelets on the real line.
Then wavelets on rectangular domains such as $[0,1]^d$ can be obtained through tensor product.
Vanishing moments of compactly supported wavelets are the key property for sparse wavelet representations and are closely linked to polynomial reproduction of their underlying refinable (vector) functions.
Boundary wavelets with low order vanishing moments often lead to undesired boundary artifacts as well as reduced sparsity and approximation orders near boundaries
in applications. Scalar orthogonal wavelets
and spline biorthogonal  wavelets
on the interval $[0,1]$ have been extensively studied in the literature.
Though multiwavelets enjoy some desired properties over scalar wavelets such as high vanishing moments and relatively short support, except a few concrete examples, there is currently no systematic method for constructing (bi)orthogonal multiwavelets on bounded intervals.
In contrast to current literature on constructing particular wavelets on intervals from special (bi)orthogonal (multi)wavelets,
from any arbitrarily given compactly supported (bi)orthogonal multiwavelet on the real line,
in this paper we propose two different approaches to construct/derive all possible locally supported (bi)orthogonal (multi)wavelets on $[0,\infty)$ or $[0,1]$ with or without prescribed vanishing moments, polynomial reproduction, and/or homogeneous boundary conditions.
The first approach generalizes the classical approach from scalar wavelets to multiwavelets, while the second approach is direct without explicitly involving any dual refinable functions and dual multiwavelets.
We shall also address wavelets on intervals satisfying general homogeneous boundary conditions.
Though constructing orthogonal (multi)wavelets on intervals is much easier than their biorthogonal counterparts, we show that some boundary orthogonal wavelets cannot have any vanishing moments if these orthogonal (multi)wavelets on intervals satisfy the homogeneous Dirichlet boundary condition.
In comparison with the classical approach, our proposed direct approach makes the construction of all possible locally supported (multi)wavelets on intervals easy.
Seven examples of orthogonal and biorthogonal multiwavelets on the interval $[0,1]$ will be provided to illustrate our construction approaches and proposed algorithms.
\end{abstract}


\maketitle

\bigskip

\pagenumbering{arabic}

\section{Introduction and Motivations}
\label{sec:intro}

Compactly supported orthogonal and biorthogonal (multi)wavelets on the real line have been extensively studied and constructed in the literature, e.g., see \cite{cw92,cdf92,dau88} for scalar wavelets and \cite{dgh96,ghm94,gjx00,han01,hanbook,hz10,hkp06,keibook}
and references therein for multiwavelets.
Such wavelets are useful in many applications such as signal/image processing and numerical solutions of differential equations.
To explain our motivations for studying multiwavelets on intervals, let us first recall the definition of Riesz wavelets in
the square integrable function space $\Lp{2}$.
Let $\phi:=(\phi^1,\ldots,\phi^r)^\tp$ and $\psi:=(\psi^1,\ldots,\psi^s)^\tp$ be vectors of functions in $\Lp{2}$.
For $J\in \Z$, the multiwavelet affine system $\AS_J(\phi;\psi)$, generated by $\phi$ and $\psi$ through dyadic dilation and integer shifts, is defined to be
\be \label{as0}
\AS_J(\phi;\psi):=
\{\phi^\ell_{J;k} \setsp k\in \Z, \ell=1,\ldots,r\} \cup
\{\psi^\ell_{j;k} \setsp j\ge J, k\in \Z, \ell=1,\ldots,s\},
\ee
where
\[
\phi^\ell_{J;k}:=2^{J/2}\phi^\ell(2^J\cdot-k) \quad \mbox{and}\quad \psi^\ell_{j;k}:=2^{j/2}\psi^\ell(2^j\cdot-k).
\]
We say that $\AS_J(\phi;\psi)$ is \emph{a Riesz basis} of $\Lp{2}$ if
the linear span of $\AS_J(\phi;\psi)$ is dense in $\Lp{2}$ and
$\AS_J(\phi;\psi)$ is a Riesz sequence in
$\Lp{2}$, i.e., there exist positive constants $C_1$ and $C_2$ such that
\be \label{riesz}
C_1 \sum_{h\in \AS_J(\phi;\psi)} |c_h|^2
\le \Big\| \sum_{h\in \AS_J(\phi;\psi)} c_h h \Big\|^2_{\Lp{2}}
\le C_2 \sum_{h\in \AS_J(\phi;\psi)} |c_h|^2
\ee
for all finitely supported sequences $\{ c_h\}_{h\in \AS_J(\phi;\psi)}$.
By a simple scaling argument, it is easy to see (\cite{han12,hanbook}) that $\AS_J(\phi;\psi)$ is a Riesz basis of $\Lp{2}$ for some $J\in \Z$ if and only if  $\AS_J(\phi;\psi)$ is a Riesz basis of $\Lp{2}$ for all $J\in \Z$. Hence, we say that $\{\phi;\psi\}$ is \emph{a Riesz multiwavelet in $\Lp{2}$} if $\AS_0(\phi;\psi)$ is a Riesz basis of $\Lp{2}$.
If $\AS_0(\phi;\psi)$ is an orthonormal basis of $\Lp{2}$, then $\{\phi;\psi\}$ is called \emph{an orthogonal multiwavelet} in $\Lp{2}$.
Obviously, an orthogonal multiwavelet $\{\phi;\psi\}$ is a Riesz multiwavelet in $\Lp{2}$.
For the special case $r=1$, the vector function $\phi=(\phi^1,\ldots,\phi^r)^\tp$ becomes a scalar function and a Riesz (or orthogonal) multiwavelet $\{\phi;\psi\}$ with $r=1$ is often called a scalar Riesz (or orthogonal) wavelet.
For simplicity of presentation, we shall use wavelets to stand for both scalar wavelets and multiwavelets, which are often clear in the context.
Let $\tilde{\phi}:=(\tilde{\phi}^1,\ldots,\tilde{\phi}^r)^\tp$ and $\tilde{\psi}:=(\tilde{\psi}^1,\ldots,\tilde{\psi}^s)^\tp$ be vectors of functions in $\Lp{2}$. We say that $(\{\tilde{\phi};\tilde{\psi}\},\{\phi;\psi\})$ is \emph{a biorthogonal multiwavelet in $\Lp{2}$} if both $\{\tilde{\phi};\tilde{\psi}\}$ and $\{\phi;\psi\}$ are Riesz multiwavelets in $\Lp{2}$ such that $\AS_0(\tilde{\phi};\tilde{\psi})$ and $\AS_0(\phi;\psi)$ are biorthogonal to each other in $\Lp{2}$.
If $(\{\tilde{\phi};\tilde{\psi}\},\{\phi;\psi\})$ is a biorthogonal multiwavelet in $\Lp{2}$, then
$(\AS_J(\tilde{\phi};\tilde{\psi}),\AS_J(\phi;\psi))$ forms a pair of biorthogonal Riesz bases in $\Lp{2}$ for every $J\in \Z$ (\cite{han10,han12,hanbook}) and every function $f\in \Lp{2}$ has a wavelet representation:
{\small \be \label{expr}
	f= \sum_{\ell=1}^r \sum_{k\in \Z}  \la f, \tilde{\phi}^\ell_{J;k}\ra \phi^\ell_{J;k}+\sum_{j=J}^\infty \sum_{\ell=1}^s \sum_{k\in \Z} \la f, \tilde{\psi}^\ell_{j;k}\ra \psi^\ell_{j;k}
	=\sum_{\ell=1}^r \sum_{k\in \Z}  \la f, \phi^\ell_{J;k}\ra \tilde{\phi}^\ell_{J;k}+\sum_{j=J}^\infty \sum_{\ell=1}^s \sum_{k\in \Z} \la f, \psi^\ell_{j;k}\ra \tilde{\psi}^\ell_{j;k}
	\ee
}
with the above series converging unconditionally in $\Lp{2}$.  By \cite[Proposition~5]{han12}, \eqref{expr} implies
\[
f= \sum_{j=-\infty}^\infty \sum_{\ell=1}^s \sum_{k\in \Z} \la f, \tilde{\psi}^\ell_{j;k}\ra \psi^\ell_{j;k}
=\sum_{j=-\infty}^\infty \sum_{\ell=1}^s \sum_{k\in \Z} \la f, \psi^\ell_{j;k}\ra \tilde{\psi}^\ell_{j;k},\qquad f\in \Lp{2}
\]
with the above series converging unconditionally in $\Lp{2}$.
The sparse representation of a biorthogonal wavelet comes from the vanishing moments of the wavelet functions $\psi$ and $\tilde{\psi}$ (see \cite{cdf92,dau88}).
For a compactly supported (vector) function $\psi$, we say that $\psi$ has $m$ \emph{vanishing moments} if $\int_\R x^j \psi(x)dx=0$ for all $j=0,\ldots,m-1$. In particular, we define $\vmo(\psi):=m$ with $m$ being the largest such integer.

Numerous problems in applications such as signals/images and solutions of differential equations are defined on bounded intervals or domains. Therefore, it is important in both theory and application to construct all possible locally supported wavelets on an interval such as $[0,1]$ with desired properties from wavelets on the real line. Then wavelets on rectangular domains such as $[0,1]^d$ can be easily obtained through tensor product.
To construct wavelets on an interval,
in this paper we are particularly interested in compactly supported (bi)orthogonal multiwavelets $(\{\tilde{\phi};\tilde{\psi}\},\{\phi;\psi\})$ in $\Lp{2}$.
Since each interval on the real line $\R$ has at most two endpoints, for simplicity of discussion and presentation,
following \cite{cdv93},
it suffices to consider the one-sided interval $[0,\infty)$ with only one endpoint $0$.
For a nontrivial function $f\in L_2([0,\infty))$, it is obvious that $f(2^j\cdot)\in L_2([0,\infty))$ for all $j\in \Z$, i.e., the space $L_2([0,\infty))$ is invariant under dyadic dilation. However, $f(\cdot-k)\in L_2([0,\infty))$ cannot be true for all $k\in \Z$ if $f$ is nontrivial. That is, $L_2([0,\infty))$ is not invariant under integer shifts.
To maximally preserve the desired dilation and shift structure of  a wavelet system $\AS_J(\phi;\psi)$ in \eqref{as0} on the real line, for every element $\eta\in \AS_J(\phi;\psi)$, the popular and natural approach in the literature
for constructing a modified system $\AS_J(\Phi;\Psi)_{[0,\infty)}$ in $L_2([0,\infty))$ is to
\begin{enumerate}
	\item [1)] keep $\eta$ in $\AS_J(\Phi;\Psi)_{[0,\infty)}$ whenever possible if $\mbox{supp}(\eta)\subseteq [0,\infty)$ (such a kept element $\eta$ is often called an interior wavelet element), while drop $\eta$ if $\mbox{supp}(\eta)\subseteq (-\infty,0]$;
	\item[2)] otherwise, either drop $\eta$ or modify $\eta$ into a desired element in $L_2([0,\infty))$ (such $\eta$ after modification is often called a boundary wavelet element).
\end{enumerate}
In other words, a modified system $\AS_J(\Phi;\Psi)_{[0,\infty)}$ on $[0,\infty)$ adapted from $\AS_J(\phi;\psi)$ on $\R$ is given by
\be \label{ASPhiPsiI}
\AS_J(\Phi;\Psi)_{[0,\infty)}:=\{2^{J/2}\varphi(2^J\cdot) \setsp
\varphi \in \Phi\} \cup\{ 2^{j/2}\eta(2^j\cdot) \setsp j\ge J, \eta \in \Psi\},\qquad J\in \Z
\ee
with
\be \label{PhiPsiI}
\Phi:=\{\phi^L\}\cup \{\phi(\cdot-k) \setsp k\ge n_\phi\},\qquad  \Psi:=\{\psi^L\}\cup \{\psi(\cdot-k) \setsp k\ge n_\psi\},
\ee
where the boundary elements $\phi^L$ and $\psi^L$ are vectors/sets of compactly supported functions in $L_2([0,\infty))$ and
the integers $n_\phi,n_\psi$ are chosen so that all elements in $\{\phi(\cdot-k)\setsp k\ge n_\phi\}\cup \{\psi(\cdot-k) \setsp k\ge n_\psi\}$ are supported inside $[0,\infty)$ and hence are interior elements.
Generally speaking, the difficulty in constructing wavelets on $[0,\infty)$ lies in the above item 2) for constructing suitable boundary elements.
For simplicity of discussion, we remind the reader that a vector function is often used as an ordered set of functions and vice versa in many places including \eqref{PhiPsiI} throughout the paper.

Let $\{\phi;\psi\}$ be a Riesz multiwavelet in $\Lp{2}$ with $\phi=(\phi^1,\ldots,\phi^r)^\tp$ and $\psi=(\psi^1,\ldots,\psi^s)^\tp$.
Recall that a multiwavelet is called a scalar wavelet if $r=1$.
A family of compactly supported scalar orthogonal wavelets $\{\phi;\psi\}$ with increasing orders of vanishing moments has been constructed in Daubechies~\cite{dau88}.
Such Daubechies orthogonal wavelets $\{\phi;\psi\}$ on the real line $\R$ have been used in Meyer~\cite{mey91} and Cohen, Daubechies and Vial~\cite{cdv93} to construct orthogonal wavelets $\AS_J(\Phi;\Psi)_{[0,1]}$ on the unit interval $[0,1]$ such that $\vmo(\psi^L)=\vmo(\psi)$, that is, the boundary wavelet $\psi^L$ preserves the order of vanishing moments of $\psi$.
Also, see \cite{ahjp94,a93,cq04,jl93,mad97,pst95} and references therein for constructing scalar orthogonal wavelets on intervals from Daubechies orthogonal wavelets.

For $m\in \N$, the B-spline function $B_m$ of order $m$ is defined to be
\be \label{bspline}
B_1:=\chi_{(0,1]} \quad \mbox{and}\quad
B_m:=B_{m-1}*B_1=\int_0^1 B_{m-1}(\cdot-x) dx.
\ee
The B-spline $B_m$ possesses many desired properties such as
$B_m\in C^{m-2}(\R)$ with support $[0,m]$ and $B_m|_{(k,k+1)}$ is a nonnegative polynomial of degree $m-1$ for all $k\in \Z$. Having many desired properties, B-splines are extensively studied and used in approximation theory, wavelet analysis, and applied mathematics.
Employing some special properties of B-splines, \cite{dku99,mas96} with further
improvements in \cite{cer19,cf11,cf12,ds98,jia09,pri10}
have adapted
spline biorthogonal scalar wavelets $(\{\tilde{\phi};\tilde{\psi}\},\{\phi;\psi\})$ in \cite{cdf92} with $\phi$ being a $B$-spline function to the interval $[0,1]$.
The semi-orthogonal spline wavelets in \cite{cw92} are also adapted to the interval $[0,1]$ in \cite{cq04}.
Further developments on spline scalar wavelets on $[0,1]$ have been reported in \cite{jia09,cer19} and references therein.
But we are not aware of any published work for adapting a general non-spline biorthogonal scalar wavelets to the interval $[0,1]$.

On one hand, it is widely known that multiwavelets can  achieve higher orders of vanishing moments and better smoothness than scalar wavelets for a given prescribed support.
On the other hand, it is known (\cite{dau88}) that except the discontinuous Haar orthogonal wavelet and its trivial variants, a compactly supported real-valued scalar dyadic orthogonal wavelet cannot have symmetry.
However, real-valued dyadic orthogonal multiwavelets (\cite{ghm94})
are known to be able to achieve both symmetry and continuity. All these desired features make them particularly attractive for constructing wavelets on an interval.
However,
there are not much known constructed multiwavelets on intervals from (bi)orthogonal multiwavelets.
Though multiwavelets on intervals can be constructed from any compactly supported symmetric biorthogonal multiwavelets in \cite{ahl17,hanbook,hm18}, the constructed boundary wavelets often have low vanishing moments.
The study of multiwavelets is often much more involved and complicated than scalar wavelets (\cite{hanbook,keibook}).
As a consequence, it is not surprising that there are much fewer papers discussing how to derive wavelets on intervals from (bi)orthogonal multiwavelets on the real line. For some examples of multiwavelets on intervals, see \cite{ak12,ahl17,cer19,dhjk00,ds10,hanbook,hj02,hm18,hm99,kei15} and references therein.
In addition most known constructions concentrate on orthogonal multiwavelets and the constructed boundary wavelets $\psi^L$ often have lower orders of vanishing moments than $\psi$ (i.e., $\vmo(\psi^L)<\vmo(\psi)$).
Due to the importance of vanishing moments for reducing boundary artifacts and increased sparsity,
it is the purpose of this paper to systematically construct wavelets on an interval from any compactly supported (bi)orthogonal (multi)wavelets on the real line such that the boundary wavelets $\psi^L$ satisfy $\vmo(\psi^L)=\vmo(\psi)$. Our study reveals that it is much more difficult and complicated to construct biorthogonal multiwavelets than scalar orthogonal wavelets on intervals such that the boundary wavelets have high vanishing moments. This difficulty is largely due to

\begin{enumerate}
	\item[(1)] For biorthogonal wavelets, the primal wavelets often have short supports (which are highly desired in wavelet methods for numerical algorithms),
	while their dual wavelets have much longer supports. This makes
	construction of locally supported general biorthogonal wavelets on an interval much more involved and complicated than orthogonal wavelets on an interval.
	
	\item[(2)] The Riesz sequence property similar to \eqref{riesz} for an orthogonal wavelet on an interval
	trivially holds by taking $C_1=C_2=1$. But the establishment of the Riesz sequence property (in particular, the lower bound in \eqref{riesz}) of a biorthogonal wavelet on intervals is nontrivial. In this paper we shall provide a relatively simple proof in \cref{thm:phi:bessel} for this crucial property.
	
	\item[(3)] Though orthogonal (multi)wavelets on intervals can be easily constructed (see \cref{alg:Phi:orth}), by \cref{cor:ow:novm} and~\cref{thm:ow:vm}, their boundary wavelets cannot have high vanishing moments and satisfy prescribed homogeneous boundary conditions simultaneously. So, it is unavoidable to study general biorthogonal (multi)wavelets on intervals for some applications.
	
	\item[(4)] The study of refinable vector functions and biorthogonal multiwavelets with matrix-valued filters/masks is
	much more complicated than their scalar counterparts (e.g., see \cite{ak12,gjx00,han01,han03,hanbook,hj02,hz10,hm18,jrz99,keibook}).
	This makes the construction of orthogonal and biorthogonal multiwavelets on an interval from refinable vector functions more technical than their scalar counterparts.
	\item[(5)] Published works so far only concern about constructing \emph{one particular} (bi)orthogonal wavelet on intervals derived from \emph{some special} compactly supported (bi)orthogonal wavelet on the real line. Even for an orthogonal scalar wavelet $\{\phi;\psi\}$ on the real line, its derived orthogonal wavelets $\AS_0(\Phi;\Psi)_{[0,\infty)}$ on $[0,\infty)$ are not (essentially) unique and pathologically, all elements in $\Phi\cup\Psi$ could be supported inside $[N,\infty)$ for any $N\in \N$, see Example~\ref{expl:00} for details.
\end{enumerate}

From
any arbitrarily given compactly supported (bi)orthogonal multiwavelet on the real line, in this paper we shall propose two different systematic approaches with tractable algorithms for constructing
all possible locally supported
(bi)orthogonal (multi)wavelets on
$\I$ with or without prescribed vanishing moments, polynomial
reproduction, and/or homogeneous boundary conditions, where $\I=[0,\infty)$ or $\I=[0,N]$ with $N\in \N$.
This allows us to find suitable/optimal wavelets on intervals in applications.
The first approach generalizes the classical approach from scalar wavelets to multiwavelets by first constructing primal and dual refinable vector functions in $L_2(\I)$ and then deriving associated primal and dual wavelets in $L_2(\I)$.
Most known constructions of wavelets on intervals basically use this classical approach.
As we shall see in this paper,
it is necessarily much more complicated for the classical approach to construct biorthogonal multiwavelets than scalar orthogonal wavelets on intervals such that the boundary wavelets have high vanishing moments.
The second approach is by directly constructing the primal refinable vector functions and primal multiwavelets in $L_2(\I)$ without explicitly involving the dual refinable vector functions and dual multiwavelets.
In comparison with the classical approach, this direct approach is more general and
makes the construction of both scalar wavelets and multiwavelets on intervals easy.

To have a glimpse of the classical and direct approaches before jumping into the technical details and proofs, here we provide some outlines and road maps
of the classical and direct approaches
for constructing compactly supported biorthogonal wavelets on $[0,\infty)$ from an arbitrarily given compactly supported biorthogonal wavelet on the real line.
According to \cref{thm:bw} in \cref{sec:bwRplus},
a compactly supported biorthogonal wavelet
$(\{\tilde{\phi};\tilde{\psi}\},\{\phi;\psi\})$ in $\Lp{2}$ must satisfy
\[
\phi=2\sum_{k\in \Z} a(k)\phi(2\cdot-k),\;\;
\psi=2\sum_{k\in \Z} b(k)\phi(2\cdot-k),
\;\;\tilde{\phi}=2\sum_{k\in \Z}\tilde{a}(k)
\tilde{\phi}(2\cdot-k),\;\;
\tilde{\psi}=2\sum_{k\in \Z} \tilde{b}(k)
\tilde{\phi}(2\cdot-k),
\]
where $a,b,\tilde{a},\tilde{b}\in \lrs{0}{r}{r}$ and by $\lrs{0}{r}{r}$ we denote the space of all finitely supported sequences $u=\{u(k)\}_{k\in \Z}: \Z \rightarrow \C^{r\times r}$.
The above multiscale relations (called the refinable structures in this paper) are well known to play the key role for a fast multiwavelet transform.
By $\PL_{m-1}$ we denote the space of all polynomials of degree less than $m$. Define $m:=\vmo(\tilde{\psi})$ and $\tilde{m}:=\vmo(\psi)$ for vanishing moments. From such a given compactly supported biorthogonal wavelet $(\{\tilde{\phi};\tilde{\psi}\},\{\phi;\psi\})$ in $\Lp{2}$, we are interested in
deriving a compactly supported Riesz basis
$\AS_0(\Phi;\Psi)_{[0,\infty)}$
in $L_2([0,\infty))$ satisfying
\be \label{I:psi:0}
\phi^L=2A_L\phi^L(2\cdot)+2\sum_{k=n_\phi}^{\infty} A(k) \phi(2\cdot-k),\qquad
\psi^L=2B_L\phi^L(2\cdot)+2\sum_{k=n_\phi}^{\infty} B(k) \phi(2\cdot-k),
\ee
for some matrices $A_L,B_L$ and finitely supported sequences $A,B$,
where $\AS_0(\Phi;\Psi)_{[0,\infty)}$ is defined in \eqref{ASPhiPsiI} for $\Phi,\Psi$ in \eqref{PhiPsiI} with
compactly supported boundary vector functions $\phi^L,\psi^L$.
We shall prove in \cref{thm:wbd:0} that the unique dual Riesz basis of $\AS_0(\Phi;\Psi)_{[0,\infty)}$ must be given by $\AS_0(\tilde{\Phi};\tilde{\Psi})_{[0,\infty)}\subseteq L_2([0,\infty))$, defined similarly as in \eqref{PhiPsiI} with $\tilde{\Phi}=\{\tilde{\phi}^L\}\cup \{\tilde{\phi}(\cdot-k) \setsp k\ge n_{\tilde{\phi}}\}$ and $\tilde{\Psi}=\{\tilde{\psi}^L\}\cup \{\tilde{\psi}(\cdot-k) \setsp k\ge n_{\tilde{\psi}}\}$,
such that $\tilde{\phi}^L$ and $\tilde{\psi}^L$ must have \emph{compact support} and satisfy
\be \label{I:psi:dual:0}
\tilde{\phi}^L=2\tilde{A}_L\tilde{\phi}^L(2\cdot)+2\sum_{k=n_{\tilde{\phi}}}^{\infty} \tilde{A}(k) \tilde{\phi}(2\cdot-k),\qquad
\tilde{\psi}^L=2\tilde{B}_L\tilde{\phi}^L(2\cdot)+2\sum_{k=n_{\tilde{\phi}}}^{\infty} \tilde{B}(k) \tilde{\phi}(2\cdot-k),
\ee
for some matrices $\tilde{A}_L,\tilde{B}_L$ and finitely supported sequences $\tilde{A},\tilde{B}$.
I.e., $(\AS_0(\tilde{\Phi};\tilde{\Psi})_{[0,\infty)},
\AS_0(\Phi;\Psi)_{[0,\infty)})$ forms a compactly supported biorthogonal wavelet in $L_2([0,\infty))$ such that
all boundary elements $\phi^L, \psi^L$, $\tilde{\phi}^L,\tilde{\psi}^L$ have compact support and satisfy the refinable structures in \eqref{I:psi:0} and \eqref{I:psi:dual:0}.
As stated in \cref{thm:wbd}, the classical approach described in \cref{sec:classical} for constructing a compactly supported biorthogonal wavelet $(\AS_0(\tilde{\Phi};\tilde{\Psi})_{[0,\infty)},
\AS_0(\Phi;\Psi)_{[0,\infty)})$ in $L_2([0,\infty))$ has four major steps:
\begin{enumerate}
	\item[(S1)] Apply \cref{prop:phicut} and \Cref{subsec:Phi} to construct a compactly supported vector function $\phi^L$ in $L_2([0,\infty))$ satisfying the first identity in \eqref{I:psi:0}. To have polynomial reproduction, every $\pp \chi_{[0,\infty)}$ with $\pp\in \PL_{m-1}$ should be an infinite linear combination of elements in $\Phi$.
	
	\item[(S2)] Use \cref{alg:tPhi} to construct a compactly supported vector function $\tilde{\phi}^L$ in $L_2([0,\infty))$ such that the first identity in \eqref{I:psi:dual:0} holds and $\tilde{\Phi}$ is biorthogonal to $\Phi$. To have vanishing moments $\vmo(\psi^L)=\tilde{m}$ with $\tilde{m}:=\vmo(\psi)$, every $\pp \chi_{[0,\infty)}$ with $\pp\in \PL_{\tilde{m}-1}$ must necessarily be an infinite linear combination of elements in $\tilde{\Phi}$, see \cref{lem:vm} for details.
	
	\item[(S3)] Employ \cref{prop:Psi} to construct a compactly supported boundary primal wavelet $\psi^L$ such that the second identity in \eqref{I:psi:0} holds and $\Psi$ is perpendicular to $\tilde{\Phi}$.
	
	\item[(S4)] Employ \cref{prop:tPsi} to construct a compactly supported boundary dual wavelet $\tilde{\psi}^L$ such that the second identity in \eqref{I:psi:dual:0} holds, $\tilde{\Psi}$ is perpendicular to $\Phi$, and $\tilde{\Psi}$ is biorthogonal to $\Psi$.
\end{enumerate}

The classical approach for constructing special vector functions $\phi^L$ in (S1) is quite simple, because each entry of $\phi^L$ is either some $\phi(\cdot-k)\chi_{[0,\infty)}, k\in \Z$ or their linear combinations.
Once (S1) and (S2) are done,
based on two simple observations in  \cref{thm:rz} and \cref{lem:W}, $\psi^L$ in (S3) can be easily constructed by \cref{prop:Psi}. Though $\psi^L$ itself is not unique, the remark after \cref{prop:Psi} shows that the finite-dimensional space generated by $\psi^L$ modulated by the space spanned by $\{\psi(\cdot-k) \setsp k\ge n_\psi\}$ is uniquely determined by $\Phi$ and $\tilde{\Phi}$.
Once (S1)--(S3) are given, $\tilde{\psi}^L$ in (S4) can be easily constructed through  \cref{prop:tPsi} and both $\tilde{\psi}^L$ and $\tilde{\Psi}$ are uniquely determined by $\Phi,\tilde{\Phi}$ and $\Psi$. The Bessel property for the stability of $\AS_0(\tilde{\Phi};\tilde{\Psi})_{[0,\infty)}$ and $\AS_0(\Phi;\Psi)_{[0,\infty)}$
is guaranteed by \cref{thm:phi:bessel}.
To apply wavelet-based methods for numerically solving boundary value problems, all the elements in the Riesz basis $\AS_0(\Phi;\Psi)_{[0,\infty)}$ must satisfy prescribed homogeneous boundary conditions.
This can be easily done by applying \cref{prop:mod} to the constructed $\phi^L$ and $\Phi$ in (S1).

For a given orthogonal (multi)wavelet $\{\phi;\psi\}$ in $\Lp{2}$, because $\{\phi; \psi\}$ is biorthogonal to itself, (S2) can be avoided. Hence, adapting an orthogonal (multi)wavelet from the real line to $[0,\infty)$ becomes quite simple, because (S1) for constructing $\phi^L$ and (S3) for constructing $\psi^L$ are fairly easy, see \cref{alg:Phi:orth} for more details. However, by \cref{cor:ow:novm} and~\cref{thm:ow:vm}, their boundary wavelets $\psi^L$ cannot possess high vanishing moments and satisfy prescribed homogeneous boundary conditions simultaneously. By \cref{thm:ow:vm},
this also holds for nonstationary orthonormal wavelets on $[0,\infty)$.

The main complexity/difficulty for the classical approach is (S2) in \cref{alg:tPhi} for constructing vector functions $\tilde{\phi}^L$, whose entries are finite linear combinations of $\tilde{\phi}(2^j\cdot-k)\chi_{[0,\infty)}, k\in \Z$ with $j\in \{0,1\}$.
The complexity of (S2) is largely due to two facts: (1) The support of $\tilde{\phi}$ is often much longer than that of $\phi$. Therefore, there are many more elements $\tilde{\phi}(\cdot-k)$ essentially touch the endpoint $0$. (2) Because $\tilde{\Phi}$ is biorthogonal to $\Phi$, we must have $\#\tilde{\phi}^L=\#\phi^L+(n_{\tilde{\phi}}-n_\phi)(\#\phi)$ and consequently, we do not have any freedom about the length of $\tilde{\phi}^L$.
Therefore, it is no longer that easy or simple to construct even particular $\tilde{\phi}^L$ such that the first identity in \eqref{I:psi:dual:0} holds and $\tilde{\Phi}$ is biorthogonal to $\Phi$.

The difficulty in (S2) for the classical approach motivates us to propose a direct approach, which is more general but simpler than the classical approach.
The direct approach constructs $\phi^L$ and $\psi^L$ in \cref{thm:Phi:direct,thm:direct} without explicitly involving $\tilde{\phi}^L$ and $\tilde{\psi}^L$. Though the particularly constructed $\phi^L$ in (S1) by the classical approach
can be reused,
the direct approach in \cref{thm:Phi:direct} constructs all possible general vector functions $\phi^L$ in (S1) by directly employing the first identity in \eqref{I:psi:0} under the condition $\rho(A_L)<2^{-1/2}$, i.e., the spectral radius of $A_L$ is less than $2^{-1/2}$.
Without explicitly constructing $\tilde{\Phi}$ and $\tilde{\Psi}$,
inspired by \cref{lem:W}, the direct approach constructs $\psi^L$ in \cref{thm:direct} through the second identity in \eqref{I:psi:0} under
some necessary and sufficient conditions stated in \cref{thm:direct}. Now we can easily derive from \eqref{I:psi:0} matrices $\tilde{A}_L, \tilde{B}_L$ and finitely supported sequences $\tilde{A},\tilde{B}$ from \eqref{I:psi:0}. Then we only need to check the condition $\rho(\tilde{A}_L)<2^{-1/2}$ to obtain a compactly supported biorthogonal wavelet
$(\AS_0(\tilde{\Phi};\tilde{\Psi})_{[0,\infty)},
\AS_0(\Phi;\Psi)_{[0,\infty)})$ in $L_2([0,\infty))$, where $\tilde{\phi}^L$ and $\tilde{\psi}^L$ are defined in
\eqref{I:psi:dual:0}.
The proof of \cref{thm:direct} builds on \cref{thm:rz}, \cref{thm:phi:bessel} for stability, and convergence property of non-standard vector cascade algorithms (which are closely linked to nonstandard vector subdivision schemes). In addition, the direct approach can also improve the classical approach by constructing all possible general vector functions $\tilde{\phi}^L$ in (S2) through \cref{thm:direct:tPhi} by only requiring that $\rho(\tilde{A}_L)<2^{-1/2}$ and $A_L, \tilde{A}_L, A, \tilde{A}$ in
\eqref{I:psi:0} and \eqref{I:psi:dual:0}
should satisfy the identity in \eqref{filter:biorth}.

The procedure stated in \cref{thm:bw:0N} is well known (but without a proof) in the literature (e.g., see \cite{cdv93}) for adapting a compactly supported biorthogonal wavelet
$(\AS_0(\tilde{\Phi};\tilde{\Psi})_{[0,\infty)},
\AS_0(\Phi;\Psi)_{[0,\infty)})$ in $L_2([0,\infty))$ to a bounded interval $[0,N]$ with $N\in \N$. We shall provide a rigorous proof for \cref{thm:bw:0N} in this paper. The main idea of \cref{thm:bw:0N}
is quite simple: one constructs a closely related biorthogonal wavelet on the interval $(-\infty,N]$ whose interior elements are still given by $\psi_{j,k}:=2^{j/2}\psi(2^j\cdot-k)$ for some $j\in \NN$ and $k\in \Z$.
To obtain a locally supported biorthogonal wavelet on $[0,N]$, these two biorthogonal wavelets on $[0,\infty)$ and $(-\infty,N]$ are fused together in a straightforward way such that their boundary elements and all the common interior elements are kept.
The main steps in \cref{thm:bw:0N} for obtaining a biorthogonal wavelet on $(-\infty,N]$ are as follows. (1) Flip functions about the origin, that is, we define
\[
\mathring{\phi}:=\phi(-\cdot),\quad \mathring{\psi}:=\psi(-\cdot),\quad
\tilde{\mathring{\phi}}:=\tilde{\phi}(-\cdot) \quad \mbox{and}\quad \tilde{\mathring{\psi}}:=\tilde{\psi}(-\cdot).
\]
(2) Construct a biorthogonal wavelet
$(\AS_0(\tilde{\mathring{\Phi}};\tilde{\mathring{\Psi}})_{[0,\infty)},
\AS_0(\mathring{\Phi};\mathring{\Psi})_{[0,\infty)})$ in $L_2([0,\infty))$ from the flipped biorthogonal wavelet $(\{\tilde{\mathring{\phi}};\tilde{\mathring{\psi}}\},
\{\mathring{\phi};\mathring{\psi}\})$ in $\Lp{2}$. (3) Then $\{\tilde{h}(N-\cdot)\}_{ \tilde{h}\in \AS_0(\tilde{\mathring{\Phi}};\tilde{\mathring{\Psi}})_{[0,\infty)}}$ and $\{ h(N-\cdot)\}_{h\in \AS_0(\mathring{\Phi};\mathring{\Psi})_{[0,\infty)}}$ form a biorthogonal wavelet in $L_2((-\infty,N])$.
If all the vector functions in $\phi,\psi,\tilde{\phi},\tilde{\psi}$ possess symmetry, then a biorthogonal wavelet on $(-\infty,N]$ can be directly obtained from
the constructed biorthogonal wavelet
$(\AS_0(\tilde{\Phi};\tilde{\Psi})_{[0,\infty)},
\AS_0(\Phi;\Psi)_{[0,\infty)})$ in $L_2([0,\infty))$, see the remark after \cref{thm:bw:0N} for details.

The structure of the paper is as follows. In \cref{sec:bwRplus}, we shall study some basic properties of wavelets on the interval $[0,\infty)$ such as their Bessel properties and vanishing moments.
In \cref{sec:classical}, we shall generalize the classical approach from scalar wavelets to multiwavelets for constructing compactly supported biorthogonal wavelets on the interval $[0,\infty)$.
We shall also discuss in \cref{alg:Phi:orth} the construction of orthogonal (multi)wavelets on $[0,\infty)$ in \cref{sec:classical}.
In \cref{sec:direct}, we shall present the direct approach for constructing all possible compactly supported biorthogonal wavelets on $[0,\infty)$ from any given compactly supported biorthogonal (multi)wavelets on the real line. We shall also discuss how to further improve the classical approach by the direct approach.
In \cref{sec:bc}, we shall address stationary and nonstationary (multi)wavelets on $[0,\infty)$ satisfying any prescribed general homogeneous boundary conditions including Robin boundary conditions.
In \cref{sec:wbd}, we shall discuss how to construct wavelets on the interval $[0,N]$ with $N\in \N$ from wavelets on $[0,\infty)$.
Using the classical approach and the direct approach,
we shall present in \cref{sec:expl} a few examples of orthogonal and biorthogonal wavelets on the interval $[0,1]$ such that the boundary wavelets have high vanishing moments and prescribed homogeneous boundary conditions.
For improved readability, some technical proofs are postponed to \cref{sec:proof}.

\section{Properties of Biorthogonal Wavelets on the Interval $[0,\infty)$}
\label{sec:bwRplus}

In this section, we shall first recall some results on biorthogonal (multi)wavelets on the real line. Then we shall study some properties of biorthogonal wavelets on the interval $[0,\infty)$ which are derived from a compactly supported biorthogonal wavelet on the real line $\R$. Throughout the paper, for simplicity, wavelets stand for both scalar wavelets and multiwavelets.

\subsection{Biorthogonal wavelets on the real line}

To recall some results on biorthogonal wavelets on the real line,
let us first recall some definitions.
The Fourier transform used in this paper is defined to be $\wh{f}(\xi):=\int_\R f(x) e^{-ix\xi} d x, \xi\in \R$ for $f\in \Lp{1}$ and is naturally extended to square integrable functions in $\Lp{2}$.
By $\lrs{0}{r}{s}$ we denote the set of all finitely supported sequences $u=\{u(k)\}_{k\in \Z}: \Z \rightarrow \C^{r\times s}$. For $u=\{u(k)\}_{k\in \Z}\in \lrs{0}{r}{s}$, its Fourier series is defined to be
\[
\wh{u}(\xi):=\sum_{k\in \Z} u(k) e^{-ik\xi}\qquad \mbox{for}\quad \xi\in \R,
\]
which is an $r\times s$ matrix of $2\pi$-periodic trigonometric polynomials. An element in $\lrs{0}{r}{s}$ is often called a (matrix-valued) mask or filter in the literature.
By $\td$ we denote the Dirac sequence
such that $\td(0)=1$ and $\td(k)=0$ for all $k\in \Z\bs\{0\}$. Note that $\wh{\td}=1$.
By $f\in (\Lp{2})^{r\times s}$ we mean that $f$ is an $r\times s$ matrix of functions in $\Lp{2}$ and we define
\[
\la f,g\ra:=\int_{\R} f(x) \ol{g(x)}^\tp dx,\qquad
f\in (\Lp{2})^{r\times t}, g\in (\Lp{2})^{s\times t}.
\]

According to \cite[Theorem~4.5.1]{hanbook} and \cite[Theorem~7]{han12},
any biorthogonal wavelet $(\{\tilde{\phi};\tilde{\psi}\},\{\phi;\psi\})$ in $\Lp{2}$ must be intrinsically derived from refinable vector functions and biorthogonal wavelet filter banks. For simplicity, we only state the following result for compactly supported biorthogonal wavelets $(\{\tilde{\phi};\tilde{\psi}\},\{\phi;\psi\})$ in $\Lp{2}$.

\begin{theorem}\label{thm:bw} (\cite[Theorem~4.5.1]{hanbook} and \cite[Theorem~7]{han12})
	Let $\phi,\tilde{\phi}$ be $r\times 1$ vectors of compactly supported distributions and $\psi,\tilde{\psi}$ be $s\times 1$ vectors of compactly supported distributions on $\R$. Then $(\{\tilde{\phi};\tilde{\psi}\},\{\phi;\psi\})$ is a biorthogonal wavelet in $\Lp{2}$ if and only if the following are satisfied
	\begin{enumerate}
		\item[(1)] $\phi,\tilde{\phi}\in (\Lp{2})^r$ and $\ol{\wh{\phi}(0)}^\tp \wh{\tilde{\phi}}(0)=1$.
		\item[(2)] $\phi$ and $\tilde{\phi}$ are biorthogonal to each other: $\la \phi,\tilde{\phi}(\cdot-k)\ra= \td(k) I_r$ for all $k\in \Z$.
		\item[(3)] There exist low-pass filters $a,\tilde{a}\in \lrs{0}{r}{r}$ and high-pass filters
		$b,\tilde{b}\in \lrs{0}{s}{r}$ such that 
		\begin{align}
		&\phi=2\sum_{k\in \Z} a(k)\phi(2\cdot-k),\qquad
		\psi=2\sum_{k\in \Z} b(k)\phi(2\cdot-k), \label{refstr}\\
		&\tilde{\phi}=2\sum_{k\in \Z}\tilde{a}(k)
		\tilde{\phi}(2\cdot-k),\qquad
		\tilde{\psi}=2\sum_{k\in \Z} \tilde{b}(k)
		 \tilde{\phi}(2\cdot-k),\label{refstr:dual}
		\end{align}
		and $(\{\tilde{a};\tilde{b}\},\{a;b\})$ is a biorthogonal wavelet filter bank, i.e., $s=r$ and
		\be \label{bwfb}
		\left [ \begin{matrix}
			\wh{\tilde{a}}(\xi) &\wh{\tilde{a}}(\xi+\pi)\\
			\wh{\tilde{b}}(\xi) &\wh{\tilde{b}}(\xi+\pi)
		\end{matrix}\right]
		\left[ \begin{matrix}
			\ol{\wh{a}(\xi)}^\tp &\ol{\wh{b}(\xi)}^\tp\\
			\ol{\wh{a}(\xi+\pi)}^\tp &\ol{\wh{b}(\xi+\pi)}^\tp
		\end{matrix}\right]
		=I_{2r}, \qquad \xi\in \R.
		\ee
		\item[(4)] Both $\AS_0(\phi;\psi)$ and $\AS_0(\tilde{\phi};\tilde{\psi})$ are Bessel sequences in $\Lp{2}$, that is, there exists a positive constant $C$ such that
		\[
		\sum_{h\in \AS_0(\phi;\psi)} |\la f, h\ra|^2\le C\|f\|^2_{\Lp{2}}
		\quad \mbox{and}\quad
		\sum_{\tilde{h}\in \AS_0(\tilde{\phi};\tilde{\psi})} |\la f, \tilde{h}\ra|^2\le C\|f\|^2_{\Lp{2}},\qquad \forall\, f\in \Lp{2}.
		\]
	\end{enumerate}
\end{theorem}

A vector function $\phi$ satisfying \eqref{refstr} is called \emph{a refinable vector function} with a refinement filter/mask $a\in \lrs{0}{r}{r}$.
For a vector function $\phi$ we also regard $\phi$ as an ordered set and vice versa. We define $\#\phi$ to be the number of entries in $\phi$, that is, the cardinality of the set/vector $\phi$. For $r=1$, a refinable vector function is often called a (scalar) refinable function. By \cite[Theorems~4.6.5 and~6.4.6]{hanbook} or \cite[Theorem~2.3]{han03}, item (4) of \cref{thm:bw} can be replaced by
\begin{enumerate}
	\item[(4')] both $\psi$ and $\tilde{\psi}$ have at least one vanishing moment, i.e., $\int_{\R} \psi(x)dx=\int_{\R}\tilde{\psi}(x) dx=0$.
\end{enumerate}
Orthogonal and biorthogonal wavelets on the real line which are derived from refinable functions have been extensively studied, for example, see \cite{cw92,cdf92,dau88} for scalar wavelets and \cite{cdp97,ghm94,gjx00,han01,hanbook,hj02,hm07,hz10,jj03,jrz99,jiang99,keibook} and references therein for multiwavelets.
It is well known in these papers that the study and construction of multiwavelets and refinable vector functions are often much more involved and complicated than their scalar counterparts, largely because the refinable vector function $\phi$ and multiwavelet $\psi$ in \eqref{refstr} are vector functions with matrix-valued filters $a$ and $b$.

\subsection{The dual of a Riesz basis $\AS_0(\Phi;\Psi)_{[0,\infty)}$ on $[0,\infty)$}\label{subsec:dual}

To improve readability and to reduce confusion later about some notations,
in the following let us first state our conventions on some notations.
If not explicitly stated, $(\{\tilde{\phi};\tilde{\psi}\},\{\phi;\psi\})$ in this paper is always a compactly supported biorthogonal (multi)wavelet in $\Lp{2}$ satisfying all items (1)--(4) of \cref{thm:bw}, which necessarily implies $\#\phi=\#\psi$.
For a compactly supported (vector) function $\phi$ (or a finitely supported filter $a\in \lrs{0}{r}{s}$), we define $\fs(\phi)$ (or $\fs(a)$) to be the shortest interval with \emph{integer endpoints} such that $\phi$ (or $a$) vanishes outside $\fs(\phi)$ (or $\fs(a)$).
Throughout the paper we always define
\begin{align}
&[l_\phi, h_\phi]:=\fs(\phi),\quad
[l_\psi, h_\psi]:=\fs(\psi),\quad
[l_a, h_a]:=\fs(a),\quad
[l_b, h_b]:=\fs(b),\label{fs:phi:a}\\
&[l_{\tilde{\phi}}, h_{\tilde{\phi}}]:=\fs(\tilde{\phi}),\quad [l_{\tilde{\psi}}, h_{\tilde{\psi}}]:=\fs(\tilde{\psi}),\quad
[l_{\tilde{a}}, h_{\tilde{a}}]:=\fs(\tilde{a}),\quad [l_{\tilde{b}}, h_{\tilde{b}}]:=\fs(\tilde{b}).\label{fs:tphi:ta}
\end{align}
One can easily deduce from \eqref{refstr} (called the refinable structure in this paper)  that
\be \label{fs:relation}
[l_\phi,h_\phi]\subseteq[l_a,h_a]
\quad \mbox{and}\quad
[l_\psi,h_\psi]\subseteq [\lfloor \tfrac{l_b+l_\phi}{2}\rfloor,\lceil
\tfrac{h_b+h_\phi}{2}\rceil].
\ee
For the scalar case $r=1$, both $\subseteq$ in \eqref{fs:relation} become identities. But strict $\subseteq$ in \eqref{fs:relation} can happen for the multiwavelet case $r>1$, e.g., see \cref{ex:hardin}.
Note that $\fs(\phi(\cdot-k))=[k+l_\phi, k+h_\phi]$. Hence,
$\supp(\phi(\cdot-k))\subseteq (-\infty,0]$ for all integers $k\le-h_\phi$ and $\supp(\phi(\cdot-k))\subseteq [0,\infty)$ for all integers $k\ge -l_\phi$.
In other words, the point $0$ is an interior point of $\fs(\phi(\cdot-k))$ if and only if $1-h_\phi\le k\le -1-l_\phi$. On the other hand, we deduce from the refinable structure in \eqref{refstr} that
\begin{align}
&\phi(\cdot-k_0)=2\sum_{k=l_a+2k_0}^{h_a+2k_0}
a(k-2k_0)\phi(2\cdot-k),
\qquad k_0\in \Z, \label{refstr:k0}\\
&\psi(\cdot-k_0)=2\sum_{k=l_b+2k_0}^{h_b+2k_0}
b(k-2k_0)\phi(2\cdot-k),
\qquad k_0\in \Z. \label{refstr:k0:psi}
\end{align}
For any integer $n_\phi$ satisfying $n_\phi\ge\max(-l_\phi,-l_a)$, we have $l_a+2k_0\ge l_a+2n_\phi \ge n_\phi$ for all $k_0\ge n_\phi$ and consequently, we deduce from \eqref{refstr:k0} that
\be \label{nphi}
\fs(\phi(\cdot-k_0))\subseteq [0,\infty)
\quad \mbox{and}\quad
\phi(\cdot-k_0)=2\sum_{k=n_\phi}^\infty a(k-2k_0) \phi(2\cdot-k),\qquad \forall\; k_0\ge n_\phi.
\ee
Similarly, for any integer $n_\psi$ satisfying $n_\psi\ge \max(-l_\psi, \frac{n_\phi-l_b}{2})$, we have $l_b+2k_0\ge l_b+2n_\psi\ge n_\phi$ for all $k_0\ge n_\psi$ and hence we deduce from \eqref{refstr:k0:psi} that
\be \label{npsi}
\fs(\psi(\cdot-k_0))\subseteq [0,\infty)
\quad \mbox{and}\quad
\psi(\cdot-k_0)=2\sum_{k=n_\phi}^\infty b(k-2k_0) \phi(2\cdot-k),\qquad \forall\; k_0\ge n_\psi.
\ee
Throughout the paper, the integers $n_\phi$ and $n_\psi$ are always chosen (not necessary to be the smallest) such that \eqref{nphi} and \eqref{npsi} hold. We make the same convention for $n_{\tilde{\phi}}$ and $n_{\tilde{\psi}}$ similarly.
Let $\tilde{\phi}^L$ and $\tilde{\psi}^L$ be vector functions in $L_2([0,\infty))$. Similarly to $\Phi$ and $\Psi$ in \eqref{PhiPsiI}, we define
\be \label{tPhiPsiI}
\tilde{\Phi}:=\{ \tilde{\phi}^L\}\cup\{ \tilde{\phi}(\cdot-k) \setsp k\ge n_{\tilde{\phi}}\},\qquad
\tilde{\Psi}:=\{ \tilde{\psi}^L\}\cup\{ \tilde{\psi}(\cdot-k) \setsp k\ge n_{\tilde{\psi}}\}.
\ee
Under the following conditions for matching cardinality between $\Phi\cup \Psi$ and $\tilde{\Phi}\cup \tilde{\Psi}$:
\be\label{cardinality}
\#\tilde{\phi}^L-
\#\phi^L=(n_{\tilde{\phi}}-n_\phi)(\#\phi)
\quad \mbox{and}\quad
\#\tilde{\psi}^L-\#\psi^L=(n_{\tilde{\psi}}-n_\psi)(\#\psi),
\ee
throughout the paper, the mapping $\sim: \Phi\rightarrow \tilde{\Phi}$ with $h\mapsto \tilde{h}$ is always the default bijection between $\Phi$ and $\tilde{\Phi}$ such that $\phi(\cdot-k)$ corresponds to $\tilde{\phi}(\cdot-k)$ for all $k\ge \max(n_\phi,n_{\tilde{\phi}})$, and the bijection $\sim$ for other elements is determined by their corresponding positions in the ordered sets/vectors $\Phi$ and $\tilde{\Phi}$.
The bijection $\sim: \Psi\rightarrow \tilde{\Psi}$ is defined similarly by mapping $\psi(\cdot-k)$ to $\tilde{\psi}(\cdot-k)$ for all $k\ge \max(n_\psi,n_{\tilde{\psi}})$.


As we explained in \cref{sec:intro}, the pair $(\AS_J(\tilde{\phi};\tilde{\psi}),\AS_J(\phi;\psi))$ on the real line will be modified into a pair $(\AS_J(\tilde{\Phi};\tilde{\Psi})_{[0,\infty)},
\AS_J(\Phi;\Psi)_{[0,\infty)})$
of biorthogonal systems in $L_2([0,\infty))$ by keeping their elements supported inside $[0,\infty)$ as interior elements and by modifying their elements essentially touching the endpoint $0$ into boundary elements.
By the definition in \eqref{ASPhiPsiI} and a simple scaling argument as in \cite{han12}, it is straightforward to see that $\AS_J(\Phi;\Psi)_{[0,\infty)}$ is a Riesz (or orthonormal) basis of $L_2([0,\infty))$ for all $J\in \Z$ if and only if $\AS_0(\Phi;\Psi)_{[0,\infty)}$ is a Riesz (or orthonormal) basis of $L_2([0,\infty))$.

We now study the structure of compactly supported Riesz wavelets on $[0,\infty)$ in
the following result, whose proof is presented in \cref{sec:proof} and which plays a key role in our study of locally supported biorthogonal wavelets on intervals.

\begin{theorem}\label{thm:wbd:0}
	Let $(\{\tilde{\phi};\tilde{\psi}\},\{\phi;\psi\})$ be a compactly supported biorthogonal wavelet in $\Lp{2}$ with 
	a biorthogonal wavelet filter bank $(\{\tilde{a}\;\tilde{b}\},\{a;b\})$ satisfying items (1)--(4) of \cref{thm:bw}.
	Let $\phi^L$ and $\psi^L$ be vectors of compactly supported functions in $L_2([0,\infty))$. Define $l_\phi,h_\phi,l_a,h_a$ and $l_{\tilde{\phi}},h_{\tilde{\phi}},l_{\tilde{a}}, h_{\tilde{a}}$ as in \eqref{fs:phi:a} and \eqref{fs:tphi:ta}.
	Define $\Phi, \Psi$ as in \eqref{PhiPsiI} with integers $n_\phi\ge \max(-l_\phi,-l_a)$ and $n_\psi\ge \max(-l_\psi, \frac{n_\phi-l_b}{2})$.
	If $\AS_0(\Phi;\Psi)_{[0,\infty)}$ in \eqref{ASPhiPsiI} is a Riesz basis of $L_2([0,\infty))$
	and satisfies
	\begin{align}
	 &\phi^L=2A_L\phi^L(2\cdot)+2\sum_{k=n_\phi}^{\infty} A(k) \phi(2\cdot-k),
	\label{I:phi}\\
	 &\psi^L=2B_L\phi^L(2\cdot)+2\sum_{k=n_\phi}^{\infty} B(k) \phi(2\cdot-k),
	\label{I:psi}
	\end{align}
	for some matrices $A_L,B_L$ and finitely supported sequences $A,B$ of matrices,
	then
	\begin{enumerate}
		\item[(1)] there must exist compactly supported vector functions $\tilde{\phi}^L, \tilde{\psi}^L$ in $L_2([0,\infty))$ and integers $n_{\tilde{\phi}}\ge \max(-l_{\tilde{\phi}},-l_{\tilde{a}},n_\phi)$ and $n_{\tilde{\psi}}\ge \max(-l_{\tilde{\psi}},\frac{n_{\tilde{\phi}}-l_{\tilde{b}}}{2},n_\psi)$ satisfying \eqref{cardinality} such that $\AS_0(\tilde{\Phi};\tilde{\Psi})_{[0,\infty)}$ is the dual Riesz basis of $\AS_0(\Phi;\Psi)_{[0,\infty)}$ in $L_2([0,\infty))$,
		where
\[
\AS_0(\tilde{\Phi};\tilde{\Psi})_{[0,\infty)}:=\tilde{\Phi} \cup \{ 2^{j/2} \tilde{\eta}(2^j\cdot) \setsp j\in \N\cup\{0\}, \tilde{\eta}\in \tilde{\Psi}\}
\]
and $\tilde{\Phi}, \tilde{\Psi}$ are defined in \eqref{tPhiPsiI};
		
		\item[(2)] there exist matrices $\tilde{A}_L,\tilde{B}_L$ and finitely supported sequences $\tilde{A},\tilde{B}$ of matrices such that
		\begin{align}
		 &\tilde{\phi}^L=2\tilde{A}_L\tilde{\phi}^L(2\cdot)+2\sum_{k=n_{\tilde{\phi}}}^{\infty} \tilde{A}(k) \tilde{\phi}(2\cdot-k),
		\label{I:phi:dual}\\
		 &\tilde{\psi}^L=2\tilde{B}_L\tilde{\phi}^L(2\cdot)+2\sum_{k=n_{\tilde{\phi}}}^{\infty} \tilde{B}(k) \tilde{\phi}(2\cdot-k),
		\label{I:psi:dual}
		\end{align}
		and
		\begin{align}
		 &\fs(\tilde{\phi}(\cdot-k_0))\subseteq[0,\infty),\quad
		 \tilde{\phi}(\cdot-k_0)=2\sum_{k=n_{\tilde{\phi}}}^\infty
		\tilde{a}(k-2k_0) \tilde{\phi}(2\cdot-k),\qquad \forall\; k_0\ge n_{\tilde{\phi}}, \label{I:tphi}\\
		 &\fs(\tilde{\psi}(\cdot-k_0))\subseteq[0,\infty),\quad
		 \tilde{\psi}(\cdot-k_0)=2\sum_{k=n_{\tilde{\phi}}}^\infty
		\tilde{b}(k-2k_0) \tilde{\phi}(2\cdot-k),\qquad \forall\; k_0\ge n_{\tilde{\psi}}; \label{I:tpsi}
		\end{align}
		\item[(3)] Every element in $\Phi(2\cdot):=\{\phi^L(2\cdot)\}\cup\{
		\phi(2\cdot-k) \setsp k\ge n_\phi\}$ can be uniquely written as a finite linear combination of elements in $\Phi\cup\Psi$.
	\end{enumerate}
\end{theorem}

Basically, \cref{thm:wbd:0} says that a compactly supported Riesz basis $\AS_0(\Phi;\Psi)_{[0,\infty)}$ of $L_2([0,\infty))$ satisfying \eqref{I:phi} and \eqref{I:psi} must have a dual \emph{compactly supported} Riesz basis $\AS_0(\tilde{\Phi};\tilde{\Psi})_{[0,\infty)}$ of $L_2([0,\infty))$ satisfying \eqref{I:phi:dual} and \eqref{I:psi:dual}.
\cref{thm:wbd:0} serves as our foundation for developing the classical approach through items (1) and (2) of \cref{thm:wbd:0} and the direct approach through item (3) of \cref{thm:wbd:0} for deriving wavelets on intervals from biorthogonal multiwavelets in $\Lp{2}$.

\subsection{Vanishing moments of biorthogonal wavelets on $[0,\infty)$}

Recall that $\psi$ has $m$ vanishing moments
if $\int_{\R} x^j \psi(x) dx=0$ for all $j=0,\ldots,m-1$. In particular, we define $\vmo(\psi):=m$ with $m$ being the largest such nonnegative integer.
By $\PL_{m-1}$ we denote the space of all polynomials of degree less than $m$.
Define $\NN:=\N\cup\{0\}$.
Let us now discuss the known relation between vanishing moments and polynomial reproduction for biorthogonal wavelets on the interval $[0,\infty)$.

\begin{lemma}\label{lem:vm}
	Let $\phi,\psi,\tilde{\phi},\tilde{\psi}$ be vectors of compactly supported functions in $\Lp{2}$.
	Let $\phi^L, \psi^L, \tilde{\phi}^L,\tilde{\psi}^L$ be vectors of compactly supported functions in $L_2([0,\infty))$.
	Suppose that $\AS_0(\tilde{\Phi};\tilde{\Psi})_{[0,\infty)}$ and $\AS_0(\Phi;\Psi)_{[0,\infty)}$ form a pair of biorthogonal Riesz bases in $L_2([0,\infty))$, where $\Phi, \Psi$ and $\tilde{\Phi},
	\tilde{\Psi}$ are defined in \eqref{PhiPsiI} and \eqref{tPhiPsiI}, respectively.
	Then $\min(\vmo(\psi^L),\vmo(\psi))\ge m$ if and only if every polynomial $\pp \chi_{[0,\infty)}$ on $[0,\infty)$ with $\pp\in \PL_{m-1}$
	can be written as an infinite linear combination of elements in $\tilde{\Phi}$.
\end{lemma}

\bp  Necessity. Suppose that $\min(\vmo(\psi^L),\vmo(\psi))\ge m$.
Since all functions in $\phi\cup\psi\cup \tilde{\phi}\cup\tilde{\psi}$ and $\phi^L\cup\psi^L\cup \tilde{\phi}^L\cup\tilde{\psi}^L$ have compact support, we assume that
they are supported inside $[-N_0,N_0]$ for some $N_0\in \N$.
For every polynomial $\pp\in \PL_{m-1}$ and $N\in \N$, we have $\pp\chi_{[0,N)}\in \Lp{2}$ and hence
\[
\pp\chi_{[0,N)}=\sum_{h\in \Phi} \la \pp\chi_{[0,N)}, h\ra \tilde{h}+\sum_{j=0}^\infty \sum_{h\in \Psi}\la \pp\chi_{[0,N)}, 2^{j/2} h(2^j\cdot)\ra 2^{j/2} \tilde{h}(2^j\cdot).
\]
Note that $\tilde{\psi}_{j;k}$ and $\psi_{j;k}$ are supported inside $[2^{-j}(k-N_0),2^{-j}(k+N_0)]$.
Since $\la \pp, h(2^j\cdot)\ra=0$ for all $h\in \Psi$ and $j\in \NN$,
we observe that
$\la \pp\chi_{[0,N)}, h(2^j\cdot)\ra \tilde{h}(2^j x)=0$ a.e. $x\in [0,N-2N_0)$ for all $h\in \Psi$ and $j\in \NN$. Therefore, we conclude from the above identity that
\[
\pp(x)\chi_{[0,\infty)}(x)=\pp(x)\chi_{[0,N)}(x)=\sum_{h\in \Phi} \la \pp\chi_{[0,N)}, h\ra \tilde{h}(x)=\sum_{h\in \Phi} \la \pp, h\ra \tilde{h}(x),\qquad a.e.\, x\in [0,N-2N_0).
\]
Taking $N\to \infty$ in the above identity, we conclude from the above identity that $\pp\chi_{[0,\infty)}$ with $\pp\in \PL_{m-1}$ can be written as an infinite linear combination of elements in $\tilde{\Phi}$.

Sufficiency. Suppose that
$\pp\chi_{[0,\infty)}$ with $\pp\in
\PL_{m-1}$ can be written as $\pp \chi_{[0,\infty)}=\sum_{h\in \Phi} c_{h}\tilde{h}$.
Since every element in $\Psi$ is perpendicular to $\tilde{\Phi}$ and all elements in $\Psi\cup\tilde{\Phi}$ have compact support, we have $\la \pp\chi_{[0,\infty)}, g\ra=
\sum_{h\in \Phi} c_{h} \la \tilde{h}, g\ra= 0$ for all $g\in \Psi$.
This proves $\min(\vmo(\psi^L), \vmo(\psi))=\vmo(\Psi)\ge m$.
\ep

The above same argument in \cref{lem:vm} can be also applied to biorthogonal wavelets on the real line $\R$ or intervals $[0,N]$ with $N\in \N$.
For a biorthogonal wavelet $(\{\tilde{\phi};\tilde{\psi}\},\{\phi;\psi\})$ in $\Lp{2}$,
that every polynomial in $\PL_{m-1}$ can be written as an infinite linear combination of $\phi(\cdot-k), k\in \Z$ if and only if $\vmo(\tilde{\psi})\ge m$.
We say that a (matrix-valued) filter $a\in \lrs{0}{r}{r}$ has \emph{order $m$ sum rules with a (moment) matching filter $\vgu\in \lrs{0}{1}{r}$} if $\wh{\vgu}(0)\wh{\phi}(0)=1$ and
\be \label{sr}
[\wh{\vgu}(2\cdot)\wh{a}]^{(j)}(0)=\wh{\vgu}^{(j)}(0)
\quad \mbox{and}\quad [\wh{\vgu}(2\cdot)\wh{a}(\cdot+\pi)]^{(j)}(0)=0,\qquad \forall\; j=0,\ldots,m-1.
\ee
In particular, we define $\sr(a)=m$ with $m$ being the largest such nonnegative integer.
Let $(\{\tilde{\phi};\tilde{\psi}\},\{\phi;\psi\})$ be
a compactly supported biorthogonal wavelet in $\Lp{2}$
with a finitely supported biorthogonal wavelet filter bank $(\{\tilde{a};\tilde{b}\},\{a;b\})$ in \cref{thm:bw}. Then
$\vmo(\tilde{\psi})\ge m$ if and only if $\sr(a)\ge m$. That is, $\vmo(\tilde{\psi})=\sr(a)$ and $\vmo(\psi)=\sr(\tilde{a})$.
Moreover, from \eqref{sr}, we further have $[\wh{v}\wh{\phi}]^{(j)}(2\pi k)=\td(k)\td(j)$ for all $j=0,\ldots,m-1$ and $k\in \Z$ (see \cite[(5.6.6)]{hanbook} and \cite{han01,han03jat}) and consequently, for all $\pp\in \PL_{m-1}$,
\be \label{polyrepr}
\pp=\sum_{k\in \Z} [\pp*\vgu](k) \phi(\cdot-k)
=\sum_{k\in \Z} \pp_{\vgu}(k) \phi(\cdot-k)
\quad \mbox{with} \quad
\pp_{\vgu}:=\pp*\vgu=\sum_{j=0}^\infty \frac{(-i)^j}{j!} \pp^{(j)}(\cdot) \wh{\vgu}^{(j)}(0),
\ee
since $\pp*\vgu:=\sum_{n\in \Z} \pp(\cdot-n) \vgu(n)=\pp_{\vgu}
\in \PL_{m-1}$ (see \cite[Lemma~1.2.1 and Theorem~5.5.1]{hanbook}).
Moreover, one can easily deduce
from \eqref{sr} that
the quantities $\wh{\vgu}^{(j)}(0), j=0,1,\ldots,m-1$ are determined
(see \cite[(5.6.10)]{hanbook}) through
$\wh{\vgu}(0)\wh{a}(0)=\wh{\vgu}(0)$ with $\wh{\vgu}(0)\wh{\phi}(0)=1$, and the following recursive formula
\be \label{vgu:expr}
\wh{\vgu}^{(j)}(0)=\sum_{k=0}^{j-1}
\frac{2^k j!}{k!(j-k)!}
\wh{\vgu}^{(k)}(0)\wh{a}^{(j-k)}(0) [I_r-2^j\wh{a}(0)]^{-1},\qquad j=1,\ldots, m-1
\ee
provided that $2^{-j}$ is not an eigenvalue of $\wh{a}(0)$ for all $j=1,\ldots,m-1$.
For $r=1$ and a scalar filter $a\in \lp{0}$ with $\wh{a}(0)=1$, a scalar filter/mask $a$ has order $m$ sum rules if and only if $(1+e^{-i\xi})^m \mid \wh{a}(\xi)$, which is equivalent to
$\wh{a}^{(j)}(\pi)=0$ for all $j=0,\ldots,m-1$.
For the scalar case $r=1$ and $\wh{a}(0)=1$, because $\wh{\phi}(\xi):=\prod_{j=1}^\infty \wh{a}(2^{-j}\xi)$ is well defined, we must have
$\wh{\vgu}^{(j)}(0)=[1/\wh{\phi}]^{(j)}(0)$ for all $j\in \NN$, which can be computed via \eqref{vgu:expr} by starting with $\wh{\vgu}(0)=1$.
The sum rules in \eqref{sr} for matrix-valued filters in \eqref{sr} make it more involved to study refinable vector functions and matrix-valued filters than their scalar counterparts. Biorthogonal multiwavelets in $L_2(\R)$ with high vanishing moments can be easily constructed by a coset by coset (CBC) algorithm in \cite[Theorem~3.4]{han01} or \cite[Algorithm~6.5.2]{hanbook}.
Moreover, the values $\wh{\tilde{\vgu}}^{(j)}(0), j\in \NN$ of the matching filter $\tilde{\vgu}$ for the dual mask $\tilde{a}$ are uniquely determined by the primal mask $a$ as given in \cite[Theorem~6.5.1]{hanbook} or \cite[Theorem~3.1]{han01} through the following identities: $\wh{\tilde{\vgu}}(0) \ol{\wh{a}(0)}^\tp=\wh{\tilde{\vgu}}(0)$ with $\wh{\tilde{\vgu}}(0)\wh{\tilde{\phi}}(0)=1$, and the following recursive formula:
\[
\wh{\tilde{\vgu}}^{(j)}(0)=
\sum_{k=0}^{j-1} \frac{j!}{k!(j-k)!}
\wh{\tilde{\vgu}}^{(k)}(0) \ol{\wh{a}^{(j-k)}(0)}^\tp [2^j I_r-\ol{\wh{a}(0)}^\tp]^{-1},\qquad j\in \N
\]
provided that $2^j$ is not an eigenvalue of $\wh{a}(0)$ for all $j\in \N$.

The following result constructs special $\phi^L$ satisfying \eqref{I:phi} with polynomial reproduction property.

\begin{prop}\label{prop:phicut}
	Let $\phi$ be an $r\times 1$ vector of compactly supported functions in $\Lp{2}$ such that $\phi=2\sum_{k\in \Z} a(k) \phi(2\cdot-k)$ for some finitely supported sequence $a\in \lrs{0}{r}{r}$. Define $[l_\phi,h_\phi]:=\fs(\phi)$ and $[l_a,h_a]:=\fs(a)$.
	For any integer $n_\phi\in \Z$ satisfying $n_\phi\ge \max(-l_\phi,-l_a)$,
	then
	\begin{enumerate}
		\item[(i)]
		the column vector function $\phi^c:=(\phi(\cdot-k)\chi_{[0,\infty)})_{1-h_\phi\le k\le n_\phi-1}$ (note that $\phi^c$ contains interior elements $\phi(\cdot-k), -l_\phi\le k\le n_\phi-1$) satisfies the following refinement equation:
		\be \label{phic:refstr}
		\phi^c=2A_{L_c} \phi^c(2\cdot)+2\sum_{k=n_\phi}^\infty A_c(k) \phi(2\cdot-k),
		\ee
		where $A_{L_c}=[a(k_0-2n)]_{1-h_\phi\le n,k_0\le n_\phi-1}$ and $A_c$ is a finitely supported sequence given by $A_c(k):=[a(k-2n)]_{1-h_\phi\le n\le n_\phi-1}$ for $k \ge n_\phi$, where $n$ is row index and $k_0$ is column index;
		
		\item[(ii)] if in addition the filter $a$ has order $m$ sum rules in \eqref{sr} with a matching filter $\vgu\in \lrs{0}{1}{r}$ satisfying $\wh{\vgu}(0)\wh{\phi}(0)=1$, then for any $j_0,\ldots,j_\ell\in \{0,\ldots,m-1\}$, the vector function $\phi^{\pp}$, whose entries are linear combinations of elements in $\phi^c$, is given by
		\be \label{phip}
		\phi^{\pp}:=
		 \sum_{k=1-h_{\phi}}^{n_{\phi}-1}
		\sum_{j=0}^{m-1}
		\frac{(-i)^j}{j!}\pp^{(j)}(k) \wh{\vgu}^{(j)}(0)\phi(\cdot-k)\chi_{[0,\infty)}
		\quad \mbox{with}\quad
		 \pp(x):=(x^{j_0},\ldots,x^{j_\ell})^\tp
		\ee
		must satisfy \eqref{I:phi} with $\phi^L$ being replaced by $\phi^{\pp}$, more precisely,
		\be \label{phip:refstr}
		\phi^{\pp}=
		2A_{L_\pp} \phi^{\pp}(2\cdot)+2\sum_{k=n_\phi}^\infty
		A_\pp(k)\phi(2\cdot-k)
		\quad \mbox{with}\quad
		 A_{L_\pp}:=\mbox{diag}(2^{-1-j_0},\ldots,2^{-1-j_\ell}),
		\ee
		where $A_\pp=\{A_\pp(k)\}_{k= n_\phi}^\infty$ is a finitely supported sequence given by $A_\pp(k)=0$ for $k<n_\phi$ and
		\be \label{App}
		A_\pp(k):=A_{L_\pp} \pp_{\vgu}(k)-\sum_{n=n_\phi}^\infty
		\pp_{\vgu}(n)a(k-2n)
		\quad \mbox{with}
		\quad
		\pp_{\vgu}(k):=\sum_{j=0}^{m-1} \frac{(-i)^j}{j!} \pp^{(j)}(k) \wh{\vgu}^{(j)}(0), \quad k\ge n_\phi.
		\ee
		In fact, $A_\pp(k)=0$ for all $k \ge l_a+2n_\phi$, where $[l_a,h_a]:=\fs(a)$.
	\end{enumerate}
\end{prop}

\bp
By the refinement equation $\phi=2\sum_{k\in \Z} a(k) \phi(2\cdot-k)$, we deduce that
\[
\phi(\cdot-n)\chi_{[0,\infty)}=2\sum_{k\in \Z} a(k-2n) \phi(2\cdot-k)\chi_{[0,\infty)}
=2\sum_{k\in \Z} a(k-2n) \Big(\phi(\cdot-k)\chi_{[0,\infty)}\Big)(2\cdot).
\]
Note that $\phi(\cdot-n)\chi_{[0,\infty)}=0$ for all $n\le -h_\phi$ and $\phi(\cdot-n)\chi_{[0,\infty)}=\phi(\cdot-n)$ for all $n\ge -l_\phi$.
For $1-h_\phi\le n\le n_\phi-1$, by $n_\phi\ge -l_\phi$, we have
\be \label{phi:c:refstr}
\phi(\cdot-n)\chi_{[0,\infty)}
=2\sum_{k=1-h_\phi}^{n_\phi-1} a(k-2n) \Big(\phi(\cdot-k)\chi_{[0,\infty)}\Big)(2\cdot)
+2\sum_{k=n_\phi}^{\infty} a(k-2n) \phi(2\cdot-k).
\ee
Therefore, \eqref{phic:refstr} holds and we proved item (i).

To prove item (ii), by \eqref{polyrepr}
we have $\pp(x)=\phi^{\pp}(x)+\sum_{k=n_\phi}^\infty \pp_{\vgu}(k) \phi(x-k)$ for all $x\in [0,\infty)$.
Because
$\pp(x)=2A_{L_\pp} \pp(2x)$ trivially holds for $x\in [0,\infty)$, we have
$\pp=2A_{L_\pp}\pp(2\cdot)=
2 A_{L_\pp}\phi^{\pp}(2\cdot)+
\sum_{k=n_\phi}^\infty 2A_{L_\pp} \pp_{\vgu}(k) \phi(2\cdot-k)$.
By $n_\phi\ge \max(-l_\phi,-l_a)$,
\eqref{nphi} must hold and
then on $[0,\infty)$ we have
\begin{align*}
\phi^{\pp}&=\pp-\sum_{n=n_\phi}^\infty \pp_{\vgu}(n)\phi(\cdot-n)
=2A_{L_\pp}\phi^{\pp}(2\cdot)+\sum_{k=n_\phi}^\infty 2A_{L_\pp} \pp_{\vgu}(k) \phi(2\cdot-k)
-\sum_{n=n_\phi}^\infty \pp_{\vgu}(n)\phi(\cdot-n)\\
&=2A_{L_\pp}\phi^{\pp}(2\cdot)+2\sum_{k=n_\phi}^\infty A_{L_\pp} \pp_{\vgu}(k) \phi(2\cdot-k)
-2\sum_{n=n_\phi}^\infty \sum_{k=n_\phi}^\infty
\pp_{\vgu}(n)a(k-2n) \phi(2\cdot-k)\\
&=2A_{L_\pp}\phi^{\pp}(2\cdot)+2\sum_{k=n_\phi}^\infty A_\pp (k)\phi(2\cdot-k).
\end{align*}
We now prove $A_\pp(k)=0$ for all $k \ge l_a+2n_\phi$.
Define the subdivision operator
\[
[\sd_a v](n):=2\sum_{k\in \Z} v(k)a(n-2k) \quad \mbox{for} \quad n\in \Z.
\]
We conclude from the definition of $A_\pp$ in \eqref{App} that $A_\pp(k)=A_{L_\pp} \pp_{\vgu}(k)-2^{-1} \sd_a \pp_{\vgu}(k)$ for all $k\ge l_a+2n_\phi$. Define $\wh{u}(\xi):=\wh{\vgu}(2\xi)\wh{a}(\xi)$.
We deduce from \eqref{sr} that
\be \label{vgu2a}
\wh{u}^{(j)}(0)=\wh{\vgu}^{(j)}(0)
\quad \mbox{and}\quad
\wh{u}^{(j)}(\pi)=0,\qquad \forall\, j=0,\ldots,m-1.
\ee
Since $\pp_{\vgu}=\pp*\vgu$, by \cite[Theorem~1.2.4 and Lemma~1.2.1]{hanbook}, we conclude from \eqref{vgu2a} that
\[
\sd_a \pp_{\vgu}=\sd_u \pp=(\pp(2^{-1}))*u
=2A_{L_\pp} (\pp*u)=2A_{L_\pp} (\pp*\vgu)=
2A_{L_\pp} \pp_{\vgu}.
\]
This proves $A_\pp(k)=A_{L_\pp} \pp_{\vgu}(k)-2^{-1} \sd_a \pp_{\vgu}(k)=0$ for all $k\ge l_a+2n_\phi$. Therefore, item (ii) holds.
\ep

\subsection{Stability and construction of biorthogonal wavelets on $[0,\infty)$}
We shall adopt the following notation:
\be \label{subspace}
\si_j(H):=\ol{\mbox{span}\{f(2^j\cdot) \setsp f\in H\}},\qquad j\in \Z, H\subseteq \Lp{2},
\ee
where the overhead bar refers to closure in $\Lp{2}$. For a countable subset $H$ of $\Lp{2}$ or $L_2([0,\infty))$, we define $\ell_2(H)$ to be the linear space of all sequences $\{c_h\}_{h\in H}$ satisfying $\sum_{h\in H} |c_h|^2<\infty$. For a Bessel sequence $H$ in $\Lp{2}$, we say that $H$ is \emph{$\ell_2$-linearly independent} if $\sum_{h\in H} c_h h=0$ for some $\{c_h\}_{h\in H}\in \ell_2(H)$, then we must have $c_h=0$ for all $h\in H$.

In \cref{thm:wbd:0}, it is necessary that $\Phi,\Psi$ in \eqref{PhiPsiI} and $\tilde{\Phi},\tilde{\Psi}$ in \eqref{tPhiPsiI} must be Riesz sequences in $L_2([0,\infty))$.
We need the following result later for constructing wavelets on $[0,\infty)$.

\begin{theorem}\label{thm:rz}
	Let $(\{\tilde{\phi};\tilde{\psi}\},
	\{\phi;\psi\})$ be a compactly supported biorthogonal wavelet in $\Lp{2}$ satisfying items (1)--(4) of \cref{thm:bw}. Let $\Phi:=\{\phi^L\}\cup\{\phi(\cdot-k) \setsp k\ge n_\phi\}\subseteq L_2([0,\infty))$ as in \eqref{PhiPsiI} with $\phi^L$ having compact support.
	Then the following statements are equivalent:
	\begin{enumerate}
		\item[(1)] $\Phi$ is a Riesz sequence in $L_2([0,\infty))$, i.e., there exists a positive constant $C$ such that
		\be \label{riesz:Phi}
		C^{-1} \sum_{h\in \Phi} |c_h|^2\le \Big\| \sum_{h\in \Phi} c_h h\Big\|_{\Lp{2}}^2\le C \sum_{h\in \Phi} |c_h|^2,\qquad \forall\, \{c_h\}_{h\in \Phi}\in \ell_2(\Phi).
		\ee

		\item[(2)] $\Phi$ is $\ell_2$-linearly independent, i.e., $c_h=0$ for all $h\in \Phi$ if $\sum_{h\in \Phi} c_h h=0$ with $\{c_h\}_{h\in \Phi}\in \ell_2(\Phi)$.
		\item[(3)] $\{\phi^L\} \cup\{\phi(\cdot-k) \setsp n_\phi\le k<N_\phi\}$ is linearly independent, where $N_\phi:=\max(n_\phi,h_{\phi^L}-l_{\tilde{\phi}})$ with $[l_{\phi^L},h_{\phi^L}]:=\fs(\phi^L)$ and $[l_{\tilde{\phi}},h_{\tilde{\phi}}]:=\fs(\tilde{\phi})$.
		\item[(4)] There exists $\tilde{H}:=\{\tilde{\eta}^L\}\cup\{\tilde{\phi}(\cdot-k) \setsp k\ge N_\phi\}\subset L_2([0,\infty))$ such that $\tilde{\eta}^L$ has compact support, $\#\tilde{\eta}^L=\#\phi^L+(N_\phi-n_\phi)(\#\phi)$ and $\tilde{H}$ is biorthogonal to $\Phi$, where $N_\phi$ is defined in item (3).
	\end{enumerate}
	Moreover, for $\Phi=\{\phi^L\}\cup\{\phi(\cdot-k) \setsp k\ge n_\phi\}$ such that
	item (3) fails, perform the following procedure:
	\begin{enumerate}
		\item[(S1)] Initially take $E:=\{\phi(\cdot-k) \setsp n_\phi\le k<N_\phi\}$, which must be linearly independent.
		\item[(S2)] Visit all elements $\eta\in \phi^L$ one by one: replace $E$ by $E\cup\{\eta\}$ if $E\cup\{\eta\}$ is linearly independent; otherwise, delete $\eta$ from $\phi^L$.
	\end{enumerate}
	Update $\phi^L:=E\bs \{\phi(\cdot-k) \setsp n_\phi\le k<N_\phi\}$, i.e., the updated $\phi^L$ is obtained by removing redundant elements in the original $\phi^L$. Then the new $\Phi$ is a Riesz sequence in $L_2([0,\infty))$ and preserves $\si_0(\Phi)$.
\end{theorem}

\bp (1)$\imply$(2)$\imply$(3) is obvious. Using a standard argument, we now prove (4)$\imply$(1).
Since $\phi^L$ and $\{\phi(\cdot-k) \setsp k\in \Z\}$ are obviously Bessel sequences in $\Lp{2}$, we conclude that $\Phi$ is a Bessel sequence in $L_2([0,\infty))$, i.e.,
there exists a positive constant $C$ such that
\be \label{Phi:bessel}
\sum_{h\in \Phi} |\la f, h\ra|^2 \le C \|f\|_{\Lp{2}}^2, \qquad \forall\; f\in \Lp{2},
\ee
which is well known to be equivalent to the second inequality in \eqref{riesz:Phi}.
Similarly, $\tilde{H}$ must be a Bessel sequence in $L_2([0,\infty))$, i.e.,
\eqref{Phi:bessel} holds with $\Phi$ being replaced by $\tilde{H}$ (probably with a different constant $C$).
For $\{c_h\}_{h\in \Phi}\in \ell_2(\Phi)$,
define $f:=\sum_{h\in \Phi} c_h h$. Using the biorthogonality in item (4),
we have $c_h=\la f, \tilde{h}\ra$, where $\tilde{h}$ is the corresponding element of $h$ in $\tilde{H}$. Now it follows from \eqref{Phi:bessel} with $\Phi$ being replaced by $\tilde{H}$ that the first inequality in \eqref{riesz:Phi} holds. This proves (4)$\imply$(1).

To complete the proof, let us prove the key step (3)$\imply$(4). Note that $\fs(\tilde{\phi}(\cdot-k))=[k+l_{\tilde{\phi}},k+h_{\tilde{\phi}}]$ for $k\in \Z$. Hence, for $k\ge h_{\phi^L}-l_{\tilde{\phi}}$, we have $k+l_{\tilde{\phi}}\ge h_{\phi^L}$ and trivially, $\la \phi^L,\tilde{\phi}(\cdot-k)\ra=0$ due to their essentially disjoint supports.
Let $\eta\in \{\phi^L\}\cup\{\phi(\cdot-k) \setsp n_\phi\le k<N_\phi\}$. Then $\fs(\eta)\subseteq [0, N]$ with $N:=\max(h_{\phi^L}, N_\phi+h_\phi-1,1)$.
We now prove that there exists $d(\eta)\in L_2([0, N])$ such that
\be \label{deta}
\la d(\eta), \eta\ra=1\quad \mbox{and}\quad \la d(\eta), g\ra=0 \quad \forall\, g\in \Phi\bs\{\eta\}.
\ee
Define $S:=\{f\in L_2([0, N])
\setsp \la f,g\ra=0\; \forall\, g\in \Phi\bs \{\eta\}\}$. For all integers $k\ge N-l_\phi$, we observe that $\fs(\phi(\cdot-k))$ is contained inside $[N,\infty)$ and consequently, $\phi(\cdot-k) \perp L_2([0, N])$ for all $k\ge N-l_\phi$. Therefore, we conclude that
\[
S=\{f\in L_2([0, N])
\setsp \la f,g\ra=0\; \forall\, g\in \{\phi^L\}\cup \{\phi(\cdot-k) \setsp n_\phi\le k<N-l_\phi\}, g\ne \eta\}.
\]
Since $L_2([0, N])$ has infinite dimension, the above identity forces that $S$ must have infinite dimension. In particular, $S$ is not empty. We now prove that there must exist $d(\eta)\in S$ such that $\la d(\eta),\eta\ra=1$. Suppose not. Then $\eta \perp S$. By the definition of the space $S$, we must have $\Phi\perp L_2([0,N])$, which forces $\eta=0$ by $\eta\in \Phi \cap L_2([0,N])$.
This contradicts our assumption in item (3). Hence, we proved the existence of $d(\eta)\in S$ such that $\la d(\eta),\eta\ra=1$. Now it is straightforward to check that \eqref{deta} holds. Let $\tilde{\eta}^L$ be the vector/set of all elements $d(\eta)$ for $\eta\in \{\phi^L\}\cup\{\phi(\cdot-k) \setsp n_\phi\le k<N_\phi\}$.
Then the derived $\tilde{H}$ must be biorthogonal to $\Phi$. This proves (3)$\imply$(4).

Suppose that
$\vec{c}\phi^L+\sum_{k=n_\phi}^{N_\phi-1} c_k \phi(\cdot-k)=0$. If $\vec{c}=0$, then $\sum_{k=n_\phi}^{N_\phi-1} c_k \phi(\cdot-k)=0$ and by biorthogonality of integer shifts of $\phi$ and $\tilde{\phi}$, we conclude that $c_k=0$ for all $n_\phi\le k<N_\phi$. Hence, if item (3) fails, then $\vec{c}\ne 0$ and we can remove the redundant entry in $\phi^L$ corresponding to a nonzero entry in $\vec{c}$.
In this way, the new $\Phi$ with fewer elements in $\phi^L$ can satisfy item (3) and preserves $\si_0(\Phi)$.
\ep

\cref{thm:rz} can be applied to all $\Phi,\Psi$ in \eqref{PhiPsiI} and $\tilde{\Phi},\tilde{\Psi}$ in \eqref{tPhiPsiI}. If any of $\Phi,\Psi, \tilde{\Phi},\tilde{\Psi}$ is not $\ell_2$-linearly independent, by \cref{thm:rz}, then we can always remove
the redundant elements in $\phi^L, \psi^L,\tilde{\phi}^L,\tilde{\psi}^L$ to turn $\Phi,\Psi,\tilde{\Phi}, \tilde{\Psi}$ into Riesz sequences in $L_2([0,\infty))$.

To study the Bessel property of $\AS_0(\Phi;\Psi)_{[0,\infty)}$, we need the definition of the Sobolev space $\HH{\tau}$ with $\tau\in \R$.
Recall that the Sobolev space $\HH{\tau}$ with $\tau\in \R$ consists of all tempered distributions $f$ on $\R$ such that $\int_\R |\wh{f}(\xi)|^2(1+|\xi|^2)^\tau d\xi<\infty$.

\begin{theorem} \label{thm:phi:bessel}
	Let $\phi$ be an $r\times 1$ vector of compactly supported functions in $\Lp{2}$ such that $\phi=2\sum_{k\in \Z} a(k) \phi(2\cdot-k)$ for some finitely supported sequence $a\in \lrs{0}{r}{r}$.
	Let $\psi$ be a vector of compactly supported functions in $\Lp{2}$.
	Define $[l_\phi,h_\phi]:=\fs(\phi)$
	and $[l_\psi,h_\psi]:=\fs(\psi)$.
	Let $n_\phi\ge \max(-l_\phi, -l_a)$ and $n_\psi\ge -l_\psi$.
	Define $\Phi,\Psi$ as in \eqref{PhiPsiI} with finite subsets $\phi^L\cup \psi^L$ of compactly supported functions in $L_2([0,\infty))$.
	If $\psi$ has at least one vanishing moment (i.e., $\wh{\psi}(0)=\int_\R \psi(x) dx=0$)
	and $\psi\cup \phi^L\cup \psi^L\subseteq \HH{\tau}$ for some $\tau>0$ (this latter technical condition always holds if each element in $\psi\cup \phi^L\cup \psi^L$ is a finite linear combination of $\phi(2^j\cdot-k)\chi_{[0,\infty)}$ with $j,k\in \Z$),
	then $\AS_J(\Phi;\Psi)_{[0,\infty)}$ must be a Bessel sequence in $L_2([0,\infty))$ for all $J\in \Z$, that is, there exists a positive constant $C$, which is independent of $J$, such that
	\be \label{AS:bessel}
	\sum_{h\in \AS_J(\Phi;\Psi)_{[0,\infty)}} |\la f, h\ra|^2\le C\|f\|_{L_2([0,\infty))}^2, \qquad \forall\, f\in L_2([0,\infty)).
	\ee
\end{theorem}

\bp Let $\phi^c$ be defined as in \cref{prop:phicut}.
Let $\mathring{\phi}$ be a vector function obtained by appending $\phi$ to $\phi^c$.
By the refinement equation $\phi=2\sum_{k\in \Z} a(k) \phi(2\cdot-k)$
and \eqref{phic:refstr}, it is straightforward to see that the vector function $\mathring{\phi}$ is a compactly supported refinable vector function with a finitely supported matrix-valued filter. Because all entries in $\mathring{\phi}$ belong to $\Lp{2}$ and have compact support, we conclude by \cite[Corollary~5.8.2 or Corollary~6.3.4]{hanbook} (also see \cite[Theorem~2.2]{han03}) that there exists $\tau>0$ such that every entry in $\mathring{\phi}$ belongs to $\HH{\tau}$. In particular, we conclude that all the entries in $\phi\cup \phi^c$ belong to $\HH{\tau}$.
Note that $\phi(\cdot-n)\chi_{[0,\infty)}=0$ for all $n\le -h_\phi$ and $\phi(\cdot-n)\chi_{[0,\infty)}=\phi(\cdot-n)$ for all $n\ge -l_\phi$. Consequently, by $n_\phi\ge -l_a$, $\phi(\cdot-n)\chi_{[0,\infty)}\in \HH{\tau}$ for all $n\in \Z$.
Hence, all elements of $\phi(2^j\cdot-k)\chi_{[0,\infty)}$ with $j,k\in \Z$ must belong to $\HH{\tau}$. In particular,
if each element in $\psi\cup \phi^L\cup \psi^L$ is a finite linear combination of $\phi(2^j\cdot-k)\chi_{[0,\infty)}$ with $j,k\in \Z$, then we must have $\phi\cup\psi\cup \phi^L\cup \psi^L\subseteq \HH{\tau}$. Since $\phi\cup \psi \subseteq \HH{\tau}$ with $\tau>0$ and $\psi$ has at least one vanishing moment, by \cite[Corollary~4.6.6]{hanbook} or \cite[Theorem~2.3]{han03},
$\AS_0(\phi;\psi)$ must be a Bessel sequence in $\Lp{2}$, that is, there exists a positive constant $C$ such that
\[
\sum_{h\in \AS_0(\phi;\psi)}|\la f, h\ra|^2\le C\|f\|_{\Lp{2}}^2\qquad \mbox{for all}\quad f\in \Lp{2}.
\]
Since all the elements in $\AS_0(\Phi;\Psi)_{[0,\infty)}$ but not in $\AS_0(\phi;\psi)$ are $\phi^L$ and $\psi^L_{j;0}:=2^{j/2}\psi^L(2^j\cdot)$ for $j\in \NN:=\N\cup\{0\}$,
to prove \eqref{AS:bessel} with $J=0$ and $C$ being replaced by $2C$, it suffices to prove
\be \label{bessel:L}
\|\la f, \phi^L\ra\|_{l_2}^2+
\sum_{j=0}^\infty
\|\la f, \psi^L_{j;0}\ra\|_{l_2}^2\le C\|f\|_{L_2([0,\infty))}^2,\qquad \forall\, f\in L_2([0,\infty)).
\ee
Define $\phi^S:=\phi^L-\phi^L(-\cdot)$ and
$\psi^S:=\psi^L-\psi^L(-\cdot)$. Since $\phi^L\cup \psi^L\subseteq \HH{\tau}$ with $\tau>0$, all elements in $\phi^S \cup \psi^S$ belong to $\HH{\tau}$ and have compact support with one vanishing moment. Now we conclude from \cite[Corollary~4.6.6]{hanbook} or \cite[Theorem~2.3]{han03}
that there exists a positive constant $C$ such that
\[
\sum_{j=0}^\infty \sum_{k\in \Z} \|\la f, \phi^S_{j;k}\ra\|_{l_2}^2+
\sum_{j=0}^\infty \sum_{k\in \Z}
\|\la f, \psi^S_{j;k}\ra\|_{l_2}^2\le C\|f\|_{L_2(\R)}^2,\qquad \forall\, f\in L_2(\R),
\]
where $\psi^S_{j;k}:=2^{j/2}\psi^S(2^j\cdot-k)$.
For $f\in L_2([0,\infty))$, we have $\fs(f)\subseteq [0,\infty)$ and
trivially
$\la f, \psi^S_{j;0}\ra=
\la f, \psi^L_{j;0}\ra$. Consequently, it follows directly from the above inequality that \eqref{bessel:L} must hold and hence \eqref{AS:bessel} holds for $J=0$. By a simple scaling argument (e.g., see \cite[Proposition~4 and (2.6)]{han12} and \cite[Theorem 4.3.3]{hanbook}),  \eqref{AS:bessel} must hold for all $J\in \Z$ with the same constant $C$.
\ep

Based on \cref{thm:bw,thm:wbd:0,thm:phi:bessel}, we now present the following result, whose proof is given in \cref{sec:proof}, for constructing compactly supported biorthogonal wavelets on $[0,\infty)$.

\begin{theorem}\label{thm:wbd}
	Let $(\{\tilde{\phi};\tilde{\psi}\},\{\phi;\psi\})$ be a compactly supported biorthogonal wavelet in $\Lp{2}$ with a biorthogonal wavelet filter bank $(\{\tilde{a};\tilde{b}\},\{a;b\})$
	satisfying
	items (1)--(4) of \cref{thm:bw}.
	Define integers $l_\phi,l_\psi,
	l_a,h_a$ as in \eqref{fs:phi:a} and
	$l_{\tilde{\phi}}, l_{\tilde{\psi}}, l_{\tilde{a}}, l_{\tilde{b}}$ as in \eqref{fs:tphi:ta}.
	Let $\phi^L, \psi^L, \tilde{\phi}^L, \tilde{\psi}^L$ be finite sets of compactly supported functions in $L_2([0,\infty))$.
	Let $n_\phi,n_\psi,n_{\tilde{\phi}},n_{\tilde{\psi}}$ be integers satisfying \eqref{cardinality}.
	Define $\Phi,\Psi$ as in \eqref{PhiPsiI} and $\tilde{\Phi},\tilde{\Psi}$ as in \eqref{tPhiPsiI}.
	Assume that
\[
\phi^L\cup \psi^L\cup \tilde{\phi}^L \cup\tilde{\psi}^L \subseteq \HH{\tau}\quad \mbox{for some}\quad \tau>0
\]
and
	\begin{enumerate}
		\item[(i)]  $\Phi\subset L_2([0,\infty))$
		satisfies both \eqref{nphi} and \eqref{I:phi} for some matrix $A_L$ and some finitely supported sequence $A$ of matrices.
		\item[(ii)] $\tilde{\Phi}\subset L_2([0,\infty))$ is biorthogonal to $\Phi$, and $\tilde{\Phi}$ satisfies both \eqref{I:phi:dual} and \eqref{I:tphi}
		for some matrix $\tilde{A}_L$ and some finitely supported sequence $\tilde{A}$ of matrices.
		
		\item[(iii)] $\si_0(\Psi)=\si_1(\Phi)\cap
		(\si_0(\tilde{\Phi}))^\perp$ and
		$\Psi$ satisfies both \eqref{npsi} and \eqref{I:psi} for some matrix $B_L$ and some finitely supported sequence $B$ of matrices.

		\item[(iv)]
		 $\si_0(\tilde{\Psi})=\si_1(\tilde{\Phi})\cap
		(\si_0(\Phi))^\perp$, $\tilde{\Psi}$ is biorthogonal to $\Psi$, and $\tilde{\Psi}$
		satisfies \eqref{I:psi:dual} and \eqref{I:tpsi} for some matrix $\tilde{B}_L$ and some finitely supported sequence $\tilde{B}$ of matrices.
	\end{enumerate}
	Then $\AS_J(\tilde{\Phi};\tilde{\Psi})_{[0,\infty)}$ and $\AS_J(\Phi;\Psi)_{[0,\infty)}$, as defined in \eqref{ASPhiPsiI}, form a pair of compactly supported biorthogonal Riesz bases in $L_2([0,\infty))$ for every $J\in \Z$.
\end{theorem}

\section{Classical Approach for Constructing Biorthogonal Wavelets on $[0,\infty)$}
\label{sec:classical}

The main ingredients of the classical approach for constructing (bi)orthogonal wavelets on intervals are outlined in items (i)--(iv) of \cref{thm:wbd}.
Most papers in the current literature (e.g., \cite{ahjp94,ahl17,a93,cer19,cf11,cf12,cq04,cdv93,dhjk00,dku99,hj02,hm99,jia09,mad97,mey91,pst95,pri10})
employ (variants of) the classical approach to construct particular wavelets on intervals from special (bi)orthogonal wavelets
on the real line such as Daubechies orthogonal wavelets in \cite{dau88} and spline biorthogonal scalar wavelets in \cite{cdf92}.
From any arbitrarily given compactly supported biorthogonal (multi)wavelets on the real line,
the main goal of this section is to follow the classical approach outlined in \cref{thm:wbd} for constructing all possible compactly supported biorthogonal wavelets $(\AS_0(\tilde{\Phi};\tilde{\Psi})_{[0,\infty)},
\AS_0(\Phi;\Psi)_{[0,\infty)})$
on the interval $[0,\infty)$ with or without vanishing moments and polynomial reproduction, under the restriction that every boundary element in $\Phi$ (or $\tilde{\Phi}$) is a finite linear combination of $\phi(\cdot-k)\chi_{[0,\infty)}$ (or $\tilde{\phi}(\cdot-k)\chi_{[0,\infty)}$) with $k\in \Z$.
As we shall see in this section, though adapting orthogonal (multi)wavelets from the real line to $[0,\infty)$ is easy, constructing general compactly supported biorthogonal (multi)wavelets on $[0,\infty)$ is often much more involved and complicated than their orthogonal counterparts.
The complexity of the classical approach in this section also motivates us to propose a direct approach in \cref{sec:direct} to construct all possible biorthogonal (multi)wavelets on $[0,\infty)$ without explicitly constructing the dual refinable functions $\tilde{\Phi}$ and dual wavelets $\tilde{\Psi}$, while removing the restrictions on the boundary elements in $\Phi$ and $\tilde{\Phi}$.

\subsection{Construct refinable $\Phi$ satisfying item (i) of \cref{thm:wbd}}
\label{subsec:Phi}

Though elements in $\phi^L$ in \cref{thm:wbd:0} could be any compactly supported functions in $L_2([0,\infty))$, in this section we only consider particular $\phi^L$.
The general case $\phi^L\subset L_2([0,\infty))$ will be addressed later in \cref{thm:Phi:direct}.

To satisfy item (i) of \cref{thm:wbd}, we have only two conditions \eqref{nphi} and \eqref{I:phi} on $\Phi=\{\phi^L\}\cup \{\phi(\cdot-k) \setsp k\ge n_\phi\}$.
As we discussed before, \eqref{nphi} is trivially true by choosing any integer $n_\phi$ satisfying $n_\phi\ge \max(-l_\phi,-l_a)$. So, the main task is to construct $\phi^L$ to satisfy \eqref{I:phi}.
By \cref{prop:phicut},
there are three straightforward choices of $\phi^L$ satisfying \eqref{I:phi}:
\begin{enumerate}
	\item[(C1)] $\phi^L=\phi^c$ in item (i) of \cref{prop:phicut} satisfies
	\eqref{I:phi} with $A_L=A_{L_c}$ and $A=A_c$ in \eqref{phic:refstr};
	\item[(C2)] $\phi^L=\phi^{\pp}$ in \eqref{phip} in item (ii) of \cref{prop:phicut}
	satisfies \eqref{I:phi} with $A_L=A_{L_\pp}$ and $A=A_\pp$, where $\pp(x)=(x^{j_0},\ldots,x^{j_\ell})^\tp$ such that $\{j_0,\ldots, j_\ell\}\subseteq \{0,\ldots, m-1\}$ with $m:=\sr(a)$;
	
	\item[(C3)] $\phi^L:=2\sum_{k=n_\phi}^\infty A(k) \phi(2\cdot-k)$ with a finitely supported sequence $A$ satisfies \eqref{I:phi} with $A_L=0$.
\end{enumerate}
Through the following breaking and merging steps,
many new $\phi^L$ satisfying \eqref{I:phi}
can be obtained
from the finite-dimensional space generated by known/given $\phi^L$ such as in (C1)--(C3):

\begin{enumerate}
	\item[(BS)] Breaking Step: For $\phi^L$ satisfying \eqref{I:phi}, write $A_L=C^{-1} \diag(J_1,\ldots, J_N)C$ in its Jordan normal form and define $(\phi^{L_{J_\ell}})_{1\le \ell\le N}:=C\phi^L$ with $\#\phi^{L_{J_\ell}}=N_{J_\ell}\in \N$, where $J_\ell$ is an $N_{J_\ell}\times N_{J_\ell}$ (basic Jordan block) matrix given by
	\be \label{jordan}
	J_\ell:=\left[ \begin{matrix}
		\gl_\ell &1 &0 &\cdots &0\\
		0 &\gl_\ell &1 &\cdots &0\\
		\vdots &\vdots &\ddots &\vdots &\vdots\\
		0 &0 &\cdots &\gl_\ell &1\\
		0 &0 &0 &\cdots &\gl_\ell\end{matrix}\right],
	\qquad \gl_\ell \in \C.
	\ee
	Then
	every $\phi^{L_{J_\ell}}$ and its truncated vector functions by throwing away the first $n$ entries of $\phi^{L_{J_\ell}}$ with $1\le n<N_{J_\ell}$
	satisfy \eqref{I:phi} and $\mbox{span}(\phi^L)= \mbox{span}(\cup_{\ell=1}^N \phi^{L_{J_\ell}})$,
	since
	\[
	\left[
	\begin{matrix}
	\phi^{L_{J_1}}\\
	\vdots\\
	\phi^{L_{J_N}}
	\end{matrix}\right]
	=2
	\left[
	\begin{matrix}
	J_1 \phi^{L_{J_1}}(2\cdot)\\
	\vdots\\
	J_N \phi^{L_{J_N}}(2\cdot)
	\end{matrix}\right]
	+2\sum_{k=n_\phi}^\infty C A(k) \phi(\cdot-k).
	\]

	\item[(MS)] Merging Step: For vector functions $\phi^{L_1}$ and $\phi^{L_2}$ satisfying \eqref{I:phi}, i.e.,
\[
\phi^{L_1}=2A_{L_1}\phi^{L_1}(2\cdot)+
	2\sum_{k=n_\phi}^\infty A_1(k) \phi(2\cdot-k)
\]
and
\[ \phi^{L_2}=2A_{L_2}\phi^{L_2}(2\cdot)+
	2\sum_{k=n_\phi}^\infty A_2(k) \phi(2\cdot-k),
\]
then $\phi^L:=\phi^{L_1}\cup \phi^{L_2}$ satisfies \eqref{I:phi} with $A_L:=\diag(A_{L_1},A_{L_2})$ and $A(k)=[A_1(k)^\tp, A_2(k)^\tp]^\tp$ for all $k\ge n_\phi$.
\end{enumerate}

Note that we can always add $\phi^\pp$ to $\phi^L$ by the merging step (MS) for polynomial reproduction.
$\Phi=\{\phi^L\}\cup\{\phi(\cdot-k) \setsp k\ge n_\phi\}$
satisfying item (i) of \cref{thm:wbd} is not necessarily a Riesz sequence in $L_2([0,\infty))$. However,
we can always perform the procedure in \cref{thm:rz} to remove redundant elements in $\phi^L$ so that the new $\Phi$ is a Riesz sequence and still satisfies item (i) of \cref{thm:wbd}.
The particular choice of $\phi^L=\phi^{\pp}$ in item (ii) of \cref{prop:phicut} with $\pp(x)=(1,x,\ldots,x^{m-1})^\tp$ and $m:=\sr(a)$ was first considered in \cite{cdv93} for Daubechies orthogonal wavelets in \cite{dau88} and used in \cite{dku99,mas96} for spline biorthogonal scalar wavelets. The particular choice $\phi^L=\phi^c$ in item (i) of \cref{prop:phicut} was originally employed in \cite{mey91} for Daubechies orthogonal wavelets and used in \cite{cf11,pri10} for spline biorthogonal scalar wavelets with improved condition numbers.

We further study some properties of $\Phi$ satisfying item (i) of \cref{thm:wbd} in the following result, which is useful later for us to construct $\Psi$ satisfying item (iii) of \cref{thm:wbd}.

\begin{lemma}\label{lem:W}
	Let $(\{\tilde{\phi};\tilde{\psi}\},\{\phi;\psi\})$ be a compactly supported biorthogonal wavelet in $\Lp{2}$ with a finitely supported biorthogonal wavelet filter bank $(\{\tilde{a};\tilde{b}\},\{a;b\})$
	satisfying items (1)--(4) of \cref{thm:bw}.
	Suppose that $\Phi:=\{\phi^L\}\cup\{\phi(\cdot-k) \setsp k\ge n_\phi\}\subset L_2([0,\infty))$ with $\phi^L$ having compact support satisfies item (i) of \cref{thm:wbd}.
	For any integer $n_\psi\in \Z$ satisfying \eqref{npsi}, define
	\be \label{mphi}
	m_{\phi}:=
	\max(2n_\phi+h_{\tilde{a}},
	2n_\psi+h_{\tilde{b}}),
	\ee
	where $[l_{\tilde{a}},h_{\tilde{a}}]:=\fs(\tilde{a})$ and $[l_{\tilde{b}},h_{\tilde{b}}]:=\fs(\tilde{b})$, and define
	\be \label{H}
	H:=\{\phi(\cdot-k) \setsp k\ge n_\phi\}\cup \{\psi(\cdot-k)\setsp k\ge n_\psi\},
	\ee
	then
	\be \label{psi:linearcomb}
	\phi(2\cdot-k_0) \; \mbox{is a finite linear combination of elements in $H$ for all}\; k_0\ge m_{\phi}
	\ee
	and the finite-dimensional quotient space $\si_1(\Phi)/\si_0(H)$  has a basis, which can be selected from $\{\phi^L(2\cdot)\}\cup \{\phi(2\cdot-k) \setsp n_\phi\le k< m_\phi\}$.
\end{lemma}

\bp Using \eqref{expr} with $f=\phi(2\cdot-k_0)$ and $J=0$,
we deduce from \eqref{refstr:dual} that for all $k_0\in \Z$,
\be \label{phi2k0}
\phi(2\cdot-k_0)=\sum_{k=\lceil \frac{k_0-h_{\tilde{a}}}{2}\rceil}
^{\lfloor \frac{k_0-l_{\tilde{a}}}{2}\rfloor}
\ol{\tilde{a}(k_0-2k)}^\tp\phi(\cdot-k)+
\sum_{k=\lceil \frac{k_0-h_{\tilde{b}}}{2}\rceil}
^{\lfloor \frac{k_0-l_{\tilde{b}}}{2}\rfloor}
\ol{\tilde{b}(k_0-2k)}^\tp \psi(\cdot-k).
\ee
By the definition of $m_\phi$ in \eqref{mphi}, we have $\frac{k_0-h_{\tilde{a}}}{2}\ge
n_\phi$ and
$\frac{k_0-h_{\tilde{b}}}{2}\ge n_\psi$ for all $k_0\ge m_\phi$.
Hence, \eqref{phi2k0} implies \eqref{psi:linearcomb}.
By \eqref{nphi} and \eqref{I:phi} in item (i) of \cref{thm:wbd},
we have $H\subseteq \si_1(\Phi)$ and hence $\si_0(H)\subseteq \si_1(\Phi)$. So, $\si_1(\Phi)/\si_0(H)$ is well defined.
By \eqref{psi:linearcomb},
\[
\si_0(H)+\si_0(\{\phi^L(2\cdot)\}\cup \{\phi(2\cdot-k)\}_{k=n_\phi}^{m_\phi-1})
=\si_1(\Phi).
\]
Hence, the dimension of $\si_1(\Phi)/\si(H)$ is no more than $(\#\phi^L)+(m_\phi-n_\phi)(\#\phi)<\infty$.
\ep

\subsection{Construction of orthogonal wavelets on $[0,\infty)$}

We call $\{\phi;\psi\}$ \emph{an orthogonal wavelet in $\Lp{2}$} if $(\{\phi;\psi\},\{\phi;\psi\})$ is a biorthogonal wavelet in $\Lp{2}$, i.e., $\AS_0(\phi;\psi)$ is an orthonormal basis of $\Lp{2}$. Similarly, we call $\{a;b\}$ \emph{an orthogonal wavelet filter bank} if $(\{a;b\},\{a;b\})$ is a biorthogonal wavelet filter bank.
As a direct consequence of \cref{lem:W} and \cref{thm:wbd,thm:rz}, we have

\begin{algorithm}\label{alg:Phi:orth}
	Let $\{\phi;\psi\}$ be a compactly supported orthogonal wavelet in $\Lp{2}$ associated with a finitely supported orthogonal wavelet filter bank $\{a;b\}$.
	\begin{enumerate}
		\item[(S1)] Construct $\Phi=\{\phi^L\}\cup\{\phi(\cdot-k) \setsp k\ge n_\phi\}\subseteq L_2([0,\infty))$ (e.g., by \Cref{subsec:Phi} or \cref{thm:Phi:direct}) such that item (i) of \cref{thm:wbd} holds but $\Phi$ is not necessarily a Riesz sequence in $L_2([0,\infty))$.
		
		\item[(S2)] Apply the following Gram-Schmidt orthonormalization procedure to $\Phi$:
		\begin{enumerate}
			\item[(1)] Initially take $E:=\{\phi(\cdot-k) \setsp n_\phi\le k<N_\phi\}$ with $N_\phi:=\max(n_\phi, h_{\phi^L}-l_\phi)$, where $[l_{\phi^L},h_{\phi^L}]:=\fs(\phi^L)$ and $[l_\phi,h_\phi]:=\fs(\phi)$;
			\item[(2)] Visit all elements $\eta\in \phi^L$ one by one: replace $E$ by $E\cup \{\mathring{\eta}/\|\mathring{\eta}\|_{\Lp{2}}\}$ if $\|\mathring{\eta}\|_{\Lp{2}}\ne 0$, where
			 $\mathring{\eta}:=\eta-\sum_{h\in E}\la \eta, h\ra h$; otherwise, delete $\eta$ from $\phi^L$;
			\item[(3)] Update/redefine $\phi^L:=E\bs \{\phi(\cdot-k) \setsp n_\phi\le k<N_\phi\}$.
		\end{enumerate}
		Then $\Phi$ with the updated boundary vector function $\phi^L$ is an orthonormal system in $L_2([0,\infty))$.
		\item[(S3)] Select an integer $n_\psi$ satisfying \eqref{npsi} and $\la \phi^L, \psi(\cdot-k)\ra=0$ for all $k\ge n_\psi$. For example, we can choose any integer $n_\psi$ satisfying $n_\psi\ge \max(-l_\psi, \frac{n_\phi-l_b}{2}, h_{\phi^L})$ with $[l_\psi,h_\psi]:=\fs(\psi)$ and $[l_b,h_b]:=\fs(b)$.
		Define
\[
\mathring{\psi}^L:= \{\phi^L(2\cdot)\}\cup \{\phi(2\cdot-k) \setsp n_\phi\le k< m_\phi\}
\]
with $m_\phi:=\max(2n_\phi+h_a, 2n_\psi+h_b)$.
		Define $M_\phi:= \lceil\max( \frac{h_{{\phi^L}}}{2}, \frac{h_\phi+m_\phi-1}{2})\rceil-l_\phi$
		and
		calculate
		\[
		\psi^L:=\mathring{\psi}^L-
		\la \mathring{\psi}^L,\phi^L\ra\phi^L-
		\sum_{k=n_\phi}^{M_\phi-1} \la \mathring{\psi}^L,\phi(\cdot-k)\ra\phi(\cdot-k).
		\]
		
		\item[(S4)] Update $\psi^L$ by applying the similar Gram-Schmidt orthonormalization procedure in (S2) to
		 $\{\psi^L\}\cup\{\psi(\cdot-k)\setsp n_\psi\le k<N_\psi\}$ with $N_\phi$ being replaced by $N_\psi:=\max(n_\psi, h_{\psi^L}-l_{\tilde{\psi}})$.
	\end{enumerate}
	Then $\AS_J(\Phi;\Psi)_{[0,\infty)}$ in \eqref{ASPhiPsiI} with $\Psi:=\{\psi^L\}\cup \{\psi(\cdot-k) \setsp k\ge n_\psi\}$ is an orthonormal basis of $L_2([0,\infty))$ for all $J\in \Z$.
	Moreover, if we append $\phi^{\pp}$ with $\pp(x):=(1,x,\ldots, x^{\sr(a)-1})^\tp$ in item (ii) of \cref{prop:phicut} to $\phi^L$ in (S1),
	then all elements in $\Psi$ must have $\sr(a)$ vanishing moments.
\end{algorithm}

The choice of $M_\phi$ for defining $\psi^L$ in (S3) of \cref{alg:Phi:orth}
guarantees that $\supp(\mathring{\psi}^L)\cap \fs(\phi(\cdot-k))$ is at most a singleton and hence $\la \mathring{\psi}^L, \phi(\cdot-k)\ra=0$ for all $k\ge M_\phi$.
See \cref{thm:integral} below
for calculating inner products in
the Gram-Schmidt orthonormalization procedure in (S2) and (S4) of \cref{alg:Phi:orth}.
Though biorthogonal wavelets on the real line are flexible to design (\cite{han01,hanbook}) and important in many applications, as we shall see later, the construction of biorthogonal wavelets on $[0,\infty)$ is much more involved than their orthogonal counterparts.

As we shall discuss in \cref{sec:bc},
if an orthogonal (multi)wavelet on $[0,\infty)$ satisfies homogeneous boundary conditions, then some of its boundary elements often can have no or low order of vanishing moments.
The following result is a special case of \cref{thm:ow:vm}.

\begin{cor}\label{cor:ow:novm}
	Suppose that $\AS_0(\Phi; \Psi)_{[0,\infty)}$ is an orthonormal basis of $L_2([0,\infty))$ such that all elements of $\Phi\cup \Psi$ are compactly supported and are continuous near $0$. If $\eta(0)=0$ for all $\eta\in \Psi$, then there must exist some $\eta\in \Psi$ such that $\eta$ does not have any vanishing moments, i.e., $\int_0^\infty \eta(x) dx\ne 0$.
\end{cor}

To illustrate the complexity of wavelets on intervals, let us present an ``abnormal'' example here.

\begin{example}\label{expl:00}
	\normalfont
	Let $N\in \N$ be an arbitrary integer.
	Let $\phi=\chi_{[0,1]}$ and $\psi=\chi_{(0,1/2]}-\chi_{(1/2,1]}$. Then $\{\phi;\psi\}$ is the well-known Haar orthogonal wavelet in $\Lp{2}$. Define $\phi^L:=\emptyset$ and $n_\phi:=2N$. Then $\Phi=\{\phi(\cdot-k) \setsp k\ge 2N\}$ obviously satisfies item (i) of \cref{thm:wbd}. Define
	$\psi^L=\{\phi(\cdot-k) \setsp N\le k<2N\}$
	and $n_\psi:=N$.
	For every $J\in \Z$,  then $\AS_J(\Phi;\Psi)_{[0,\infty)}$
	with $\Psi:=\{\psi^L\}\cup \{\psi(\cdot-k) \setsp k\ge N\}$
	must be an orthonormal basis in $L_2([0,\infty))$ such that all elements in $\Phi\cup\Psi$ are supported inside $[N,\infty)$ and the boundary wavelet $\psi^L$ does not have any vanishing moments.
	This appears weird but not surprising at all.
	For simplicity, we present a self-contained proof here only for $J=0$ and the general case follows by a simple scaling argument. It is easy to directly check that $\AS_0(\Phi;\Psi)_{[0,\infty)}$ is an orthonormal system in $L_2([0,\infty))$ by noting that $\Psi$ is perpendicular to $\Phi$ and
	\be \label{haar}
	\si_0(\Phi) \oplus \si_0(\Psi)=
	\si_0(\Phi\cup \psi^L)\oplus \si_0(\{\psi(\cdot-k) \setsp k\ge N\})=\si_1(\Phi).
	\ee
	To prove that $\AS_0(\Phi;\Psi)_{[0,\infty)}$ is dense in $L_2([0,\infty))$,
	using \eqref{haar}, we observe that for any $n\in \N$,
	\[
	\si_n(\Phi)=\si_0(\Phi)\oplus \si_0(\Psi)\oplus\si_1(\Psi)\oplus\cdots\oplus \si_{n-1}(\Psi)\subseteq \AS_0(\Phi;\Psi)_{[0,\infty)}.
	\]
	Note that $\si_n(\Phi)=\si_0(\Phi(2^n\cdot))=\si_0(\{ \phi(2^n(\cdot-2^{-n}k)) \setsp k\ge 2N\})$.
	Because $\lim_{n\to \infty} 2^{-n} N=0$,
	we now conclude that $\cup_{n=1}^\infty \si_n(\Phi)$ is indeed dense in $L_2([0,\infty))$. Hence, we proved that $\AS_0(\Phi;\Psi)_{[0,\infty)}$ is dense in $L_2([0,\infty))$ and thus is an orthonormal basis in $L_2([0,\infty))$.
	Similar examples can be constructed from any compactly supported orthogonal (multi)wavelets using \cref{alg:Phi:orth}.
\end{example}

To perform the Gram-Schmidt orthonormalization procedure in (S2) and (S4) of \cref{alg:Phi:orth},
we have the following result (see \cref{sec:proof} for its proof) to compute $\int_0^1 \tilde{\phi}(x-m) \ol{\phi(x-n)}^\tp dx$ for all $m,n\in \Z$ from any two arbitrary compactly supported refinable vector functions.

\begin{theorem}\label{thm:integral}
	Let $\phi,\tilde{\phi}$ be two $r\times 1$ vectors of compactly supported functions in $\Lp{2}$ such that $\phi=2\sum_{k\in\Z} a(k) \phi(2\cdot-k)$ and $\tilde{\phi}=2\sum_{k\in\Z} \tilde{a}(k) \tilde{\phi}(2\cdot-k)$ for some finitely supported filters $a,\tilde{a}\in \lrs{0}{r}{r}$. Assume that $\wh{\phi}(0)\ne 0$ and $\wh{\tilde{\phi}}(0)\ne 0$.
	Define $[l_\phi,h_\phi]:=\fs(\phi)$ and $[l_{\tilde{\phi}},h_{\tilde{\phi}}]:=\fs(\tilde{\phi})$.
	\begin{enumerate}
		\item[(S1)] Define two vector functions by
\[
\vec{\phi}:=[\phi(\cdot-1+h_\phi)\chi_{[0,1]},\ldots, \phi(\cdot+l_\phi)\chi_{[0,1]}]^\tp \quad\mbox{and}\quad \vec{\tilde{\phi}}:=[\tilde{\phi}(\cdot-1+h_{\tilde{\phi}})
		\chi_{[0,1]},\ldots, \tilde{\phi}(\cdot+l_{\tilde{\phi}})\chi_{[0,1]}]^\tp.
\]
Then
		\be \label{vecphi:refeq}
		\vec{\phi}=2A_0 \vec{\phi}(2\cdot)+2A_1 \vec{\phi}(2\cdot-1) \quad \mbox{and}\quad \vec{\tilde{\phi}}=2\tilde{A}_0 \vec{\tilde{\phi}}(2\cdot)+2\tilde{A}_1 \vec{\tilde{\phi}}(2\cdot-1)
		\ee
		with $A_\gamma:=[a(k+\gamma-2j)]_{1-h_\phi\le j, k\le -l_\phi}$ and $\tilde{A}_\gamma:=[\tilde{a}(k+\gamma-2j)]_{1-h_{\tilde\phi}\le j, k\le -l_{\tilde{\phi}}}$ for $\gamma=0,1$, where $j$ is for the row index and $k$ is for the column index.
		\item[(S2)] If all the entries in $\vec{\phi}$ are not linearly independent on $[0,1]$, then we delete as many entries as possible from $\vec{\phi}$ so that all the deleted entries are linear combinations of entries kept. Do the same for $\vec{\tilde{\phi}}$. Then \eqref{vecphi:refeq} still holds with $A_0, A_1, \tilde{A}_0$ and $\tilde{A}_1$ being appropriately modified.
		
		\item[(S3)] Define $M:=\la \vec{\tilde{\phi}}, \vec{\phi} \ra:=\int_0^1 \vec{\tilde{\phi}}(x) \ol{\vec{\phi}(x)}^\tp dx$. Then the matrix $M$ is uniquely determined by the system of linear equations given by
		\be \label{Mphi}
		M=2\tilde{A}_0 M \ol{A_0}^\tp+2\tilde{A}_1 M \ol{A_1}^\tp
		\ee
		under the normalization condition
		\be \label{Mphi:normalization}
		\vec{\tilde{v}} M \ol{\vec{v}}^\tp=1,
		\ee
		where $\vec{v}$ is the unique row vector satisfying $\vec{v}(A_0+A_1)=\vec{v}$ and $\vec{v}\wh{\vec{\phi}}(0)=1$, while similarly $\vec{\tilde{v}}$ is the unique row vector satisfying $\vec{\tilde{v}}(\tilde{A}_0+\tilde{A}_1)=\vec{\tilde{v}}$ and $\vec{\tilde{v}}\wh{\vec{\tilde{\phi}}}(0)=1$.
	\end{enumerate}
	Define $M_{\tilde{\phi},\phi}:=(\int_0^1 \tilde{\phi}(x-j) \ol{\phi(x-k)}^\tp dx)_{1-h_{\tilde{\phi}}\le j\le -l_{\tilde{\phi}}, 1-h_{\phi}\le k\le -l_{\phi}}$. If (S2) is not performed, then $M_{\tilde{\phi},\phi}$ agrees with $M$ as in (S3); Otherwise, we obtain $M_{\tilde{\phi},\phi}$ from $M$ using the linear combinations in (S2). Hence, all integrals $\int_0^1 \tilde{\phi}(x-j)\ol{\phi(x-k)}^\tp dx$ for $j,k\in \Z$ can be obtained from $M_{\tilde{\phi},\phi}$.
\end{theorem}

Suppose that $a$ has order one sum rule with a matching filter $\vgu\in \lrs{0}{1}{r}$ and $\wh{\vgu}(0)\wh{\phi}(0)=1$ and
$\tilde{a}$ has order one sum rule with a matching filter $\tilde{\vgu}\in \lrs{0}{1}{r}$ and $\wh{\tilde{\vgu}}(0)\wh{\tilde{\phi}}(0)=1$.
Then by \cite[Proposition~5.6.2]{hanbook}, we have $\wh{\vgu}(0)\wh{\phi}(2\pi k)=0$ for all $k\in \Z\bs\{0\}$.
By Poisson summation formula, we must have $\wh{\vgu}(0)\sum_{k\in \Z} \phi(\cdot-k)=1$. Consequently, we conclude that $\wh{\vgu}(0) \phi(x-1+h_\phi)+\cdots+\wh{\vgu}(0)\phi(x+l_\phi)=1$ for almost every $x\in [0,1]$. Similarly, we have
$\wh{\tilde{\vgu}}(0)\sum_{k\in \Z} \tilde{\phi}(\cdot-k)=1$ and $\wh{\tilde{\vgu}}(0) \tilde{\phi}(x-1+h_{\tilde{\phi}})+\cdots+
\wh{\tilde{\vgu}}(0)\tilde{\phi}(x+l_{\tilde{\phi}})=1$ for almost every $x\in [0,1]$.
If (S2) is not performed, then $M_{\phi,\tilde{\phi}}=M$ and it is easy to directly check that $\vec{v}=[\wh{\vgu}(0),\ldots, \wh{\vgu}(0)]$, $\vec{\tilde{v}}=[\wh{\tilde{\vgu}}(0),\ldots,\wh{\tilde{\vgu}}(0)]$, and hence, \eqref{Mphi:normalization} becomes
\be \label{Mphi:normalize}
[\wh{\tilde{\vgu}}(0),\ldots,\wh{\tilde{\vgu}}(0)] M_{\tilde{\phi},\phi}
[\ol{\wh{\vgu}(0)},\ldots,\ol{\wh{\vgu}(0)}]^\tp=1.
\ee
A particular case of \cref{thm:integral} is $\tilde{\phi}=\eta$ with
$\eta(x):=(1,x,\ldots,x^m)^\tp\chi_{[0,1]}$, which is a refinable vector function. This allows us to compute $\int_k^{k+1} x^j \phi(x) dx$ for all $j\in \NN$ and $k\in \Z$.
Note that $\eta$ is refinable:
$\eta(x)=C_0
\eta(2x)+C_1\eta(2x-1)$
with $C_0:=\mbox{diag}(1,2^{-1},\ldots,2^{-m})$ and $C_1:= [2^{-m} \binom{j}{\ell}]_{0\le j \le m, 0 \le \ell \le j}$.

\subsection{Construct refinable $\tilde{\Phi}$ satisfying item (ii) of \cref{thm:wbd}}
\label{subsec:tPhi}

Though elements in $\tilde{\phi}^L$ in \cref{thm:wbd:0} could be any compactly supported functions in $L_2([0,\infty))$, in this section
we present an algorithm to construct a particular $\tilde{\Phi}=\{\tilde{\phi}^L\}
\cup\{\tilde{\phi}(\cdot-k)\setsp k\ge n_{\tilde{\phi}}\}$ satisfying item (ii) of \cref{thm:wbd}. The general case $\tilde{\phi}^L\subseteq L_2([0,\infty))$ will be addressed in \cref{thm:direct:tPhi}.

\begin{algorithm}\label{alg:tPhi}
	Let $(\{\tilde{\phi};\tilde{\psi}\},\{\phi;\psi\})$ be a compactly supported biorthogonal wavelet in $\Lp{2}$
	with a biorthogonal wavelet filter bank $(\{\tilde{a};\tilde{b}\},\{a;b\})$
	satisfying items (1)--(4) of \cref{thm:bw}.
	Let $0\le m\le \sr(a)$ and $0\le \tilde{m}\le \sr(\tilde{a})$.
	Assume that
	$\Phi=\{\phi^L\}\cup\{\phi(\cdot-k) \setsp k\ge n_\phi\}\subseteq L_2([0,\infty))$, with $\phi^L$ having compact support, satisfies item (i) of \cref{thm:wbd} (e.g., $\Phi$
	is obtained by \Cref{subsec:Phi} or \cref{thm:Phi:direct}) and $\Phi$ is a Riesz sequence in $L_2([0,\infty))$.
	Define $[l_{\tilde{\phi}},h_{\tilde{\phi}}]:=\fs(\tilde{\phi})$ and $[l_{\tilde{a}},h_{\tilde{a}}]:=\fs(\tilde{a})$.
	\begin{enumerate}
		\item[(S1)]  Choose $n_{\tilde{\phi}}\ge \max(-l_{\tilde{\phi}},-l_{\tilde{a}},n_\phi)$ such that $n_{\tilde{\phi}}$ is the smallest integer satisfying $\la \tilde{\phi}(\cdot-k), \phi^L\ra=0$ for all $k\ge n_{\tilde{\phi}}$, e.g., we can take any $n_{\tilde{\phi}}\ge \max(-l_{\tilde{\phi}},-l_{\tilde{a}},n_\phi, h_{\phi^L})$ with $[l_{\phi^L},h_{\phi^L}]:=\fs(\phi^L)$.
		
		\item[(S2)] Since $n_{\tilde{\phi}}\ge n_\phi$, we define a vector function $\mathring{\phi}^L:=\{\phi^L\}\cup\{ \phi(\cdot-k) \setsp n_\phi\le k< n_{\tilde{\phi}}\}$.
		
		\item[(S3)] Define a vector function
		 $\breve{\phi}^L:=\tilde{\phi}^c\cup \tilde{\phi}^h$, where
		 $\tilde{\phi}^c:=\{\tilde{\phi}(\cdot-k)\chi_{[0,\infty)}\}_{1-h_{\tilde{\phi}}\le k\le n_{\tilde{\phi}}-1}$ and
		\be \label{phih}
		 \tilde{\phi}^h:=2\sum_{k=n_{\tilde{\phi}}}^{n_h-1}
		\tilde{A}(k) \tilde{\phi}(2\cdot-k) \quad \mbox{and}\quad
		\la \tilde{\phi}^h, \phi(\cdot-k)\ra=0,\qquad \forall\; k \ge n_{\tilde{\phi}},
		\ee
		where $n_h:=2\max(n_{\phi^L}, h_\phi+n_{\tilde{\phi}})-l_{\tilde{\phi}}$ with $\{\tilde{A}(k)\}_{k=n_{\tilde{\phi}}}^{n_h-1}$ to be determined.
		Note that $\breve{\phi}^L=2\breve{A}_L \breve{\phi}^L(2\cdot)+2\sum_{k=n_{\tilde{\phi}}}^\infty \breve{A}(k)\tilde{\phi}(2\cdot-k)$ for some matrix $\breve{A}_L$ and finitely supported sequence $\breve{A}$.
		
		\item[(S4)] Apply item (BS) of \Cref{subsec:Phi} to break $\breve{\phi}^L$ into short vector functions $\breve{\phi}_1,\ldots, \breve{\phi}_N$ with each satisfying \eqref{I:phi:dual}.
		Initially define $\eta^L=\emptyset$. We add/merge $\breve{\phi}_\ell$ into $\eta^L$ if $\la \eta^L \cup \breve{\phi}_\ell, \mathring{\phi}^L\ra$ has full rank. Repeat this procedure until $\#\eta^L=\#\mathring{\phi}^L$. Then $\tilde{\phi}^L:=\la \eta^L,\mathring{\phi}^L\ra^{-1} \eta^L$ is a well-defined vector function, because the square matrix $\la \eta^L,\mathring{\phi}^L\ra$ is invertible.
		
		\item[(S4')] Assume that $\{\breve{\phi}^L\}\cup \{\tilde{\phi}(\cdot-k) \setsp k\ge n_{\tilde{\phi}}\}$ is a Riesz sequence; otherwise, remove redundant elements from $\breve{\phi}^L$.
		Instead of (S4),
		we can alternatively obtain $\tilde{\phi}^L:=C \breve{\phi}^L$, where the unknown $(\#\mathring{\phi})\times (\#\breve{\phi}^L)$ matrix $C$ is determined by solving $C\la \breve{\phi}^L, \mathring{\phi}^L\ra=I_{\#\mathring{\phi}^L}$ and $C \breve{A}_L= C \breve{A}_L \la \breve{\phi}^L, \mathring{\phi}^L\ra C$.
	\end{enumerate}
	Then $\tilde{\Phi}:=\{\tilde{\phi}^L\}\cup\{\tilde{\phi}(\cdot-k)\setsp k\ge n_{\tilde{\phi}}\}$ satisfies item (ii) of \cref{thm:wbd}.
\end{algorithm}

\bp By the choice of $n_{\tilde{\phi}}$ in (S1),
for every $k\ge n_{\tilde{\phi}}$,
we have $\la \tilde{\phi}(\cdot-k), \eta\ra=0$ for all $\eta\in \Phi\bs \{\phi(\cdot-k)\}$.
Note that the integer $n_h$ in (S3) is chosen so that $\fs(\mathring{\phi}^L)$
is essentially disjoint with $\supp(\tilde{\phi}(2\cdot-k))$ for all $k\ge n_h$
and hence $\la \tilde{\phi}(2\cdot-k), \mathring{\phi}^L\ra=0$ for all $k\ge n_h$.
By the definition of $\tilde{\phi}^c$, it is trivial to see that $\la \tilde{\phi}^c, \phi(\cdot-k)\ra=0$ for all $k\ge n_{\tilde{\phi}}$.
This and \eqref{phih} imply that $\la \tilde{\phi}^L, \phi(\cdot-k)\ra=0$ for all $k\ge n_{\tilde{\phi}}$.
Now the claim holds trivially for the choice of $\tilde{\phi}^L$ in (S4).

The condition $C \la \tilde{\phi}^c, \mathring{\phi}^L\ra=
I_{\#\mathring{\phi}^L}$ in (S4') is obviously equivalent to the biorthogonality condition $\la \tilde{\phi}^L, \mathring{\phi}^L\ra=I_{\#\mathring{\phi}^L}$.
For the choice of $\tilde{\phi}^L$ in (S4'), we have
\[
\tilde{\phi}^L=C
\breve{\phi}^L=
2C\breve{A}_L \breve{\phi}^L(2\cdot)+2\sum_{k=n_{\tilde{\phi}}}^\infty C\breve{A}(k)\tilde{\phi}(2\cdot-k).
\]
Since $\{\breve{\phi}^L\}\cup \{\tilde{\phi}(\cdot-k) \setsp k\ge n_{\tilde{\phi}}\}$ is a Riesz sequence,
\eqref{I:phi:dual} holds if and only if $C \breve{A}_L=\tilde{A}_L C$. Due to the biorthogonality between $\tilde{\phi}^L$ and $\mathring{\phi}^L$, we must have $\tilde{A}_L=\la \tilde{\phi}^L,\mathring{\phi}^L(2\cdot)\ra
=C \la \breve{\phi}^L,\mathring{\phi}^L(2\cdot)\ra
=C \breve{A}_L \la \breve{\phi}^L, \mathring{\phi}^L\ra$.
Now the condition $C\breve{A}_L=C \breve{A}_L \la \breve{\phi}^L, \mathring{\phi}^L\ra C$ in (S4') guarantees that $\tilde{\phi}^L$ satisfies \eqref{I:phi:dual} with $\tilde{A}_L=C \breve{A}_L \la \breve{\phi}^L, \mathring{\phi}^L\ra$.
Hence, the claim holds for the choice of $\tilde{\phi}^L$ in (S4').
\ep

For spline biorthogonal scalar wavelets $(\{\tilde{\phi};\tilde{\psi}\},\{\phi;\psi\})$ in \cite{cdf92} with $\phi=B_m$ in \eqref{bspline} such that $\tilde{m}:=\sr(\tilde{a})\ge m$ and $m+\tilde{m}$ is an even integer,
\cite{dku99}
considered the particular choice $\phi^L=\phi^\pp$
and
$\tilde{\phi}^L=\la \tilde{\pp},\mathring{\phi}^L\ra^{-1} \tilde{\phi}^{\tilde{\pp}}$ in (S4)
with $\pp(x)=(1,x,\ldots,x^{m-1})^\tp$ and $\tilde{\pp}(x)=(1,x,\ldots,x^{\tilde{m}-1})^\tp$ as in \cref{prop:phicut}.
\cite{cf11,pri10} instead took the particular choice $\phi^L=\phi^c$ in \cref{prop:phicut}, which has more boundary elements than \cite{dku99}.
Then \cite[(4.3)]{pri10} and \cite[(32)]{cf11} proposed
to take $\eta^L=\tilde{\phi}^{\tilde{\pp}}\cup \tilde{\phi}^h$ in (S4) by properly choosing $\tilde{A}$ in \eqref{phih}.

\subsection{Construct wavelets $\Psi$ and $\tilde{\Psi}$ satisfying items (iii) and (iv) of \cref{thm:wbd}}
\label{subsec:PsitPsi}

Assume that $\Phi$ and $\tilde{\Phi}$ satisfy items (i) and (ii) of \cref{thm:wbd} but without restricting that $\phi^L$ and $\tilde{\phi}^L$ are special choices as discussed in \Cref{subsec:Phi,subsec:tPhi}.
We now address how to construct $\Psi$ and $\tilde{\Psi}$ satisfying items (iii) and (iv) of \cref{thm:wbd}.

From $\Phi$ and $\tilde{\Phi}$ satisfying items (i) and (ii) of \cref{thm:wbd},
we now construct $\Psi$ satisfying item (iii) of \cref{thm:wbd} as follows.

\begin{prop}\label{prop:Psi}
	Let $(\{\tilde{\phi};\tilde{\psi}\},\{\phi;\psi\})$ be a compactly supported biorthogonal wavelet in $\Lp{2}$ satisfying items (1)--(4) of \cref{thm:bw}.
	Suppose that $\Phi$ and $\tilde{\Phi}$ as defined in \eqref{PhiPsiI} and \eqref{tPhiPsiI}, consisting
	of compactly supported functions in $L_2([0,\infty))$,
	satisfy items (i) and (ii) of \cref{thm:wbd}. Define
	$[l_{\phi^L},h_{\phi^L}]:=\fs(\phi^L)$ and
	 $[l_{\tilde{\phi}^L},h_{\tilde{\phi}^L}]:=\fs(\tilde{\phi}^L)$.
	Take $n_\psi\in \Z$ to be the smallest integer satisfying
	\be \label{npsi:small}
	n_\psi\ge \max(-l_\psi, \tfrac{n_\phi-l_b}{2}, h_{\tilde{\phi}^L}-l_\psi)
	\ee
	and define $m_\phi:=\max(2n_\phi+h_{\tilde{a}}, 2n_\psi+h_{\tilde{b}})$ as in \eqref{mphi}.
	Let $\mathring{\psi}^L:=
	\{\phi^L(2\cdot)\}\cup\{\phi(2\cdot-k) \setsp n_\phi\le k<m_\phi\}$.
	Calculate a compactly supported vector function $\psi^L$ by
	\be \label{psiL:constr}
	\psi^L:=\mathring{\psi}^L-
	\la \mathring{\psi}^L,\tilde{\phi}^L\ra\phi^L-
	\sum_{k=n_\phi}^{M_\phi-1} \la \mathring{\psi}^L,\tilde{\phi}(\cdot-k)\ra\phi(\cdot-k),
	\ee
	where $M_\phi:= \lceil\max( \frac{h_{{\phi^L}}}{2}, \frac{h_\phi+m_\phi-1}{2})\rceil-l_{\tilde{\phi}}$.
	Then $\Psi:=\{\psi^L\}\cup \{\psi(\cdot-k) \setsp k\ge n_\psi\}$ satisfies item (iii) of \cref{thm:wbd}.
	Moreover, if we apply item (3) of \cref{thm:rz} with $\Phi$ being replaced by $\Psi$ to remove the redundant elements of $\psi^L$ in $\{\psi^L\}\cup\{\psi(\cdot-k) \setsp n_\psi\le k<N_\psi\}$ with $N_\psi:=\max(n_\psi, h_{\psi^L}-l_{\tilde{\psi}})$, then $\Psi$ with updated $\psi^L$ satisfies item (iii) of \cref{thm:wbd} and $\Psi$ is a Riesz sequence in $L_2([0,\infty))$.
\end{prop}

\bp By the definition of $n_\psi$ in \eqref{npsi:small}, \eqref{npsi} holds. Note that $\fs(\psi(\cdot-k))=[k+l_{\psi}, k+h_{\psi}]$ for $k\in \Z$. Hence, for $k\ge n_{\tilde{\psi}}$, we deduce from \eqref{npsi:small} that $n_\psi\ge h_{\tilde{\phi}^L}-l_\psi$. Therefore, due to essentially disjoint support, we trivially have $\la \psi(\cdot-k), \tilde{\phi}^L\ra=0$ and
$\psi(\cdot-k) \perp \tilde{\Phi}$
for all $k\ge n_{\tilde{\psi}}$.

Due to essentially disjoint support, we have $\la \mathring{\psi}^L,\tilde{\phi}(\cdot-k)\ra=0$ for all $k\ge M_\phi$.
Since $\tilde{\Phi}$ is biorthogonal to $\Phi$ by item (ii) of \cref{thm:wbd}, we trivially deduce from the definition of $\psi^L$ in \eqref{psiL:constr} that $\psi^L \perp \tilde{\Phi}$. Therefore, $\Psi$ is perpendicular to $\tilde{\Phi}$ and $\si_0(\Psi) \perp \si_0(\tilde{\Phi})$.
By \eqref{psi:linearcomb} in \cref{lem:W},
we have $\si_1(\Phi)=\si_0(\Phi)+\si_0(
\{\mathring{\psi}^L\} \cup\{\psi(\cdot-k) \setsp k\ge n_{\psi}\})$ and hence
$\si_1(\Phi)=\si_0(\Phi)+\si_0(\Psi)$.
Therefore,
we must have $\si_0(\Psi)=\si_1(\Phi)\cap (\si_0(\tilde{\Phi}))^\perp$.
Because all functions in $\Phi\cup\tilde{\Phi}$ are compactly supported, $\psi^L$ defined in \eqref{psiL:constr} obviously has compact support. We can also directly check \eqref{I:psi} using the definition of $\psi^L$ in \eqref{psiL:constr} and the relations in \eqref{I:phi} and \eqref{nphi}.
Hence,
$\Psi$ satisfies item (iii) of \cref{thm:wbd}.
\ep

The integer $n_\psi$ satisfying
\eqref{npsi:small} can be replaced by the smallest $n_\psi\in \Z$ such that $\psi(\cdot-k)\in \si_1(\Phi)$ and $\psi(\cdot-k)\perp \tilde{\Phi}$ for all $k\ge n_\psi$.
The choice of $M_\phi$ for defining $\psi^L$ in \eqref{psiL:constr}
guarantees that $\supp(\mathring{\psi}^L)\cap \fs(\tilde{\phi}(\cdot-k))$ is at most a singleton and hence $\la \mathring{\psi}^L, \tilde{\phi}(\cdot-k)\ra=0$ for all $k\ge M_\phi$.
To guarantee that $\Psi$ is a Riesz sequence in \cref{prop:Psi}, we can avoid using \cref{thm:rz} to reduce the redundant elements in $\psi^L$ by replacing $\mathring{\psi}^L$ with
a suitable subset of $\mathring{\psi}^L$, which forms a basis of
the quotient space $\si_1(\Phi)/\si_0(\Phi\cup \{\psi(\cdot-k) \setsp k\ge n_\psi\})$.
See the remark after \cref{thm:direct} about how to find such desired subset of $\mathring{\psi}^L$.
Though $\psi^L$ itself is not unique, the finite-dimensional space $\si_0(\Psi)/\si_0(\{\psi(\cdot-k) \setsp k\ge n_\psi\})$ (or equivalently,
$\mbox{span}(\psi^L) \mod \mbox{span}\{\psi(\cdot-k) \setsp k\ge n_\psi\}$)
is uniquely determined by $\Phi$ and $\tilde{\Phi}$ satisfying items (i) and (ii) of \cref{thm:wbd}.

For $\Phi,\tilde{\Phi}$ and $\Psi$ satisfying items (i)--(iii) of \cref{thm:wbd}, it is easy to construct the dual wavelet $\tilde{\Psi}$ satisfying item (iv) of \cref{thm:wbd}, mainly due to the uniqueness of $\tilde{\Psi}$.

\begin{prop}\label{prop:tPsi}
	Let $(\{\tilde{\phi};\tilde{\psi}\},\{\phi;\psi\})$ be a compactly supported biorthogonal wavelet in $\Lp{2}$ satisfying items (1)--(4) of \cref{thm:bw}.
	Suppose that $\Phi,\tilde{\Phi}$ and $\Psi$ as defined
	in \eqref{PhiPsiI} and \eqref{tPhiPsiI}, consisting
	of compactly supported functions in $L_2([0,\infty))$,
	satisfy items (i)--(iii) of \cref{thm:wbd}. Assume that $\Psi$ is a Riesz sequence in $L_2([0,\infty))$. Define
	\[
	 [l_{\phi^L},h_{\phi^L}]:=\fs(\phi^L),\quad
	 [l_{\psi^L},h_{\psi^L}]:=\fs(\psi^L),\quad
	 [l_{\tilde{\phi}^L},h_{\tilde{\phi}^L}]:=\fs(\tilde{\phi}^L).
	\]
	Take $n_{\tilde{\psi}}\in \Z$ to be the smallest integer satisfying
	\be \label{ntpsi:small}
	n_{\tilde{\psi}}\ge \max(-l_{\tilde{\psi}}, \tfrac{n_{\tilde{\phi}}
		-l_{\tilde{b}}}{2}, h_{\phi^L}-l_{\tilde{\psi}}, h_{\psi^L}-l_{\tilde{\psi}}, n_\psi)
	\ee
	and $m_{\tilde{\phi}}:=\max(2n_{\tilde{\phi}}+h_{\tilde{a}},2n_{\tilde{\psi}}+h_{\tilde{b}})$.
	For each element $\eta \in \{\psi^L\}\cup\{\psi(\cdot-k) \setsp n_\psi\le k<n_{\tilde{\psi}}\}$, there exists a unique sequence $\{c_{\eta}(h)\}_{h\in \Phi}\in \ell_2(\Phi)$ such that
	\be \label{FindtPsi}
	\la d_\eta, g\ra=
	\begin{cases}
		1 &\text{if $g=\eta$},\\
		0 &\text{if $g\in (\Phi \cup \Psi)\bs \{\eta\}$}
	\end{cases}
	\qquad \mbox{with}\qquad
	d_\eta:=\sqrt{2} \sum_{h\in \Phi} c_{\eta}(h) \tilde{h}(2\cdot)\in \si_1(\tilde{\Phi})
	\ee
	and
	\be \label{mtphi:cond}
	c_\eta(\phi(\cdot-k))=0\qquad \forall\; k\ge m_{\tilde{\phi}}.
	\ee
	Then
	$\tilde{\Psi}:=\{\tilde{\psi}^L\}\cup \{ \tilde{\psi}(\cdot-k) \setsp k\ge n_{\tilde{\psi}}\}$ with $\tilde{\psi}^L:=\{d_\eta \setsp \eta\in \{\psi^L\}\cup\{\psi(\cdot-k) \setsp n_{\psi}\le k<n_{\tilde{\psi}}\}\}$
	satisfies item (iv) of \cref{thm:wbd} and all functions in $\tilde{\psi}^L$ have compact support.
\end{prop}

\bp
Note that all $\Phi,\tilde{\Phi}$ and $\Psi$ are Riesz sequences.
We first prove that $\Phi\cup \Psi$ is a Riesz sequence. Suppose not. Then the lower Riesz bound of $\Phi\cup \Psi$ is zero and there exists a sequence $\{c_n\}_{n=1}^\infty$ in $\ell_2(\Phi\cup \Psi)$ such that $\sum_{h\in \Phi\cup \Psi} |c_n(h)|^2=1$ and $\lim_{n\to \infty} \|F_n\|_{\Lp{2}}=0$, where $F_n:=f_n+g_n$ with $f_n:=\sum_{h\in \Phi} c_n(h) h$ and $g_n:=\sum_{h\in \Psi} c_n(h)h$. Because $\tilde{\Phi}$ is biorthogonal to $\Phi$ and $\tilde{\Phi} \perp \Psi$, for $\tilde{h}\in \tilde{\Phi}$, we have $\la F_n, \tilde{h}\ra=\la f_n,\tilde{h}\ra=c_n(h)$, where $h$ is the corresponding element of $\tilde{h}$ in $\Phi$. Then
\[
\lim_{n\to \infty} \sum_{h\in \Phi} |c_n(h)|^2=\lim_{n\to \infty}\sum_{h\in \Phi} |\la F_n, \tilde{h}\ra|^2=\lim_{n\to \infty}\sum_{\tilde{h}\in \tilde{\Phi}} |\la F_n, \tilde{h}\ra|^2=0,
\]
because $\tilde{\Phi}$ is a Riesz sequence and $\lim_{n\to \infty} \|F_n\|_{\Lp{2}}=0$.
Then $\lim_{n\to \infty} \|f_n\|_{\Lp{2}}=0$
since $\Phi$ is a Riesz sequence.
For any $\gep>0$, there exists $N\in \N$ such that $\sum_{h\in \Phi} |c_n(h)|^2\le \gep$, $\|F_n\|_{\Lp{2}}\le \gep$, and
$\|f_n\|_{\Lp{2}}\le \gep$
for all $n\ge N$.
Thus,
\[
\sum_{h\in \Psi} |c_n(h)|^2=1-\sum_{h\in \Phi} |c_n(h)|^2\ge 1-\gep
\]
and $\|g_n\|_{\Lp{2}}
=\|F_n-f_n\|_{\Lp{2}}
\le 2\gep$. This shows that the lower Riesz bound of $\Psi$ cannot be larger than $\frac{2\gep}{\sqrt{1-\gep}}$ for all $0<\gep<1$,
which contradicts that $\Psi$ is a Riesz sequence. Hence, $\Phi\cup \Psi$ must be a Riesz sequence.

Because $\Phi\cup \Psi$ is a Riesz sequence and $\si_0(\Phi\cup \Psi)=\si_1(\Phi)$ by item (iii) of \cref{thm:wbd}, $\Phi\cup\Psi$ is a Riesz basis of $\si_1(\Phi)$.
Since $\Phi(2\cdot)$ is also a Riesz basis of $\si_1(\Phi)$, for every $g\in \Phi\cup \Psi$, we define $w_g\in \ell_2(\Phi)$ to be the unique sequence satisfying $g=\sum_{h\in \Phi} w_g(h) h(2\cdot)$.
Let $\eta\in \Psi$ and define $W_\eta$ to be the closed linear span of $w_g, g\in \Phi\cup \Psi\bs \{\eta\}$. Then there exists a unique element $v_\eta\in \ell_2(\Phi)$ such that $w_\eta-v_\eta\in W_\eta$ and $v_\eta \perp W_\eta$.
Because $\Phi\cup \Psi$ is a Riesz basis of $\si_1(\Phi)$, we must have $v_\eta\ne 0$ and $\la v_\eta, w_\eta\ra_{\ell_2(\Phi)} \ne 0$.
Define $f:=\sum_{h\in \Phi} v_\eta(h) \tilde{h}(2\cdot)$. Then $f\in \si_1(\tilde{\Phi})$ and $f\ne 0$ by $v_\eta\ne 0$.
Because $\tilde{\Phi}$ is biorthogonal to $\Phi$,
we must have $\la f,\eta\ra\ne 0$ by $\la v_\eta, w_\eta\ra_{\ell_2(\Phi)} \ne 0$
and $f\perp (\Phi\cup\Psi)\bs\{\eta\}$ by
$v_\eta\perp W_\eta$.
Obviously, $d_\eta:=
\la f,\eta\ra^{-1} f\in \si_1(\tilde{\Phi})$ must satisfy \eqref{FindtPsi}.
Suppose that both $d_\eta, d'_\eta\in \si_1(\tilde{\Phi})$ satisfy \eqref{FindtPsi}. Then $d_\eta-d'_\eta \perp \Phi\cup \Psi$.
Hence, $d_\eta-d'_\eta \perp \si_0(\Phi\cup\Psi)=\si_1(\Phi)$.
Since $\tilde{\Phi}$ is biorthogonal to $\Phi$ and $d_\eta-d'_\eta\in \si_1(\tilde{\Phi})$, we must have $d_\eta=d'_\eta$. This proves the existence and uniqueness of $d_\eta$.

We now prove that $d_{\psi(\cdot-m)}=\tilde{\psi}(\cdot-m)$ for all $m\ge n_{\tilde{\psi}}$.
Since $n_{\tilde{\psi}}\ge \max(-l_{\tilde{\psi}}, \frac{n_{\tilde{\phi}}-l_{\tilde{b}}}{2})$, we observe that \eqref{I:tpsi} holds.
Note that $\fs(\tilde{\psi}(\cdot-k))=[k+l_{\tilde{\psi}},k+h_{\tilde{\psi}}]$ for $k\in \Z$.
Since $n_{\tilde{\psi}}\ge h_{\phi^L}-l_{\tilde{\psi}}$ by \eqref{ntpsi:small},
for $m\ge n_{\tilde{\psi}}$
we have $m+l_{\tilde{\psi}}\ge h_{\phi^L}$ and hence $\la \phi^L,\tilde{\psi}(\cdot-m)\ra=0$. Consequently, $\tilde{\psi}(\cdot-m)$ is perpendicular to all elements in $\Phi$.
By the same argument and $n_{\tilde{\psi}}\ge h_{\psi^L}-l_{\tilde{\psi}}$,
$\tilde{\psi}(\cdot-m)$ is also perpendicular to all elements in $\{\psi^L\}\cup \{\psi(\cdot-k) \setsp n_\psi\le k<n_{\tilde{\psi}}\}$. Now we conclude that
\eqref{FindtPsi} must hold with $\eta=\psi(\cdot-m)$ and $d_\eta=\tilde{\psi}(\cdot-m)$. This proves $d_{\psi(\cdot-m)}=\tilde{\psi}(\cdot-m)$ for all $m\ge n_{\tilde{\psi}}$.

Let $\eta \in \{\psi^L\}\cup\{\psi(\cdot-k) \setsp n_\psi\le k<n_{\tilde{\psi}}\}$. We now prove that \eqref{mtphi:cond} must hold.
By \cref{lem:W} with $n_\psi$ being replaced with $n_{\tilde{\psi}}$, we conclude from \eqref{psi:linearcomb} that
\be \label{phi:0}
\phi(2\cdot-k_0) \; \mbox{is a finite linear combination of elements in $H$ for all}\; k_0\ge m_{\tilde{\phi}},
\ee
where $H:=\{\phi(\cdot-k)\setsp k\ge n_\phi\}\cup \{\psi(\cdot-k)\setsp k\ge n_{\tilde{\psi}}\}$.
By the definition of $d_\eta$ in \eqref{FindtPsi} and $\eta \in \{\psi^L\}\cup\{\psi(\cdot-k) \setsp n_\psi\le k<n_{\tilde{\psi}}\}$, we have $d_\eta \perp H$. For any integer $k\ge m_{\tilde{\phi}}$, by the biorthogonality of $\Phi$ and $\tilde{\Phi}$, we deduce from \eqref{phi:0} and $d_\eta \perp H$ that
$c_{\eta}(\phi(\cdot-k))=
\la d_\eta, \sqrt{2} \phi(2\cdot-k)\ra=0$. This proves \eqref{mtphi:cond}.
Hence, $\tilde{\psi}^L$ has compact support and $\tilde{\Psi}$ satisfies item (iv) of \cref{thm:wbd}.
\ep

\section{Direct Approach for Constructing Biorthogonal Wavelets on $[0,\infty)$}
\label{sec:direct}

The general construction
using the classical approach in \cref{sec:classical} (in particular, \Cref{subsec:tPhi})
is often complicated and it restricts the choices of $\phi^L$ and $\tilde{\phi}^L$ in \Cref{subsec:Phi,subsec:tPhi}.
In this section, we propose a direct approach to construct all possible compactly supported biorthogonal wavelets on $[0,\infty)$ without explicitly involving the dual refinable functions $\tilde{\Phi}$ and dual wavelets $\tilde{\Psi}$.

We first address how to construct general $\Phi=\{\phi^L\}\cup\{\phi(\cdot-k) \setsp k\ge n_\phi\}$ with $\phi^L\subseteq L_2([0,\infty))$ satisfying item (i) of \cref{thm:wbd} and having compact support.

\begin{theorem}\label{thm:Phi:direct}
	Let $\phi$ be a compactly supported refinable vector function satisfying $\phi=2\sum_{k\in \Z} a(k) \phi(2\cdot-k)$ for some finitely supported matrix-valued filter $a\in \lrs{0}{r}{r}$. Take $n_\phi\in \Z$ satisfying \eqref{nphi}.
	Let $A_L$ be an $N\times N$ matrix satisfying
	\be \label{rho:AL}
	\rho(A_L)<2^{-1/2}, \quad \mbox{that is, the spectral radius of $A_L$ is less than}\; 2^{-1/2}.
	\ee
	Define an $N\times 1$ vector function $\phi^L$ by
	\be \label{phiL:direct}
	\phi^L:=\sum_{j=1}^\infty 2^{j-1} A_L^{j-1} g(2^j\cdot)\quad \mbox{with}\quad g:=2\sum_{k=n_\phi}^\infty A(k) \phi(\cdot-k),
	\ee
	where $A$ is a finitely supported sequence of $N\times (\#\phi)$ matrices.
	Then $\phi^L$ is a well-defined compactly supported vector function in $L_2([0,\infty))\cap \HH{\tau}$ for some $\tau>0$, \eqref{I:phi} holds, i.e., $\phi^L=2A_L\phi^L(2\cdot)+2\sum_{k=n_\phi}^\infty A(k) \phi(2\cdot-k)$,
	and
	 $\Phi=\{\phi^L\}\cup\{\phi(\cdot-k)\setsp k\ge n_\phi\}$ satisfies item (i) of \cref{thm:wbd}.
\end{theorem}

\bp
Since $\phi$ is a compactly supported refinable vector function with a finitely supported filter $a\in \lrs{0}{r}{r}$, by \cite[Corollary~5.8.2 or Corollary~6.3.4]{hanbook} (also see \cite[Theorem~2.2]{han03}),
the refinable vector function $\phi$ must belong to $\HH{t}$ for some $t>0$.
By our assumption in \eqref{rho:AL}, we have $\log_2\rho(A_L)<-1/2$ and hence, $t_{A_L}:=-1/2-\log_2 \rho(A_L)>0$. So,
$(0,t]\cap (0, t_{A_L})$ is nonempty.
Let $\tau\in (0,t]\cap (0, t_{A_L})$. Then $\tau>0$.
Since $\phi\subseteq \HH{\tau}$ by $0<\tau\le t$ and $n_\phi$ satisfies \eqref{nphi},
we see that $g$ must be a compactly supported vector function in $L_2([0,\infty))\cap \HH{\tau}$ and hence there exists a positive constant $C$ independent of $j$ (e.g., see \cite[(3.7)]{han10}) such that
\[
C^{-1} \|g\|_{\HH{\tau}}\le
2^{(-1/2-\tau) j}\|2^j g(2^j\cdot)\|_{\HH{\tau}}\le C\|g\|_{\HH{\tau}}
\]
for all $j\in \N\cup\{0\}$.
Applying the triangle inequality, we deduce that
\[
\left\| \sum_{j=1}^\infty 2^{j-1} A_L^{j-1} g(2^{j}\cdot)\right\|_{\HH{\tau}}\le
\sum_{j=1}^\infty 2^{j-1} \|A_L^{j-1}\| \|g(2^{j}\cdot)\|_{\HH{\tau}}
\le 2^{-1} C \|g\|_{\HH{\tau}} \sum_{j=1}^\infty
2^{(1/2+\tau)j}\|A_L^{j-1}\|.
\]
Since $\tau<t_{A_L}$, we have $t_{A_L}-\tau>0$.
For any $0<\gep<t_{A_L}-\tau$, since $\rho(A_L)=\lim_{j\to \infty} \|A_L^j\|^{1/j}$, there exists a positive constant $C_\gep$ such that $\|A_L^j\|\le C_\gep 2^{\gep j} (\rho(A_L))^j$ for all $j\in \N \cup\{0\}$. Therefore,
\[
\sum_{j=1}^\infty 2^{(1/2+\tau)j}\|A_L^{j-1}\|
\le C_\gep \sum_{j=1}^\infty 2^{(1/2+\tau)j} 2^{\gep(j-1)} (\rho(A_L))^{j-1}=
2^{(1/2+\tau)}C_\gep \sum_{j=0}^\infty 2^{(\tau+\gep-t_{A_L})j}<\infty,
\]
because $\tau+\gep-t_{A_L}<0$ by $0<\gep<t_{A_L}-\tau$. Hence we proved $\phi^L\subseteq \HH{\tau}$ with $\tau>0$ and thus, $\phi^L\subseteq L_2([0,\infty))$.
Since $g$ has compact support, $\phi^L$ in \eqref{phiL:direct} obviously has compact support and
\[
\phi^L=g(2\cdot)+2 A_L \sum_{j=1}^\infty 2^{j-1} A_L^{j-1} g(2^{j+1}\cdot)=
g(2\cdot)+2A_L \phi^L(2\cdot)=2A_L \phi^L(2\cdot)+2\sum_{k=n_\phi}^\infty A(k) \phi(2\cdot-k).
\]
This proves \eqref{I:phi}.
Hence, $\Phi=\{\phi^L\}\cup\{\phi(\cdot-k)\setsp k\ge n_\phi\}$ satisfies item (i) of \cref{thm:wbd}.
\ep

To achieve polynomial reproduction, we can simply append the vector function $\phi^\pp$ in item (ii) of \cref{prop:phicut} to the vector function $\phi^L$ in
\cref{thm:Phi:direct} to create a new $\phi^L$; or equivalently, we can append the associated refinement coefficients of $\phi^\pp$ to the coefficients $A_L$ and $A$ in \eqref{phiL:direct} instead.
Then remove redundant elements in the new $\phi^L$ by \cref{thm:rz}.
To construct orthogonal wavelets on $[0,\infty)$, we can simply replace (S1) of \cref{alg:Phi:orth} by \cref{thm:Phi:direct}.
As we discussed in \Cref{subsec:Phi}, without loss of generality, $A_L$ in \eqref{phiL:direct} can be taken as a block diagonal matrix of Jordan matrices in \eqref{jordan}.
Define a vector function $\mathring{\phi}^L$ by appending $\phi$ to $\phi^L$. Since $\phi^L$ satisfies \eqref{I:phi}, we see that $\mathring{\phi}^L$ is a compactly supported vector function associated with a finitely supported filter. Consequently, we can apply \cref{thm:integral} to compute $\int_0^1 \phi^L(x-m) \ol{\eta(x-n)}^\tp dx$ for all $m,n\in \Z$ for any compactly supported refinable vector function $\eta$.
Generalizing results on vector subdivision schemes and refinable vector functions in \cite[Chapter~5]{hanbook}, in fact we can prove through a technical argument that the condition in \eqref{rho:AL} must hold if $\Phi=\{\phi^L\}\cup\{\phi(\cdot-k) \setsp k\ge n_\phi\}$ satisfies item (i) of \cref{thm:wbd} and is a Riesz sequence.
We shall address this technical issue elsewhere.

The direct approach is to construct $\Psi$ directly.
The following result will be proved in \cref{sec:proof}.

\begin{theorem}\label{thm:direct}
	Let $(\{\tilde{\phi};\tilde{\psi}\},\{\phi;\psi\})$ be a compactly supported biorthogonal wavelet in $\Lp{2}$ with a biorthogonal wavelet filter bank $(\{\tilde{a};\tilde{b}\},\{a;b\})$ satisfying items (1)--(4) of \cref{thm:bw}.
	Suppose that $\Phi=\{\phi^L\}\cup\{\phi(\cdot-k)\setsp k \ge n_\phi\}\subseteq L_2([0,\infty))$ satisfies
	item~(i) of \cref{thm:wbd}, $\Phi$ is a Riesz sequence in $L_2([0,\infty))$, and $\phi^L$ is a compactly supported vector function in $L_2([0,\infty))\cap \HH{\tau}$ for some $\tau>0$.
	Take an integer $n_\psi\ge \max(-l_\psi,\frac{n_\phi-l_b}{2})$
	and define
	 $m_{\phi}:=\max(2n_\phi+h_{\tilde{a}},2n_\psi+h_{\tilde{b}})$ as in \eqref{mphi}.
	Let $m, n_0\in \NN$ (we often take $0\le m\le \sr(\tilde{a})$ and $n_0=0$).
	Construct $\Psi:=\{\psi^L\}\cup\{\psi(\cdot-k) \setsp k\ge n_\psi\}$ such that
	\begin{enumerate}
		\item[(i)] $\psi^L\subseteq \mbox{span}(\{\phi^L(2\cdot)\}\cup \{ \phi(2\cdot-k) \setsp n_\phi\le k<m_\phi+n_0\})$, $\psi^L$ has $m$ vanishing moments with $\vmo(\psi^L)\ge m$, and
		the set $\psi^L$ (which is regarded as in $\si_1(\Phi)/\si_0(\Phi\cup \{\psi(\cdot-k)\setsp k\ge n_\psi\})$)
		is a basis of the finite-dimensional quotient space $\si_1(\Phi)/\si_0(\Phi\cup \{\psi(\cdot-k)\setsp k\ge n_\psi\})$.
		
		\item[(ii)]  Every element in $\phi^L(2\cdot)\cup \phi^E(2\cdot)$ is a finite linear combination of elements in $\Phi\cup \Psi$, where $\phi^E:=\{\phi(\cdot-k)\setsp n_\phi\le k<m_\phi\}$.
		That is, for some integers $h_C\ge n_\phi$ and $h_D\ge n_\psi$,
		\be \label{inv:phiL}
		\left[ \begin{matrix}
			\phi^L(2\cdot)\\
			 \phi^E(2\cdot)\end{matrix}\right]
		=
		A_0 \phi^L+B_0 \psi^L+\sum_{n_\phi\le k<h_C} C(k) \phi(\cdot-k)+\sum_{n_\psi\le k<h_D}
		D(k) \psi(\cdot-k)
		\ee
		for some matrices $A_0,B_0$, $C(k)$ with $n_{\phi}\le k< h_C$, and $D(k)$ with $n_{\psi}\le k< h_D$.
	\end{enumerate}
	Define
	$n_{\tilde{\phi}}:=
	\max(m_\phi, h_C, -l_{\tilde{\phi}}, 1-l_{\tilde{a}})$ and
	$n_{\tilde{\psi}}:=\max(n_\psi, h_D,
	-l_{\tilde{\psi}}, \lceil \tfrac{n_{\tilde{\phi}}
		-l_{\tilde{b}}+1}{2}\rceil)$.
	Then we must have
	\be \label{phi2k0:2}
	 \phi(2\cdot-k_0)=\sum_{k=n_\phi}^{n_{\tilde{\phi}}-1}
	\ol{\tilde{a}(k_0-2k)}^\tp \phi(\cdot-k)+
	\sum_{k=n_\psi}^{n_{\tilde{\psi}}-1}
	\ol{\tilde{b}(k_0-2k)}^\tp \psi(\cdot-k),\qquad \forall\, m_\phi\le k_0<n_{\tilde{\phi}}.
	\ee
	Now we can rewrite/combine \eqref{inv:phiL} and \eqref{phi2k0:2} together into the following equivalent form:
	\be \label{inv:phiL:2}
	 \mathring{\phi}^L(2\cdot)=\ol{\tilde{A}_L}^\tp \mathring{\phi}^L+\ol{\tilde{B}_L}^\tp \mathring{\psi}^L,
	\ee
	where
	$\mathring{\phi}^L:=\{\phi^L\} \cup \{\phi(\cdot-k) \setsp n_\phi\le k<n_{\tilde{\phi}}\}$, $\mathring{\psi}^L:=\{\psi^L\}\cup \{\psi(\cdot-k) \setsp n_\psi\le k<n_{\tilde{\psi}}\}$, and the matrices $\tilde{A}_L, \tilde{B}_L$ are uniquely determined by $A_0, B_0, \{C(k)\}_{k=n_\phi}^{h_C-1}$, $\{D(k)\}_{k=n_\psi}^{h_D-1}$ and the filters $\tilde{a}, \tilde{b}$.
	If
	\be \label{tAL}
	\rho(\tilde{A}_L)<2^{-1/2}, \quad \mbox{that is, the spectral radius of $\tilde{A}_L$ is less than}\; 2^{-1/2},
	\ee
	then the following $\tilde{\phi}^L$ and $\tilde{\psi}^L$ are well-defined compactly supported vector functions in $L_2([0,\infty))$:
	\begin{align}
	&\tilde{\phi}^L:=\sum_{j=1}^\infty 2^{j-1} \tilde{A}_L^{j-1} \tilde{g}(2^j\cdot)\quad \mbox{with}\quad
	 \tilde{g}:=2\sum_{k=n_{\tilde{\phi}}}^{\infty} \tilde{A}(k) \tilde{\phi}(\cdot-k),
	\label{tphi:implicit}\\
	&\tilde{\psi}^L:=2\tilde{B}_L \tilde{\phi}^L(2\cdot)+2\sum_{k= n_{\tilde{\phi}}}^\infty \tilde{B}(k) \tilde{\phi}(2\cdot-k),
	\label{tpsi:implicit}
	\end{align}
	where the $(\#\mathring{\phi}^L)\times (\#\phi)$ matrices $\tilde{A}(k)$ and
	$(\#\mathring{\psi}^L)\times (\#\phi)$ matrices $\tilde{B}(k), k\in \Z$ are defined by
	\be \label{tAB}
	\tilde{A}(k):=
	\left[ \begin{matrix}
		0_{(\#\phi^L)\times (\#\phi)}\\
		\tilde{a}(k-2n_{\phi})\\
		\vdots\\
		\tilde{a}(k-2(n_{\tilde{\phi}}-1))
	\end{matrix}\right]
	\qquad \mbox{and}\qquad
	\tilde{B}(k):=
	\left[ \begin{matrix}
		0_{(\#\psi^L)\times (\#\phi)}\\
		\tilde{b}(k-2n_{\psi})\\
		\vdots\\
		\tilde{b}(k-2(n_{\tilde{\psi}}-1))
	\end{matrix}\right],\qquad k\in \Z.
	\ee
	Then $\AS_J(\tilde{\Phi};\tilde{\Psi})_{[0,\infty)}$ and $\AS_J(\Phi;\Psi)_{[0,\infty)}$ form a pair of biorthogonal Riesz bases of
	$L_2([0,\infty))$ for every $J\in \Z$, where $\tilde{\Phi}:=\{\tilde{\phi}^L\}\cup \{\tilde{\phi}(\cdot-k)\setsp k\ge n_{\tilde{\phi}}\}$ and $\tilde{\Psi}:=\{\tilde{\psi}^L\}\cup \{\tilde{\psi}(\cdot-k)\setsp k\ge n_{\tilde{\psi}}\}$.
\end{theorem}

By item (3) of \cref{thm:wbd:0}, items (i) and (ii) of \cref{thm:direct} are necessary conditions on $\Psi$ for $\AS_J(\Phi;\Psi)_{[0,\infty)}$ to be a Riesz basis of $L_2([0,\infty))$ satisfying both \eqref{I:phi} and \eqref{I:psi}.
We often take $n_\psi\in \Z$ to be the smallest integer such that $\psi(\cdot-k)\in \si_1(\Phi)$ for all $k\ge n_\psi$.
We now discuss how to construct all possible \[
\psi^L\subseteq \mbox{span}(\{\phi^L(2\cdot)\}\cup \{\phi(2\cdot-k) \setsp n_\phi\le k<m_\phi+n_0\})
\]
satisfying both items (i) and (ii) of \cref{thm:direct}.
Since $\Phi(2\cdot)$ is a Riesz basis of $\si_1(\Phi)$,
we observe that
\[
\eta=\sum_{h\in \Phi} c_\eta(h) h(2\cdot)
\quad \mbox{with}\quad
c_\eta:=\{c_\eta(h)\}_{h\in \Phi}\in \ell_2(\Phi),
\qquad \mbox{for}\; \eta\in \si_1(\Phi).
\]
Let $S:=\phi^L\cup \phi^E$ and $T:=\{\phi(\cdot-k) \setsp k\ge m_\phi\}$.
Then $\Phi=S\cup T$ and we can write $c_\eta=c_\eta \chi_S+c_\eta \chi_T$ for all $\eta\in \si_1(\Phi)$.
Define $H:=\Phi\cup \{\psi(\cdot-k) \setsp k\ge n_\psi\}$. By \eqref{refstr} and \eqref{I:phi}, we have $\si_0(H)\subseteq \si_1(\Phi)$.
Now we can find a finite subset $H_0\subseteq H$ such that $M_0:=\{c_\eta\chi_{S} \setsp \eta\in H_0\}$ is a basis of the finite-dimensional space spanned by $c_\eta\chi_S, \eta\in H$. In other words, $H_0 \cup T(2\cdot)$ is another Riesz basis of $\si_1(\Phi)$.
Next we find a generating set (not necessarily a basis) $\psi^b$ of
the finite-dimensional space
\[
W:=\{ g\in \mbox{span}(\{\phi^L(2\cdot)\}\cup\{\phi(2\cdot-k)\setsp n_\phi\le k<m_\phi+n_0\}) \setsp \vmo(g)\ge m\},
\]
which is not empty by taking $n_0$ large enough (we often set $n_0=0$). If $\psi(\cdot-k)\in W$ for some $k\in \Z$,
to have as many interior wavelets as possible, then we always keep $\psi(\cdot-k)$ in $\psi^b$.
We now find a subset $\psi^L\subseteq \psi^b$ such that $M_0 \cup M_1$ is a basis of $\R^{\#S}$, where
$M_1:=\{c_\eta\chi_S \setsp \eta\in \psi^L\}$.
Since both $\Phi(2\cdot)$ and $H_0 \cup T(2\cdot)$ are Riesz bases of $\si_1(\Phi)$,
item (i) must hold for the constructed $\psi^L$.
We now prove that $\psi^L$ satisfies item (ii) of \cref{thm:direct}.
Let $\eta\in S$. Then $c_{\eta(2\cdot)}=\td_\eta \in \R^{\#S}$.
Since $M_0 \cup M_1$ is a basis of $\R^{\#S}$,
by $M_0\cup M_1=\{c_h \chi_S \setsp h\in \psi^L\cup H_0\}$, we have
$c_{\eta(2\cdot)}=\sum_{h\in \psi^L\cup H_0} d_{\eta,h} c_h \chi_S$
for some $d_{\eta,h}\in \C$. Observing $c_h=c_h\chi_S+c_h \chi_T$, we obtain
\[
c_{\eta(2\cdot)}-\sum_{h\in \psi^L\cup H_0} d_{\eta,h} c_h=-\sum_{h\in \psi^L\cup H_0}
d_{\eta,h} c_h \chi_T.
\]
Using the refinable structure in \eqref{refstr} and \eqref{I:phi},
we conclude that every sequence $c_h$ for $h\in \psi^L\cup H$ must have only finitely many nonzero entries.
Hence, by $H_0\subseteq H$, the sequence $\sum_{h\in \psi^L\cup H_0} d_{\eta,h} c_h \chi_T$ has only finitely many nonzero entries. Therefore, we conclude from \eqref{psi:linearcomb} that
$\sum_{h\in \psi^L\cup H_0}
d_{\eta,h} c_h \chi_T$
must be a finite linear combination of elements in $H$.
Consequently, $\eta(2\cdot)-\sum_{h\in \psi^L\cup H_0}
d_{\eta,h} h$ must be a finite linear combination of elements in $H$.
This proves item (ii) of \cref{thm:direct}.

Instead of constructing $\Psi$ first
in \cref{thm:direct},  we can first construct $\tilde{\Phi}$ satisfying item (ii) of \cref{thm:wbd} below. Then construct $\Psi$ by \cref{prop:Psi} and $\tilde{\Psi}$ by \cref{prop:tPsi}.
The classical approach in \cref{sec:classical} can be improved by the following result, whose proof is given in \cref{sec:proof}.

\begin{theorem}\label{thm:direct:tPhi}
	Let $(\{\tilde{\phi};\tilde{\psi}\},\{\phi;\psi\})$ be a compactly supported biorthogonal wavelet in $\Lp{2}$
	with a biorthogonal wavelet filter bank $(\{\tilde{a}; \tilde{b}\},\{a;b\})$
	satisfying items (1)--(4) of \cref{thm:bw}.
	Suppose that $\Phi=\{\phi^L\}\cup\{\phi(\cdot-k)\setsp k \ge n_\phi\}\subseteq L_2([0,\infty))$ with $\phi^L$ having compact support satisfies
	item~(i) of \cref{thm:wbd} and $\Phi$ is a Riesz sequence in $L_2([0,\infty))$.
	Let $n_{\tilde{\phi}}$ be chosen as in item (S1) of \cref{alg:tPhi}.
	Define $N:=\#\phi^L+(n_{\tilde{\phi}}-n_\phi)(\#\phi)$ and  let $\tilde{A}_L$ be an $N\times N$ matrix satisfying \eqref{tAL}.
	For a finitely supported sequence $\tilde{A}$ of $N\times (\#\phi)$ matrices, define $\tilde{\phi}^L$ as in \eqref{tphi:implicit}.
	By \cref{thm:Phi:direct}, $\tilde{\phi}^L$ is a well-defined compactly supported vector function in $L_2([0,\infty))\cap \HH{\tau}$ for some $\tau>0$.
	If
	\be \label{filter:biorth}
	\tilde{A}_L \ol{A_L}^\tp+
	\sum_{k=n_{\tilde{\phi}}}^\infty \tilde{A}(k) \ol{A(k)}^\tp=I_{N},
	\ee
	where $A_L$ and $\{A(k)\}_{k=n_{\tilde{\phi}}}^\infty$ are augmented version in \eqref{I:phi} with $\phi^L$ being replaced by $\mathring{\phi}^L:=\{\phi^L\}\cup\{\phi(\cdot-k) \setsp n_\phi\le k<n_{\tilde{\phi}}\}$,
	then $\tilde{\Phi}$ is biorthogonal to $\Phi$ and satisfies item (ii) of \cref{thm:wbd}, where $\tilde{\Phi}:=\{\tilde{\phi}^L\}\cup\{\tilde{\phi}(\cdot-k) \setsp k\ge n_{\tilde{\phi}}\}$.
\end{theorem}

If $\tilde{\Phi}$ is biorthogonal to $\Phi$, then \eqref{filter:biorth} must hold. Hence, \eqref{filter:biorth} is a necessary condition for the biorthogonality between $\tilde{\Phi}$ and $\Phi$.
\cref{thm:direct:tPhi} generalizes \cref{alg:tPhi} for the classical approach.

\section{Biorthogonal wavelets on $[0,\infty)$ satisfying homogeneous boundary conditions}
\label{sec:bc}

In this section we study (bi)orthogonal wavelets on $[0,\infty)$ satisfying given boundary conditions.

For a polynomial $\pp(x)=\sum_{j=0}^\infty c_j x^j$, it is convenient to use the notation
$\pp(\tfrac{d}{dx})=\sum_{j=0}^\infty c_j \frac{d^j}{dx^j}$ for a differential operator. Let $\I:=[0,\infty)$.
To study wavelets on $\I$ with general homogeneous boundary conditions such as Robin boundary conditions, it is necessary for us to study nonstationary wavelet systems in $L_2(\I)$.
For subsets $\Phi_j$ and $\Psi_j$ of functions in $L_2(\I)$ with $j\in \Z$, we define
\be \label{ASPhijPsij}
\AS_J(\Phi_J;\{\Psi_j\}_{j=J}^\infty)_{\I}
:=\{ 2^{J/2}\varphi(2^J\cdot) \setsp \varphi\in \Phi_J\} \cup\{ 2^{j/2} \eta(2^j\cdot)\setsp \eta\in \Psi_j, j\ge J\}, \qquad J\in \Z.
\ee
If $\Phi_j=\Phi$ and $\Psi_j=\Psi$ for all $j\in \Z$, then $\AS_J(\Phi_J;\{\Psi_j\}_{j=J}^\infty)_{[0,\infty)}
=\AS_J(\Phi;\Psi)_{[0,\infty)}$ as in \eqref{ASPhiPsiI}.

For wavelets $\Psi_j(2^j\cdot), j\ge J$ satisfying prescribed boundary conditions, the following result shows that
all the elements in
$\AS_J(\Phi_J;\{\Psi_j\}_{j=J}^\infty)$ must satisfy the same prescribed boundary conditions.

\begin{theorem}\label{thm:bc}
	Let $J\in \Z$ and $\I=[0,\infty)$.
	Let $(\{\tilde{\phi};\tilde{\psi}\},\{\phi;\psi\})$ be a compactly supported biorthogonal wavelet in $\Lp{2}$ satisfying items (1)--(4) of \cref{thm:bw}.
	Let $n_\phi\in \Z$ such that
	$\fs(\eta(\cdot-k_0))\subseteq \I$
	for all $k_0\ge n_\phi$ and $\eta\in \phi\cup\psi\cup\tilde{\phi}\cup \tilde{\psi}$.
	Let
	$\{\phi_J^L, \tilde{\phi}_J^L\}\cup \{\psi^L_j, \tilde{\psi}_j^L\}_{j=J}^\infty\subseteq L_2(\I)$ have compact support and satisfy
	$\lim_{j\to \infty} 2^{-j} h_{\tilde{\psi}^L_j}=0$, where $[l_{\tilde{\psi}^L_j}, h_{\tilde{\psi}^L_j}]:=\fs(\tilde{\psi}_j^L)$.
	Define $\Phi_J$ and $\Psi_j, j\ge J$  by
	\be \label{def:PhiPsij}
	\Phi_J:=\{\phi^L_J\}\cup \{\phi(\cdot-k) \setsp k \ge n_\phi\},\qquad
	\Psi_j:=\{\psi^L_j\}\cup \{ \psi(\cdot-k) \setsp k\ge n_\phi\}
	\ee
	and
	\[
	\tilde{\Phi}_J:=\{\tilde{\phi}^L_J\}\cup \{\tilde{\phi}(\cdot-k) \setsp k\ge n_\phi\},
	\qquad
	\tilde{\Psi}_j:=\{\tilde{\psi}^L_j\}\cup \{ \tilde{\psi}(\cdot-k) \setsp k\ge n_\phi\}.
	\]
	Suppose that $\AS_J(\Phi_J; \{\Psi_j\}_{j=J}^\infty)_{\I}$ and $\AS_J(\tilde{\Phi}_J; \{\tilde{\Psi}_j\}_{j=J}^\infty)_{\I}$ form a pair of biorthogonal Riesz bases in $L_2(\I)$.
	Let $\pp_0,\ldots,\pp_\ell \in \PL_{m-1}$ be polynomials of degree less than $m$.
	Suppose that each function $\eta\in \AS_J(\Phi_J; \{\Psi_j\}_{j=J}^\infty)_{\I}$
	has continuous derivatives of all orders less than $m$ on $[0,\gep_\eta)$ for some $\gep_\eta>0$.
	If all the wavelet functions in $\AS_J(\Phi_J;\{\Psi_j\}_{j=J}^\infty)_{\I}$
	satisfy the following homogeneous boundary conditions prescribed by $\pp_0,\ldots,\pp_{\ell}$ as follows:
	\be \label{bc:psi}
	\pp_0(\tfrac{d}{dx})(\eta(2^j x))|_{x=0}=\cdots=
	 \pp_\ell(\tfrac{d}{dx})(\eta(2^jx))|_{x=0}=0\quad \forall\;
	\eta\in \Psi_j, j\ge J,
	\ee
	then the refinable functions $\Phi_J(2^J\cdot)$ in $\AS_J(\Phi_J;\{\Psi_j\}_{j=J}^\infty)_{\I}$ must satisfy the same boundary conditions
	\be \label{bc:phi}
	 \pp_0(\tfrac{d}{dx})(\varphi(2^Jx))|_{x=0}
	=\cdots=
	 \pp_\ell(\tfrac{d}{dx})(\varphi(2^Jx))|_{x=0}=0\quad \forall\;
	\varphi\in \Phi_J.
	\ee
	That is, if all the elements in 
	$\{2^{j/2}\eta(2^j\cdot) \setsp \eta\in \Psi_j, j\ge J\}$
	satisfy the homogeneous boundary conditions in \eqref{bc:psi},
	then all elements in $\AS_J(\Phi_J; \{\Psi_j\}_{j=J}^\infty)_{\I}$ must satisfy the same boundary conditions.
\end{theorem}

\bp
Let $\gep>0$ and consider functions $f\in L_2([2\gep,\infty))$.
Since
\[
\AS_J(\Phi_J; \{\Psi_j\}_{j=J}^\infty)_{\I} \quad \mbox{and}\quad
\AS_J(\tilde{\Phi}_J; \{\tilde{\Psi}_j\}_{j=J}^\infty)_{\I}
\]
form a pair of biorthogonal Riesz bases of $L_2(\I)$, we have
\[
f(x)=\sum_{\varphi\in \Phi}2^J \la f, \tilde{\varphi}(2^J\cdot)\ra \varphi(2^J x)+
\sum_{j=J}^\infty \sum_{\eta\in \Psi_j} 2^j \la f, \tilde{\eta}(2^j\cdot)\ra \eta(2^jx).
\]
Because $\lim_{j\to \infty} 2^{-j}h_{\tilde{\psi}^L_j}=0$ and $\supp(\tilde{\psi}_j^L(2^j\cdot))\subseteq [0, 2^{-j} h_{\tilde{\psi}^L_j}]$, there exists $\tilde{J}_\gep\in \N$ such that
\be \label{supptpsiL}
\supp(\tilde{\psi}^L_j(2^j\cdot))\subseteq [0,2\gep], \qquad \forall\, j\ge \tilde{J}_\gep.
\ee
Since $\phi,\psi,\tilde{\phi}$ and $\tilde{\psi}$ have compact support, we assume that all of them
are supported inside $[-N,N]$ for some $N\in \N$. Since
$\supp(\eta(2^j\cdot-k))\subseteq [2^{-j}(k-N),2^{-j}(k+N)]$ for $\eta\in \phi\cup\psi\cup\tilde{\phi}\cup\tilde{\psi}$,
we observe
\be \label{tpsi:supp}
\supp(\tilde{\psi}(2^j\cdot-k))\subseteq [0,2\gep],\qquad \forall\; n_\phi\le k\le 2^{j+1}\gep-N
\ee
and
\be \label{psi:supp}
\supp(\phi(2^j\cdot-k))\cup\supp(\psi(2^j\cdot-k))\subseteq [\gep,\infty),\qquad \forall\; k\ge 2^j\gep+N.
\ee
Let $J_\gep\in \N$ such that $J_\gep \ge \max(\tilde{J}_\gep, \log_2 \frac{2N}{\gep})$. For $j\ge J_\gep$ and $k\in \Z$, then either $k\le 2^{j+1}\gep-N$ or $k\ge 2^j\gep+N$ must hold. Consequently, one of \eqref{tpsi:supp} and \eqref{psi:supp} must hold for all $k\ge n_\phi$.
Hence, by $\supp(f)\subseteq [2\gep,\infty)$ and $J_\gep\ge \tilde{J}_\gep$,
we deduce from \eqref{tpsi:supp} and \eqref{psi:supp} that
\be \label{vanishing}
\la f, \tilde{\eta}(2^j\cdot)\ra \eta(2^jx)=0 \qquad \forall\; x\in [0,\gep),  \eta\in \Psi_j, j\ge J_\gep.
\ee
From \eqref{psi:supp}, for $x\in [0,\gep)$, we have $\la f, \tilde{\phi}(2^J\cdot-k)\ra \phi(2^Jx-k)=0$ for all $k\ge 2^J\gep+N$ and $\la f,\tilde{\psi}(2^j\cdot-k)\ra\psi(2^jx-k)=0$ for all $k\ge 2^j\gep+N$.
Consequently, by \eqref{psi:supp} and \eqref{vanishing}, we obtain
\be \label{f:expr:0}
\begin{split}
	f(x)=
	&\la f,(\tilde{\phi}^L_J)_{J;0}\ra
	(\phi^L_J)_{J;0}(x)+
	\sum_{k=n_\phi}^{\lfloor 2^j\gep +N\rfloor}
	\la f, \tilde{\phi}_{J;k}\ra
	\phi_{J;k}(x)\\
	&\qquad\qquad +\sum_{j=J}^{J_\gep-1} \Big(\la f, (\tilde{\psi}^L_j)_{j;0}\ra
	(\psi^L_j)_{j;0}(x)+\sum_{k=n_\phi}^{ \lfloor 2^j\gep+N \rfloor} \la f, \tilde{\psi}_{j;k}\ra
	\psi_{j;k}(x)\Big)
\end{split}
\ee
for almost every $x\in [0,\gep)$, where $\psi_{j;k}:=2^{j/2}\psi(2^j\cdot-k)$. By assumption,
each function $\eta\in \AS_J(\Phi_J; \{\Psi_j\}_{j=J}^\infty)_{\I}$
is continuous on $[0,\gep_\eta)$ and has continuous derivatives of all orders less than $m$ on $[0,\gep_\eta)$ for some $\gep_\eta>0$. Because there are only finitely many terms in \eqref{f:expr:0}, there must exist $0<\gep_0<\gep$ such that \eqref{f:expr:0} holds for all $x\in [0,\gep_0)$ and all terms in \eqref{f:expr:0} have continuous derivatives of all orders less than $m$ on $[0,\gep_0)$.
Applying our assumption in
\eqref{bc:psi} and using the fact $\supp(f)\subset [2\gep,\infty)$, we conclude from \eqref{f:expr:0} that $\pp_i(\tfrac{d}{dx}) f(x)|_{x=0}
=0$ and
\be \label{BC:phi}
2^J \la f, \tilde{\phi}^L_J(2^J\cdot)\ra \pp_i(\tfrac{d}{dx})\phi^L_J(2^Jx)|_{x=0}
+\sum_{k=n_\phi}^{\lfloor 2^j\gep +N\rfloor}
2^J \la f, \tilde{\phi}(2^J\cdot-k)\ra \pp_i(\tfrac{d}{dx})\phi(2^Jx-k)|_{x=0}
=0
\ee
for all $i=0,\ldots,\ell$.
By the choice of $n_\phi$ satisfying \eqref{nphi}, $\supp(\phi(\cdot-k))\subseteq [1,\infty)$ for all $k\ge n_\phi+1$ and hence trivially
$\pp_i(\tfrac{d}{dx})\phi(2^Jx-k)|_{x=0}=0$.
For simplicity, we may group $\phi(\cdot-n_\phi)$ into $\phi^L$. Hence, \eqref{BC:phi} becomes
\be \label{fgep}
\la f, \tilde{\phi}^L_J(2^J\cdot)\ra \pp_i(\tfrac{d}{dx})\phi^L_J(2^Jx)|_{x=0}=0,
\qquad \forall\; i=0,\ldots, \ell\quad \mbox{and}\quad f\in L_2([2\gep,\infty)).
\ee
In particular, \eqref{fgep} must hold with $\gep=0$.
Since $\tilde{\phi}^L_J$ must be a Riesz sequence, the mapping $L_2([0,\infty))\rightarrow \C^{\#{\tilde{\phi}^L_J}}$ with $f\mapsto \la f, \tilde{\phi}^L_J(2^J\cdot)\ra$ is onto.
Consequently, we deduce from \eqref{fgep} that \eqref{bc:phi} holds for all $\varphi\in \phi^L_J$, from which we conclude that \eqref{bc:phi} holds for all $\varphi\in \Phi_J$.
\ep

As a direct consequence of \cref{thm:bc}, we now claim that any orthogonal wavelet basis on $[0,\infty)$ satisfying boundary conditions often cannot have high vanishing moments.

\begin{theorem}\label{thm:ow:vm}
	Let $J\in \Z$ and $\I=[0,\infty)$.
	Let $\{\phi;\psi\}$ be a compactly supported orthogonal wavelet in $\Lp{2}$.
	Let
	 $\{\phi_J^L\}\cup\{\psi_j^L\}_{j=J}^\infty\subseteq L_2(\I)$
	have compact support
	and satisfy $\lim_{j\to \infty} 2^{-j} h_{\psi^L_j}=0$, where $[l_{\psi^L_j}, h_{\psi^L_j}]:=\fs(\psi^L_j)$.
	Let $n_\phi\in \Z$ such that
	$\fs(\eta(\cdot-k_0))\subseteq \I$
	for all $k_0\ge n_\phi$ and $\eta\in \phi\cup\psi$.
	Define $\Phi_J$ and $\Psi_j, j\ge J$ as in \eqref{def:PhiPsij}.
	Let $\pp_0,\ldots,\pp_\ell \in \PL_{m-1}$ and define $n_{\bcm}$ to be the largest nonnegative integer such that
	\be \label{nBC}
	\pp_i(\tfrac{d}{dx}) (x^j)|_{x=0}=0\qquad \forall\, i=0,\ldots,\ell\quad \mbox{and}\quad j=0,\ldots,n_{\bcm}-1.
	\ee
	Suppose that each function $\eta\in \AS_J(\Phi_J; \{\Psi_j\}_{j=J}^\infty)_{\I}$
	has continuous derivatives of all orders less than $m$ on $[0,\gep_\eta)$ for some $\gep_\eta>0$.
	If $\AS_J(\Phi_J; \{\Psi_j\}_{j=J}^\infty)_{\I}$ is an orthonormal basis of $L_2(\I)$ such that the boundary conditions in \eqref{bc:psi} are satisfied,
	then for any $j_0\in \N$, there exists  at least one element $\eta_{j_0}\in \cup_{j=j_0}^\infty \Psi_j$ such that $\eta_{j_0}$ has no more than $n_{\bcm}$
	vanishing moments, i.e., $\vmo(\eta_{j_0})\le n_B$.
\end{theorem}

\bp Suppose not. Then there exists $j_0\in \N$ such that all the elements in $\cup_{j=j_0}^\infty \Psi_j$ have $n_{\bcm}+1$ vanishing moments.
Since $\phi$ and $\psi$ have compact support, we can assume that $\phi$ and $\psi$ are supported inside $[-N,N]$ for some $N\in \N$.
Define $\gep:=2^{1-j_0}N>0$ and let $h: [0,\infty)\rightarrow \R$ be a compactly supported continuous function such that $h(x)=1$ for all $x\in [0,2\gep]$.
Define $f(x):=x^{n_{\bcm}} h(x)$. Then $f\in L_2(\I)$ has compact support and $f(x)=x^{\bcm}$ for all $x\in [0,2\gep]$.
Noting that
$\supp(\eta(2^j\cdot-k))\subseteq [2^{-j}(k-N),2^{-j}(k+N)]$ for all $\eta\in \phi\cup\psi$,
we can easily verify that
\eqref{psi:supp} holds and
\be \label{psi:supp:2}
\supp(\psi(2^j\cdot-k))\subset [0,2\gep],\qquad \forall\; n_\phi\le k\le 2^{j+1}\gep-N.
\ee
By $\lim_{j\to \infty} 2^{-j} h_{\psi^L_j}=0$, there exists an integer $\tilde{J}_\gep\ge j_0$ such that $\supp(\psi_j^L(2^j\cdot))\subseteq [0,2\gep]$ for all $j\ge \tilde{J}_\gep$.
Since $f(x)=x^{\bcm}$ for $x\in [0,2\gep]$ and all elements in $\cup_{j=j_0}^\infty \Psi_j$ have $n_{\bcm}+1$ vanishing moments, we have $\la f, \psi_j^L(2^j\cdot)\ra=0$ for all $j\ge \tilde{J}_\gep$.
Let $J_\gep\in \N$ such that $J_\gep \ge \max(\tilde{J}_\gep, \log_2 \frac{2N}{\gep})$. Note that $J_\gep \ge j_0$ by $\gep=2^{1-j_0} N$.
For all $j\ge J_\gep$,
one of \eqref{psi:supp} and \eqref{psi:supp:2} must hold.
Now it follows from the same argument as in the proof of \cref{thm:bc} that
\be \label{f:expr:1}
\begin{split}
	f(x)=
	&\la f,(\phi^L_J)_{J;0}\ra
	(\phi^L_J)_{J;0}(x)+
	\sum_{k=n_\phi}^{\lfloor 2^j\gep +N\rfloor}
	\la f, \phi_{J;k}\ra
	\phi_{J;k}(x)\\
	&\qquad\qquad +\sum_{j=J}^{J_\gep-1} \Big(\la f, (\psi^L_j)_{j;0}\ra
	(\psi^L_j)_{j;0}(x)+\sum_{k=n_\phi}^{ \lfloor 2^j\gep+N \rfloor} \la f, \psi_{j;k}\ra
	\psi_{j;k}(x)\Big)
\end{split}
\ee
for almost every $x\in [0,\gep)$, and there exists $0<\gep_0<\gep$ such that
\eqref{f:expr:1} holds for all $x\in [0,\gep_0)$ and all terms in \eqref{f:expr:1} have continuous derivatives of all orders less than $m$ on $[0,\gep_0)$.

On the other hand, all the conditions in \cref{thm:bc} are satisfied. Consequently, all the elements in $\AS_J(\Phi_J; \{\Psi_j\}_{j=J}^\infty)_{\I}$ must satisfy the prescribed homogeneous boundary conditions.
Therefore, we deduce from \eqref{f:expr:1} that $\pp_i(\frac{d}{dx})f(x)|_{x=0}=0$ for all $i=0,\ldots,\ell$. Because $f(x)=x^{n_{\bcm}}$ for all $x\in [0,2\gep]$, we conclude that \eqref{nBC} holds with $n_{\bcm}$ being replaced by $n_{\bcm}+1$,
which contradicts the definition of the maximum integer $n_{\bcm}$ in \eqref{nBC}. This proves the claim.
\ep

A popular choice of homogeneous boundary conditions in the literature is
\be \label{bc:standard}
\pp_0(\tfrac{d}{dx})=\frac{d^{j_0}}{d x^{j_0}},\quad \ldots,\quad
\pp_\ell(\tfrac{d}{dx})=\frac{d^{j_\ell}}{d x^{j_\ell}} \quad \mbox{with}\quad 0\le j_0<\ldots<j_\ell<m.
\ee
Moreover, the particular choice $j_0=0,\ldots, j_\ell=\ell$ in \eqref{bc:standard} is commonly used
in the
variational formulation of the boundary value problems in numerical partial differential equations, where the derivatives are in the weak/distributional sense and boundary values at $0$ are interpreted in the trace sense.
Spline scalar wavelets on $[0,1]$ satisfying homogeneous Dirichlet boundary conditions
have been addressed in \cite{cer19,cf12,ds98,ds10,hanbook,hm18,jia09} and references therein.

Let $(\AS_J(\tilde{\Phi};\tilde{\Psi})_
{[0,\infty)}, \AS_J(\Phi;\Psi)_{[0,\infty)})$
be a (stationary) biorthogonal wavelet on $[0,\infty)$.
Let $\pp_0,\ldots,\pp_\ell $ $\in \PL_{m-1}$. It is easy to check that \eqref{bc:psi} holds with $\Psi_j:=\Psi$ for all $j\ge J$ if and only if $\pp(\frac{d}{dx})\eta(x)|_{x=0}=0$ for all $\eta\in \Psi$ and $\pp\in \mathcal{P}:=\mbox{span}\{\pp_0(2^j\cdot),\ldots, \pp_\ell(2^j\cdot)\setsp j\ge J\}$. Note that $\mathcal{P}$ is generated by all the nonzero monomial terms in the polynomials $\pp_0,\ldots,\pp_\ell$. Hence, if $\mbox{span}\{\pp_0,\ldots,\pp_\ell\}$ does not have a basis of monomials as in \eqref{bc:standard}, then the dimension of $\mathcal{P}$ will be greater than $\ell+1$.
To avoid increasing the number of boundary conditions, it is necessary to consider nonstationary wavelets in \eqref{ASPhijPsij}.
To construct a biorthogonal wavelet
$(\AS_J(\tilde{\Phi};\tilde{\Psi})_{[0,\infty)},
\AS_J(\Phi;\Psi)_{[0,\infty)})$
on $[0,\infty)$ such that
\[
\frac{d^{j_0}}{d x^{j_0}} \eta(x)|_{x=0}=\cdots=
\frac{d^{j_\ell}}{d x^{j_\ell}} \eta(x)|_{x=0}=0,\qquad \forall\, \eta\in \Psi,
\]
by \cref{thm:bc} and the refinable structure in \eqref{I:phi} and \eqref{I:psi}, it is necessary and sufficient that
\be \label{bc:phi:standard}
\frac{d^{j_0}}{d x^{j_0}} \varphi(x)|_{x=0}=\cdots=
\frac{d^{j_\ell}}{d x^{j_\ell}} \varphi(x)|_{x=0}=0,\qquad \forall\, \varphi\in \Phi.
\ee
Consequently, \eqref{bc:phi:standard} holds if and only if all the elements
in $\AS_J(\Phi;\Psi)_{[0,\infty)}$ satisfies the same prescribed homogeneous boundary conditions given by \eqref{bc:standard}.
For $\Phi$ satisfying the boundary conditions in \eqref{bc:phi:standard}, to achieve high approximation orders near the endpoint $0$, it is important to have
\be \label{polyprop:bc}
\PL_{m-1}\chi_{[0,\infty)} \subseteq \mbox{span}\{\pp_0\chi_{[0,\infty)},\ldots,
\pp_\ell\chi_{[0,\infty)}\}+\si_0(\Phi)\quad \mbox{with}\quad \pp_0,\ldots, \pp_\ell \; \mbox{in}\; \eqref{bc:standard}.
\ee
For any $\Phi$ satisfying
item (i) of \cref{thm:wbd},
we can easily obtain a new $\Phi^{bc}$ satisfying item (i) of \cref{thm:wbd} and the boundary conditions in \eqref{bc:phi:standard}.

\begin{prop} \label{prop:mod}
	Let $\Phi=\{\phi^L\}\cup\{\phi(\cdot-k) \setsp k\ge n_\phi\}\subseteq L_2([0,\infty))$ satisfy item (i) of \cref{thm:wbd}, where $\phi^L$ and $\phi$ have compact support. Let $\pp(x):=(x^{j_0},\ldots,x^{j_\ell})^\tp$ with $\{j_0,\ldots,j_\ell\}\subseteq \{0,1,\ldots, m-1\}$. Suppose that every element $\eta\in \Phi$ has continuous derivatives of all order less than $m$ on $[0,\gep_\eta)$ for some $\gep_\eta>0$.
	Then there exists an invertible $(\#\phi^L+\#\phi)\times (\#\phi^L+\#\phi)$ matrix $C_{\phi^L}$ such that
	\be \label{mod:phiL}
	M_\pp(\phi^{L,I})=\{0\}\; \mbox{and $M_\pp(\phi^{L,E})$ is a basis of}\; \mbox{span} (M_\pp(\Phi))
	\;\; \mbox{with}\;\;
	\begin{bmatrix}
		\phi^{L,E}\\
		\phi^{L,I}\end{bmatrix}
	:=C_{\phi^L} \begin{bmatrix}
		\phi^L\\
		\phi(\cdot-n_\phi)\end{bmatrix},
	\ee
	where $M_\pp(S):=\{ \pp(\frac{d}{dx})\eta(x)|_{x=0} \setsp \eta\in S\}$ for $S\subseteq \si_0(\Phi)$.
	Then
	\[
	\Phi^{bc}:=\{\phi^{L,I}\}\cup\{
	\phi(\cdot-k) \setsp k\ge n_\phi+1\}
	\]
	satisfies item (i) of \cref{thm:wbd},
	the homogeneous boundary conditions $\pp(\tfrac{d}{dx})\eta(x)|_{x=0}=0$ for all  $\eta\in \Phi^{bc}$, and $\si_0(\Phi^{bc})=\{\eta\in \si_0(\Phi) \setsp \pp(\frac{d}{dx})\eta(x)|_{x=0}=0\}$.
\end{prop}

\bp Note that $\fs(\phi(\cdot-k))\subseteq [1,\infty)$ for all $k\ge n_\phi+1$. Trivially, $\pp(\frac{d}{dx})\phi(x-k)|_{x=0}=0$ for all $k\ge n_\phi+1$.
Thus, $M_\pp(\Phi)=M_\pp(\phi^L\cup \phi(\cdot-n_\phi))$, which is a finite subset of $\R^{\#\pp}$.
Therefore, there exists an invertible matrix $C_{\phi^L}$ such that \eqref{mod:phiL} holds.
Hence, all the elements in $\Phi^{bc}$ satisfy the homogeneous boundary conditions prescribed by $\pp$.
Since $\si_0(\phi^{L,E}\cup \Phi^{bc})=\si_0(\Phi)$ and $C_{\phi^L}$ is invertible, by \eqref{nphi} and \eqref{I:phi}, we have
\be \label{expr:eta}
\eta=\sum_{f\in \phi^{L,E}} c_\eta(f) f(2\cdot)+
\sum_{f\in \phi^{L,I}} c_\eta(f) f(2\cdot)
+\sum_{k=n_\phi+1}^\infty c_\eta(\phi(\cdot-k)) \phi(2\cdot-k),\qquad \eta\in \si_0(\Phi).
\ee
Because $M_\pp(g)=\{0\}$ for all $g\in \Phi^{bc}$, for $\eta\in \si_0(\Phi)$ with $M_{\pp}(\eta)=\{0\}$ in \eqref{expr:eta}, we have
\[
\{0\}=M_\pp(\eta)=
[c_\eta(f)]_{f\in \phi^{L,E}} M_\pp(\phi^{L,E}(2\cdot))
=[c_\eta(f)]_{f\in \phi^{L,E}} \diag(2^{j_0},\ldots,2^{j_\ell}) M_\pp(\phi^{L,E}).
\]
Since $M_\pp(\phi^{L,E})$ is linearly independent, the above identity forces $c_\eta(f)=0$ for all $f\in \phi^{L,E}$.
By \eqref{expr:eta},
$\eta=
\sum_{f\in \phi^{L,I}} c_\eta(f) f(2\cdot)
+\sum_{k=n_\phi+1}^\infty c_\eta(\phi(\cdot-k)) \phi(2\cdot-k)$ for all $\eta\in \Phi^{bc}$. This proves \eqref{I:phi} and \eqref{nphi} for $\Phi^{bc}$ with $n_\phi$ being replaced by $n_{\phi}+1$. Hence, $\Phi^{bc}$ satisfies item (i) of \cref{thm:wbd}. The identity $\si_0(\Phi^{bc})=\{\eta\in \si_0(\Phi) \setsp \pp(\frac{d}{dx})\eta(x)|_{x=0}=0\}$ follows directly from \eqref{mod:phiL} and \eqref{expr:eta}.
\ep

\section{Orthogonal and Biorthogonal Wavelets on Bounded Intervals}
\label{sec:wbd}

In this section we discuss how to construct locally supported biorthogonal wavelets on a bounded interval $[0,N]$ with $N\in \N$ from compactly supported biorthogonal wavelets on $[0,\infty)$.

Recall that $\NN:=\N\cup\{0\}$ and $f_{j;k}:=2^{j/2}f(2^j\cdot-k)$ for $j,k\in \Z$. We remind the reader that a vector function is also used as an ordered set and vice versa throughout the paper.
Using the classical approach in \cref{sec:classical} or the direct approach in \cref{sec:direct} for constructing (bi)orthogonal wavelets on $[0,\infty)$, we now discuss how to construct a locally supported (bi)orthogonal wavelet in $L_2([0,N])$ with $N\in \N$ from a compactly supported biorthogonal wavelet in $\Lp{2}$.
We shall provide a detailed proof in \cref{sec:proof} for
the following result, which is often employed but without a proof in the literature.

\begin{theorem}\label{thm:bw:0N}
	Let $(\{\tilde{\phi}; \tilde{\psi}\},\{\phi;\psi\})$ be a compactly supported biorthogonal wavelet in $\Lp{2}$ with a biorthogonal wavelet filter bank $(\{\tilde{a};\tilde{b}\},\{a;b\})$
	satisfying items (1)--(4) of \cref{thm:bw}. A locally supported biorthogonal wavelet on the interval $[0,N]$ with $N\in \N$ can be constructed as follows:
	\begin{enumerate}
		\item[(S1)] From the biorthogonal wavelet $(\{\tilde{\phi}; \tilde{\psi}\},\{\phi;\psi\})$ in $\Lp{2}$, use either the classical approach in \cref{sec:classical} or the direct approach in \cref{sec:direct} to
		construct compactly supported $\Phi,\Psi,\tilde{\Phi},\tilde{\Psi}$ as in
		\eqref{PhiPsiI} and \eqref{tPhiPsiI}
		such that
		 $(\AS_J(\tilde{\Phi};\tilde{\Psi})_{[0,\infty)},\AS_J(\Phi;\Psi)_{[0,\infty)})$ is a pair of biorthogonal Riesz bases in $L_2([0,\infty))$ for every $J\in \NN$ and satisfies items (i)--(iv) of \cref{thm:wbd}.

		\item[(S2)] Similarly, perform item (S1) to the (flipped) biorthogonal wavelet
		 $(\{\tilde{\mathring{\phi}};\tilde{\mathring{\psi}}\}, \{\mathring{\phi};\mathring{\psi}\})$
		in $\Lp{2}$ to construct compactly supported $\mathring{\Phi},\mathring{\Psi},\tilde{\mathring{\Phi}},\tilde{\mathring{\Psi}}$, where
		\be \label{mphipsi}
		 \mathring{\phi}:=\phi(-\cdot),\quad
		 \mathring{\psi}:=\psi(-\cdot),\quad
		 \tilde{\mathring{\phi}}:=\tilde{\phi}(-\cdot),\quad
		 \tilde{\mathring{\psi}}:=\tilde{\psi}(-\cdot),
		\ee
		such that $(\AS_J(\tilde{\mathring{\Phi}};\tilde{\mathring{\Psi}})_{[0,\infty)},
		 \AS_J(\mathring{\Phi};\mathring{\Psi})_{[0,\infty)})$ is a pair of biorthogonal Riesz bases in $L_2([0,\infty))$ for every $J\in \NN$ and satisfies items (i)--(iv) of \cref{thm:wbd} similarly.
		
		\item[(S3)]
		Let $J_0$ be the smallest nonnegative integer such that
		\be \label{J0}
		\max(h_A+n_{\mathring{\phi}}, h_B+n_{\mathring{\phi}}, h_{\mathring{A}}+n_\phi, h_{\mathring{B}}+n_\phi, 2n_\phi+2n_{\mathring{\phi}}-2,
		 2n_\psi+2n_{\mathring{\psi}}-2)\le 2^{J_0+1}N
		\ee
		and all elements in $\phi^L\cup\psi^L \cup\mathring{\phi}^L\cup \mathring{\psi}^L$ are supported inside $[0, 2^{J_0}N]$, where
		\be \label{lAB}
		[l_A,h_A]:=\fs(A), \quad [l_B, h_B]:=\fs(B),\quad
		 [l_{\tilde{A}},h_{\tilde{A}}]:=\fs(\tilde{A}), \quad
		[l_{\tilde{B}}, h_{\tilde{B}}]:=\fs(\tilde{B})
		\ee
		for the finitely supported filters $A,B,\tilde{A},\tilde{B}$ in \eqref{I:phi}, \eqref{I:psi}, \eqref{I:phi:dual}, and \eqref{I:psi:dual}, respectively. The integers $h_{\mathring{A}}, h_{\mathring{B}}$ and $h_{\tilde{\mathring{A}}}, h_{\tilde{\mathring{B}}}$ are defined similarly.
		
		\item[(S4)] Let $\tilde{J}_0$ be the smallest nonnegative integer such that
		\be \label{tJ0}
		 \max(h_{\tilde{A}}+n_{\tilde{\mathring{\phi}}},
		 h_{\tilde{B}}+n_{\tilde{\mathring{\phi}}},
		 h_{\tilde{\mathring{A}}}+n_{\tilde{\phi}},
		 h_{\tilde{\mathring{B}}}+n_{\tilde{\phi}}, 2n_{\tilde{\phi}}+2n_{\tilde{\mathring{\phi}}}-2,
		 2n_{\tilde{\psi}}+2n_{\tilde{\mathring{\psi}}}-2)
		\le 2^{\tilde{J_0}+1}N,
		\ee
		all elements in $\tilde{\phi}^L\cup \tilde{\psi}^L \cup\tilde{\mathring{\phi}}^L\cup \tilde{\mathring{\psi}}^L$ are supported inside $[0, 2^{\tilde{J}_0}N]$, and for all $j\ge \tilde{J}_0$,
		\be \label{disjointPhiPsi}
		\{\phi^L_{j;0}, \psi^L_{j;0}\} \perp \{\tilde{\phi}^R_{j; 2^j N-N}, \tilde{\psi}^R_{j;2^j N-N}\}
		\quad \mbox{and}\quad
		\{\tilde{\phi}^L_{j;0}, \tilde{\psi}^L_{j;0}\} \perp \{\phi^R_{j; 2^j N-N}, \psi^R_{j;2^j N-N}\},
		\ee
		where the right boundary refinable functions and right boundary wavelets are defined by
		\be \label{reflection}
		 \phi^R:=\mathring{\phi}^L(N-\cdot),\quad
		 \psi^R:=\mathring{\psi}^L(N-\cdot),\quad
		 \tilde{\phi}^R:=\tilde{\mathring{\phi}}^L(N-\cdot),\quad
		 \tilde{\psi}^R:=\tilde{\mathring{\psi}}^L(N-\cdot).
		\ee
	\end{enumerate}
	Without loss of generality, we assume $J_0\le \tilde{J}_0$.
	Then the following statements hold:
	\begin{enumerate}
		\item[(1)] for each $j\ge J_0$,
		there exist matrices $A_j$ and $B_j$ such that
		$\Phi_j=A_j\Phi_{j+1}$ and $\Psi_j=B_j \Phi_{j+1}$
		with $\#\Psi_{j}=\#\Phi_{j+1}-\#\Phi_j=2^j N(\#\phi)$ and $\#\Phi_j=\#\phi^L+\#\mathring{\phi}^L
		 +(2^jN-n_{\mathring{\phi}}-n_\phi+1)(\#\phi)$,
		where
		\begin{align}
		&\Phi_j:=\{\phi^L_{j;0}\} \cup \{\phi_{j;k} \setsp n_\phi\le k\le 2^jN-n_{\mathring{\phi}}\}\cup \{ \phi^R_{j;2^jN-N} \}, \label{Phij}\\
		&\Psi_j:=\{ \psi^L_{j;0} \} \cup \{\psi_{j;k} \setsp n_\psi\le k\le 2^jN-n_{\mathring{\psi}}\}\cup \{ \psi^R_{j;2^jN-N}\}. \label{Psij}
		\end{align}
		\item[(2)] for every $j\ge \tilde{J}_0$, there exist matrices $\tilde{A}_j$ and $\tilde{B}_j$ such that
		 $\tilde{\Phi}_j=\tilde{A}_j\tilde{\Phi}_{j+1}$ and $\tilde{\Psi}_j=\tilde{B}_j \tilde{\Phi}_{j+1}$ hold with
		$\#\tilde{\Psi}_j=\#\Psi_j=2^j N(\#\phi)$ and $\#\tilde{\Phi}_j=\#\Phi_j$, where
		\begin{align}
		 &\tilde{\Phi}_j:=\{\tilde{\phi}^L_{j;0}\} \cup \{\tilde{\phi}_{j;k} \setsp n_{\tilde{\phi}}\le k\le 2^jN-n_{\tilde{\mathring{\phi}}}\}\cup \{ \tilde{\phi}^R_{j;2^jN-N}\}, \label{tPhij}\\
		&\tilde{\Psi}_j:= \{ \tilde{\psi}^L_{j;0}\} \cup \{ \tilde{\psi}_{j;k} \setsp n_{\tilde{\psi}}\le k\le 2^jN-n_{\tilde{\mathring{\psi}}}\}\cup \{ \tilde{\psi}^R_{j;2^jN-N} \}. \label{tPsij}
		\end{align}
		\item[(3)]
		For every $J\ge \tilde{J}_0$,
		$(\tilde{\cB}_J, \cB_J)$ forms a pair of biorthogonal Riesz bases of $L_2([0,N])$ and for all integers $j\ge \tilde{J}_0$, the matrix
		$[\ol{A_j}^\tp, \ol{B_j}^\tp]$ must be an invertible square matrix satisfying
		\be \label{refstr:inv}
		\left[ \begin{matrix} \tilde{A}_j\\ \tilde{B}_j\end{matrix}\right]=[\ol{A_j}^\tp, \ol{B_j}^\tp]^{-1},\quad \mbox{that is},\quad
		\left[ \begin{matrix} \tilde{A}_j\\ \tilde{B}_j\end{matrix}\right]
		[\ol{A_j}^\tp, \ol{B_j}^\tp]=
		\left[ \begin{matrix} A_j\\ B_j\end{matrix}\right]
		[\ol{\tilde{A}_j}^\tp, \ol{\tilde{B}_j}^\tp]=I_{\#\Phi_{j+1}},
		\ee
		where
		$\cB_J:=\Phi_J \cup\{ \Psi_j \setsp\; j\ge J\}$ and
		$\tilde{\cB}_J:=\tilde{\Phi}_J \cup\{ \tilde{\Psi}_j \setsp\; j\ge J\}$.
		
		\item[(4)] $\PL_{m-1}\chi_{[0,N]}\subseteq \mbox{span}(\Phi_j)$ for some (or all) $j\ge \tilde{J}_0$ if and only if $\vmo(\tilde{\psi}^L\cup \tilde{\psi}^R\cup \tilde{\psi})\ge m$. Similarly, $\PL_{\tilde{m}-1}\chi_{[0,N]}\subseteq \mbox{span}(\tilde{\Phi}_j)$ for some (or all) $j\ge \tilde{J}_0$ if and only if $\vmo(\psi^L\cup \psi^R\cup \psi)\ge \tilde{m}$.
		
		\item[(5)] If $[\ol{A_J}^\tp, \ol{B_J}^\tp]$ is invertible for every $J_0\le J<\tilde{J}_0$,
		then $(\tilde{\cB}_J, \cB_J)$ forms a pair of biorthogonal Riesz bases of $L_2([0,N])$ for every $J_0\le J< \tilde{J}_0$,
		where we recursively define
		$\tilde{\Phi}_j:=\tilde{A}_j \tilde{\Phi}_{j+1}$ and
		 $\tilde{\Psi}_j:=\tilde{B}_j\tilde{\Phi}_{j+1}$ for $j$ going from $\tilde{J_0}-1$ to $J_0$ with the matrices $\tilde{A}_j$ and $\tilde{B}_j$ in \eqref{refstr:inv}.
	\end{enumerate}
\end{theorem}

We now make some remarks on \cref{thm:bw:0N}.
Note that there are no interior elements of $\Phi_j$ in \eqref{Phij} if $2^jN<n_\phi+n_{\mathring{\phi}}$.
Since $J_0$ is often much smaller than $\tilde{J}_0$,
item (5) allows us to have a locally supported Riesz basis $\cB_{J_0}$ with simple structures and the smallest coarse scale level $J_0$.
Suppose that $\phi=(\phi^1,\ldots,\phi^r)^\tp,
\psi=(\psi^1,\ldots,\psi^r)^\tp$, $\tilde{\phi},\tilde{\psi}\in (\Lp{2})^r$ in \cref{thm:bw:0N} have the following symmetry:
\begin{align}
&\phi^\ell(c_\ell^\phi-\cdot)=\eps_\ell^\phi \phi^\ell, \quad
\tilde{\phi}^\ell(c_\ell^\phi-\cdot)=\eps_\ell^\phi \tilde{\phi}^\ell \quad\mbox{with}\quad
c_\ell^\phi\in \Z, \eps_\ell^\phi\in \{-1,1\}, \quad \ell=1,\ldots,r, \label{sym:phi}\\
&\psi^\ell(c_\ell^\psi-\cdot)=\eps_\ell^\psi \psi^\ell, \quad
\tilde{\psi}^\ell(c_\ell^\psi-\cdot)=\eps_\ell^\psi \tilde{\psi}^\ell \quad\mbox{with}\quad
c_\ell^\psi\in \Z, \eps_\ell^\psi\in \{-1,1\}, \quad \ell=1,\ldots,r.\label{sym:tphi}
\end{align}
As a consequence of the above symmetry property, up to a possible sign change of some elements, $\AS_0(\phi(-\cdot);\psi(-\cdot))$ is the same as $\AS_0(\phi;\psi)$, while
$\AS_0(\tilde{\phi}(-\cdot);\tilde{\psi}(-\cdot))$ is the same as
$\AS_0(\tilde{\phi};\tilde{\psi})$. Hence, using the definition in \eqref{mphipsi}, for item (S2) in \cref{thm:bw:0N} we can simply choose
\begin{align}
&\mathring{\Phi}=\{\phi^L\}\cup \{\mathring{\phi}^\ell(\cdot-k) \setsp k\ge n_\phi+c^\phi_\ell\}_{\ell=1}^r,\quad
\mathring{\Psi}=\{\psi^L\}\cup \{\mathring{\psi}^\ell(\cdot-k) \setsp k\ge n_\psi+c^\psi_\ell\}_{\ell=1}^r, \label{mPhiPsi}\\
&\tilde{\mathring{\Phi}}=\{\tilde{\phi}^L\}\cup \{\tilde{\mathring{\phi}}^\ell(\cdot-k) \setsp k\ge n_{\tilde{\phi}}+c^\phi_\ell\}_{\ell=1}^r,\quad
\tilde{\mathring{\Phi}}=\{\tilde{\psi}^L\}\cup \{\tilde{\mathring{\psi}}^\ell(\cdot-k) \setsp k\ge n_{\tilde{\psi}}+c^\psi_\ell\}_{\ell=1}^r. \label{mtPhiPsi}
\end{align}
In other words, up to a possible sign change of some elements,
$\mathring{\Phi}$, $\mathring{\Psi}$, $\tilde{\mathring{\Phi}}$, and $\tilde{\mathring{\Psi}}$ are the same as $\Phi,\Psi, \tilde{\Phi}, \tilde{\Psi}$, respectively.
If a compactly supported biorthogonal wavelet
$(\{\tilde{\phi};\tilde{\psi}\},\{\phi;\psi\})$ in $\Lp{2}$ has the symmetry properties in \eqref{sym:phi} and \eqref{sym:tphi}, then we always take $\mathring{\Phi}$ and
$\mathring{\Psi}$ in \eqref{mPhiPsi} and
$\tilde{\mathring{\Phi}}$ and
$\tilde{\mathring{\Psi}}$
in \eqref{mtPhiPsi} for item (S2) in \cref{thm:bw:0N} for all our examples in the next section.

\section{Examples of Orthogonal and Biorthogonal Wavelets on $[0,1]$}
\label{sec:expl}

In this section we provide a few examples to illustrate our general construction methods and algorithms.
Since the construction of orthogonal wavelets on $[0,1]$ is much simpler than biorthogonal wavelets on $[0,1]$, let us first provide a few examples of orthogonal multiwavelets on $[0,1]$ by \cref{alg:Phi:orth} such that the boundary wavelets have the same order of vanishing moments as the interior wavelets. We shall provide examples of wavelets on $[0,1]$ satisfying homogeneous boundary conditions as well.
All our examples have the polynomial reproduction property in \eqref{polyprop:bc} for $\Phi$ satisfying \eqref{bc:phi:standard}, i.e., for $m:=\sr(a)$, $\PL_{m-1}\subseteq \mbox{span}(\{x^n \setsp 0\le n\le \ell \})+\mbox{span}(\Phi_j)$ holds on $[0,1]$ and
$h^{(n)}(0)=h^{(n)}(1)=0$ for all
$0\le n\le \ell$, $h\in \Phi_j$ and $j\ge J_0$, where $\ell=-1$ (no boundary conditions) or $\ell\in \{0,1\}$.
To avoid possible confusion, we shall use the notations $\phi^{L,bc}, \phi^{L,bc1}$ for $\phi^L$, and $\psi^{L,bc}, \psi^{L,bc1}$ for $\psi^L$ if they satisfy the homogeneous Dirichlet boundary conditions for $\ell=0$ or $\ell=1$, respectively.

Before presenting our examples, let us recall a technical quantity.
For $\tau\in \R$, recall that $\phi\in (\HH{\tau})^r$ if $\int_{\R} \|\wh{\phi}(\xi)\|_{l_2}^{2}(1+|\xi|^{2})^{\tau} d\xi <\infty$. We define the smoothness exponent $\sm(\phi):=\sup\{\tau\in \R \setsp \phi\in (\HH{\tau})^r\}$.
For $a,\tilde{a}\in \lrs{0}{r}{r}$, let $\phi,\tilde{\phi}$ be compactly supported distributions satisfying $\wh{\phi}(2\xi)=\wh{a}(\xi)\wh{\phi}(\xi)$ and
$\wh{\tilde{\phi}}(2\xi)=\wh{\tilde{a}}(\xi)\wh{\tilde{\phi}}(\xi)$
with $\ol{\wh{\phi}(0)}^\tp \wh{\tilde{\phi}}(0)=1$. It is known (e.g., see \cite[Theorem~6.4.5]{hanbook} and \cite{han03jat}) that items (1) and (2) in \cref{thm:bw} can be equivalently replaced by $\sm(a)>0$ and $\sm(\tilde{a})>0$, where the technical quantity $\sm(a)$ is defined in \cite[(5.6.44)]{hanbook} (also see \cite[(4.3)]{han03jat} and \cite[(3.2)]{han06}) and can be computed (see \cite{jj03}, \cite[Theorem~7.1]{han03jat}, and \cite[Theorem~5.8.4]{hanbook}).
The quantity $\sm(a)$ is closely linked to the smoothness of a refinable vector function $\phi$ through
the inequality $\sm(\phi)\ge \sm(a)$. For any refinable vector function $\phi$ in a biorthogonal wavelet, $\{\phi(\cdot-k)\setsp k\in \Z\}$ must be a Riesz sequence in $\Lp{2}$ and hence we always have $\sm(\phi)=\sm(a)$ (e.g., see \cite[Theorem~6.3.3]{hanbook}).
See \cite{cdp97,han01,han03jat,han06,hanbook,jj03,jiang99} for more details on smoothness $\sm(\phi)$ of refinable vector functions and the quantity $\sm(a)$. Recall that $\sr(a)$ is the highest order of sum rules satisfied by the filter $a$ in \eqref{sr}, while $\vmo(\psi)$ stands for the highest order of vanishing moments satisfied by $\psi$. We shall always take $n_\psi$ in \cref{thm:direct} to be the smallest integer such that $\psi(\cdot-k)\in \si_1(\Phi)$ for all $k\ge n_\psi$.

\begin{example} \label{ex:hardin}
	\normalfont
	Consider the compactly supported orthogonal multiwavelet $\{\phi;\psi\}$ in \cite{ghm94}
	with $\phi=(\phi^1,\phi^2)^\tp$ and $\psi=(\psi^1,\psi^2)^\tp$
	satisfying $\wh{\phi}(2\xi)=\wh{a}(\xi)\wh{\phi}(\xi)$ and $\wh{\psi}(2\xi)=\wh{b}(\xi)\wh{\phi}(\xi)$ with $\wh{\phi}(0)=(\sqrt{6}/3,\sqrt{3}/3)^{\tp}$ and an associated finitely supported orthogonal wavelet filter bank $\{a;b\}$ given by
	\begin{align*} 
	&a=\left\{\begin{bmatrix} \tfrac{3}{10} & \tfrac{2\sqrt{2}}{5}\\[0.1em]
	-\tfrac{\sqrt{2}}{40} &-\tfrac{3}{20}\end{bmatrix},
	\begin{bmatrix} \tfrac{3}{10} &0 \\[0.3em]
	\tfrac{9\sqrt{2}}{40} & \frac{1}{2}\end{bmatrix},
	\begin{bmatrix} 0 & 0\\[0.1em]
	\tfrac{9\sqrt{2}}{40} & -\tfrac{3}{20} \end{bmatrix},
	\begin{bmatrix} 0 & 0\\[0.1em]
	-\tfrac{\sqrt{2}}{40} & 0\end{bmatrix}\right\}_{[-1,2]},\\
	&b=\left\{\begin{bmatrix} -\tfrac{\sqrt{2}}{40} & -\tfrac{3}{20}\\[0.1em]
	\tfrac{1}{20} & \tfrac{3\sqrt{2}}{20}\end{bmatrix},
	\begin{bmatrix} \tfrac{9\sqrt{2}}{40} & -\frac{1}{2} \\[0.3em]
	-\tfrac{9}{20} & 0 \end{bmatrix},
	\begin{bmatrix} \tfrac{9\sqrt{2}}{40} & -\tfrac{3}{20}\\[0.1em]
	\tfrac{9}{20} & -\tfrac{3\sqrt{2}}{20} \end{bmatrix},
	\begin{bmatrix} -\tfrac{\sqrt{2}}{40} & 0\\[0.1em]
	-\tfrac{1}{20} & 0\end{bmatrix}\right\}_{[-1,2]}.
	\end{align*}
	Note that $\fs(\phi) = \fs(\psi) =[-1,1]$.
	Then $\sm(a)=1.5$, $\sr(a)=2$ and its matching filter $\vgu\in \lrs{0}{1}{2}$ satisfying $\wh{\vgu}(0)\wh{\phi}(0)=1$ is given by
	$\wh{\vgu}(0)=(\sqrt{6}/3,\sqrt{3}/3)$ and $\wh{\vgu}'(0)=i(-\sqrt{6}/6,0)$.
	Using item (i) of \cref{prop:phicut} with $n_\phi=1$, we have the left boundary refinable function
	$\phi^{L}:= \sqrt{2} \phi^{2}\chi_{[0,\infty)}$  with $\#\phi^L=1$ satisfying $\|\phi^L\|_{\Lp{2}}=1$ and the refinement equation in \eqref{I:phi} below
	\[
	\phi^L= \phi^L(2\cdot)+
	\left[\tfrac{9}{10}, -\tfrac{3\sqrt{2}}{10}\right]\phi(2\cdot-1)+
	\left[-\tfrac{1}{10}, 0\right]\phi(2\cdot-2).
	\]
	Taking $n_\psi=1$ in \cref{alg:Phi:orth}, we obtain the left boundary wavelet $\psi^L$ with $\#\psi^L=1$ defined by
	\[
	\psi^L := \phi^L(2\cdot)+
	\left[-\tfrac{9}{10}, \tfrac{3\sqrt{2}}{10}\right]\phi(2\cdot-1)+
	\left[\tfrac{1}{10}, 0\right]\phi(2\cdot-2).
	\]
	Since  $\phi$ and $\psi$ have symmetry, we obtain through \eqref{mPhiPsi} that
	 $\mathring{\Phi}=\{\phi^L,\mathring{\phi}^1\}\cup\{\mathring{\phi}(\cdot-k) \setsp k\ge n_{\mathring{\phi}}\}$ and $\mathring{\Psi}=\{\psi^L\} \cup\{
	\mathring{\psi}(\cdot-k) \setsp k \ge n_{\mathring{\psi}}\}$ with $n_{\mathring{\phi}}=n_{\mathring{\psi}}=1$.
	According to \cref{alg:Phi:orth} and \cref{thm:bw:0N} with $N=1$, we obtain an orthonormal basis $\cB_J=\Phi_J \cup \{\Psi_j\}_{j=J}^\infty$ of $L_{2}([0,1])$ for every $J \in \NN$, where $\Phi_j$
	and $\Psi_j$ in \eqref{Phij} and \eqref{Psij} with $n_\phi=n_\psi=n_{\mathring{\phi}}=n_{\mathring{\psi}}=1$ are given by
	\begin{align*}
	&\Phi_j=\{\phi_{j;0}^L\}
	\cup \{\phi_{j;k}\setsp 1\le k\le 2^j-1\} \cup
	\{\phi^1_{j;2^j},\phi^R_{j;2^j-1}\},\\
	&\Psi_j=\{\psi_{j;0}^L\}
	\cup \{\psi_{j;k} \setsp 1\le k\le 2^j-1\}\cup
	\{\psi^R_{j;2^j-1}\},
	\end{align*}
	where $\phi^R:=\phi^L(1-\cdot)$ and
	$\psi^R:=\psi^L(1-\cdot)$
	with $\#\phi^L=\#\psi^L=1$, $\#\Phi_j=2^{j+1}+1$ and $\#\Psi_j=2^{j+1}$.
	Note that $\vmo(\psi^L)=\vmo(\psi^R)=\vmo(\psi)=2=\sr(a)$ and $\PL_{1}\chi_{[0,1]}\subset \mbox{span}(\Phi_j)$ for all $j\in \NN$.
	
	Using the classical approach in \cref{sec:classical} and \cref{thm:bw:0N} with $N=1$,
	we obtain a Riesz basis
	$\cB_J^{bc}:=\Phi_J^{bc} \cup \{\Psi_j^{bc} \setsp j\ge J\}$ of $L_{2}([0,1])$ for every $J\ge J_0:=1$ such that
	$h(0)=0$ for all $h\in \cB_J^{bc}$, where $\Phi_j^{bc}$
	and $\Psi_j^{bc}$ in \eqref{Phij} and \eqref{Psij} with  $n_\phi=n_\psi=n_{\mathring{\phi}}=n_{\mathring{\psi}}=1$ are given by
	\begin{align*}
	&\Phi_j^{bc}=\{\phi_{j;1}\}
	\cup \{\phi_{j;k} \setsp 2\le k\le 2^j-2\} \cup \{\phi^1_{j;2^j-1}, \phi^1_{j;2^j}, \phi^2_{j;2^j-1}\},\\
	&\Psi_j^{bc}=\{\psi_{j;0}^{L,bc}, \psi_{j;1}\}
	\cup \{\psi_{j;k} \setsp 2\le k\le 2^j-2\}\cup
	 \{\psi^{R,bc}_{j;2^j-1},\psi_{j;2^j-1}\},
	\end{align*}
	with $\#\psi^{L,bc}=\#\psi^{R,bc}=1$, $\#\Phi^{bc}_j=2^{j+1}-1$ and $\#\Psi_j^{bc}=2^{j+1}$,
	where $\phi^{L,bc}=\emptyset$, $\phi^{R,bc}=\emptyset$ and
	\[
	\psi^{L,bc}:=[1,-2\sqrt{2}] \phi(2\cdot -1) + [1,0]\phi(2\cdot -2)\qquad \mbox{and}\qquad
	\psi^{R,bc}:=\psi^{L,bc}(1-\cdot).
	\]
	Note that $\vmo(\psi^{L,bc})=\vmo(\psi^{R,bc})=\vmo(\psi)=2$ and $\Phi_j^{bc}=\Phi_j\bs\{\phi^L_{j;0},\phi^R_{j;2^j-1}\}$ as in \cref{prop:mod}.
	Moreover, the dual Riesz wavelet basis $\tilde{\cB}_J^{bc}$ of $\cB_J^{bc}$ is given by
	 $\tilde{\cB}_J^{bc}=\tilde{\Phi}_J^{bc}\cup \{\tilde{\Psi}_j^{bc} \setsp j\ge J\}$ for $J\ge \tilde{J}_0:=2$, where
	$\tilde{\Phi}^{bc}_j$ and $\tilde{\Psi}^{bc}_j$
	in \eqref{tPhij} and \eqref{tPsij} with $n_{\tilde{\phi}}=n_{\tilde{\mathring{\phi}}}=
	 n_{\tilde{\psi}}=n_{\tilde{\mathring{\psi}}}=2$ are given by
	\begin{align*}
	&\tilde{\Phi}_j^{bc}=
	\{\tilde{\phi}^{L,bc}_{j;0}\}
	\cup \{\phi_{j;k} \setsp 2\le k\le 2^j-2\} \cup \{\phi^1_{j;2^j-1}, \tilde{\phi}^{R,bc}_{j;2^{j-1}}\}, \quad \mbox{with} \quad \tilde{\phi}^{R,bc}:=\tilde{\phi}^{L,bc}(1-\cdot),\\
	 &\tilde{\Psi}^{bc}_j=\{\tilde{\psi}^{L,bc}_{j;0}\}
	\cup \{\psi_{j;k} \setsp 2\le k\le 2^j-2\} \cup\{\tilde{\psi}^{R,bc}_{j;2^j-1}\}, \quad \mbox{with}
	\quad \tilde{\psi}^{R,bc}:=\text{diag}(1,1,-1)\tilde{\psi}^{L,bc}(1-\cdot)
	\end{align*}
	with $\#\tilde{\phi}^{L,bc}=\#\tilde{\phi}^{R,bc}=2$ and $\#\tilde{\psi}^{L,bc}=\#\tilde{\psi}^{R,bc}=3$, where $\tilde{\phi}^{L,bc} :=[\sqrt{2},-1]^{\tp} \phi^2\chi_{[0,\infty)} + \phi(\cdot -1)$ and
	%
	\[
	\tilde{\psi}^{L,bc}:=
	\begin{bmatrix}
	\tfrac{9}{10} & -\tfrac{3\sqrt{2}}{10}\\
	-\tfrac{1}{\sqrt{2}} & 0\\
	1 & 0\\
	 \end{bmatrix}\tilde{\phi}^{L,bc}(2\cdot)+
	\begin{bmatrix}
	-\tfrac{1}{10} & 0\\
	\tfrac{1}{\sqrt{2}} & -1\\
	-1 & 0\\
	\end{bmatrix}\phi(2\cdot -2)
	+2
	\begin{bmatrix}
	0_{1\times 2}\\
	b(1)
	\end{bmatrix}
	\phi(2\cdot-3)
	+
	2\begin{bmatrix}
	0_{1\times 2}\\
	b(2)
	\end{bmatrix}
	\phi(2\cdot-4).
	\]
	Note that $\vmo(\tilde{\psi}^{L,bc})=\vmo(\tilde{\psi}^{R,bc})=0$ and
	$\tilde{\phi}^{L,bc}$ satisfies the refinement equation in \eqref{I:phi:dual} below
	\[
	\tilde{\phi}^{L,bc}
	=\begin{bmatrix}
	\tfrac{3}{2} & \tfrac{\sqrt{2}}{2}\\
	-\frac{\sqrt{2}}{2} & 0
	\end{bmatrix}
	\tilde{\phi}^{L,bc}(2\cdot) +
	\begin{bmatrix}
	\tfrac{1}{2} & 0\\
	\frac{\sqrt{2}}{2} & 1
	\end{bmatrix}
	\phi(2\cdot-2)
	+2a(1) \phi(2\cdot-3) + 2a(2) \phi(2\cdot-4).
	\]
	We can also directly check that all the conditions in \cref{thm:direct} are satisfied for the Riesz basis $\cB^{bc}_J$ with $J\ge 2$. Indeed, by $\phi^{L,bc}=\emptyset$ and $n_\phi=1$, taking $n_{\psi}=1$ and $m_{\phi}=3$ in \cref{thm:direct}, we see that the above $\psi^{L,bc}$ satisfies both items (i) and (ii) of \cref{thm:direct} with $A_0=[0,0,0,0]^\tp$ and
	\[
	 B_{0}=\left[\tfrac{9}{20},-\tfrac{3\sqrt{2}}{20},-\tfrac{1}{20},0\right]^{\tp},\quad
	C(1)=
	\begin{bmatrix}
	\frac{3}{4} & \frac{\sqrt{2}}{4} & \frac{1}{4} & 0 \\
	-\frac{\sqrt{2}}{4} & 0 & \frac{\sqrt{2}}{4} & \frac{1}{2}
	\end{bmatrix}^{\tp}, \quad
	D(1)=\begin{bmatrix}
	-\frac{\sqrt{2}}{4} & 0 & \frac{\sqrt{2}}{4} & -\frac{1}{2}\\
	\frac{1}{2} & 0 & -\frac{1}{2} & 0
	\end{bmatrix}^{\tp}
	\]
	with $\fs(C)=\fs(D)=[1,1]$.
	Note that \eqref{tAL} is satisfied with
	$\rho(\tilde{A}_{L}^{bc})=1/2$. We reconcile the results obtained from direct and classical approaches. Since $\tilde{\phi}^{L,bc}$ in \eqref{tphi:implicit} of the direct approach contains interior elements, we can rewrite $\tilde{\phi}^{L,bc}$ in \eqref{tphi:implicit} as $\{\tilde{\phi}^{L,bc},\phi(\cdot-2)\}$ with $\tilde{\phi}^{L,bc}$ as in the classical approach with $\# \tilde{\phi}^{L,bc}=2$. On the other hand, $\tilde{\psi}^{L,bc}$ in \eqref{tpsi:implicit} is the same as in the classical approach with $\# \tilde{\psi}^{L,bc}=3$. Note that the dual Riesz basis $\tilde{\cB}_J^{bc}$ for $J=1$ has to be computed via item (5) of \cref{thm:bw:0N}. See \cref{fig:ghmorth} for the graphs of $\phi,\psi$ and all boundary elements.
\end{example}

\begin{figure}[htbp]
	\centering \begin{subfigure}[b]{0.24\textwidth} \includegraphics[width=\textwidth,height=0.6\textwidth]{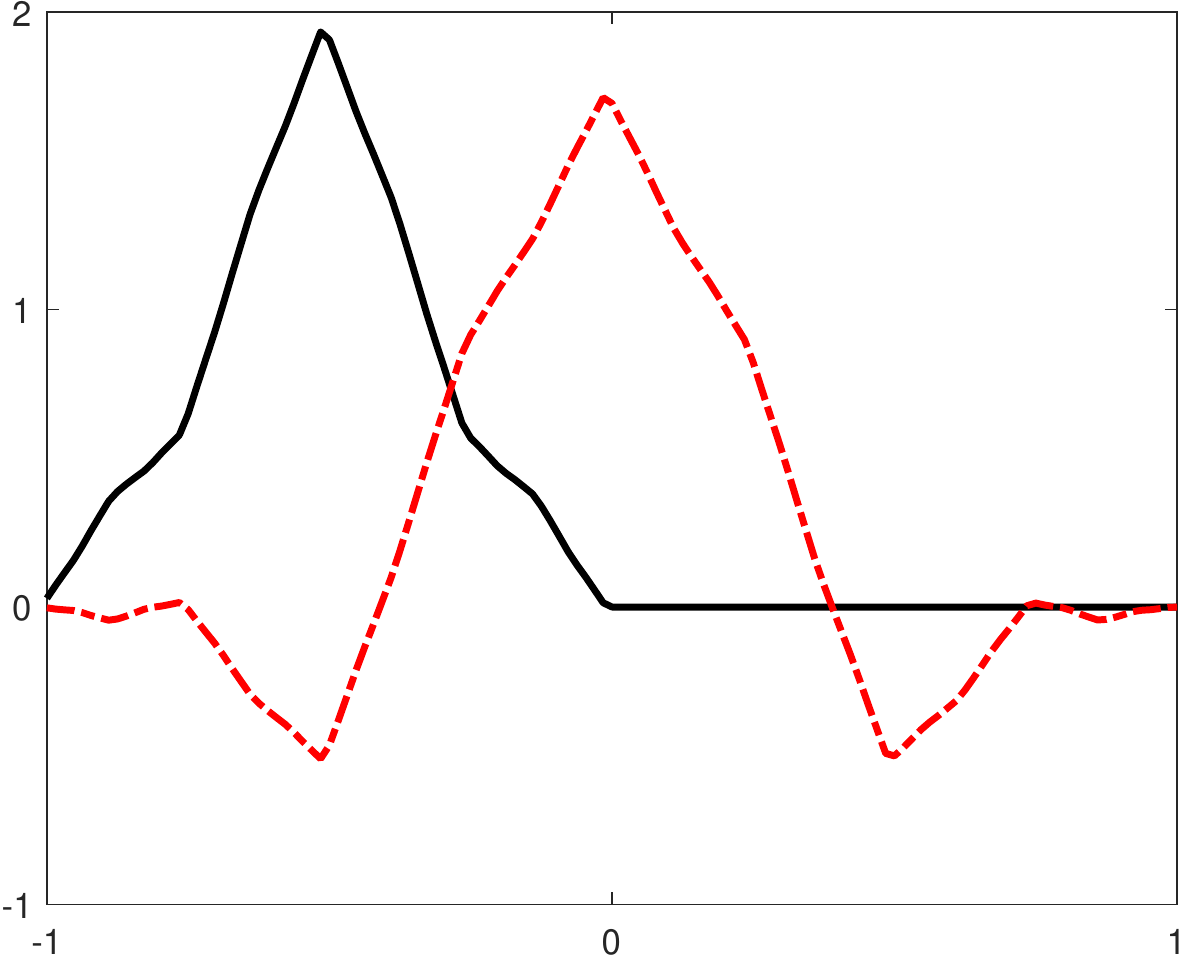} \caption{$\phi=(\phi^1,\phi^2)^\tp$}
	\end{subfigure}
	\begin{subfigure}[b]{0.24\textwidth} \includegraphics[width=\textwidth,height=0.6\textwidth]{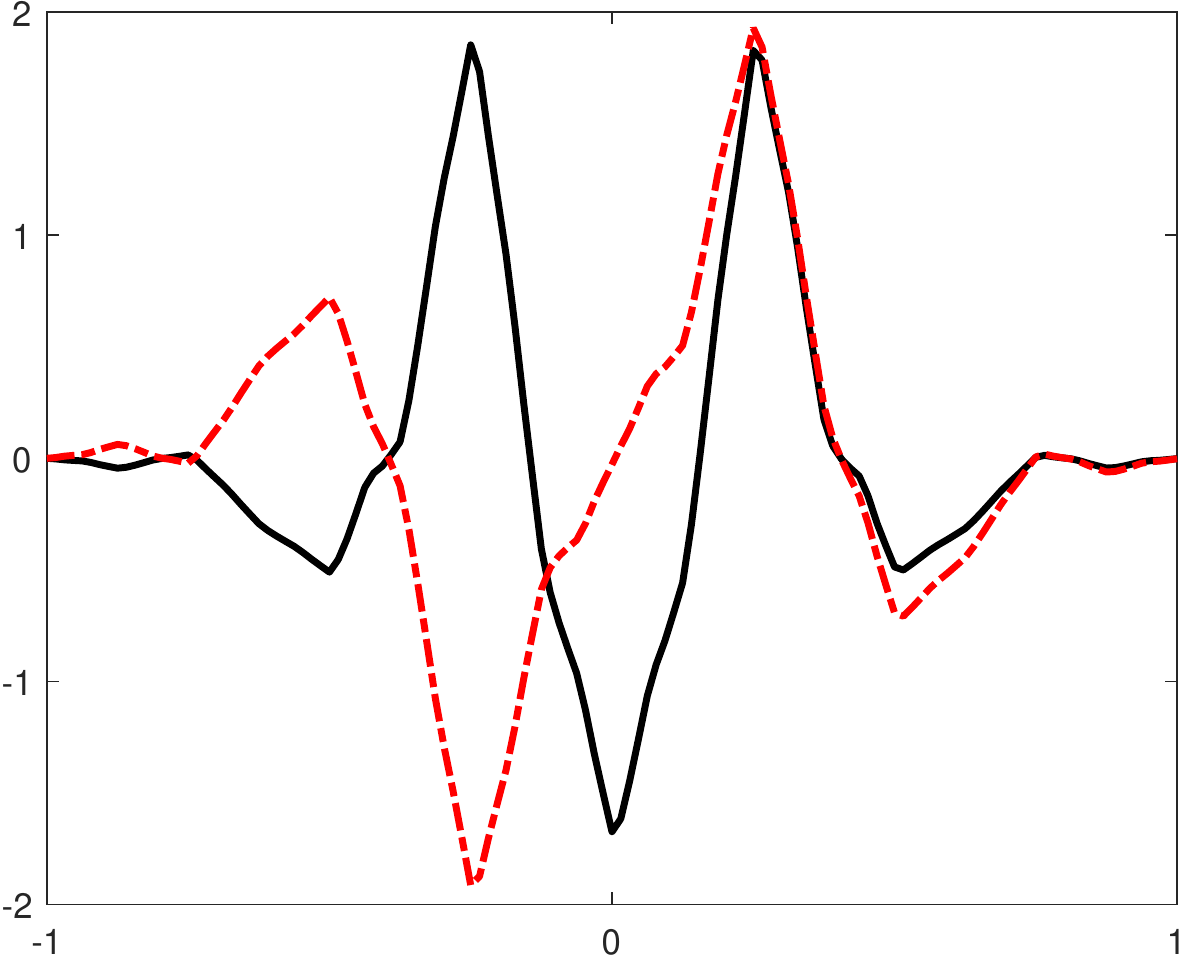} \caption{$\psi=(\psi^1,\psi^2)^\tp$}
	\end{subfigure} \begin{subfigure}[b]{0.24\textwidth}
		 \includegraphics[width=\textwidth,height=0.6\textwidth]{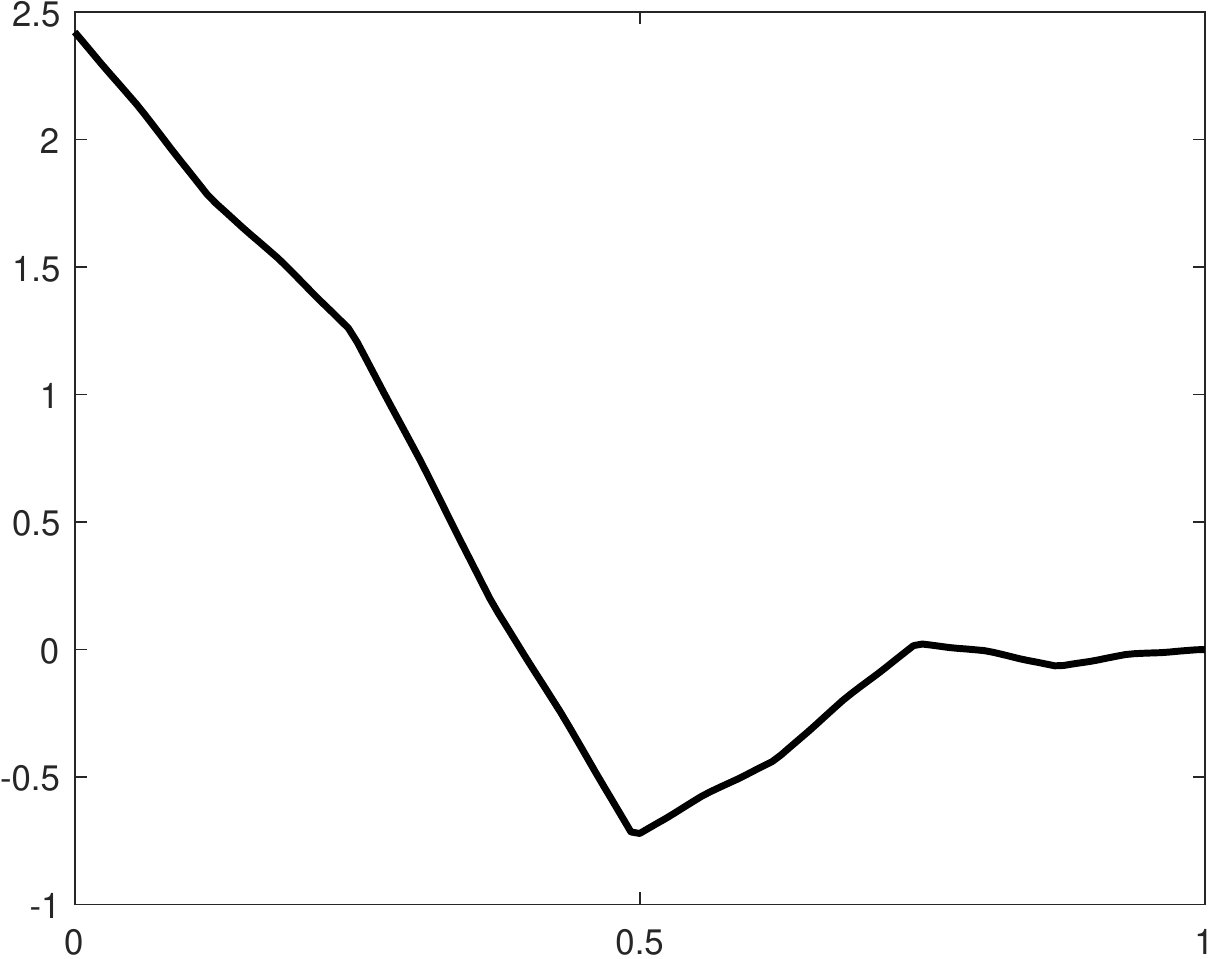}
		\caption{$\phi^{L}$}
	\end{subfigure}
	\begin{subfigure}[b]{0.24\textwidth} \includegraphics[width=\textwidth,height=0.6\textwidth]{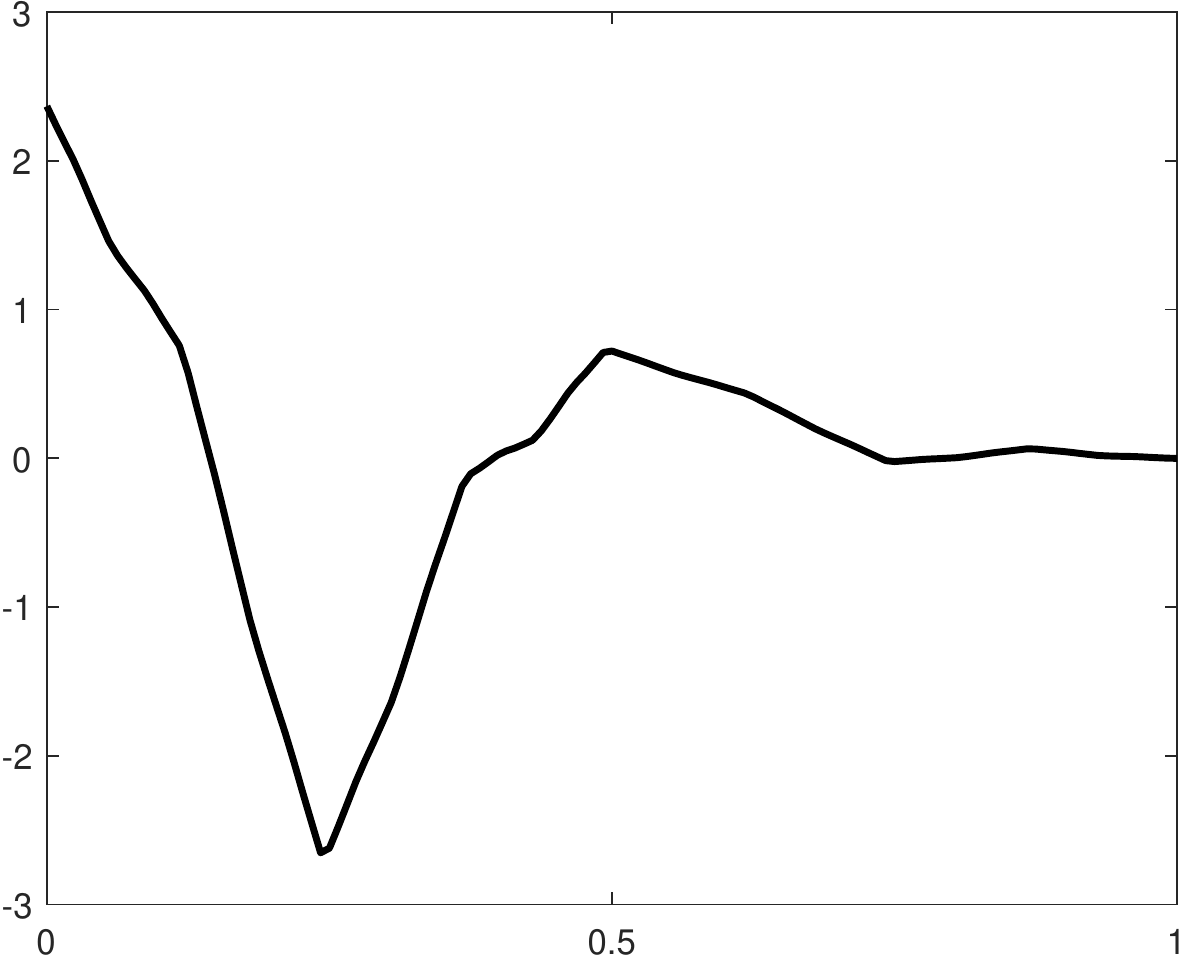}
		\caption{$\psi^{L}$}
	\end{subfigure}\\ \begin{subfigure}[b]{0.24\textwidth} \includegraphics[width=\textwidth,height=0.6\textwidth]{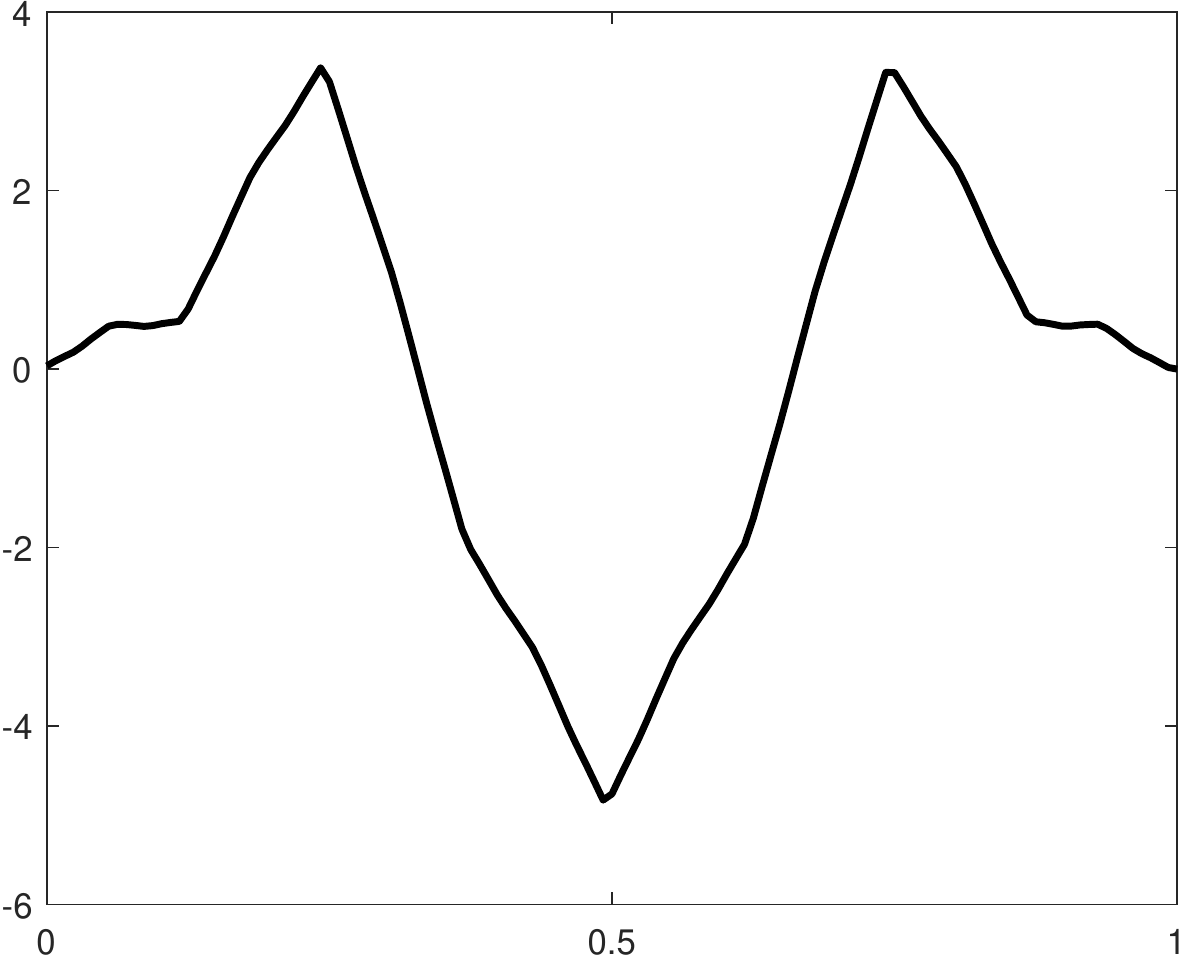}
		\caption{$\psi^{L,bc}$}
	\end{subfigure} \begin{subfigure}[b]{0.24\textwidth} \includegraphics[width=\textwidth,height=0.6\textwidth]{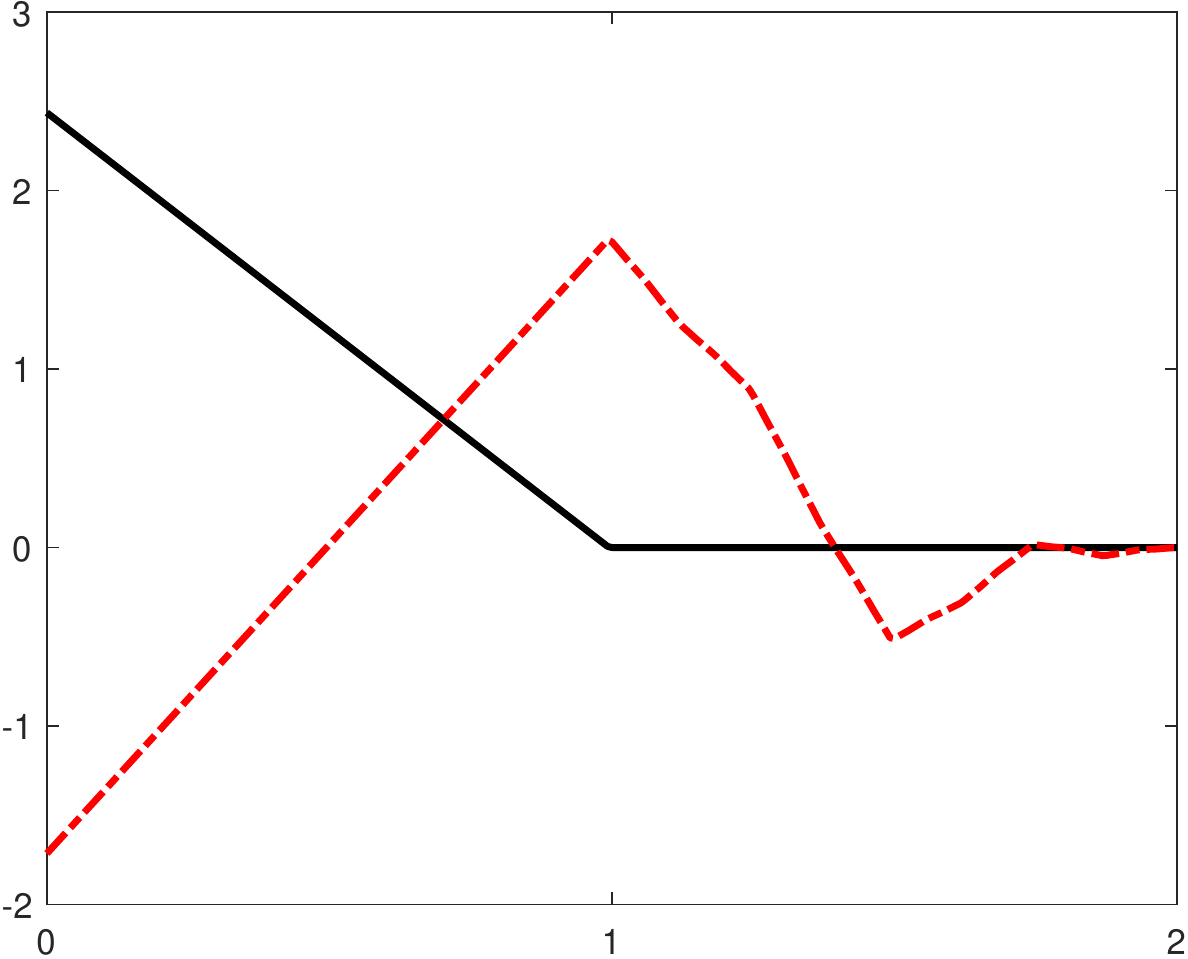}
		\caption{$\tilde{\phi}^{L,bc}$}
	\end{subfigure} \begin{subfigure}[b]{0.24\textwidth} \includegraphics[width=\textwidth,height=0.6\textwidth]{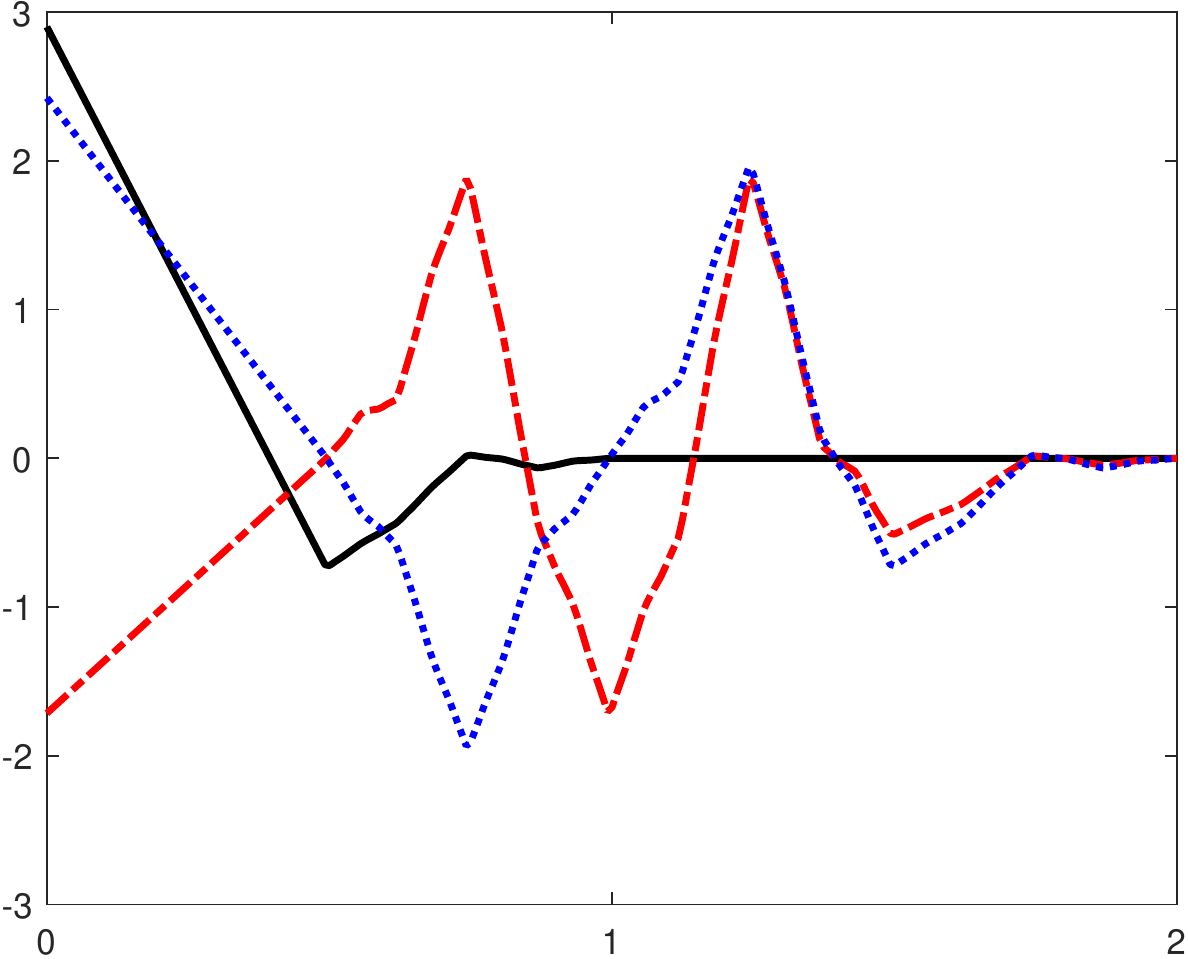}
		\caption{$\tilde{\psi}^{L,bc}$}
	\end{subfigure}
	\begin{subfigure}[b]{0.24\textwidth} \includegraphics[width=\textwidth,height=0.6\textwidth]{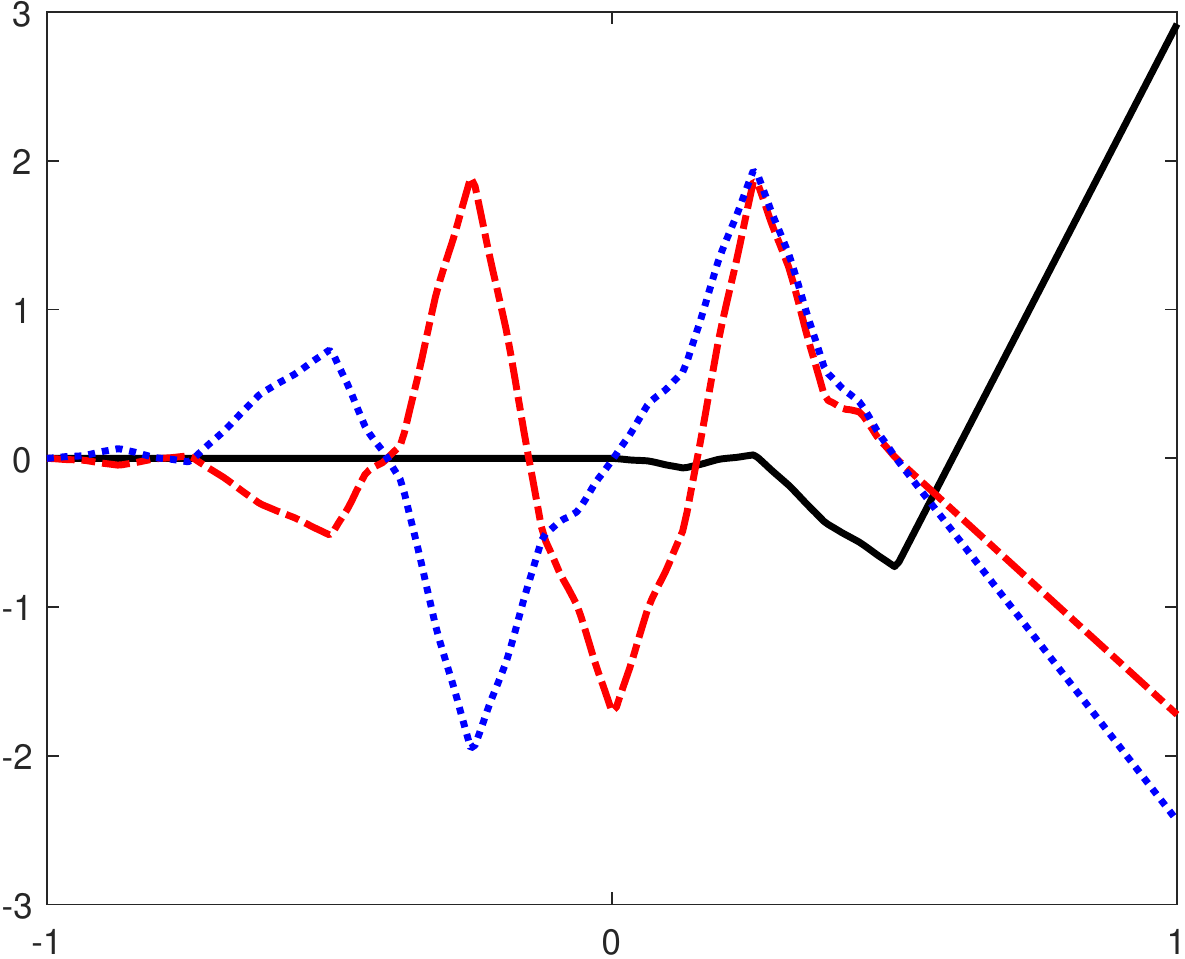} \caption{$\tilde{\psi}^{R,bc}$}
	\end{subfigure}
	\caption{The generators of the orthonormal basis $\cB_J$ and the Riesz basis $\cB_J^{bc}$ of $L_{2}([0,1])$ in \cref{ex:hardin} with $J\ge 1$ such that
		$h(0)=h(1)=0$ for all $h\in \cB_J^{bc}$. The black, red, and blue lines correspond to the first, second, and third components of a vector function. Note that $\phi^{L,bc}=\emptyset$ and
		 $\vmo(\psi^L)=\vmo(\psi^{L,bc})=\vmo(\psi)=2$. 
	}\label{fig:ghmorth}
\end{figure}

\begin{example} \label{ex:hardinr3}
	\normalfont
	Consider the compactly supported orthogonal wavelet $\{\phi;\psi\}$ in \cite{dgh96} satisfying $\wh{\phi}(2\xi)=\wh{a}(\xi)\wh{\phi}(\xi)$ and $\wh{\psi}(2\xi)=\wh{b}(\xi)\wh{\phi}(\xi)$ with $\wh{\phi}(0)=(\sqrt{7/33},\sqrt{3/4},\sqrt{5/132})^{\tp}$ and an associated finitely supported
	orthogonal wavelet filter bank $\{a;b\}$ given by
	\begin{align*}
	a=&
	{\left\{
		\begin{bmatrix} 0 & -{\frac {\sqrt {2}\sqrt {154} \left( 3+2
				\,\sqrt {5} \right) }{7392}} & {\frac {\sqrt {2}\sqrt {14} \left( 2+5\,
				\sqrt {5} \right) }{2464}}\\ 0 & 0 & 0 \\ 0 & 0 & 0 \end{bmatrix}
		\begin{bmatrix} -{\frac{3}{44}}-\frac{\sqrt {5}}{22} & {\frac {\sqrt {2}\sqrt {154} \left( 67+30\,\sqrt {5} \right) }{7392}} & {\frac {\sqrt {2}\sqrt {14} \left( -10+\sqrt {5} \right) }{224}}
		\\ 0 & 0 & 0 \\0 & 0 & 0\end{bmatrix},\right.}\\
	&{\left.
		\;\; \begin{bmatrix} \frac{1}{2} &{\frac {\sqrt {2}\sqrt {154} \left( 67
				-30\,\sqrt {5} \right) }{7392}}&{\frac {\sqrt {2}\sqrt {14} \left( 10+\sqrt {5} \right) }{224}}\\0 & \frac{3}{8} &{\frac {\sqrt {2}
				\sqrt {22} \left( -4+\sqrt {5} \right) }{176}}\\ 0&{\frac {\sqrt {2}\sqrt {22} \left( 32+7\,\sqrt {5} \right) }{528}}&-{
			 \frac{5}{88}}+\frac{\sqrt{5}}{22}\end{bmatrix}, \begin{bmatrix} -{\frac{3}{44}}+\frac{\sqrt {5}}{22}&{\frac {\sqrt {2}\sqrt {154} \left( -3+2\,\sqrt {5} \right) }{7392}}&{\frac {
				\sqrt {2}\sqrt {14} \left( -2+5\,\sqrt {5} \right) }{2464}}
		\\ \frac{\sqrt{2}\sqrt{154}}{44} & \frac{3}{8} &{\frac {\sqrt {2}
				\sqrt {22} \left( 4+\sqrt {5} \right) }{176}}\\-\frac{\sqrt{2}\sqrt {70}}{44} & {\frac {\sqrt {2}\sqrt {22} \left( -32+7\,
				\sqrt {5} \right) }{528}}&-{\frac{5}{88}}-\frac{\sqrt {5}}{22}\end{bmatrix}
		\right\}_{[-2,1]}},\\
	b=&{\left\{
		\begin{bmatrix} 0 & 0 & 0\\ 0&{\frac {\sqrt
				{2}\sqrt {154} \left( 3+2\,\sqrt {5} \right) }{7392}}&-{\frac {\sqrt {
					2}\sqrt {14} \left( 2+5\,\sqrt {5} \right) }{2464}}
		\\ 0 & -{\frac {\sqrt {2}\sqrt {7} \left( 1+\sqrt {5}
				\right) }{672}}&{\frac {\sqrt {2}\sqrt {77} \left( -1+3\,\sqrt {5}
				\right) }{2464}}\end{bmatrix},
		\begin{bmatrix} 0&0&0\\ {\frac{3}{44}}+\frac{\sqrt {5}}{22} & -{\frac {\sqrt {2}\sqrt {154} \left( 67+30\,\sqrt {5}
				\right) }{7392}} & {\frac {\sqrt {2} \left( 10-\sqrt {5} \right) \sqrt
				{14}}{224}}\\-{\frac {\sqrt {2}\sqrt {11} \left( 1+
				\sqrt {5} \right) }{88}} & {\frac {\sqrt {2}\sqrt {7} \left( 29+13\,
				\sqrt {5} \right) }{672}} & {\frac {\sqrt {2}\sqrt {77} \left( -75+17\,
				\sqrt {5} \right) }{2464}}\end{bmatrix}, 	
		\right.}\\
	&{\left.
		\;\;\begin{bmatrix} 0&{\frac {\sqrt {2}\sqrt {77} \left( -2+
				\sqrt {5} \right) }{528}}&{\frac {\sqrt {2}\sqrt {7} \left( 13-6\,
				\sqrt {5} \right) }{176}}\\ \frac{1}{2} & {\frac {\sqrt {2}
				\sqrt {154} \left( -67+30\,\sqrt {5} \right) }{7392}}&-{\frac {\sqrt {
					2}\sqrt {14} \left( 10+\sqrt {5} \right) }{224}}\\ 0
		&{\frac {\sqrt {2}\sqrt {7} \left( -29+13\,\sqrt {5} \right) }{672}}&-
		{\frac {\sqrt {2}\sqrt {77} \left( 75+17\,\sqrt {5} \right) }{2464}}
		\end{bmatrix},
		\begin{bmatrix} {\frac {13\,\sqrt {2}}{44}}&-{\frac {
				\sqrt {2}\sqrt {77} \left( \sqrt {5}+2 \right) }{528}}&-{\frac {\sqrt
				{2}\sqrt {7} \left( 13+6\,\sqrt {5} \right) }{176}}
		\\ {\frac{3}{44}}-\frac{\sqrt{5}}{22}&{\frac {\sqrt {2}
				\sqrt {154} \left( 3-2\,\sqrt {5} \right) }{7392}}&{\frac {\sqrt {2}
				\left( 2-5\,\sqrt {5} \right) \sqrt {14}}{2464}}\\
		{\frac {\sqrt {2}\sqrt {11} \left( 1-\sqrt {5} \right) }{88}}&{\frac {
				\sqrt {2}\sqrt {7} \left( 1-\sqrt {5} \right) }{672}}&-{\frac {\sqrt {
					2}\sqrt {77} \left( 3\,\sqrt {5}+1 \right) }{2464}}\end{bmatrix}
		\right\}_{[-2,1]}.}
	\end{align*}
	Note that
	$\phi=(\phi^1,\phi^2,\phi^3)^\tp$ is a continuous piecewise linear vector function without symmetry and $\fs(\phi)=\fs(\psi)=[-1,1]$.
	Then $\sm(a)=1.5$, $\sr(a)=2$, and its matching filter $\vgu\in \lrs{0}{1}{2}$ with $\wh{\vgu}(0)\wh{\phi}(0)=1$ is given by
	 $\wh{\vgu}(0)=(\sqrt{7/33},\sqrt{3}/2,\sqrt{5/132})$ and $\wh{\vgu}'(0)=i(0,\sqrt{3}/4,\sqrt{165}/132-\sqrt{1/33})$.
	By item (i) of \cref{prop:phicut} with $n_\phi=2$, the left boundary refinable vector functions consisting of  interior elements $\phi^2, \phi^3, \phi(\cdot-1)$
	and a true boundary element
	\[
	 \phi^{L}:=\sqrt{\tfrac{7(7+\sqrt{5})}{22}}
	\phi^1 \chi_{[0,\infty)}
	\quad \mbox{satisfying}\quad
	\phi^L=  \phi^L(2\cdot)+ 2[\sqrt{\tfrac{7(7+\sqrt{5})}{22}},0,0]a(1) \phi(2\cdot-1).
	\]
	Hence, we can reset $n_\phi:=1$ and use only $\{\phi^L,\phi^2,\phi^3\}$ as the left boundary refinable vector function. Let $[a(k)]_{j,:}$ denote the $j$th row of the matrix $a(k)$.
	Using $n_\phi=1$ and $n_\psi=1$ in \cref{alg:Phi:orth},
	we obtain the left boundary wavelet $\{\psi^L,\psi^1\}$ with $\#\psi^L=1$ as follows:
	\[
	\psi^{L} :=
	2 \left([\tfrac{1}{2},0,0] + \lambda_{1} [b(0)]_{3,:}\right)
	\phi^L(2\cdot)+
	2  \lambda_{1}[b(1)]_{3,:} \phi(2\cdot-1)\quad \mbox{with}\quad
	 \gl_1:=\tfrac{1}{2}\sqrt{\tfrac{7(7-\sqrt{5})}{11}}.
	\]
	Note that
	 $\mathring{\phi}=(\mathring{\phi}^1,\mathring{\phi}^2,\mathring{\phi}^3)^\tp:=\phi(-\cdot)$ has no symmetry.
	Using item (ii) of \cref{prop:phicut} with $\pp(x)=(1,x)^{\tp}$, we have $n_{\mathring{\phi}}=1$ and the left boundary refinable vector function
	\[
	 \mathring{\phi}^{L}:=\tfrac{\sqrt{14}}{\sqrt{7+\sqrt{5}}}
	\mathring{\phi}^{1}\chi_{[0,\infty)}
	\quad \mbox{satisfying}\quad
	\mathring{\phi}^{L}=
	\mathring{\phi}^{L}(2\cdot)
	+ \tfrac{2\sqrt{14}}{\sqrt{7+\sqrt{5}}}[a(-1)]_{1,:} \mathring{\phi}(2\cdot-1)+ \tfrac{2\sqrt{14}}{\sqrt{7+\sqrt{5}}}[a(-2)]_{1,:} \mathring{\phi}(2\cdot-2).
	\]
	Using $n_{\mathring{\psi}}=1$ in \cref{alg:Phi:orth}, we obtain the left boundary wavelet $\mathring{\psi}^L$ with $\#\mathring{\psi}^L=1$ as follows:
	\[
	 \mathring{\psi}^{L}:=\mathring{\phi}^{L}(2\cdot) + 2\lambda_{2} [b(-1)]_{2,:}\mathring{\phi}(2\cdot-1) + 2\lambda_{2} [b(-2)]_{2,:} \mathring{\phi}(2\cdot-2) \quad \mbox{with}\quad
	 \gl_2:=\sqrt{\tfrac{7(7-\sqrt{5})}{22}}.
	\]
	By \eqref{reflection}, we have
	$\phi^R:=\mathring{\phi}^{L}(1-\cdot)$ and $\psi^R:=\mathring{\psi}^{L}(1-\cdot)$. According to \cref{alg:Phi:orth} and \cref{thm:bw:0N} with $N=1$, we conclude that
	$\cB_J=\Phi_J \cup \{\Psi_j \setsp j\ge J\}$ is an orthonormal basis of $L_{2}([0,1])$ for every $J \in \NN$, where $\Phi_j$
	and $\Psi_j$ in \eqref{Phij} and \eqref{Psij} with $n_\phi=n_{\psi}=n_{\mathring{\phi}}=n_{\mathring{\psi}}=1$
	are given by
	\begin{align*}
	\Phi_{j}  &= \{\phi^{L}_{j;0}, \phi^2_{j;0},\phi^3_{j;0}\} \cup \{\phi_{j;k}:1\le k \le 2^{j}-1\} \cup \{\phi^{R}_{j;2^{j}-1}\},\\
	\quad \Psi_{j}  &= \{\psi^{L}_{j;0},\psi^1_{j;0}\} \cup \{\psi_{j;k}:1\le k \le 2^{j}-1\} \cup \{\psi^{R}_{j;2^{j}-1}\},
	\end{align*}
	with $\#\phi^L=\#\phi^R=\psi^L=\psi^R=1$,
	$\#\Phi_j=3(2^j)+1$ and $\#\Psi_j=3(2^j)$.
	Note that $\vmo(\psi^L)=\vmo(\psi^R)=\vmo(\psi)=2=\sr(a)$ and $\PL_{1}\chi_{[0,1]}\subset \mbox{span}(\Phi_j)$ for all $j\in \NN$.
	
	Using the classical approach in \cref{sec:classical} and \cref{thm:bw:0N},
	we obtain a Riesz basis
	$\cB_J^{bc}:=\Phi_J^{bc} \cup \{\Psi_j^{bc} \setsp j\ge J\}$ of $L_{2}([0,1])$ for every $J\ge J_0:=1$ such that
	$h(0)=h(1)=0$ for all $h\in \cB_J^{bc}$, where
	\begin{align*}
	&\Phi_j^{bc}=
	 \{\phi^{2}_{j;0},\phi^{3}_{j;0},\phi_{j;1}\}
	\cup \{\phi_{j;k} \setsp 2\le k\le 2^j-2\} \cup \{\phi_{j;2^{j}-1}\},\\
	&\Psi_j^{bc}=\{\psi_{j;0}^{L,bc}, \psi_{j;0}^{1}, \psi_{j;1}\}
	\cup \{\psi_{j;k} \setsp 2\le k\le 2^j-2\}\cup
	 \{\psi_{j;2^{j}-1},\psi^{R,bc}_{j;2^j-1}\}
	\end{align*}
	with $\#\psi^{L,bc}=\#\psi^{R,bc}=1$,
	$\#\Phi^{bc}_j=3(2^j)-1$ and
	$\#\Psi^{bc}_j=3(2^j)$, where $\psi^{R,bc}:=\mathring{\psi}^{L,bc}(1-\cdot)$ and
	\begin{align*}
	\psi^{L,bc} &:=-\tfrac{\sqrt{11}(4+\sqrt{5})}{33} \phi^{2}(2\cdot) + \phi^{3}(2\cdot) + \left[0,\tfrac{\sqrt{11}(4-\sqrt{5})}{33},1\right]\phi(2\cdot-1),\\
	\mathring{\psi}^{L,bc}  &:= \left[\tfrac{2}{\sqrt{7}},\tfrac{\sqrt{11}(\sqrt{5}-4)}{33},-1\right] \mathring{\phi}(2\cdot -1).
	\end{align*}
	Note that $\vmo(\psi^{L,bc})=\vmo(\psi^{R,bc})=\vmo(\psi)=2$ and $\Phi^{bc}_j=\Phi_j\bs \{\phi^L_{j;0},\phi^R_{j;2^j-1}\}$ as in \cref{prop:mod}.
	Moreover, the dual Riesz basis $\tilde{\cB}_J^{bc}$ of $\cB_J^{bc}$ is given by
	 $\tilde{\cB}_J^{bc}=\tilde{\Phi}_j^{bc}\cup \{\tilde{\Psi}_j^{bc} \setsp j\ge J\}$ with $J\ge \tilde{J}_0=2$ and
	\begin{align*}
	&\tilde{\Phi}_j^{bc}=
	\{\tilde{\phi}^{L,bc}_{j;0}\}
	\cup \{\phi_{j;k} \setsp 2\le k\le 2^j-2\} \cup\{\tilde{\phi}^{R,bc}_{j;2^{j-1}}\}, \quad \mbox{with} \quad \tilde{\phi}^{R,bc}:=\tilde{\mathring{\phi}}^{L,bc}(1-\cdot),\\
	 &\tilde{\Psi}^{bc}_j=\{\tilde{\psi}^{L,bc}_{j;0}\}
	\cup \{\psi_{j;k} \setsp 2\le k\le 2^j-2\} \cup\{\tilde{\psi}^{R,bc}_{j;2^j-1}\}, \quad \mbox{with}
	\quad \tilde{\psi}^{R,bc}:=\tilde{\mathring{\psi}}^{L,bc}(1-\cdot),
	\end{align*}
	 $\#\tilde{\phi}^{L,bc}=\#\tilde{\psi}^{L,bc}=5$, $\#\tilde{\phi}^{R,bc}=3$, and $\#\tilde{\psi}^{R,bc}=4$,
	where
	\begin{align*}
	\tilde{\phi}^{L,bc} & :=
	 \sqrt{\tfrac{22}{7(7+\sqrt{5})}}\left[0,\tfrac{4\sqrt{7}(4-\sqrt{5})}{11},\tfrac{21-8\sqrt{5}}{11},0,0\right]^{\tp}\phi^{L}+ [\phi^{2}(\cdot-1),\phi^{3}(\cdot-1),\phi(\cdot-2)]^{\tp},\\
	\tilde{\mathring{\phi}}^{L,bc} & := \lambda_2^{-1}\left[\tfrac{21+8\sqrt{5}}{11},0,-\tfrac{4\sqrt{7}(4+\sqrt{5})}{11}\right]^{\tp}\mathring{\phi}^{L}+ \mathring{\phi}(\cdot -1),\\
	\tilde{\psi}^{L,bc} & :=
	{\begin{bmatrix}
		 \frac{\sqrt{11}(114-51\sqrt{5})}{104} &\frac{315-54\sqrt{5}}{104} &\frac{9\sqrt{7}(-2+\sqrt{5})}{26} & \frac{3\sqrt{11}(-2+\sqrt{5})}{104} & \frac{18\sqrt{5}-27}{104}\\
		 \frac{3\sqrt{154}(31\sqrt{5}-69)}{1144} & \frac{\sqrt{14}(167\sqrt{5}-41)}{1144} & \frac{\sqrt{2}(207-49\sqrt{5})}{143} & -\frac{3\sqrt{154}(3\sqrt{5}-1)}{1144} & -\frac{9\sqrt{14}(1+\sqrt{5})}{104} \\
		0 & 0 & 0 & 0 & 0\\
		\frac{3\sqrt{77}(6\sqrt{5}-13)}{308} & \frac{\sqrt{7}(27\sqrt{5}-86)}{308} & \tfrac{8-2\sqrt{5}}{11} & -\frac{3\sqrt{77}(3+2\sqrt{5})}{308} &
		-\frac{\sqrt{7}(-6+\sqrt{5})}{28}\\
		-\frac{3\sqrt{14}(5\sqrt{5}-11)}{56} & -\frac{\sqrt{154}(25\sqrt{5}-67)}{616} & \frac{\sqrt{22}(\sqrt{5}-3)}{11} & \frac{3\sqrt{14}(1+\sqrt{5})}{56} & \frac{\sqrt{154}(13\sqrt{5}-47)}{616}\\
		\end{bmatrix}
		\tilde{\phi}^{L,bc}(2\cdot)}\\
	& \quad
	+ {\begin{bmatrix}
		0_{2 \times 3}\\
		2b(0)
		\end{bmatrix}
		{\phi}(2 \cdot -2) +
		\begin{bmatrix}
		0_{2 \times 3}\\
		2b(1)
		\end{bmatrix}
		{\phi}(2 \cdot -3)},
	\\
	\tilde{\mathring{\psi}}^{L,bc} & :=
	{\begin{bmatrix}
		 -\tfrac{\sqrt{2}(837+377\sqrt{5})}{11} & \tfrac{3\sqrt{154}(1503\sqrt{5}+3361)}{616} & -\tfrac{\sqrt{14}(8683\sqrt{5}+19321)}{616}\\
		\frac{8+2\sqrt{5}}{11} & -\frac{3\sqrt{77}(13+6\sqrt{5})}{308} & \frac{\sqrt{7}(27\sqrt{5}+86)}{308}\\
		\frac{\sqrt{22}(\sqrt{5}+3)}{11} & -\frac{3\sqrt{14}(11+5\sqrt{5})}{56} & \frac{\sqrt{154}(67+25\sqrt{5})}{616}\\
		\frac{\sqrt{7}(504+225\sqrt{5})}{11} & -\frac{\sqrt{11}(1347\sqrt{5}+3012)}{44} & \frac{2592\sqrt{5}+5715}{44}
		 \end{bmatrix}\tilde{\mathring{\phi}}^{L,bc}(2\cdot)
	}\\
	& \quad +
	{\begin{bmatrix}
		0 & -\frac{3\sqrt{154}(189+83\sqrt{5})}{616} & \frac{9\sqrt{14}(13\sqrt{5}+31)}{56}\\
		1 & \frac{3\sqrt{77}(-3+2\sqrt{5})}{308} & -\frac{\sqrt{7}(6+\sqrt{5})}{28}\\
		0 & \frac{3\sqrt{14}(\sqrt{5}-1)}{56} & -\frac{\sqrt{154}(13\sqrt{5}+47)}{616}\\
		0& \frac{\sqrt{11}(75\sqrt{5}+168)}{44} & -\frac{81+36\sqrt{5}}{4}
		\end{bmatrix}
	} \mathring{\phi}(2\cdot -2)
	+ \begin{bmatrix}
	2b(-1)\\
	0_{1 \times 3}
	\end{bmatrix} \mathring{\phi}(2\cdot -3)
	+ \begin{bmatrix}
	2b(-2)\\
	0_{1 \times 3}
	\end{bmatrix} \mathring{\phi}(2\cdot -4).
	\end{align*}
	Note that $\tilde{\phi}^{L,bc}$ and $\tilde{\mathring{\phi}}^{L,bc}$ satisfy the refinement equation in \eqref{I:phi:dual} as follows:
	\begin{align*}
	\tilde{\phi}^{L,bc} & =
	{\begin{bmatrix}
		\tfrac{3}{4} & \tfrac{\sqrt{11}(-4+\sqrt{5})}{44} & \tfrac{\sqrt{77}}{11} & \tfrac{3}{4} & \tfrac{\sqrt{11}(4+\sqrt{5})}{44} \\
		-\tfrac{9\sqrt{11}(-4+\sqrt{5})}{44} & \frac{65-8\sqrt{5}}{44} & \tfrac{\sqrt{7}(-4+\sqrt{5})}{11} & \tfrac{3\sqrt{11}(-4+ \sqrt{5})}{44} & -\tfrac{1}{4}\\
		 \tfrac{\sqrt{77}(39-18\sqrt{5})}{308} & \tfrac{\sqrt{7}(86-27\sqrt{5})}{308} & \tfrac{2\sqrt{5}-8}{11} & \tfrac{\sqrt{77}(9+6\sqrt{5})}{308} & \tfrac{\sqrt{7}(\sqrt{5}-6)}{28}\\
		\multicolumn{5}{c}{0_{2\times 5}}
		\end{bmatrix}
	}
	\tilde{\phi}^{L,bc}(2\cdot)\\
	&\quad +
	{\begin{bmatrix}
		0_{2\times 3}\\
		2a(0)
		\end{bmatrix}} \phi(2\cdot-2) +
	{\begin{bmatrix}
		0_{2 \times 3}\\
		2a(1)
		\end{bmatrix}} \phi(2\cdot-3),\\
	\tilde{\mathring{\phi}}^{L,bc} & =
	{\begin{bmatrix}
		-\tfrac{8+2\sqrt{5}}{11} & \tfrac{\sqrt{77}(18\sqrt{5}+39)}{308} & -\tfrac{\sqrt{7}(27\sqrt{5}+86)}{308}\\
		\tfrac{\sqrt{77}}{11} & \frac{3}{4} & \tfrac{\sqrt{11}(4+\sqrt{5})}{44}\\
		\tfrac{\sqrt{7}(4+\sqrt{5})}{11} & -\tfrac{9\sqrt{11}(4+\sqrt{5})}{44} & \tfrac{65+8\sqrt{5}}{44}
		\end{bmatrix}} \tilde{\mathring{\phi}}^{L,bc}(2\cdot)
	+{\begin{bmatrix}
		1 & -\tfrac{\sqrt{77}(6\sqrt{5}-9)}{308} & \tfrac{\sqrt{7}(6+\sqrt{5})}{28}\\
		0 & \tfrac{3}{4} & \tfrac{\sqrt{11}(-4+\sqrt{5})}{44}\\
		0 & \tfrac{3\sqrt{11}(4+\sqrt{5})}{44} & -\frac{1}{4}
		\end{bmatrix}}
	\phi(2\cdot-2)\\
	& \quad + 2a(-1)\phi(2\cdot-3) + 2a(-2)\phi(2\cdot-4).
	\end{align*}
	We can also directly check that all the conditions in \cref{thm:direct} are satisfied for the Riesz basis $\cB^{bc}_J$ with $J\ge 2$. To avoid complicated presentation, we only mention that
	the condition in \eqref{tAL} is satisfied with $\rho(\tilde{A}_L^{bc})=1/2$ for both left and right dual boundary elements. Note that the dual Riesz basis $\tilde{\cB}^{bc}_J$ for $J=1$ has to be computed via item (5) of \cref{thm:bw:0N}.
	See \cref{fig:hardinr3} for the graphs of $\phi,\psi$ and all boundary elements.
\end{example}

\begin{figure}[htbp]
	\centering
	\begin{subfigure}[b]{0.24\textwidth} \includegraphics[width=\textwidth,height=0.6\textwidth]{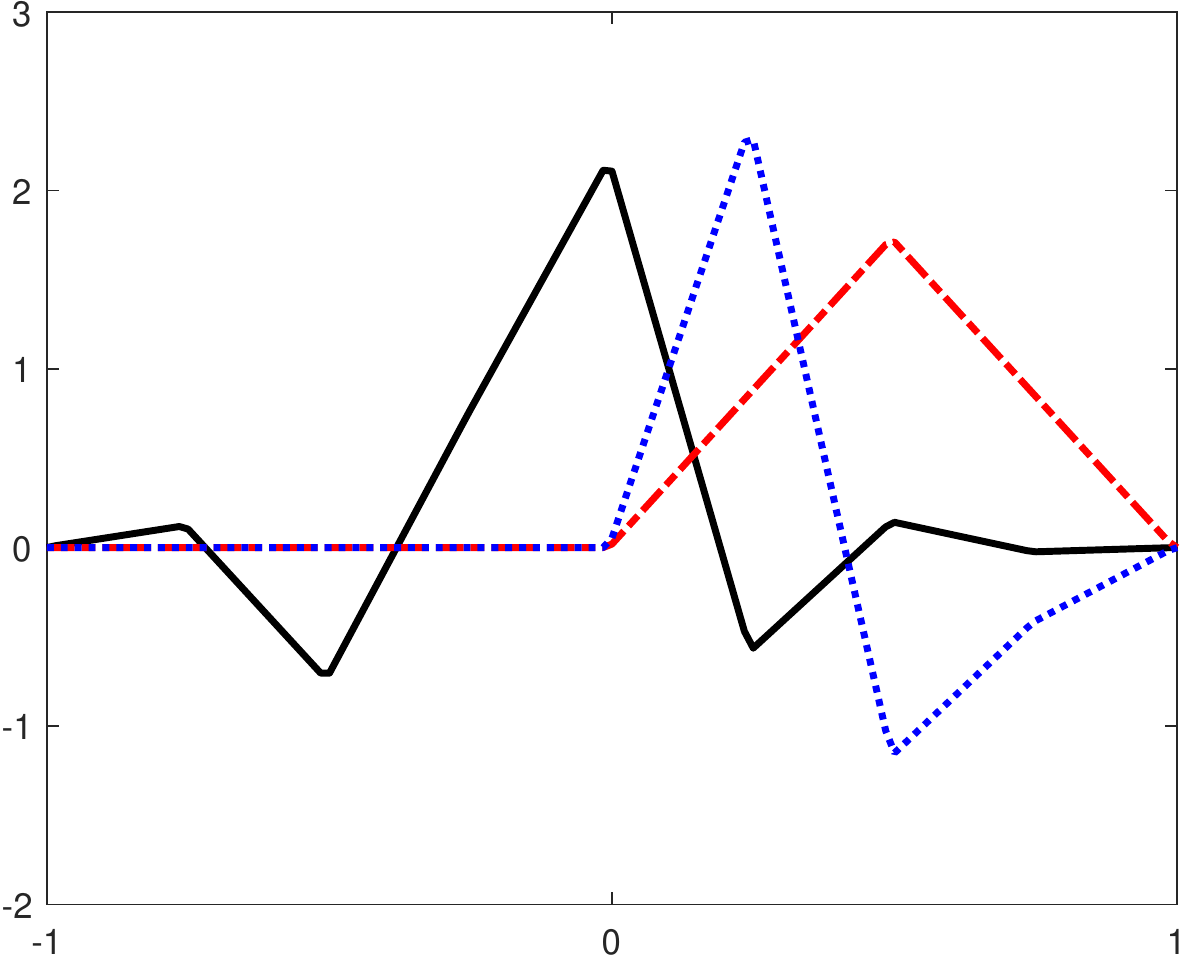} \caption{$\phi=(\phi^1,\phi^2,\phi^3)^\tp$}
	\end{subfigure} \begin{subfigure}[b]{0.24\textwidth} \includegraphics[width=\textwidth,height=0.6\textwidth]{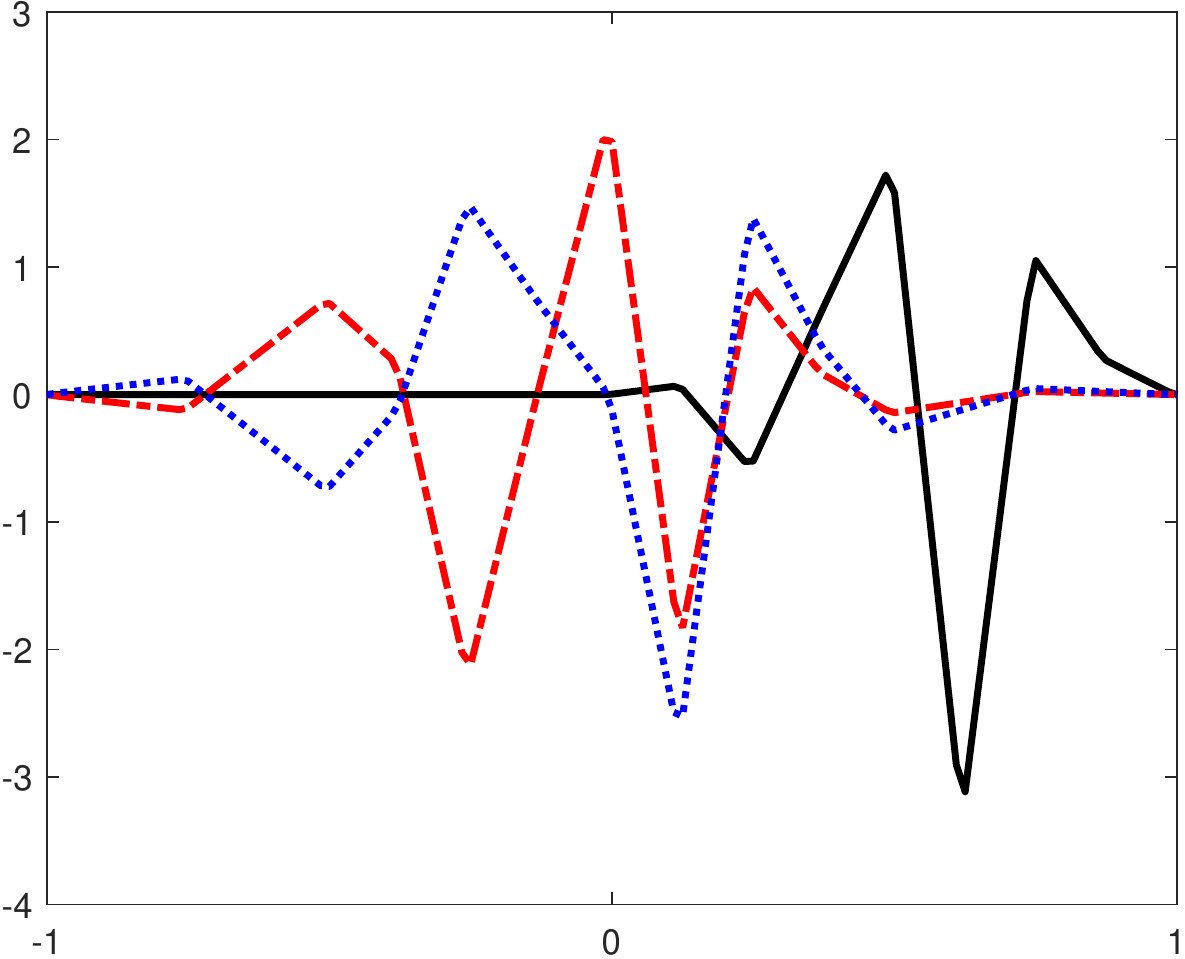} \caption{$\psi=(\psi^1,\psi^2,\psi^3)^\tp$}
	\end{subfigure}
	\begin{subfigure}[b]{0.24\textwidth} \includegraphics[width=\textwidth,height=0.6\textwidth]{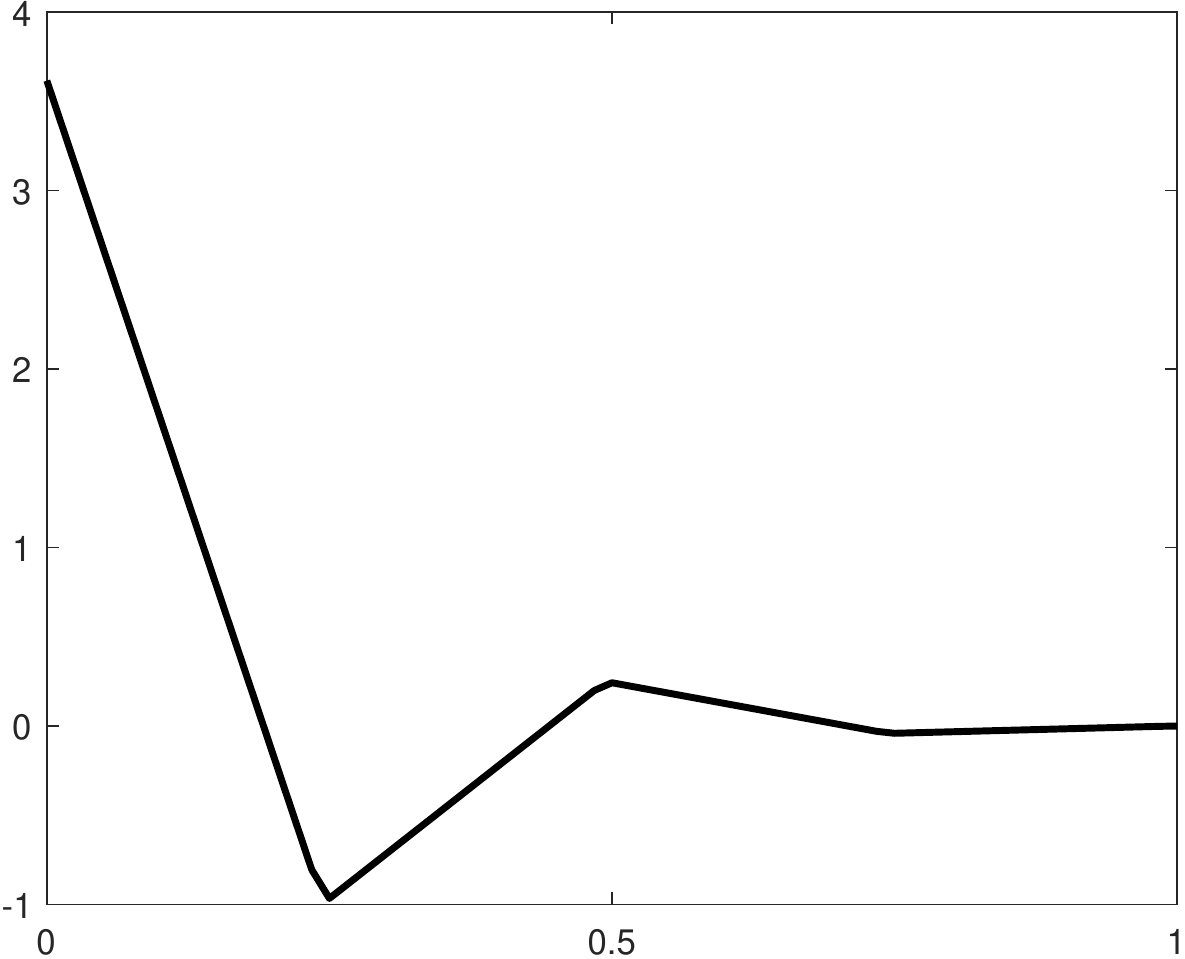}
		\caption{$\phi^{L}$}
	\end{subfigure} \begin{subfigure}[b]{0.24\textwidth} \includegraphics[width=\textwidth,height=0.6\textwidth]{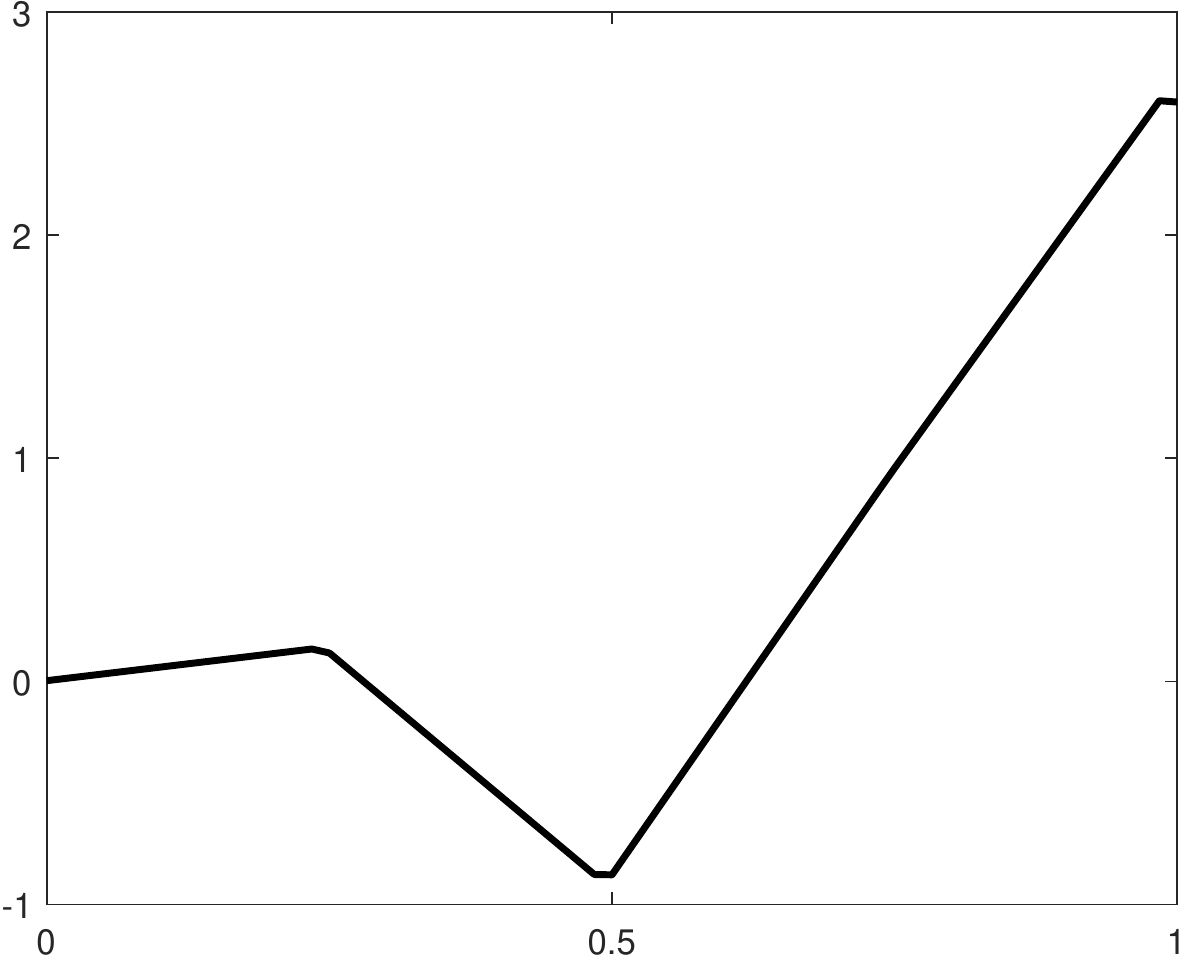}
		\caption{$\phi^{R}$}
	\end{subfigure}\\ \begin{subfigure}[b]{0.24\textwidth} \includegraphics[width=\textwidth,height=0.6\textwidth]{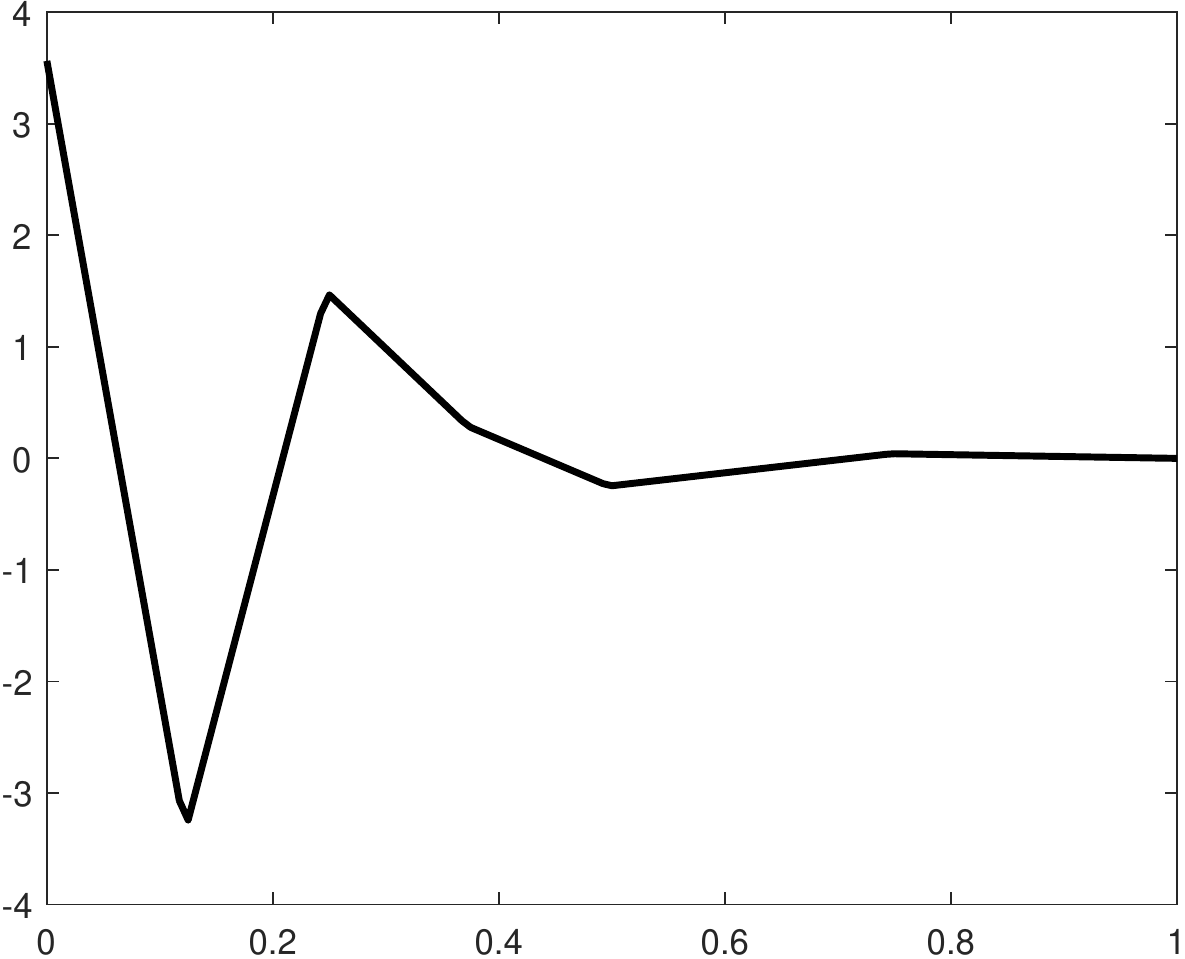}
		\caption{$\psi^{L}$}
	\end{subfigure}
	\begin{subfigure}[b]{0.24\textwidth} \includegraphics[width=\textwidth,height=0.6\textwidth]{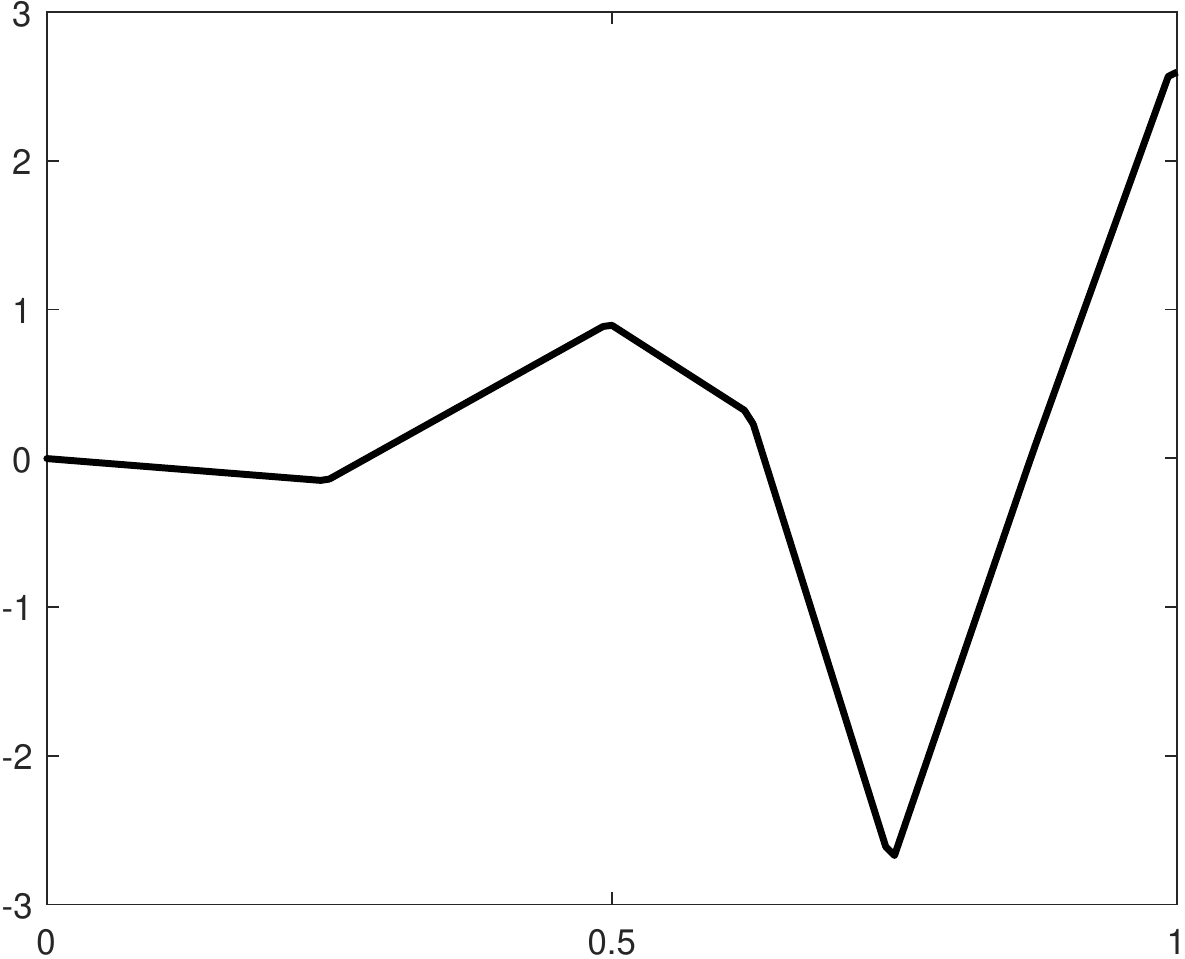}
		\caption{$\psi^{R}$}
	\end{subfigure}
	\begin{subfigure}[b]{0.24\textwidth} \includegraphics[width=\textwidth,height=0.6\textwidth]{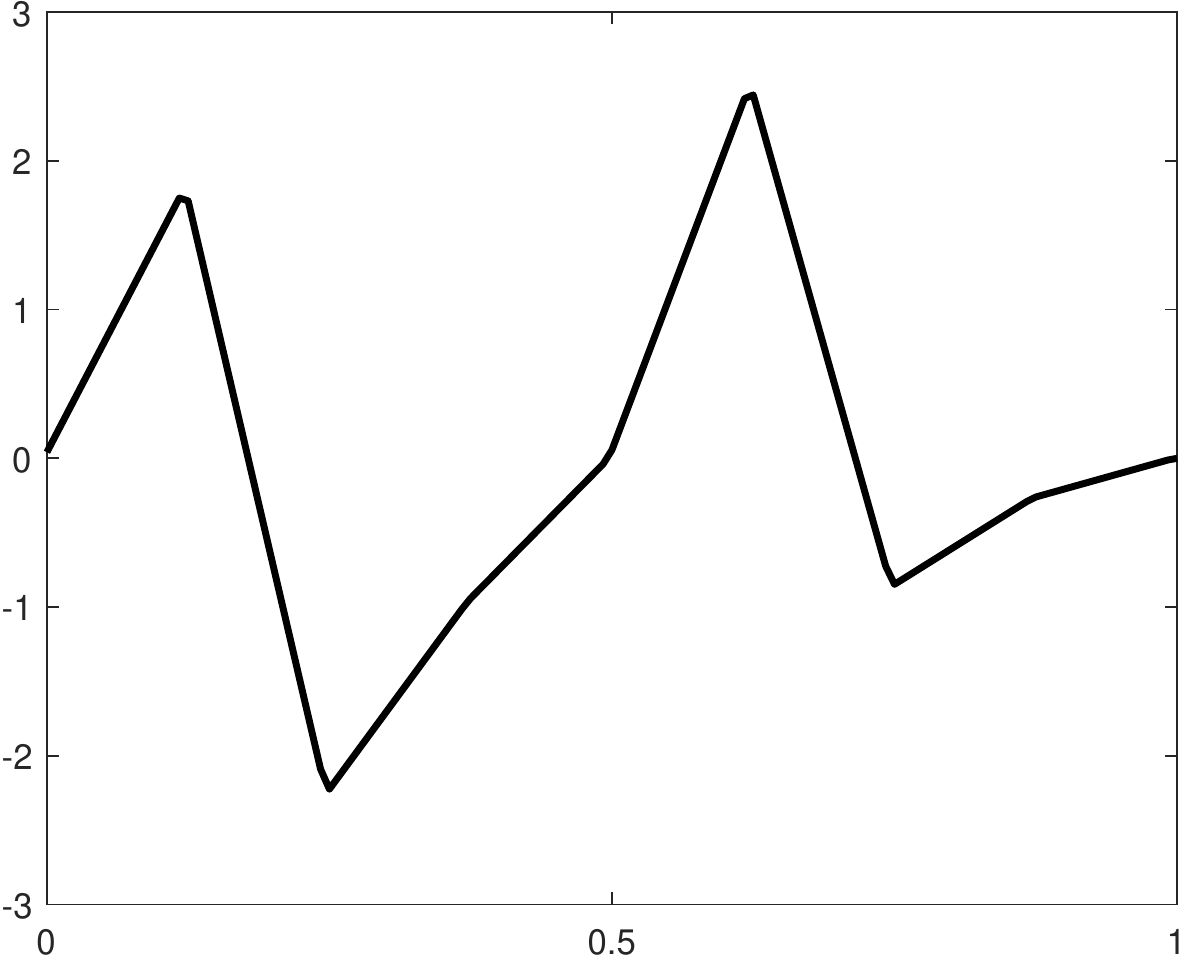}
		\caption{$\psi^{L,bc}$}
	\end{subfigure}
	\begin{subfigure}[b]{0.24\textwidth} \includegraphics[width=\textwidth,height=0.6\textwidth]{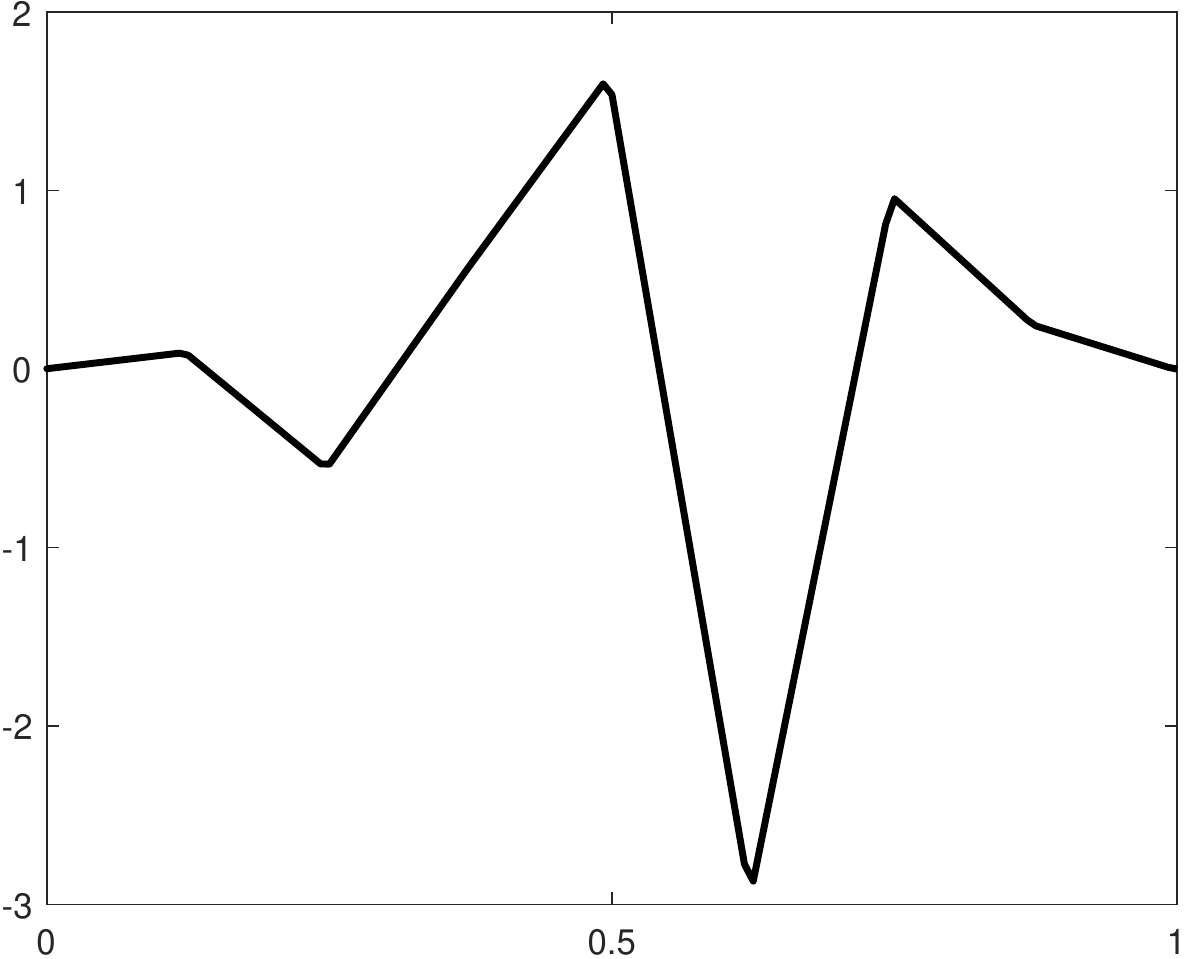}
		\caption{$\psi^{R,bc}$}
	\end{subfigure}
	\caption{The generators of the orthonormal basis $\cB_J$ and the Riesz basis $\cB_J^{bc}$ of $L_{2}([0,1])$ in \cref{ex:hardinr3} with $J \ge 1$ such that $h(0)=h(1)=0$ for all $h\in \cB_J^{bc}$. The black, red, and blue lines correspond to the first, second, and third components of a vector function. Note that $\vmo(\psi^L)=\vmo(\psi^R)=\vmo(\psi^{L,bc})=\vmo(\psi^{R,bc})=\vmo(\psi)=2$.}
	\label{fig:hardinr3}
\end{figure}

\begin{example} \label{ex:hmtquad}
	\normalfont Consider a biorthogonal wavelet $(\{\tilde{\phi};\tilde{\psi}\},\{\phi;\psi\})$ in \cite[Example~6.5.2]{hanbook} with $\wh{\phi}(0)=\wh{\tilde{\phi}}(0)=(1,0)^\tp$ and a biorthogonal wavelet filter bank $(\{\tilde{a};\tilde{b}\},\{a;b\})$ given by
	\begin{align*} 
	a=&{\left\{ \begin{bmatrix} \tfrac{1}{4} &\tfrac{1}{2}\\[0.3em]
		-\tfrac{1}{16} & -\tfrac{1}{8}\end{bmatrix},
		\begin{bmatrix} \tfrac{1}{2} &0 \\[0.3em]
		0 &\tfrac{1}{4}\end{bmatrix},
		\begin{bmatrix} \tfrac{1}{4} &-\tfrac{1}{2}\\[0.3em]
		\tfrac{1}{16} &-\tfrac{1}{8}\end{bmatrix}\right\}_{[-1,1]}}, \quad
	&&b={\left\{
		\begin{bmatrix} -\tfrac{1}{4} & -\tfrac{1}{2}\\[0.3em]
		\tfrac{37}{1456} & \tfrac{37}{728}\end{bmatrix},
		\begin{bmatrix} \tfrac{1}{2} &0 \\[0.3em]
		0 &\tfrac{1}{4}\end{bmatrix},
		\begin{bmatrix} -\tfrac{1}{4} & \frac{1}{2} \\[0.3em]
		-\frac{37}{1456} &\tfrac{37}{728}\end{bmatrix}\right\}_{[-1,1]}},\\
	\tilde{a}=&{\left\{
		\begin{bmatrix} \tfrac{1}{4} &\tfrac{1}{8}\\[0.3em]
		-\tfrac{91}{64} &-\tfrac{91}{128}\end{bmatrix},
		\begin{bmatrix} \tfrac{1}{2} &0 \\[0.3em]
		0 &\tfrac{37}{64}\end{bmatrix},
		\begin{bmatrix} \tfrac{1}{4} &-\tfrac{1}{8}\\[0.3em]
		\tfrac{91}{64} &-\tfrac{91}{128}\end{bmatrix}\right\}_{[-1,1]}}, \quad
	&&\tilde{b}={\left\{
		\begin{bmatrix} -\tfrac{1}{4} & -\tfrac{1}{8}\\[0.3em]
		\tfrac{91}{64} & \tfrac{91}{128}\end{bmatrix},
		\begin{bmatrix} \tfrac{1}{2} &0 \\[0.3em]
		0 &\tfrac{91}{64}\end{bmatrix},
		\begin{bmatrix} -\tfrac{1}{4} & \tfrac{1}{8}\\[0.3em]
		-\tfrac{91}{64} & \tfrac{91}{128}\end{bmatrix}\right\}_{[-1,1]}}.
	\end{align*}
	Note that $\phi= (\phi^{1},\phi^{2})^{\tp}$ is the Hermite quadratic splines (\cite{hm07}).
	Then $\sm(a)=2.5$, $\sm(\tilde{a})\approx 0.184258$, $\sr(a)=3$, $\sr(\tilde{a})=1$, and its matching filters $\vgu,\tilde{\vgu} \in (l_{0}(\Z))^{1 \times 2}$ with $\wh{\vgu}(0)\wh{\phi}(0)=\wh{\tilde{\vgu}}(0)\wh{\tilde{\phi}}(0)=1$ are given by $\wh{\vgu}(0)=\wh{\tilde{\vgu}}(0)=(1,0)$, $\wh{\vgu}'(0)=i(0,1)$ and $\wh{\vgu}''(0)=(0,0)$ (see \cite[Theorem~6.2.3]{hanbook} and \cite{han01,han03jat}).
	We use the direct approach in \cref{sec:direct}.
	By item (i) of \cref{prop:phicut} with $n_\phi=1$, the left boundary refinable vector function is given by $\phi^{L}:=\phi \chi_{[0,\infty)}$ with $\#\phi^L=2$ and satisfies
	\be \label{phiL:ex}
	\phi^{L}=(\phi^L_1,\phi^L_2)^\tp
	:=(\phi^L_1(2\cdot),
	\tfrac{1}{2}\phi^L_2(2\cdot))^\tp
	+ 2a(1)
	\phi(2\cdot-1).	
	\ee
	Taking $n_\psi=1$ and $m_{\phi}=2$ in \cref{thm:direct}, we have  $\psi^{L}:= [\tfrac{1}{6},-1]\phi^{L}(2\cdot)$ with $\#\psi^L=1$ satisfying both items (i) and (ii) of \cref{thm:direct} with $\fs(C)=\fs(D)=[1,1]$, $B_0=[-3, -\tfrac{3}{2}, 3, -\tfrac{3}{2}]^{\tp}$ and
	\[
	A_0={\begin{bmatrix}
		\tfrac{3}{2} &\tfrac{1}{4} &-\tfrac{1}{2} &\tfrac{1}{4}\\
		-6 & -1 & 6 & -3\end{bmatrix}^\tp},\quad
	C(1):=
	{\begin{bmatrix}
		0 & 0 & \tfrac{1}{4} & \tfrac{1}{8} \\[0.2em]
		0 & 0 & -\tfrac{91}{64} & -\frac{91}{128}
		\end{bmatrix}^\tp}, \quad
	D(1)=
	{\begin{bmatrix}
		0 & 0 & -\tfrac{1}{4} & -\tfrac{1}{8}\\[0.2em]
		0 & 0 & {\frac{91}{64}} & {\frac{91}{128}}
		\end{bmatrix}^{\tp}}.
	\]
	Since \eqref{tAL} is satisfied with $\rho(\tilde{A}_{L})=1/2$,
	we conclude from \cref{thm:direct,thm:bw:0N} that
	$\cB_J=\Phi_J \cup \{\Psi_j \setsp j\ge J\}$ is a Riesz basis of $L_2([0,1])$ for every $J\ge J_0:=0$, where $\Phi_j$ and $\Psi_j$ in \eqref{Phij} and \eqref{Psij} with $n_{\phi}=n_{\mathring{\phi}}=n_{\psi}=n_{\mathring{\psi}}=1$ are given by
	\begin{align*}
	\Phi_{j} & = \{\phi^{L}_{j;0}\} \cup \{\phi_{j;k}:1\le k \le 2^{j}-1\} \cup \{\phi^{R}_{j;2^{j}-1}\}, \quad \mbox{with} \quad \phi^{R}:=\phi^{L}(1-\cdot),\\
	\Psi_{j} & = \{\psi^{L}_{j;0}\} \cup \{\psi_{j;k}:1\le k \le 2^{j}-1\} \cup \{\psi^{R}_{j;2^{j}-1}\}, \quad \mbox{with} \quad \psi^{R}:=\psi^{L}(1-\cdot)
	\end{align*}
	with $\#\phi^L=\#\phi^R=2$, $\#\psi^L=\#\psi^R=1$, $\#\Phi_j=2^{j+1}+2$ and $\#\Psi_j=2^{j+1}$.
	Note that $\vmo(\psi^L)=\vmo(\psi^R)=\vmo(\psi)=1=\sr(\tilde{a})$.
	The dual Riesz basis $\tilde{\cB}_J$ of $\cB_J$ for $J\ge \tilde{J}_0=0$
	is given by \cref{thm:direct} through \eqref{tphi:implicit} and \eqref{tpsi:implicit}.
	Since $\tilde{\phi}^L$ in \eqref{tphi:implicit} contains interior elements,
	we rewrite $\tilde{\phi}^{L}$ as $\{\tilde{\phi}^L, \tilde{\phi}(\cdot-1)\}$ with true boundary elements $\tilde{\phi}^L$ and $\#\tilde{\phi}^L=2$.
	Similarly, we rewrite $\tilde{\psi}^{L}$
	in \eqref{tpsi:implicit} as $\{\tilde{\psi}^L, \psi(\cdot-1)\}$ with true boundary elements $\tilde{\psi}^L$ and $\#\tilde{\psi}^L=1$.
	Hence,
	$\tilde{\cB}_J=\tilde{\Phi}_J \cup \{\tilde{\Psi}_j \setsp j\ge J\}$ is given by
	\begin{align*}
	\tilde{\Phi}_{j} & = \{\tilde{\phi}^{L}_{j;0}\} \cup \{\tilde{\phi}_{j;k}:1\le k \le 2^{j}-1\} \cup \{\tilde{\phi}^{R}_{j;2^{j}-1}\}, \quad \mbox{with} \quad \tilde{\phi}^{R}:=\tilde{\phi}^{L}(1-\cdot),\\
	\tilde{\Psi}_{j} & = \{\tilde{\psi}^{L}_{j;0}\} \cup \{\tilde{\psi}_{j;k}:1\le k \le 2^{j}-1\} \cup \{\tilde{\psi}^{R}_{j;2^{j}-1}\}, \quad \mbox{with} \quad \tilde{\psi}^{R}:=\tilde{\psi}^{L}(1-\cdot).
	\end{align*}
	Note that $\vmo(\tilde{\psi}^L)=\vmo(\tilde{\psi}^R)=\vmo(\tilde{\psi})=3$
	and hence $\PL_{2}\chi_{[0,1]}\subset\mbox{span}(\Phi_j)$ for all $j\in \NN$. According to Theorem \ref{thm:bw:0N} with $N=1$, $(\tilde{\cB}_J,\cB_J)$ forms a biorthogonal Riesz basis of $L_{2}([0,1])$ for every $J \in \NN$.
	
	By item (ii) of \cref{prop:phicut} with $n_\phi=1$ and $\textsf{p}(x)=(x,x^{2})^\tp$, the left boundary refinable vector function is given by $\phi^{L,bc}:=\phi^{2} \chi_{[0,\infty)}=\phi^L_2$ (the second entry of $\phi^L$) with $\#\phi^{L,bc}=1$
	and satisfies $\phi^{L,bc} = \frac{1}{2}\phi^{L,bc}(2\cdot) + [\tfrac{1}{8},-\tfrac{1}{4}] \phi(2\cdot-1)$ by \eqref{phiL:ex}.
	Taking $n_\psi=1$ and $m_{\phi}=2$ in \cref{thm:direct},
	we have
	\[
	\psi^{L,bc}:= -\phi^L_2(2\cdot) + [\tfrac{1}{12},-\tfrac{1}{6}]\phi(2\cdot-1)
	=\psi^L+[-\tfrac{1}{6}, 0]\phi^L(2\cdot)+ [\tfrac{1}{12}, -\tfrac{1}{6}]\phi(2\cdot-1)
	\]
	with $\#\psi^{L,bc}=1$ satisfying
	both items (i) and (ii) of \cref{thm:direct}, where $A_0^{bc}$, $B_0^{bc}$, $C^{bc}$ and $D^{bc}$ can be easily derived from $A_0, B_0, C$ and $D$. More precisely,
	$A_0^{bc}$ is obtained from $U^{-1} A_0$ by taking out its first row and first column, and $B_0^{bc}, C^{bc}, D^{bc}$ are obtained from $U^{-1}B_0, U^{-1}C, U^{-1}D$, respectively by removing their first rows, where the invertible matrix $U$ is given by
	\[
	U:=I_4+B_0[-\tfrac{1}{6}, 0, \tfrac{1}{12}, -\tfrac{1}{6}]
	=I_4+
	[-3, -\tfrac{3}{2}, 3, -\tfrac{3}{2}]^{\tp}[-\tfrac{1}{6}, 0, \tfrac{1}{12}, -\tfrac{1}{6}].
	\]
	Explicitly, we have
	$A_0^{bc} =[\tfrac{1}{2}, 3, -\tfrac{3}{2}]^\tp$,
	$B_0^{bc}=[-\tfrac{3}{4}, \tfrac{3}{2}, -\tfrac{3}{4}]^\tp$ and
	$\fs(C^{bc})=\fs(D^{bc})=[1,1]$, where
	\[
	C^{bc}(1)=
	{\begin{bmatrix}
		0 & \tfrac{1}{4} & \tfrac{1}{8}\\[0.2em]	 0&-{\frac{91}{64}}&-{\frac{91}{128}}
		\end{bmatrix}^{\tp}}, \qquad
	D^{bc}(1)=
	{\begin{bmatrix}
		0 & -\tfrac{1}{4} & -\tfrac{1}{8}\\[0.2em]
		0&{\frac{91}{64}}&{\frac{91}{128}}
		\end{bmatrix}^{\tp}}.
	\]
	Since \eqref{tAL} is satisfied with
	$\rho(\tilde{A}_{L}^{bc})=1/2$, we conclude from \cref{thm:direct,thm:bw:0N}
	that $\cB_J^{bc}=\Phi_J^{bc} \cup \{\Psi_j^{bc} \setsp j\ge J\}$ is a Riesz basis of $L_2([0,1])$ for every $J\ge J_0:=0$ such that $h(0)=h(1)$ for all $h\in \cB^{bc}_J$, where $\Phi_j^{bc}$ and $\Psi_j^{bc}$ in \eqref{Phij} and \eqref{Psij} with $n_{\phi}=n_{\mathring{\phi}}=n_{\psi}=n_{\mathring{\psi}}=1$ are given by
	\begin{align*}
	\Phi_{j}^{bc} & = \{\phi^{L,bc}_{j;0}\} \cup \{\phi_{j;k}:1\le k \le 2^{j}-1\} \cup \{\phi^{R,bc}_{j;2^{j}-1}\}, \quad \mbox{with} \quad \phi^{R,bc}:=\phi^{L,bc}(1-\cdot),\\
	\Psi_{j}^{bc} & = \{\psi^{L,bc}_{j;0}\} \cup \{\psi_{j;k}:1\le k \le 2^{j}-1\} \cup \{\psi^{R,bc}_{j;2^{j}-1}\}, \quad \mbox{with} \quad \psi^{R,bc}:=\psi^{L,bc}(1-\cdot)
	\end{align*}
	with $\#\phi^{L,bc}=\#\psi^{L,bc}=1$ and
	$\#\Phi^{bc}_j=\#\Psi^{bc}_j=2^{j+1}$.
	Note that $\vmo(\psi^{L,bc})=\vmo(\psi^{R,bc})=\vmo(\psi)=1$
	and $\Phi^{bc}_j=\Phi_j\bs \{\phi^{L,1}_{j;0}, \phi^{R,1}_{j;2^j-1}\}$ as in \cref{prop:mod}.
	The dual Riesz basis $\tilde{\cB}_J^{bc}$ of $\cB_J^{bc}$ with $j\ge \tilde{J}_0:=0$ is given by \cref{thm:direct} through
	\eqref{tphi:implicit} and \eqref{tpsi:implicit}.
	We rewrite $\tilde{\phi}^{L,bc}$ in \eqref{tphi:implicit} as $\{\tilde{\phi}^{L,bc}, \tilde{\phi}(\cdot-1)\}$ with true boundary elements $\tilde{\phi}^{L,bc}$ and $\#\tilde{\phi}^{L,bc}=1$.
	Similarly, we rewrite $\tilde{\psi}^{L,bc}$
	in \eqref{tpsi:implicit} as $\{\tilde{\psi}^{L,bc}, \tilde{\psi}(\cdot-1)\}$ with true boundary elements $\tilde{\psi}^{L,bc}$ and $\#\tilde{\psi}^{L,bc}=1$.
	Hence,
	$\tilde{\cB}_J^{bc}=\tilde{\Phi}_J^{bc} \cup \{\tilde{\Psi}_j^{bc} \setsp j\ge J\}$ is given by
	\begin{align*}
	\tilde{\Phi}_{j}^{bc} & = \{\tilde{\phi}^{L,bc}_{j;0}\} \cup \{\tilde{\phi}_{j;k}:1\le k \le 2^{j}-1\} \cup \{\tilde{\phi}^{R,bc}_{j;2^{j}-1}\}, \quad \mbox{with} \quad \tilde{\phi}^{R}:=\tilde{\phi}^{L,bc}(1-\cdot),\\
	\tilde{\Psi}_{j}^{bc} & = \{\tilde{\psi}^{L,bc}_{j;0}\} \cup \{\tilde{\psi}_{j;k}:1\le k \le 2^{j}-1\} \cup \{\tilde{\psi}^{R,bc}_{j;2^{j}-1}\}, \quad \mbox{with} \quad \tilde{\psi}^{R}:=\tilde{\psi}^{L,bc}(1-\cdot).
	\end{align*}
	By Theorem \ref{thm:bw:0N} with $N=1$, $(\tilde{\cB}_J^{bc},\cB_J^{bc})$ forms a biorthogonal Riesz basis of $L_{2}([0,1])$ with $h(0)=h(1)=0$ for all $h\in \cB_J^{bc}$
	for $J \in \NN$. See \cref{fig:hmtquad} for the graphs of $\phi$, $\psi$, $\tilde{\phi}$, $\tilde{\psi}$, and all boundary elements.
\end{example}

\begin{figure}[htbp]
	\centering \begin{subfigure}[b]{0.24\textwidth} \includegraphics[width=\textwidth,height=0.6\textwidth]{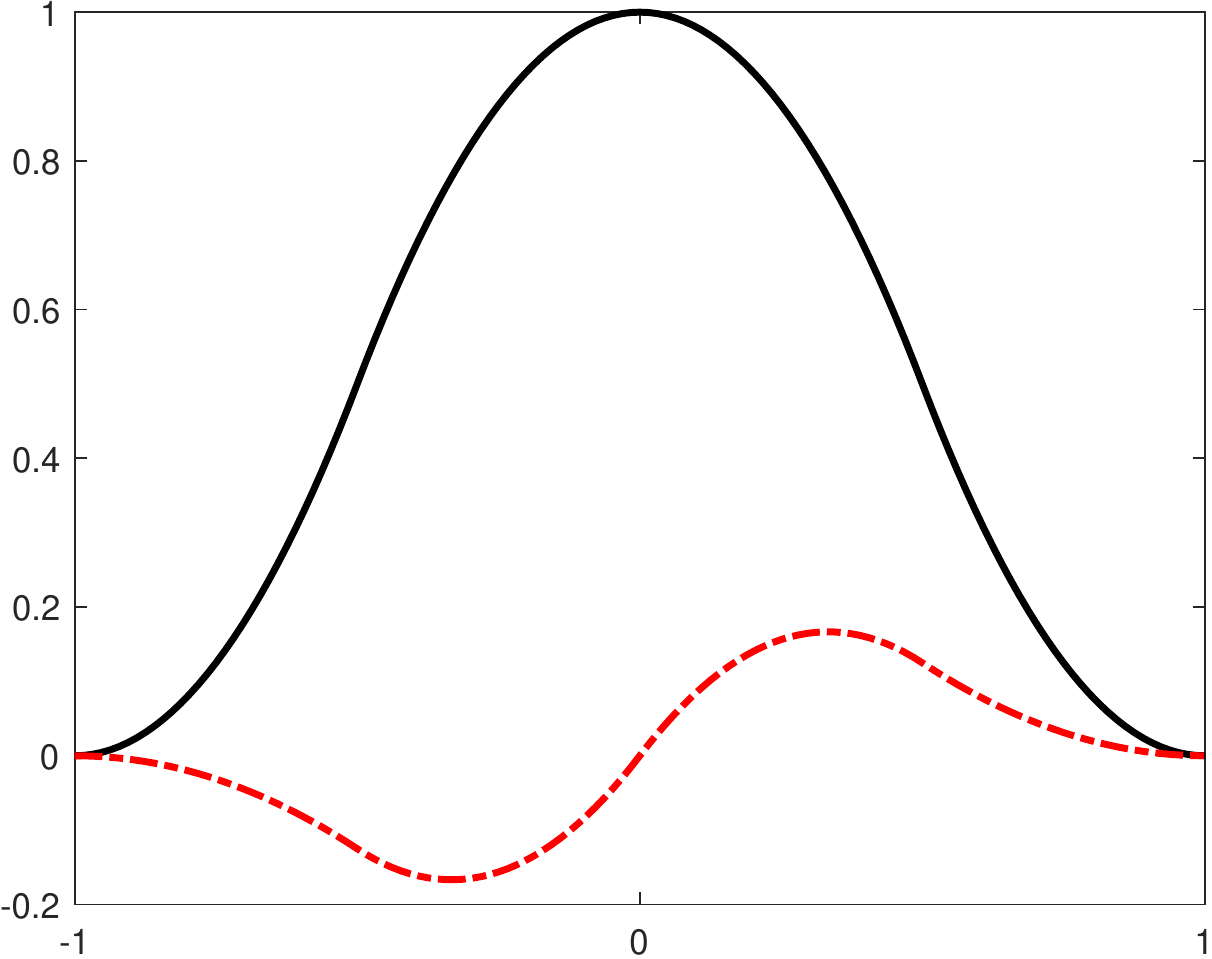}
		 \caption{$\phi=(\phi^{1},\phi^{2})^{\tp}$}
	\end{subfigure} \begin{subfigure}[b]{0.24\textwidth} \includegraphics[width=\textwidth,height=0.6\textwidth]{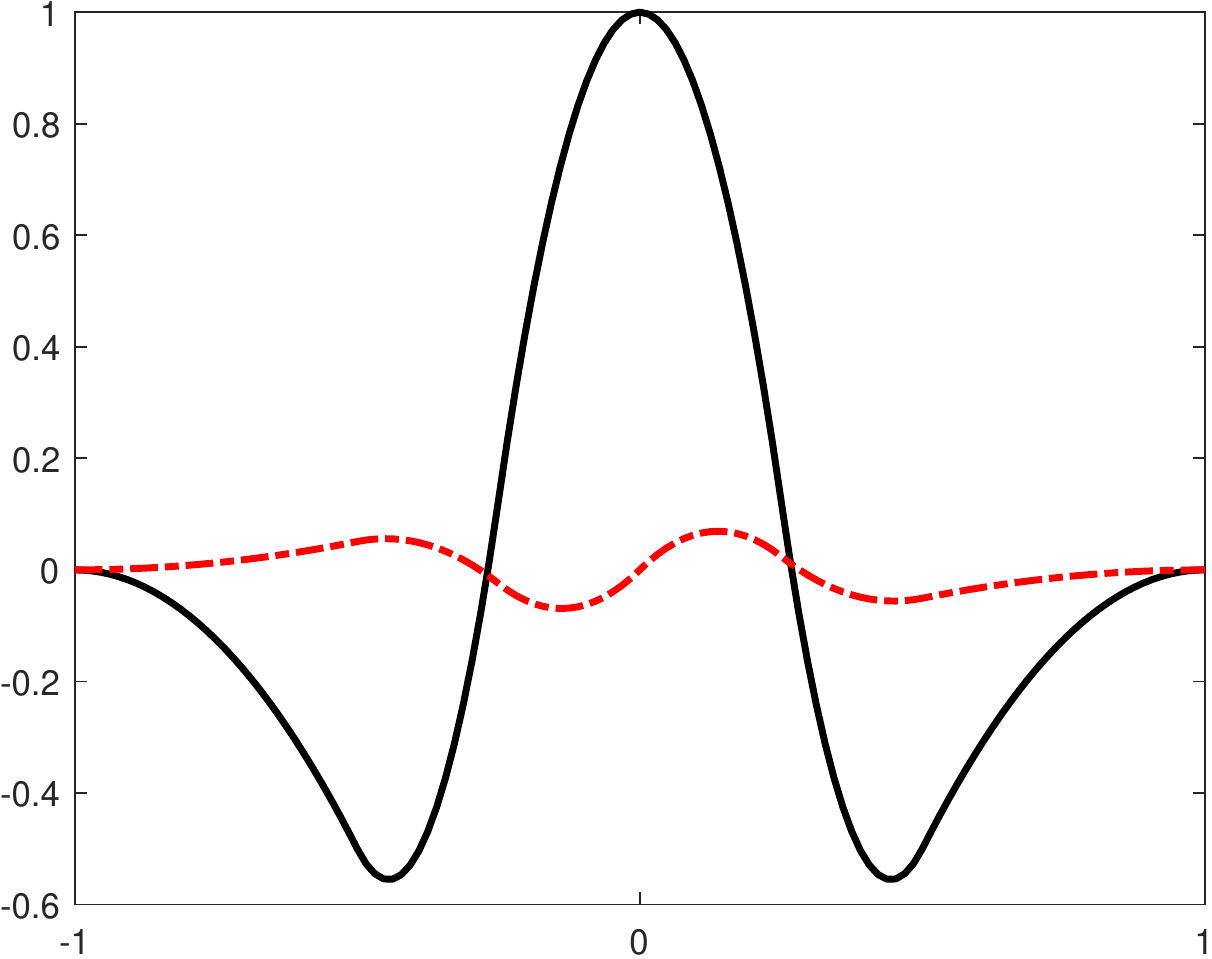}
		 \caption{$\psi=(\psi^{1},\psi^{2})^{\tp}$}
	\end{subfigure} \begin{subfigure}[b]{0.24\textwidth} \includegraphics[width=\textwidth,height=0.6\textwidth]{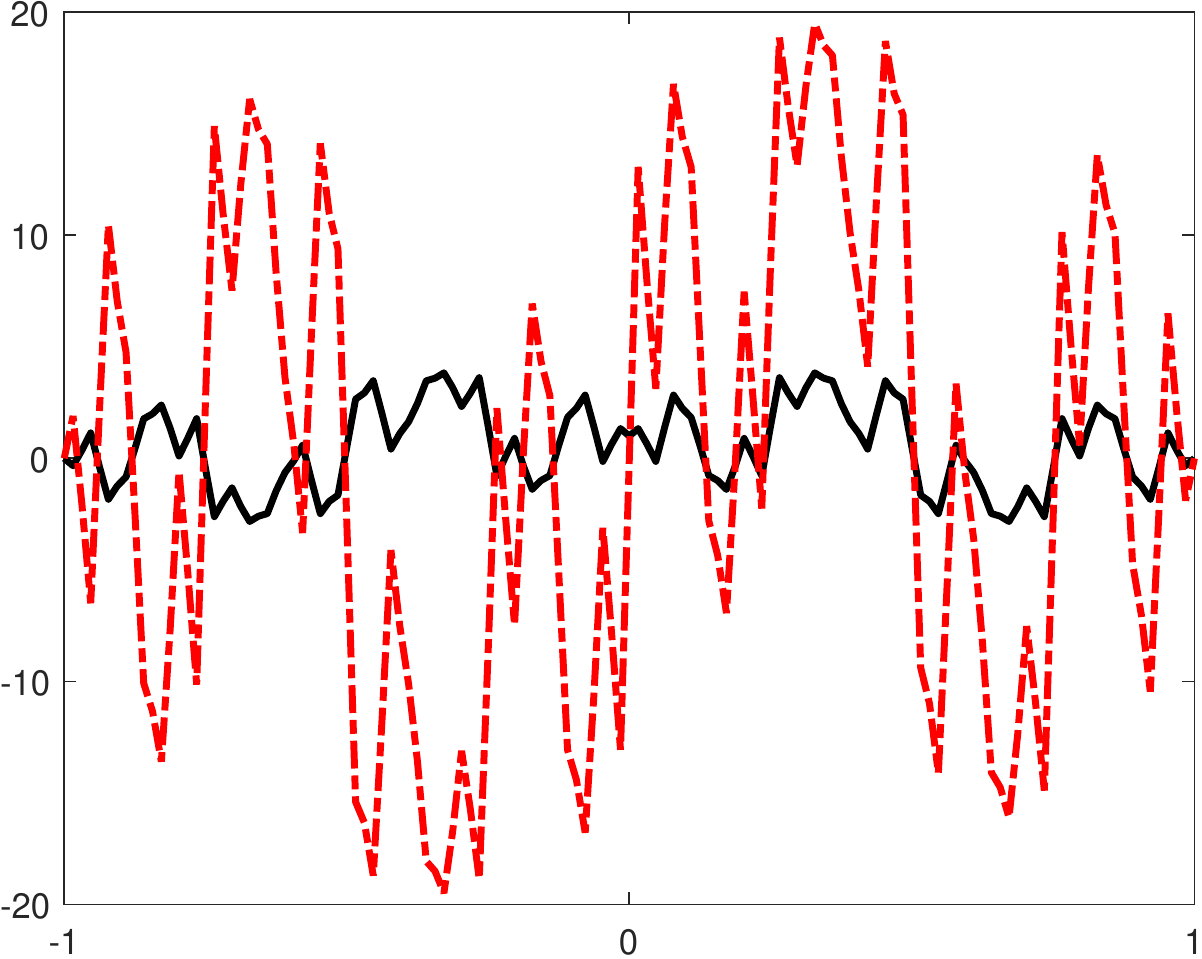}
		\caption{$\tilde{\phi} = (\tilde{\phi}^{1},\tilde{\phi}^{2})^{\tp}$}
	\end{subfigure} \begin{subfigure}[b]{0.24\textwidth} \includegraphics[width=\textwidth,height=0.6\textwidth]{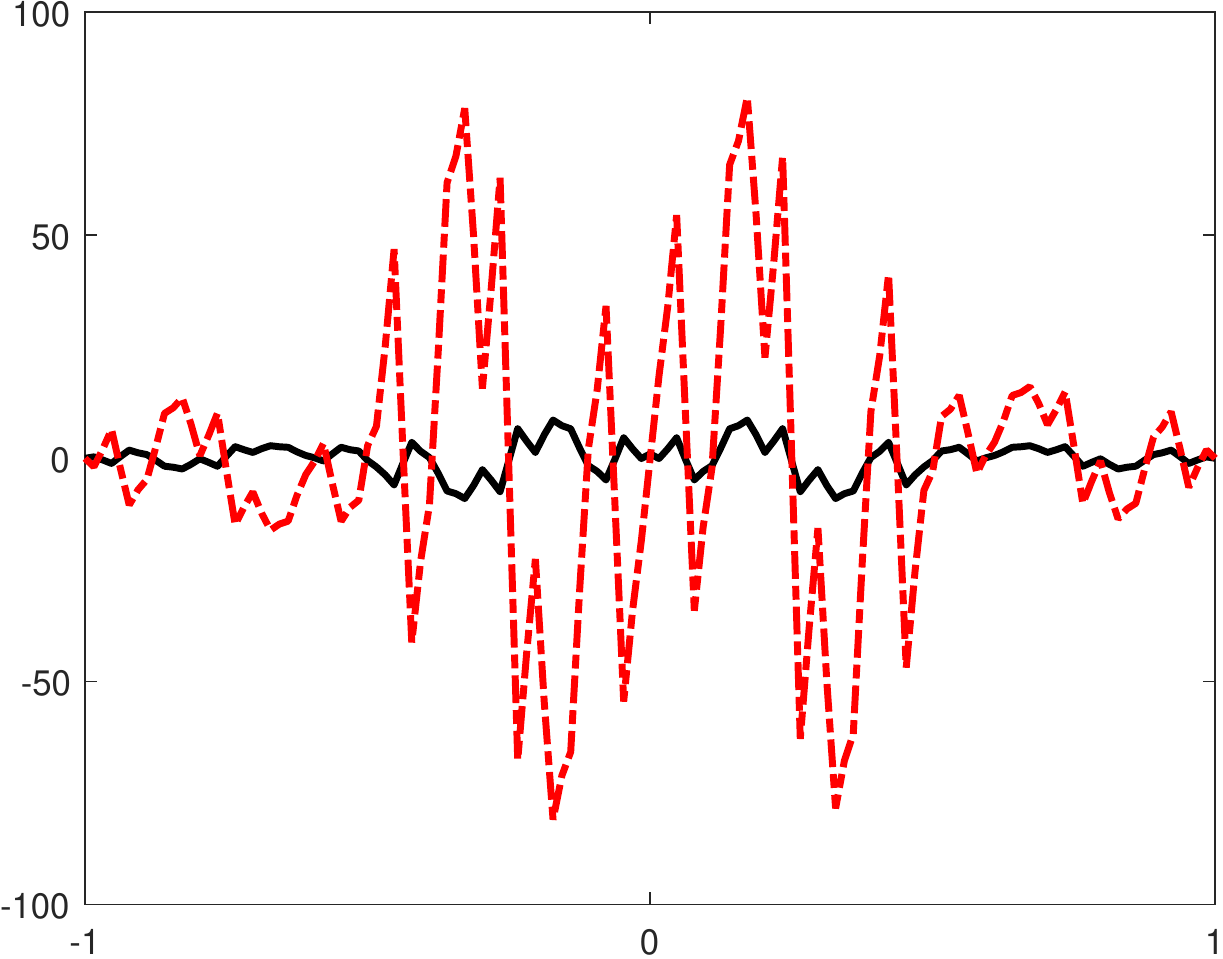}
		\caption{$\tilde{\psi}= (\tilde{\psi}^{1},\tilde{\psi}^{2})^{\tp}$}
	\end{subfigure} \begin{subfigure}[b]{0.24\textwidth} \includegraphics[width=\textwidth,height=0.6\textwidth]{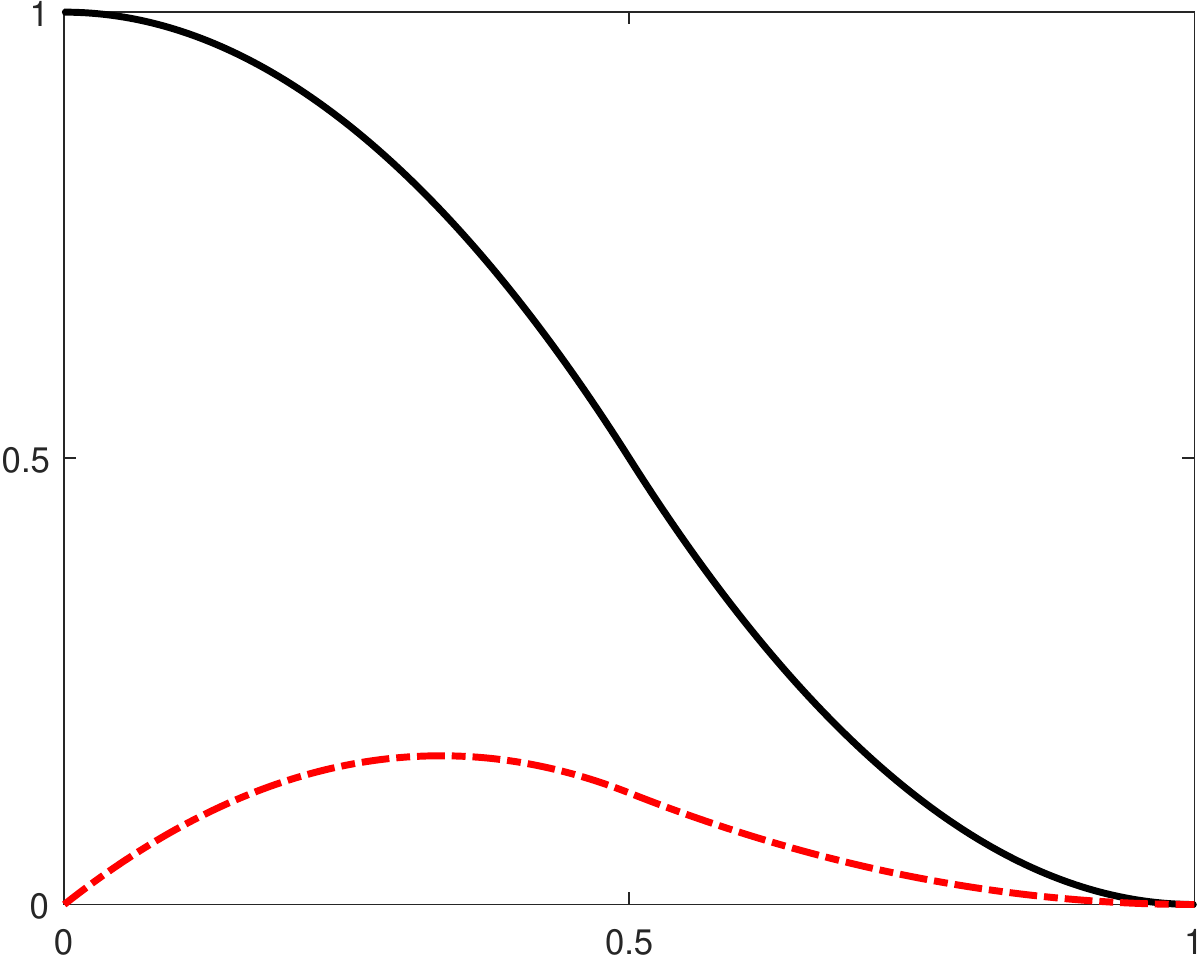}
		\caption{$\phi^{L}$}
	\end{subfigure} \begin{subfigure}[b]{0.24\textwidth} \includegraphics[width=\textwidth,height=0.6\textwidth]{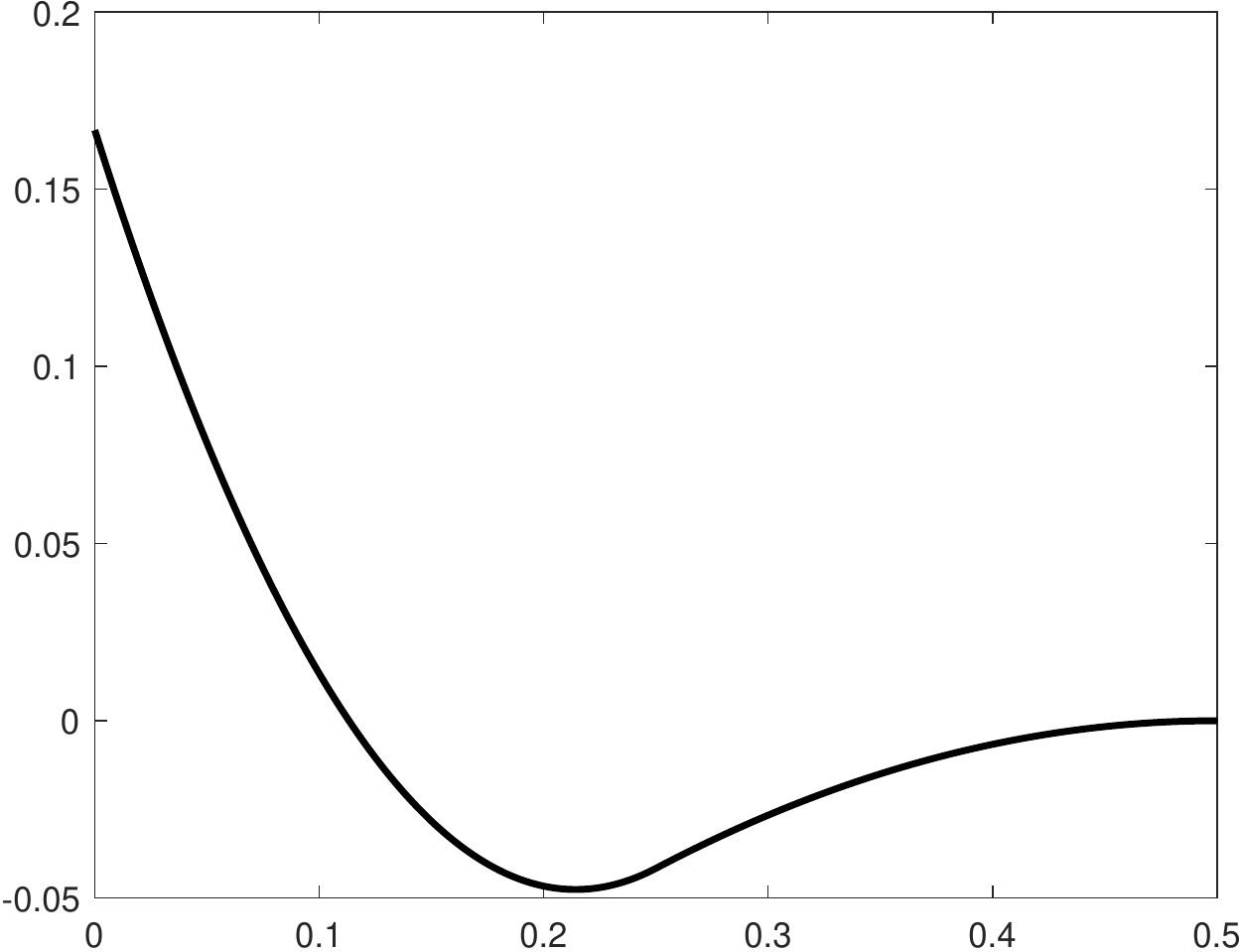}
		\caption{$\psi^{L}$}
	\end{subfigure} \begin{subfigure}[b]{0.24\textwidth} \includegraphics[width=\textwidth,height=0.6\textwidth]{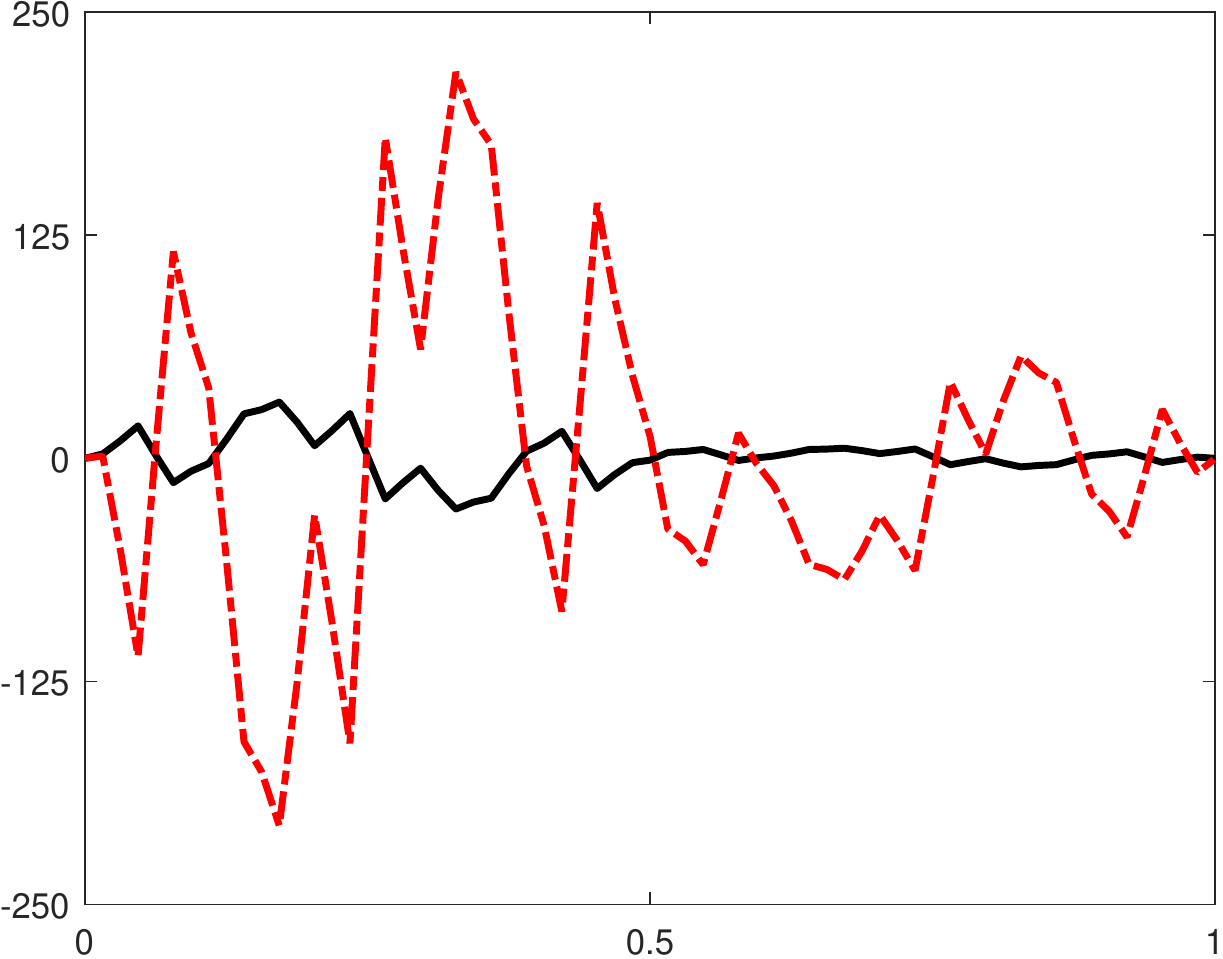}
		\caption{$\tilde{\phi}^{L}$}
	\end{subfigure} \begin{subfigure}[b]{0.24\textwidth} \includegraphics[width=\textwidth,height=0.6\textwidth]{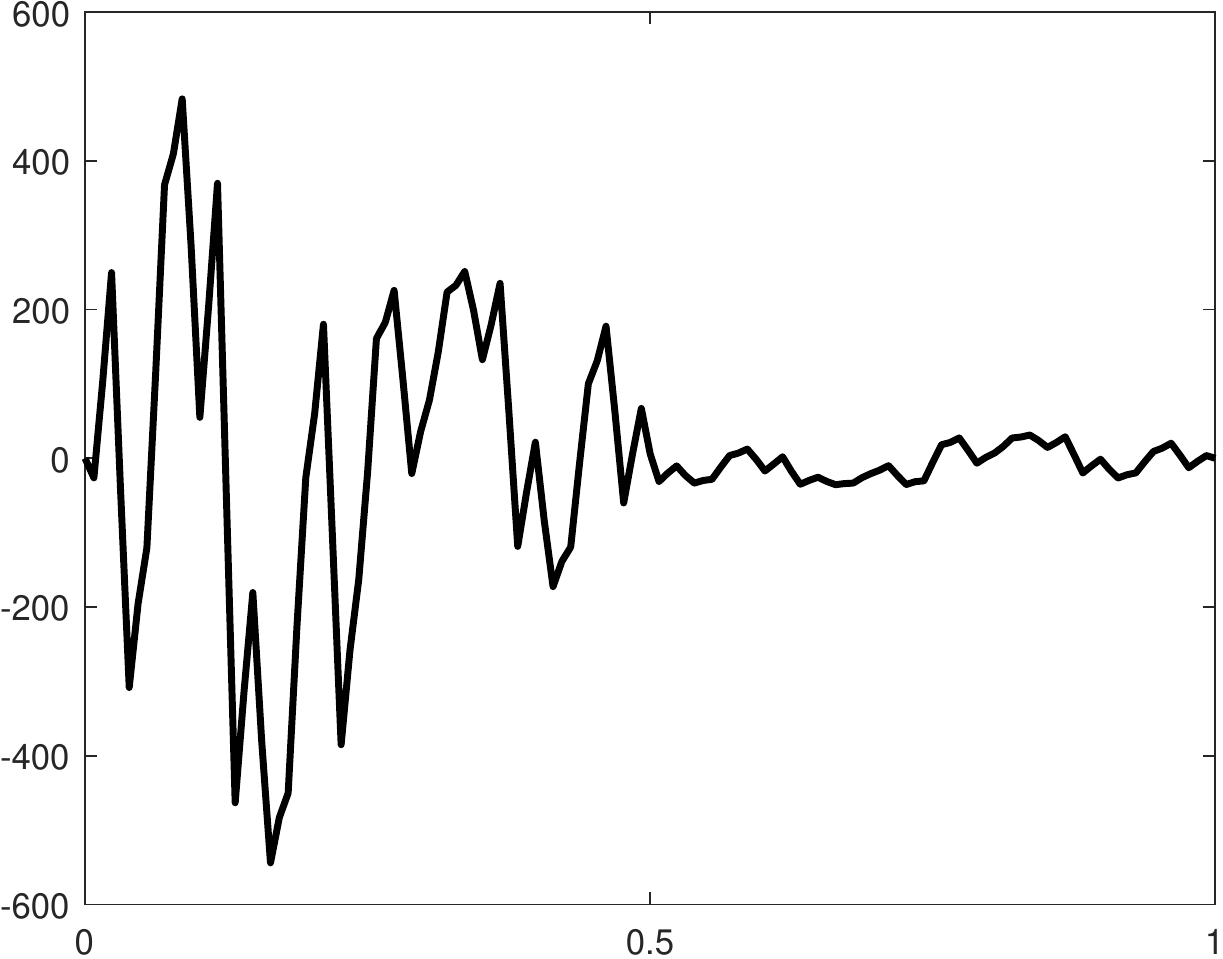}
		\caption{$\tilde{\psi}^{L}$}
	\end{subfigure} \begin{subfigure}[b]{0.24\textwidth} \includegraphics[width=\textwidth,height=0.6\textwidth]{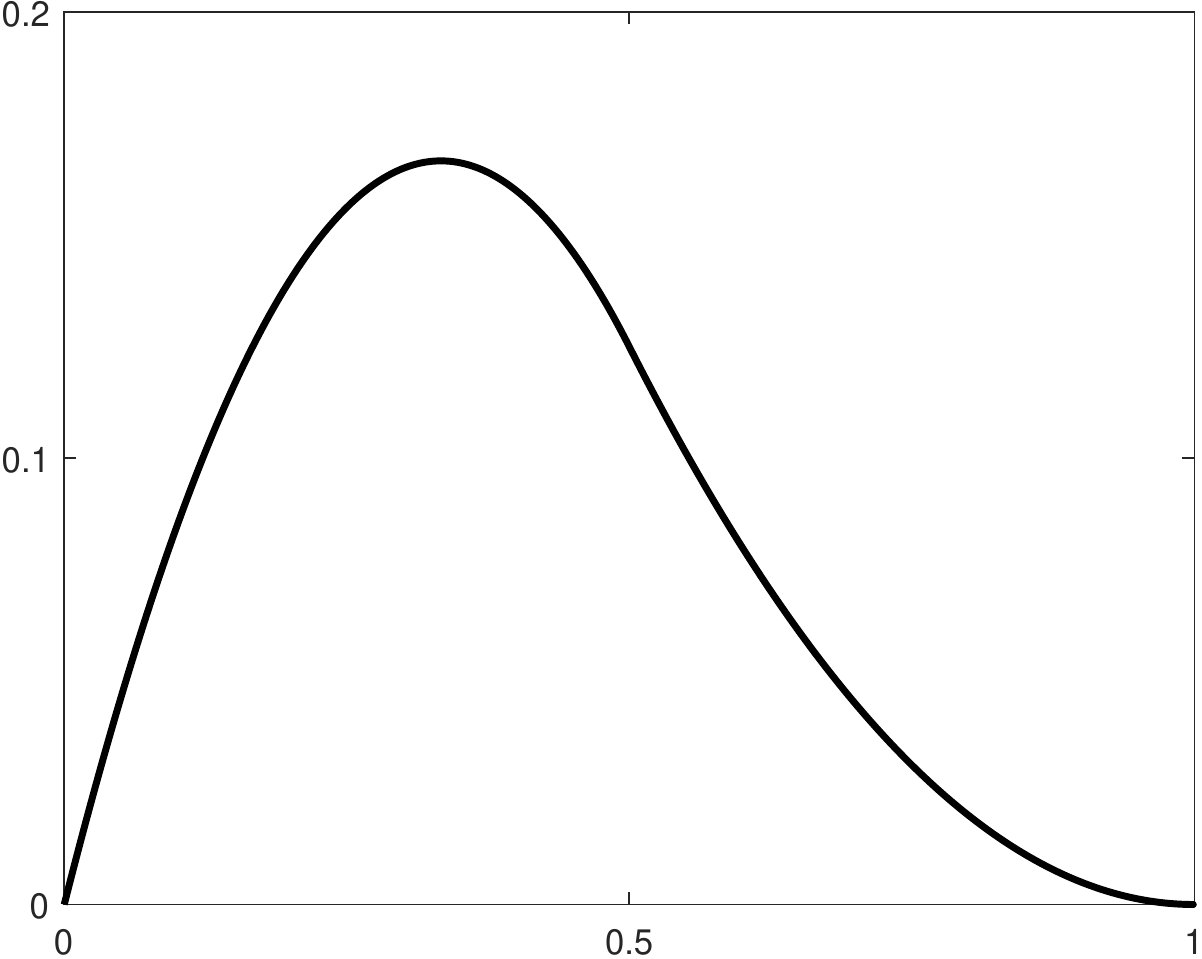}
		\caption{$\phi^{L,bc}$}
	\end{subfigure} \begin{subfigure}[b]{0.24\textwidth} \includegraphics[width=\textwidth,height=0.6\textwidth]{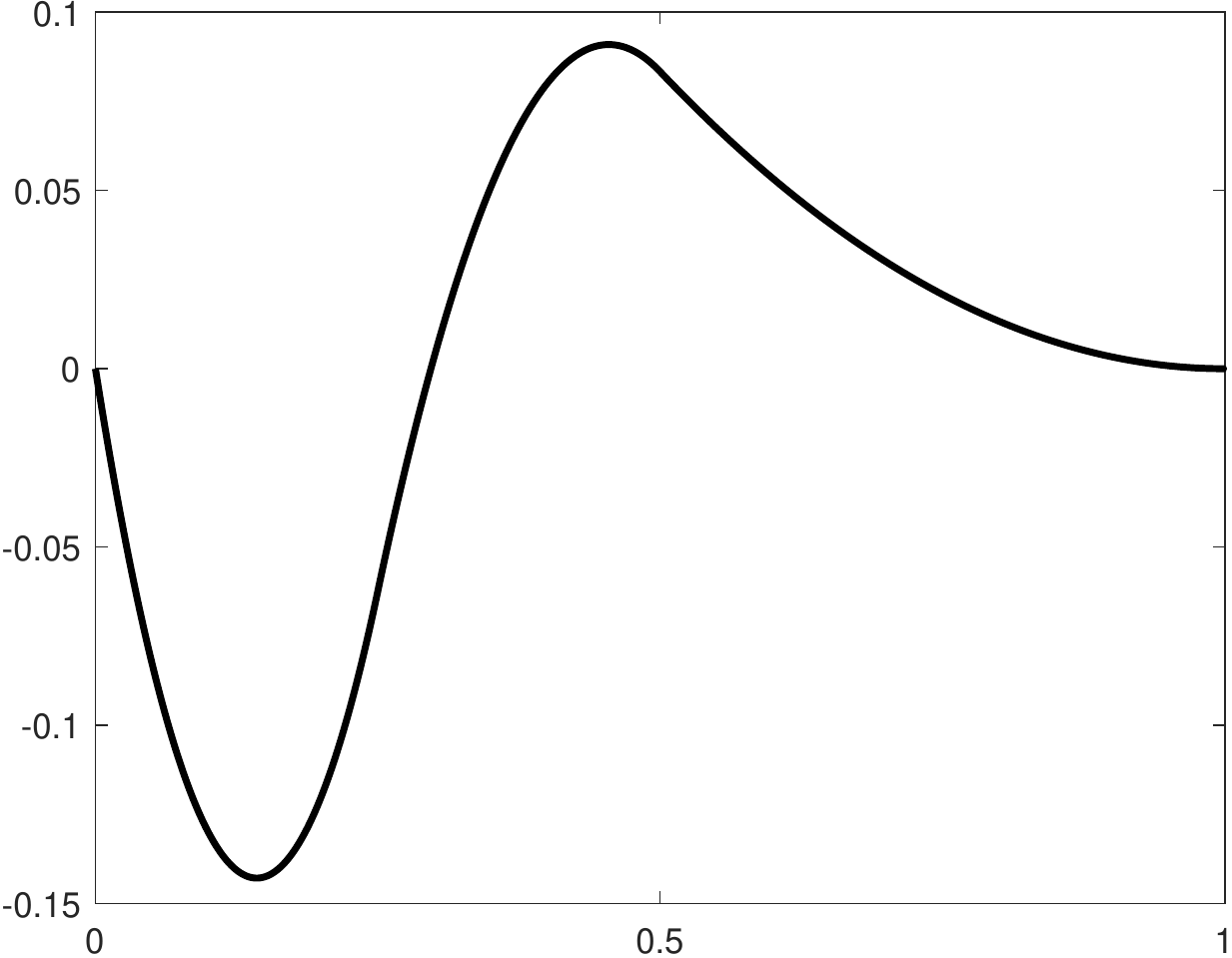}
		\caption{$\psi^{L,bc}$}
	\end{subfigure} \begin{subfigure}[b]{0.24\textwidth} \includegraphics[width=\textwidth,height=0.6\textwidth]{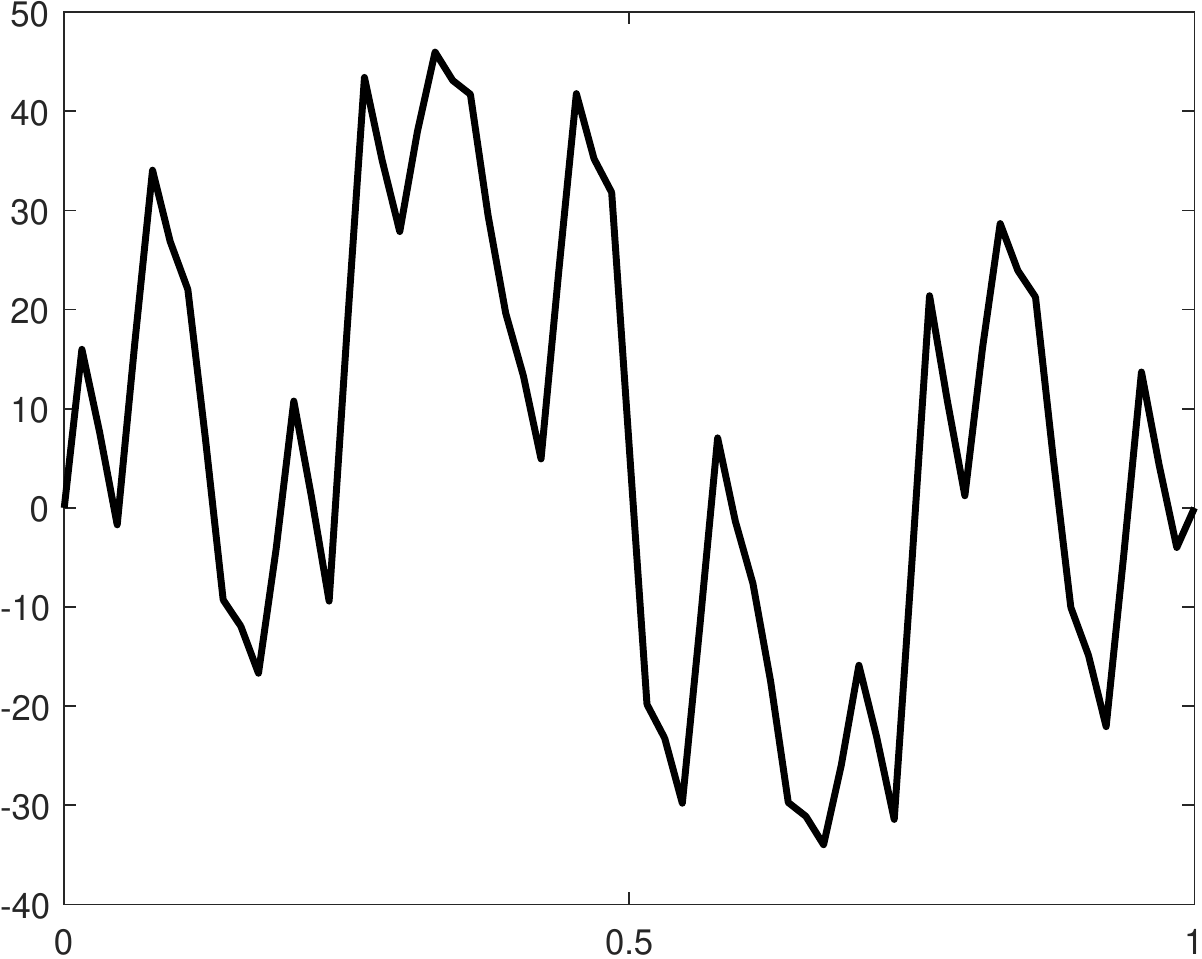} \caption{$\tilde{\phi}^{L,bc}$}
	\end{subfigure} \begin{subfigure}[b]{0.24\textwidth} \includegraphics[width=\textwidth,height=0.6\textwidth]{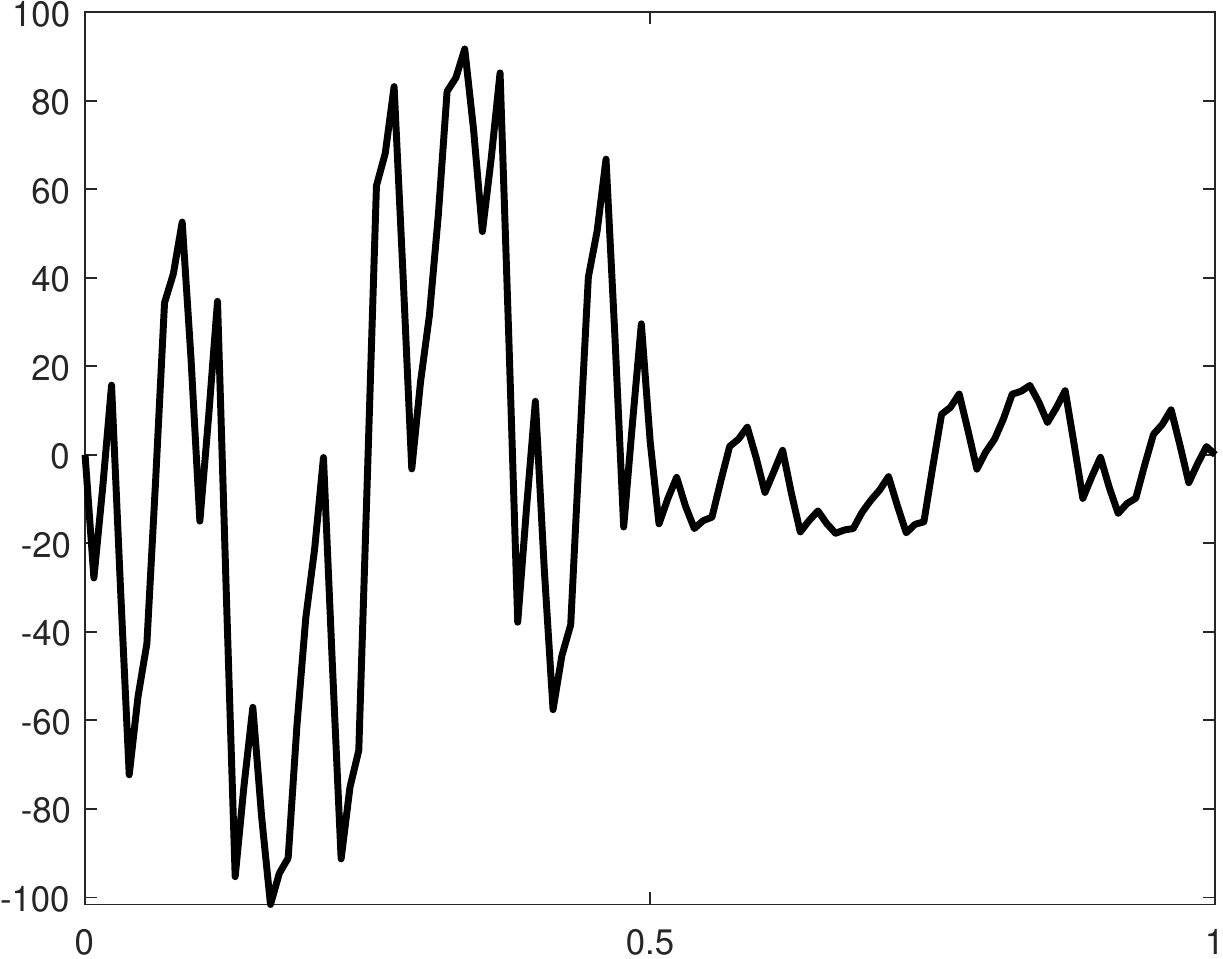}
		\caption{$\tilde{\psi}^{L,bc}$}
	\end{subfigure}
	\caption{The generators of the biorthogonal wavelet bases $(\tilde{\cB}_J,\cB_J)$ and  $(\tilde{\cB}_J^{bc},\cB_J^{bc})$ of $L_{2}([0,1])$ for $J\in \NN$ in \cref{ex:hmtquad} with $h(0)=h(1)=0$ for all $h\in \cB_J^{bc}$. The black and red lines correspond to the first, second, and third components of a vector function.
	}\label{fig:hmtquad}
\end{figure}

\begin{example} \label{ex:hmtvm4}
	\normalfont
	Using the CBC algorithm in \cite{han01,hanbook}, we construct
	a  biorthogonal wavelet $(\{\tilde{\phi};\tilde{\psi}\},\{\phi;\psi\})$ with $\wh{\phi}(0)=\wh{\tilde{\phi}}(0)=(1,0)^\tp$ and a biorthogonal wavelet filter bank $(\{\tilde{a};\tilde{b}\},\{a;b\})$ given by
	\begin{align*} 
	a=&{\left\{ \begin{bmatrix} \tfrac{1}{4} &\tfrac{3}{8}\\[0.3em]
		-\tfrac{1}{16} &-\tfrac{1}{16}\end{bmatrix},
		\begin{bmatrix} \tfrac{1}{2} &0 \\[0.3em]
		0 &\tfrac{1}{4}\end{bmatrix},
		\begin{bmatrix} \tfrac{1}{4} &-\tfrac{3}{8}\\[0.3em]
		\tfrac{1}{16} &-\tfrac{1}{16}\end{bmatrix}\right\}_{[-1,1]}},\\
	b=&{\left\{ \begin{bmatrix} 0 & 0\\[0.3em]
		\tfrac{2}{97} & \tfrac{24}{679}\end{bmatrix},
		\begin{bmatrix} -\tfrac{1}{2} & -\tfrac{15}{4}\\[0.3em]
		\tfrac{77}{1164} & \tfrac{2921}{2761}\end{bmatrix},
		\begin{bmatrix} 1 & 0 \\[0.3em]
		0 & 1 \end{bmatrix},
		\begin{bmatrix} -\tfrac{1}{2} & \tfrac{15}{4}\\[0.3em]
		-\tfrac{77}{1164} & \tfrac{2921}{2761}\end{bmatrix},
		\begin{bmatrix} 0 & 0\\[0.3em]
		-\tfrac{2}{97} & \tfrac{24}{679}\end{bmatrix}\right\}_{[-2,2]}},\\ 
	\tilde{a}=
	&{\left\{\begin{bmatrix} -\tfrac{13}{2432} & -\tfrac{91}{29184}\\[0.3em]
		\tfrac{3}{152} & \tfrac{7}{608}\end{bmatrix},
		\begin{bmatrix} \tfrac{39}{2432} &\tfrac{13}{3648}\\[0.3em]
		-\tfrac{9}{152} &-\tfrac{1}{76}\end{bmatrix},
		\begin{bmatrix} -\tfrac{1}{12} & -\tfrac{1699}{43776} \\[0.3em]
		\tfrac{679}{1216} & \tfrac{4225}{14592}\end{bmatrix},
		\begin{bmatrix} \tfrac{569}{2432} & \tfrac{647}{10944}\\[0.3em]
		-\tfrac{1965}{1216} & -\tfrac{37}{96}\end{bmatrix},
		\begin{bmatrix} \tfrac{2471}{3648} & 0\\[0.3em]
		0 &\tfrac{7291}{7296}\end{bmatrix},\right.}
	\\
	&\;\; {\left.\begin{bmatrix} \tfrac{569}{2432} & -\tfrac{647}{10944}\\[0.3em]
		\tfrac{1965}{1216} & -\tfrac{37}{96}\end{bmatrix},
		\begin{bmatrix} -\tfrac{1}{12} & \tfrac{1699}{43776} \\[0.3em]
		-\tfrac{679}{1216} & \tfrac{4225}{14592}\end{bmatrix},
		\begin{bmatrix} \tfrac{39}{2432} &-\tfrac{13}{3648}\\[0.3em]
		\tfrac{9}{152} &-\tfrac{1}{76}\end{bmatrix},
		\begin{bmatrix} -\tfrac{13}{2432} & \tfrac{91}{29184}\\[0.3em]
		-\tfrac{3}{152} & \tfrac{7}{608}\end{bmatrix}\right\}_{[-4,4]}},\\
	\tilde{b}=&{\left\{\begin{bmatrix} -\tfrac{1}{4864} & -\tfrac{7}{58368}\\[0.3em]
		0 & 0\end{bmatrix},
		\begin{bmatrix} \tfrac{3}{4864} & \tfrac{1}{7296}\\[0.3em]
		0 & 0\end{bmatrix},
		\begin{bmatrix} \tfrac{1}{24} & \tfrac{2161}{87552}\\[0.3em]
		-\tfrac{679}{4864} & -\tfrac{4753}{58368}\end{bmatrix},
		\begin{bmatrix} -\tfrac{611}{4864} & -\tfrac{605}{21888}\\[0.3em]
		\tfrac{2037}{4864} & \tfrac{679}{7296}\end{bmatrix},
		\begin{bmatrix} \tfrac{1219}{7296} & 0\\[0.3em]
		0 & \tfrac{7469}{29814}\end{bmatrix},\right.}\\
	& \;\;{\left. \begin{bmatrix} -\tfrac{611}{4864} & \tfrac{605}{21888}\\[0.3em]
		-\tfrac{2037}{4864} & \tfrac{679}{7296}\end{bmatrix},
		\begin{bmatrix} \tfrac{1}{24} & -\tfrac{2161}{87552}\\[0.3em]
		\tfrac{679}{4864} & -\tfrac{4753}{58368}\end{bmatrix},
		\begin{bmatrix} \tfrac{3}{4864} & -\tfrac{1}{7296}\\[0.3em]
		0 & 0\end{bmatrix},
		\begin{bmatrix} -\tfrac{1}{4864} & \tfrac{7}{58368}\\[0.3em]
		0 & 0\end{bmatrix}\right\}_{[-4,4]}}.
	\end{align*}
	Note that $\phi$ is the well-known Hermite cubic splines with $\fs(\phi)=[-1,1]$. Note that $\fs(\psi)=[-2,2]$ and $\fs(\tilde{\phi})=\fs(\tilde{\psi})=[-4,4]$.
	Then $\sm(a)=2.5$, $\sm(\tilde{a})=0.281008$,
	$\sr(a)=\sr(\tilde{a})=4$,
	and the matching filters $\vgu, \tilde{\vgu} \in \lrs{0}{1}{2}$ with $\wh{\vgu}(0)\wh{\phi}(0)=\wh{\tilde{\vgu}}(0)\wh{\tilde{\phi}}(0)=1$ are given (see \cite{han01,han03jat}) by $\wh{\vgu}(0,0)=(1,0)$, $\wh{\vgu}'(0)=(0,i)$, $\wh{\vgu}''(0)=\wh{\vgu}'''(0)=(0,0)$ and
	\[
	\wh{\tilde{\vgu}}(0)=(1,0),\quad \wh{\vgu}'(0)=i(0,\tfrac{1}{15}),\quad \wh{\tilde{\vgu}}''(0)=(-\tfrac{2}{15},0),\quad
	\wh{\vgu}'''(0)=i(0,-\tfrac{2}{105}).
	\]
	We use the direct approach as discussed in \cref{sec:direct}.
	By item (i) of \cref{prop:phicut} with $n_\phi=1$, the left boundary refinable vector function is $\phi^{L}:=\phi \chi_{[0,\infty)}$ with $\#\phi^L=2$ and satisfies
	\be \label{phiLhmcubic}
	\phi^{L}=(\phi^{L}_1,\phi^{L}_2)^\tp:=
	 (\phi^L_1(2\cdot),\tfrac{1}{2}\phi^L_2(2\cdot))^\tp
	+ 2a(1) \phi(2\cdot-1).
	\ee
	Taking $n_{\psi}=2$ and $m_{\phi}=7$ in \cref{thm:direct}, we have
	\begin{align*}
	\psi^{L} &=
	{\begin{bmatrix}
		\psi^L_1\\
		\psi^L_2\\
		\psi^L_3\end{bmatrix}}
	:=
	{\begin{bmatrix}
		\phi^L_1(2\cdot)\\
		\phi^L_2(2\cdot)\\
		0 \end{bmatrix}} +
	{\begin{bmatrix}
		\tfrac{31}{54} & \tfrac{533}{36}\\
		\tfrac{1}{36} & \tfrac{7}{4}\\
		1 & \tfrac{390}{61}\\
		\end{bmatrix}}\phi(2\cdot-1) +
	{\begin{bmatrix}
		-\tfrac{29}{27} & \tfrac{59}{9}\\
		-\tfrac{1}{9} & \tfrac{2}{3}\\
		-\tfrac{52}{61} & \tfrac{660}{61}\\
		\end{bmatrix}}\phi(2\cdot-2) +
	{\begin{bmatrix}
		0 & 0\\
		0 & 0\\
		-\tfrac{9}{61} & 0\\
		\end{bmatrix}}\phi(2\cdot-3),
	\end{align*}
	which satisfies items (i) and (ii) of \cref{thm:direct} with $\fs(C)=[1,5]$, $\fs(D)=[2,5]$ and
	{\footnotesize{\begin{align*}
			& A_{0} =	
			\left[\begin{smallmatrix}	 {\frac{79}{64}}&{\frac{49043}{394624}}&-{\frac{54639}{98656}}&-{\frac{1399}{24664}}&{\frac{48111}{49328}}&{\frac{3653}{24664}}&-{\frac{97227}{98656}}&{\frac{1341}{6166}}&{\frac{64593}{197312}}&-{\frac{152367}{789248}}&{\frac{675}{197312}}&-{\frac{75}{98656}}&-{\frac{225}{197312}}&{\frac{525}{789248}}\\
			 -{\frac{279}{64}}&-{\frac{58639}{394624}}&{\frac{823347}{98656}}&-{\frac{6135}{24664}}&-{\frac{561771}{49328}}&-{\frac{33789}{24664}}&{\frac{1097199}{98656}}&-{\frac{30279}{12332}}&-{\frac{729081}{197312}}&{\frac{1718679}{789248}}&-{\frac{7155}{197312}}&{\frac{795}{98656}}&{\frac{2385}{197312}}&-{\frac{5565}{789248}}
			 \end{smallmatrix}\right]^{\tp},\\
			& B_{0} =
			\left[\begin{smallmatrix}
			 -{\frac{15}{64}}&-{\frac{49043}{394624}}&{\frac{54639}{98656}}&{\frac{1399}{24664}}&-{\frac{48111}{49328}}&-{\frac{3653}{24664}}&{\frac{97227}{98656}}&-{\frac{1341}{6166}}&-{\frac{64593}{197312}}&{\frac{152367}{789248}}&-{\frac{675}{197312}}&{\frac{75}{98656}}&{\frac{225}{197312}}&-{\frac{525}{789248}}\\
			 {\frac{279}{128}}&{\frac{847887}{789248}}&-{\frac{823347}{197312}}&{\frac{6135}{49328}}&{\frac{561771}{98656}}&{\frac{
					 33789}{49328}}&-{\frac{1097199}{197312}}&{\frac{30279}{24664}}&{\frac{729081}{394624}}&-{\frac{1718679}{1578496}}&{\frac{7155}{394624}}&-{\frac{795}{197312}}&-{\frac{2385}{394624}}&{\frac{5565}{1578496}}\\ {\frac{61}{9216}}&{\frac{1222501}{170477568}}&-{\frac{282125}{4735488}}&-{\frac{82289}{2663712}}&{\frac{2913787}{
					 7103232}}&{\frac{1504321}{10654848}}&-{\frac{847107}{1578496}}&{\frac{46787}{394624}}&{\frac{563091}{3156992}}&-{\frac{1325957}{12627968}}&{\frac{4941}{3156992}}&-{\frac{549}{1578496}}&-{\frac{1647}{3156992}}&{\frac{3843}{12627968}}
			 \end{smallmatrix}\right]^{\tp},\\
			& C(1) =
			\left[\begin{smallmatrix}
			 -{\frac{1}{256}}&-{\frac{113627}{14206464}}&{\frac{95129}{1183872}}&{\frac{10735}{221976}}&{\frac{552251}{591936}}&{\frac{33025}{887904}}&-{\frac{8905}{394624}}&-{\frac{247}{98656}}&{\frac{4459}{2367744}}&-{\frac{324935}{28412928}}&{\frac{13351}{789248}}&-{\frac{13351}{3551616}}&-{\frac{13351}{2367744}}&{\frac{93457}{28412928}}\\ {\frac{3}{128}}&{\frac{99131}{2367744}}&-{\frac{80633}{197312}}&-{\frac{8923}{36996}}&-{\frac{363803}{98656}}&-{\frac{4033}{147984}}&{\frac{1200891}{197312}}&-{\frac{67723}{49328}}&-{\frac{806635}{394624}}&{\frac{5513543}{4735488}
			 }&{\frac{18123}{394624}}&-{\frac{6041}{591936}}&-{\frac{6041}{394624}}&{\frac{42287}{4735488}}
			 \end{smallmatrix}\right]^{\tp},\\
			& C(2) =	
			\left[\begin{smallmatrix}
			 0&-{\frac{19}{221976}}&{\frac{19}{18498}}&{\frac{19}{27747}}&{\frac{247}{9249}}&{\frac{76}{27747}}&{\frac{513}{6166}}&{\frac{285}{3083}}&{\frac{16363565}{22493568}}&-{\frac{7958533}{269922816}}&{\frac{1757267}{7497856}}&-{\frac{1997741}{33740352}}&-{\frac{772}{9249}}&{\frac{5248657}{134961408}}\\
			 0&{\frac{679}{1183872}}&-{\frac{679}{98656}}&-{\frac{679}{147984}}&-{\frac{8827}{49328}}&-{\frac{679}{36996}}&-{\frac
				 {54999}{98656}}&-{\frac{30555}{49328}}&-{\frac{329829}{937232}}&{\frac{565273}{468616}}&{\frac{3023955}{1874464}}&-{\frac{360659}{937232}}&-{\frac{54999}{98656}}&{\frac{2166985}{7497856}}
			 \end{smallmatrix}\right]^{\tp},\\
			& C(3) =
			\left[\begin{smallmatrix}
			 0&-{\frac{13}{2367744}}&{\frac{13}{197312}}&{\frac{13}{295968}}&{\frac{169}{98656}}&{\frac{13}{73992}}&{\frac{1053}{197312}}&{\frac{585}{98656}}&-{\frac{23611}{295968}}&-{\frac{72635}{1775808}}&{\frac{46169}{197312}}&{\frac{52487}{887904}}&{\frac{15235991}{22493568}}&{\frac{455}{89974272}}\\
			 0&{\frac{1}{49328}}&-{\frac{3}{12332}}&-{\frac{1}{6166}}&-{\frac{39}{6166}}&-{\frac{2}{3083}}&-{\frac{243}{12332}}&-{\frac{135}{6166}}&{\frac{2044109}{3748928}}&{\frac{13373051}{44987136}}&-{\frac{6058455}{3748928}}&-{\frac{2167229}{5623392}}&{\frac{15}{468616}}&{\frac{22477733}{22493568}}
			 \end{smallmatrix}\right]^{\tp},\\
			& C(4) =
			\left[\begin{smallmatrix}
			0_{1\times 8}
			 &-{\frac{13}{2432}}&-{\frac{91}{29184}}&{\frac{39}{2432}}&{\frac{13}{3648}}&-\frac{1}{12}&-{\frac{1699}{43776}}\\
			0_{1\times 8}
			 &{\frac{3}{152}}&{\frac{7}{608}}&-{\frac{9}{152}}&-{\frac{1}{76}}&{\frac{679}{1216}}
			&{\frac{4225}{14592}}	
			 \end{smallmatrix}\right]^{\tp},\qquad\quad
			C(5) =
			\left[\begin{smallmatrix}	 0_{1\times 12} &-{\frac{13}{2432}}&-{\frac{91}{29184}}\\
			0_{1\times 12} &{\frac{3}{152}}&{\frac{7}{608}}
			 \end{smallmatrix}\right]^{\tp},\\
			& D(2) =
			\left[\begin{smallmatrix}
			 0&{\frac{19}{443952}}&-{\frac{19}{36996}}&-{\frac{19}{55494}}&-{\frac{247}{18498}}&-{\frac{38}{27747
			 }}&-{\frac{513}{12332}}&-{\frac{285}{6166}}&{\frac{6259489}{44987136}}
			 &{\frac{8864935}{539845632}}&-{\frac{1886753}{14995712}}&{\frac{1868255}{67480704}}&{\frac{386}{9249}}&-{\frac{6673003}{269922816}}\\
			 0&-{\frac{679}{4735488}}&{\frac{679}{394624}}&{\frac{679}{591936}}&{\frac{8827}{197312}}&{\frac{679}{147984}}&{\frac{
					 54999}{394624}}&{\frac{30555}{197312}}&{\frac{18333}{197312}}&{\frac{79443}{394624}}&-{\frac{164997}{394624}}&{\frac{18333}{197312}}&{\frac{54999}{394624}}&-{\frac{128331}{1578496}}
			 \end{smallmatrix}\right]^{\tp},\\
			& D(3) =
			\left[\begin{smallmatrix}
			 0&-{\frac{1}{4735488}}&{\frac{1}{394624}}&{\frac{1}{591936}}&{\frac{13}{197312}}&{\frac{1}{147984}}&{\frac{81}{394624}}&{\frac{45}{197312}}&{\frac{24745}{591936}}&{\frac{87377}{3551616}}&-{\frac{49571}{394624}}&-{\frac{49085}{1775808}}&{\frac{7516339}{44987136}}&{\frac{35}{179948544}}\\
			 0&0&0&0&0&0&0&0&-{\frac{679}{4864}}&-{\frac{4753}{58368}}&{\frac{2037}{4864}}&{\frac{679}{7296}}&0&{\frac{7469}{29184}}
			 \end{smallmatrix}\right]^{\tp},\\
			& D(4) =
			\left[\begin{smallmatrix}
			0_{1\times 8}&-{\frac{1}{4864}}&-{\frac{7}{58368}}&{\frac{3}{4864}}&{\frac{1}{7296}}&\frac{1}{24}&{\frac{2161}{87552}}\\
			0_{1\times 8}&0 &0 &0 &0
			 &-{\frac{679}{4864}}&-{\frac{4753}{58368}}
			 \end{smallmatrix}\right]^{\tp},\qquad\quad
			D(5) =
			\left[\begin{smallmatrix}
			0_{1\times 12}&-{\frac{1}{4864}}&-{\frac{7}{58368}}\\
			0_{1\times 12} &0&0
			\end{smallmatrix}\right]^{\tp}.
			\end{align*}
	}}

	Since  \eqref{tAL} is satisfied with $\rho(\tilde{A}_{L})=1/2$, we conclude from \cref{thm:direct,thm:bw:0N} with $N=1$ that  $\cB_J=\Phi_J \cup \{\Psi_j \setsp j\ge J\}$ is a Riesz basis of $L_{2}([0,1])$ for all $J\ge J_{0}:=2$, where $\Phi_j$ and $\Psi_j$ in \eqref{Phij} and \eqref{Psij} with $n_{\phi}=n_{\mathring{\phi}}=1$ and $n_{\psi}=n_{\mathring{\psi}}=2$ are given by $\phi^{R}:=\phi^{L}(1-\cdot)$, $\psi^{R}:=\psi^{L}(1-\cdot)$ and
	\begin{align*}
	\Phi_{j} & := \{\phi^{L}_{j;0},\phi_{j;1},\phi_{j;2},\phi_{j;3}\} \cup \{\phi_{j;k}:4\le k \le 2^{j}-4\} \cup \{\phi^{R}_{j;2^{j}-1},\phi_{j;2^{j}-1},\phi_{j;2^{j}-2},\phi_{j;2^{j}-3}\},\\
	\Psi_{j} & := \{\psi^{L}_{j;0},\psi_{j;2},\psi_{j;3}\} \cup \{\psi_{j;k}:4\le k \le 2^{j}-4\} \cup \{\psi^{R}_{j;2^{j}-1},\psi_{j;2^{j}-2},\psi_{j;2^{j}-3}\},
	\end{align*}
	with $\#\phi^{L} = \# \phi^{R}=2$, $\#\psi^{L} = \# \psi^{R}=3$, $\#\Phi_j=2^{j+1}+2$ and $\#\Psi_j=2^{j+1}$.
	Note that $\vmo(\psi^{L})=\vmo(\psi^{R})=\vmo(\psi)=4=\sr(\tilde{a})$.
	The dual Riesz basis $\tilde{\cB}_j$ of $\cB_{j}$ with $j\ge \tilde{J}_{0} := 3$ is given by \cref{thm:direct} through \eqref{tphi:implicit} and \eqref{tpsi:implicit}. We rewrite $\tilde{\phi}^{L}$ in \eqref{tphi:implicit} as $\{\tilde{\phi}^{L},\tilde{\phi}(\cdot-4),\tilde{\phi}(\cdot-5),\tilde{\phi}(\cdot-6)\}$ with true boundary elements $\tilde{\phi}^{L}$ and $\# \tilde{\phi}^{L}=8$, and rewrite $\tilde{\psi}^{L}$ in \eqref{tpsi:implicit} as $\{\tilde{\psi}^{L},\tilde{\psi}(\cdot-4),\tilde{\psi}(\cdot-5)\}$ with true boundary elements $\tilde{\psi}^{L}$ and $\# \tilde{\psi}^{L}=7$. Hence, $\tilde{\cB}_J=\tilde{\Phi}_J \cup \{\tilde{\Psi}_j \setsp j\ge J\}$ for $J\ge 3$ is given by
	\begin{align*}
	\tilde{\Phi}_{j} & := \{\tilde{\phi}^{L}_{j;0}\} \cup \{\tilde{\phi}_{j;k}:4\le k \le 2^{j}-4\} \cup \{\tilde{\phi}^{R}_{j;2^{j}-1}\}, \quad \mbox{with} \quad \tilde{\phi}^{R}:=\tilde{\phi}^{L}(1-\cdot),\\
	\tilde{\Psi}_{j} & := \{\tilde{\psi}^{L}_{j;0}\} \cup \{\tilde{\psi}_{j;k}:4\le k \le 2^{j}-4\} \cup \{\tilde{\psi}^{R}_{j;2^{j}-1}\}, \quad \mbox{with} \quad \tilde{\psi}^{R}:=\tilde{\psi}^{L}(1-\cdot).
	\end{align*}
	Note that $\vmo(\tilde{\psi}^{L})=\vmo(\tilde{\psi}^{R})=\vmo(\tilde{\psi})=4=\sr(a)$ and $\PL_{3} \chi_{[0,1]} \subset\mbox{span}(\Phi_{j})$ for all $j \ge 2$. According to Theorem \ref{thm:bw:0N} with $N=1$, $(\tilde{\cB}_J,\cB_J)$ forms a biorthogonal Riesz basis of $L_{2}([0,1])$ for every $J \ge 3$.
	
	By item (ii) of \cref{prop:phicut} with $n_\phi=1$ and $\textsf{p}(x)=(x,x^{2},x^{3})^\tp$, the left boundary refinable vector function is $\phi^{L,bc}:=\phi^L_2$ (the second entry of $\phi^L$) and satisfies $\phi^{L,bc} = \tfrac{1}{2}\phi^{L,bc}(2\cdot) + [\tfrac{1}{8},-\tfrac{1}{8}] \phi(2\cdot-1)$ by \eqref{phiLhmcubic}.
	Taking $n_{\psi}=2$ and $m_{\phi}=7$ in \cref{thm:direct}, we have $\psi^{L,bc}:=(\psi^{L,bc}_1, \psi^L_2,\psi^L_3)^\tp$ with
	\[
	\psi^{L,bc}_1:=\psi^L_1
	-\phi^L_1(2\cdot)
	-[\tfrac{1121}{2376}, \tfrac{533}{36}]\phi(2\cdot-1)+
	[\tfrac{989}{594}, \tfrac{17}{18}]\phi(2\cdot-2)
	+[-\tfrac{61}{88}, \tfrac{195}{44}]\phi(2\cdot-3)
	\]
	with $\# \psi^{L,bc}=3$ satisfying both items (i) and (ii) of \cref{thm:direct}, where $A^{bc}_0$, $B^{bc}_0$, $C^{bc}$, and $D^{bc}$ can be easily derived from $A_{0}$, $B_{0}$, $C$, and $D$. More precisely, $A_0^{bc}$ is obtained from $U^{-1} A_0$ by taking out its first row and first column, and $B_0^{bc}, C^{bc}, D^{bc}$ are obtained from $U^{-1}B_0, U^{-1}C, U^{-1}D$, respectively by removing their first rows, where the invertible matrix $U$ is given by
	\[
	U:=I_{14}+B_0
	{\begin{bmatrix}
		 -1&0&-{\frac{1121}{2376}}&-{\frac{533}{36}}
		&{\frac{989}{594}}&{\frac{17}{18}}
		&-{\frac{61}{88}}&{\frac{195}{44}}& {0_{1 \times 6}}\\
		\multicolumn{9}{c}{0_{2\times 14}}
		\end{bmatrix}}.
	\]
	Since \eqref{tAL} is satisfied with $\rho(\tilde{A}_{L}^{bc})=1/2$, we conclude from \cref{thm:direct,thm:bw:0N} that $\cB_J^{bc}=\Phi_J^{bc} \cup \{\Psi_j^{bc} \setsp j\ge J\}$ is a Riesz basis of $L_{2}([0,1])$ for every  $J\ge J_{0}:=2$ such that $h(0)=h(1)=0$ for all $h\in \cB_J^{bc}$, where $\Phi_j^{bc}$ and $\Psi_j^{bc}$ in \eqref{Phij} and \eqref{Psij} with $n_{\phi}=n_{\mathring{\phi}}=1$ and $n_{\psi}=n_{\mathring{\psi}}=2$ are given by
	\begin{align*}
	\Phi_{j}^{bc} & := \{\phi^{L,bc}_{j;0},\phi_{j;1},\phi_{j;2},\phi_{j;3}\} \cup \{\phi_{j;k}:4\le k \le 2^{j}-4\} \cup \{\phi^{R,bc}_{j;2^{j}-1},\phi_{j;2^{j}-1},\phi_{j;2^{j}-2},\phi_{j;2^{j}-3}\},\\
	\Psi_{j}^{bc} & := \{\psi^{L,bc}_{j;0},\psi_{j;2},\psi_{j;3}\} \cup \{\psi_{j;k}:4\le k \le 2^{j}-4\} \cup \{\psi^{R,bc}_{j;2^{j}-1},\psi_{j;2^{j}-2},\psi_{j;2^{j}-3}\},
	\end{align*}
	where $\phi^{R,bc}:=\phi^{L,bc}(1-\cdot)$ and $\psi^{R,bc}:=\psi^{L,bc}(1-\cdot)$ with  $\#\phi^{L,bc}=1$, $\#\psi^{L,bc}=3$,
	and $\#\Phi^{bc}_j=\#\Psi^{bc}_j=2^{j+1}$.
	For the case $j=2$,
	$\Phi^{bc}_2=\{\phi^{L,bc}_{2;0},
	\phi_{2;1}, \phi_{2;2}, \phi_{2;3},\phi^{R,bc}_{2;3}\}$ and
	$\Psi^{bc}_2=\{\psi^{L,bc}_{2;0}, \psi_{2;2},\psi^{R,bc}_{2;3}\}$
	after removing repeated elements.
	Note $\vmo(\psi^{L,bc})=\vmo(\psi^{R,bc})=\vmo(\psi)=4$ and $\Phi^{bc}_j=\Phi_j\bs\{(\phi^L_1)_{j;0},(\phi^R_1)_{j;2^j-1}\}$ as in \cref{prop:mod}. The dual Riesz basis $\tilde{\cB}_j^{bc}$ of $\cB_{j}^{bc}$ with $j\ge \tilde{J}_{0} := 3$ is given by \cref{thm:direct} through \eqref{tphi:implicit} and \eqref{tpsi:implicit}. We rewrite $\tilde{\phi}^{L,bc}$ in \eqref{tphi:implicit} as $\{\tilde{\phi}^{L,bc},\tilde{\phi}(\cdot-4),\tilde{\phi}(\cdot-5),\tilde{\phi}(\cdot-6)\}$ with true boundary elements $\tilde{\phi}^{L,bc}$ and $\# \tilde{\phi}^{L,bc}=7$, and $\tilde{\psi}^{L,bc}$ in \eqref{tpsi:implicit} as $\{\tilde{\psi}^{L,bc},\tilde{\psi}(\cdot-4),\tilde{\psi}(\cdot-5)\}$ with true boundary elements $\tilde{\psi}^{L,bc}$ and $\# \tilde{\psi}^{L,bc}=7$. Hence, $\tilde{\cB}_J^{bc}=\tilde{\Phi}_J^{bc} \cup \{\tilde{\Psi}_j^{bc} \setsp j\ge J\}$ is given by
	\begin{align*}
	\tilde{\Phi}_{j}^{bc} & := \{\tilde{\phi}^{L,bc}_{j;0}\} \cup \{\tilde{\phi}_{j;k}:4\le k \le 2^{j}-4\} \cup \{\tilde{\phi}^{R,bc}_{j;2^{j}-1}\}, \quad \mbox{with} \quad \tilde{\phi}^{R,bc}:=\tilde{\phi}^{L,bc}(1-\cdot),\\
	\tilde{\Psi}_{j}^{bc} & := \{\tilde{\psi}^{L,bc}_{j;0}\} \cup \{\tilde{\psi}_{j;k}:4\le k \le 2^{j}-4\} \cup \{\tilde{\psi}^{R,bc}_{j;2^{j}-1}\}, \quad \mbox{with} \quad \tilde{\psi}^{R,bc}:=\tilde{\psi}^{L,bc}(1-\cdot).
	\end{align*}
	Note that $\vmo(\tilde{\psi}^{L,bc})=\vmo(\tilde{\psi}^{R,bc})=0$ and $\{x\chi_{[0,1]},x^2\chi_{[0,1]},x^3\chi_{[0,1]}\} \subset\mbox{span}(\Phi_{j}^{bc})$ for all $j \ge 2$. By Theorem \ref{thm:bw:0N} with $N=1$, $(\tilde{\cB}_J^{bc},\cB_J^{bc})$ forms a biorthogonal Riesz basis of $L_{2}([0,1])$ for every $J \ge 3$.
	
	By item (ii) of \cref{prop:phicut} with $n_\phi=1$ and $\textsf{p}(x)=(x^{2},x^{3})^\tp$, we have $\phi^{L,bc1}:=\emptyset$. Taking $n_{\psi}=2$ and $m_{\phi}=7$ in \cref{thm:direct}, we have
	 $\psi^{L,bc1}:=(\psi^{L,bc}_1,\psi^{L,bc1}_2,\psi^L_3)$ with
	\[
	\psi^{L,bc1}_2:=
	 \psi^L_2-\phi^L_2(2\cdot)-[\tfrac{1}{36},\tfrac{7}{4}]\phi(2\cdot-1)
	+[\tfrac{10}{9}, \tfrac{1048}{183}]\phi(2\cdot-2)
	 -[\tfrac{52}{61},-\tfrac{660}{61}]\phi(2\cdot-3)
	-[\tfrac{9}{61},0]\phi(2\cdot-4)
	\]
	%
	with $\# \psi^{L,bc1}=3$ satisfying both items (i) and (ii) of \cref{thm:direct}, where
	$A_0^{bc1}$ is obtained from $V^{-1} A_0$ by taking out its first two rows and the first two columns, and $B_0^{bc1}, C^{bc1}, D^{bc1}$ are obtained from $V^{-1}B_0$, $V^{-1}C$, $V^{-1}D$, respectively by removing their first two rows, where the invertible matrix $V$ is given by
	\[
	V:=I_{14}+B_0
	{\begin{bmatrix}
		-1&0&-{\frac{1121}{2376}}
		 &-{\frac{533}{36}}&{\frac{989}{594}}
		&{\frac{17}{18}}&-{\frac{61}{88}}
		&{\frac{195}{44}}&0&{0_{1 \times 5}}\\[0.2em]
		0 &-1&-{\frac{1}{36}}&-{\frac{7}{4}}
		&\frac{10}{9}&{\frac{1048}{183}}
		&-{\frac{52}{61}}&{\frac{660}{61}}
		&-{\frac{9}{61}} &{0_{1 \times 5}}\\
		\multicolumn{10}{c}{0_{1\times 14}}
		\end{bmatrix}}.
	\]
	Since \eqref{tAL} is satisfied with $\rho(\tilde{A}_{L}^{bc1})=1/2$, we conclude from \cref{thm:direct,thm:bw:0N} with $N=1$ that $\cB_J^{bc1}=\Phi_J^{bc1} \cup \{\Psi_j^{bc1} \setsp j\ge J\}$ is a Riesz basis of $L_{2}([0,1])$ for every  $J\ge J_{0}:=2$ such that $h(0)=h'(0)=h(1)=h'(1)=0$ for all $h\in \cB^{bc1}_J$, where $\Phi_j^{bc1}$ and $\Psi_j^{bc1}$ in \eqref{Phij} and \eqref{Psij} with $n_{\phi}=n_{\mathring{\phi}}=1$ and $n_{\psi}=n_{\mathring{\psi}}=2$ are given by: for $j=2,3$,
	\begin{align*}
	\Phi_{j}^{bc1} := \{\phi_{j;k}:1\le k \le 2^{j}-1\}, \qquad
	\Psi_{j}^{bc1} := \{\psi^{L,bc1}_{j;0}\} \cup \{\psi_{j;k}:2\le k \le 2^{j}-2\} \cup \{\psi^{R,bc1}_{j;2^{j}-1}\},
	\end{align*}
	where $\psi^{R,bc1}:=\psi^{L,bc1}(1-\cdot)$ with $\#\psi^{L,bc1}=\#\psi^{R,bc1}=3$,
	and for $j\ge 4$,
	\begin{align*}
	\Phi_{j}^{bc1} & := \{\phi_{j;k}:1\le k \le 4\} \cup \{\phi_{j;k}:5\le k \le 2^{j}-5\} \cup \{\phi_{j;2^{j}-k}:1\le k \le 4\},\\
	\Psi_{j}^{bc1} & := \{\psi^{L,bc1}_{j;0},\psi_{j;2},\psi_{j;3},\psi_{j;4}\} \cup \{\psi_{j;k}:5\le k \le 2^{j}-5\} \cup \{\psi^{R,bc1}_{j;2^{j}-1},\psi_{j;2^{j}-2},\psi_{j;2^{j}-3},\psi_{j;2^{j}-4}\},
	\end{align*}
	with $\#\Phi^{bc1}_j=2^{j+1}-2$ and $\#\Psi^{bc1}_j=2^{j+1}$.
	Note that $\vmo(\psi^{L,bc1})=\vmo(\psi^{R,bc1})=\vmo(\psi)=4$ and $\Phi^{bc1}_j=\Phi^{bc}_j\bs\{\phi^{L,bc}_{j;0},\phi^{R,bc}_{j,2^j-1}\}
	=\Phi_j\bs\{\phi^L_{j;0},
	\phi^R_{j;2^j-1}\}$ as in \cref{prop:mod}.
	We rewrite $\tilde{\phi}^{L,bc1}$ in \eqref{tphi:implicit} as $\{\tilde{\phi}^{L,bc1},\tilde{\phi}(\cdot-5),\tilde{\phi}(\cdot-6)\}$ with true boundary elements $\tilde{\phi}^{L,bc1}$ and $\# \tilde{\phi}^{L,bc1}=8$. Note that $\# \tilde{\psi}^{L,bc1}=9$.
	The dual Riesz basis $\tilde{\cB}_J^{bc1}:=\tilde{\Phi}_J^{bc1} \cup \{\tilde{\Psi}_j^{bc1} \setsp j\ge J\}$ of $\cB_{j}^{bc1}$ with $j\ge \tilde{J}_{0} := 4$ is given by %
	\begin{align*}
	\tilde{\Phi}_{j}^{bc1} & := \{\tilde{\phi}^{L,bc1}_{j;0}\} \cup \{\tilde{\phi}_{j;k}:5\le k \le 2^{j}-5\} \cup \{\tilde{\phi}^{R,bc1}_{j;2^{j}-1}\}, \quad \mbox{with} \quad \tilde{\phi}^{R,bc1}:=\tilde{\phi}^{L,bc1}(1-\cdot),\\
	\tilde{\Psi}_{j}^{bc1} & := \{\tilde{\psi}^{L,bc1}_{j;0}\} \cup \{\tilde{\psi}_{j;k}:5\le k \le 2^{j}-5\} \cup \{\tilde{\psi}^{R,bc1}_{j;2^{j}-1}\}, \quad \mbox{with} \quad \tilde{\psi}^{R,bc1}:=\tilde{\psi}^{L,bc1}(1-\cdot).
	\end{align*}
	Note that $\vmo(\tilde{\psi}^{L,bc1})=\vmo(\tilde{\psi}^{R,bc1})=0$ and $\{x^2\chi_{[0,1]},x^3\chi_{[0,1]}\} \subset\mbox{span}(\Phi_{j}^{bc1})$ for all $j \ge 2$. According to Theorem \ref{thm:bw:0N} with $N=1$, $(\tilde{\cB}_J^{bc1},\cB_J^{bc1})$ forms a biorthogonal Riesz basis of $L_{2}([0,1])$ for every $J \ge 4$.
	See \cref{fig:hmtvm4} for the graphs of $\phi, \psi$ and their associated boundary elements.
\end{example}

\begin{figure}[htbp]
	\centering
	\begin{subfigure}[b]{0.24\textwidth} \includegraphics[width=\textwidth,height=0.6\textwidth]{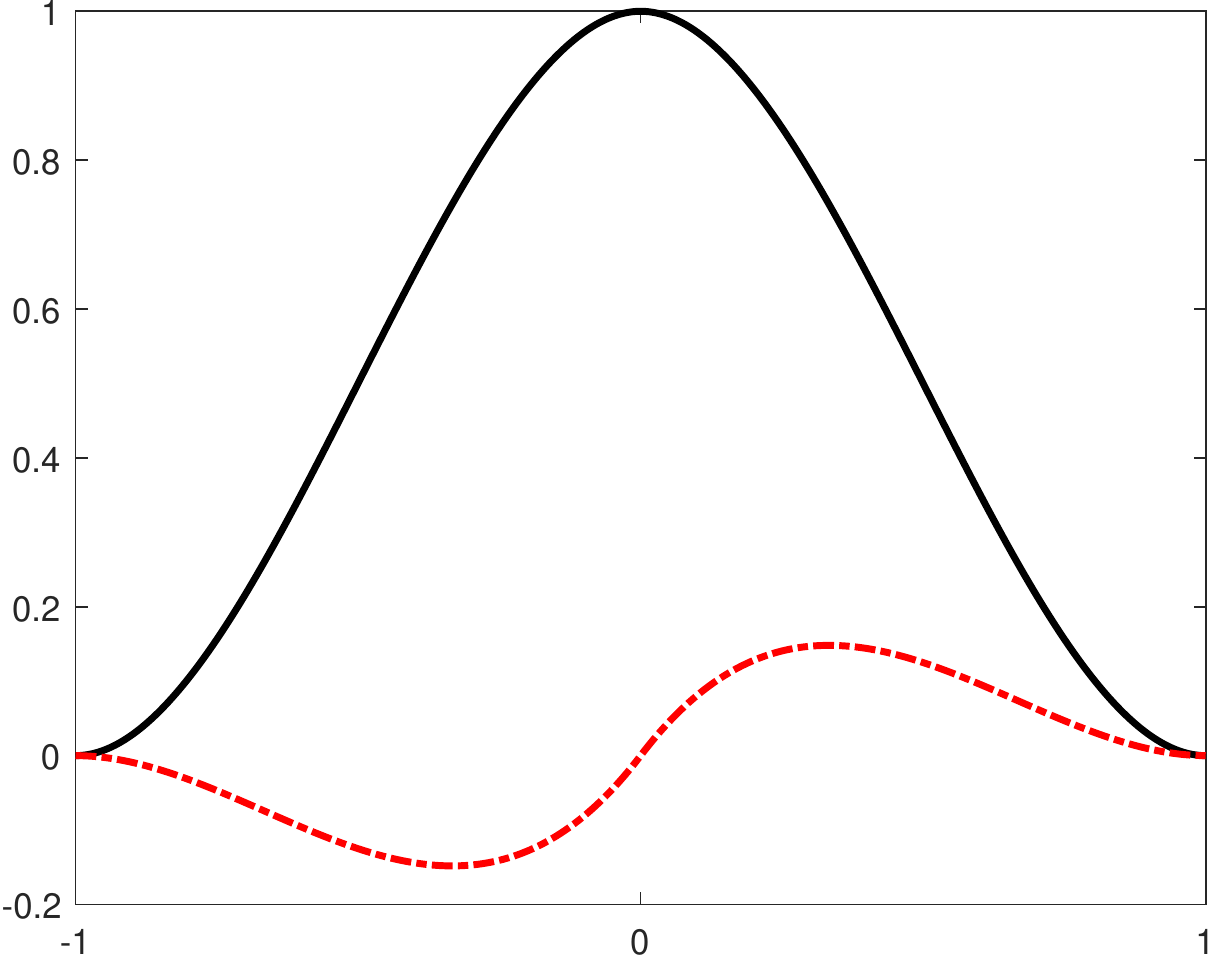}
		\caption{$\phi =(\phi^{1},\phi^{2})^{\tp}$}
	\end{subfigure}	
	\begin{subfigure}[b]{0.24\textwidth} \includegraphics[width=\textwidth,height=0.6\textwidth]{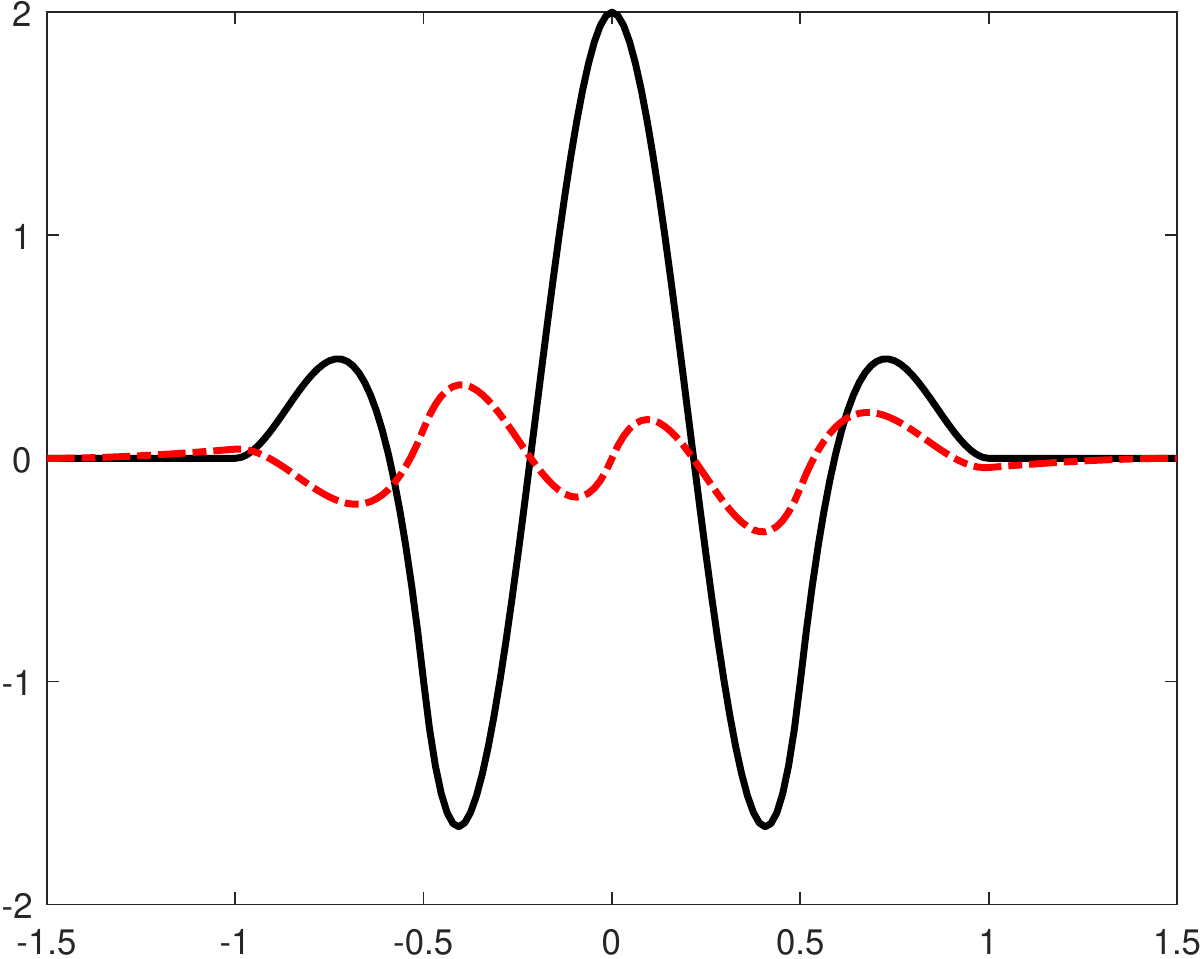}
		\caption{$\psi =(\psi^{1},\psi^{2})^{\tp}$}
	\end{subfigure}
	\begin{subfigure}[b]{0.24\textwidth} \includegraphics[width=\textwidth,height=0.6\textwidth]{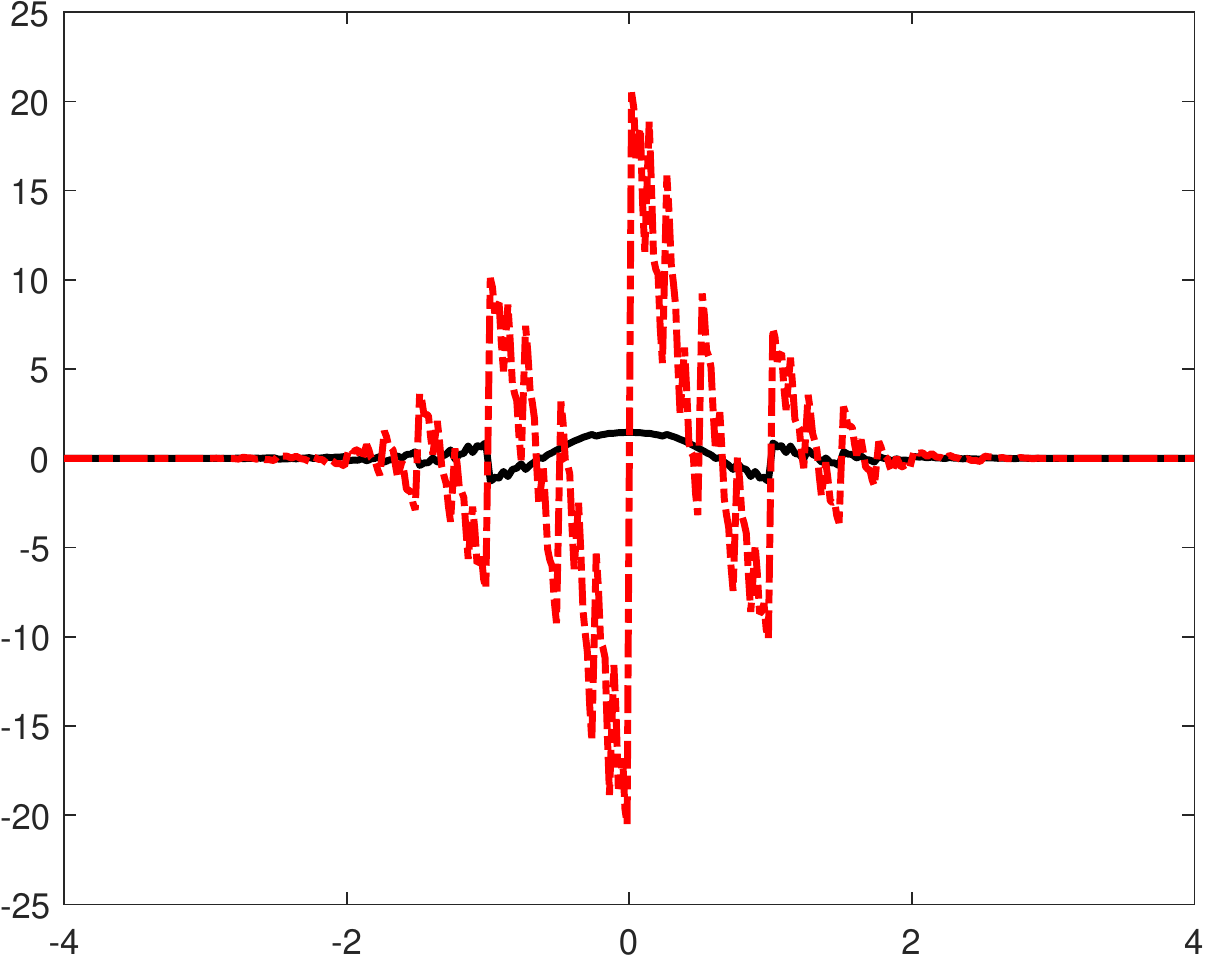}
		\caption{$\tilde{\phi} =(\tilde{\phi}^{1},\tilde{\phi}^{2})^{\tp}$}
	\end{subfigure}	
	\begin{subfigure}[b]{0.24\textwidth} \includegraphics[width=\textwidth,height=0.6\textwidth]{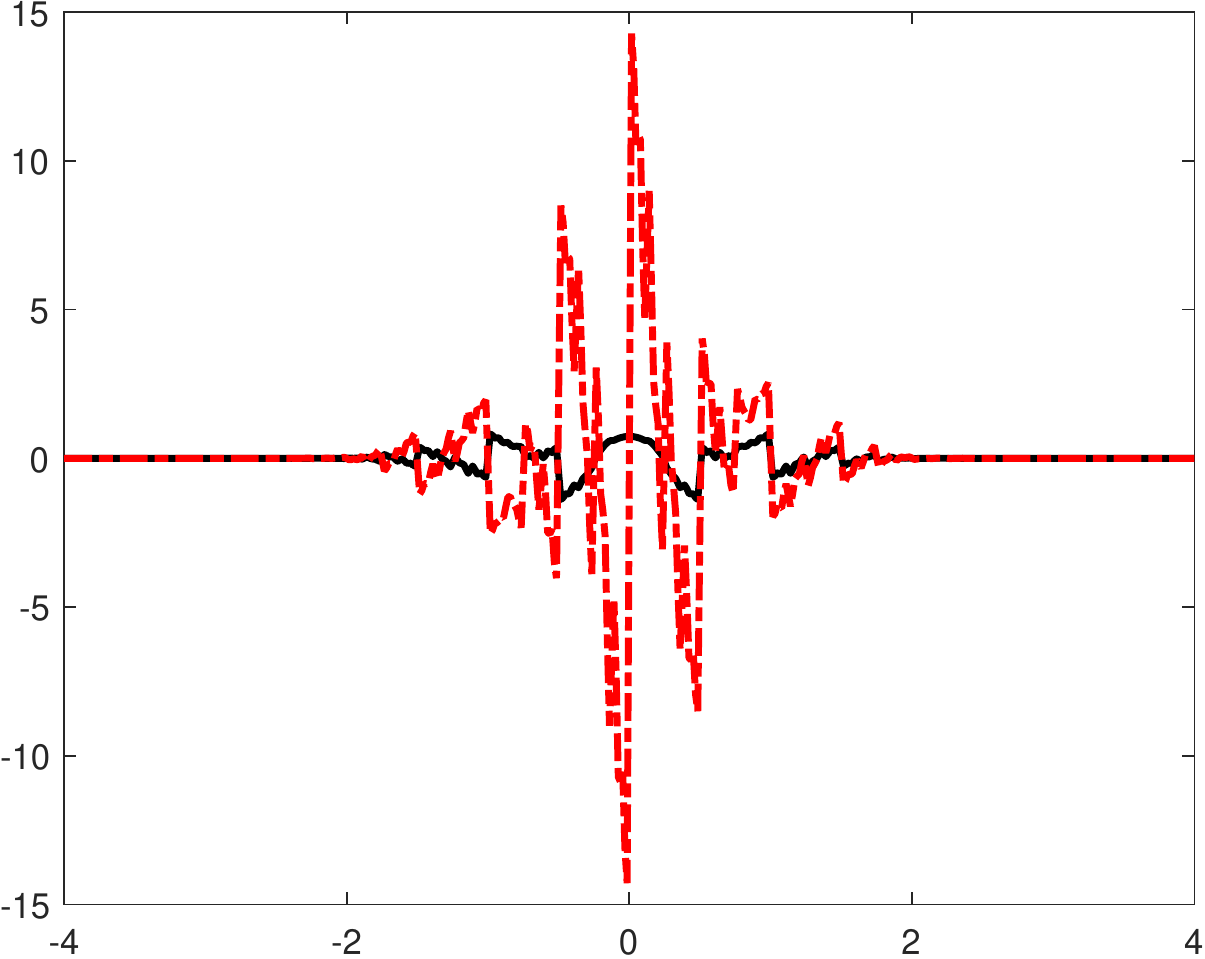}
		\caption{$\tilde{\psi} =(\tilde{\psi}^{1},\tilde{\psi}^{2})^{\tp}$}
	\end{subfigure}
	\begin{subfigure}[b]{0.24\textwidth} \includegraphics[width=\textwidth,height=0.6\textwidth]{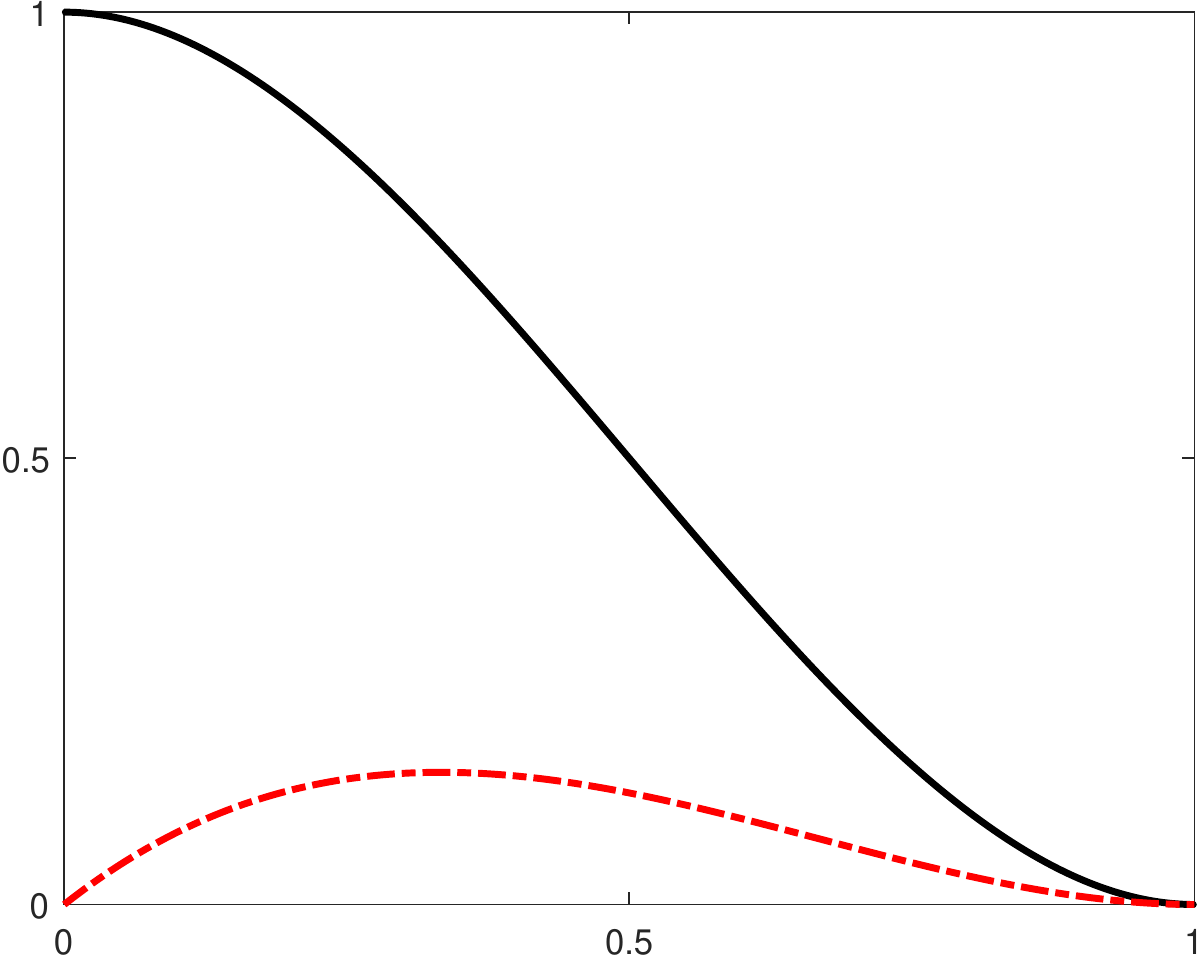}
		\caption{$\phi^{L}$}
	\end{subfigure}	
	\begin{subfigure}[b]{0.24\textwidth} \includegraphics[width=\textwidth,height=0.6\textwidth]{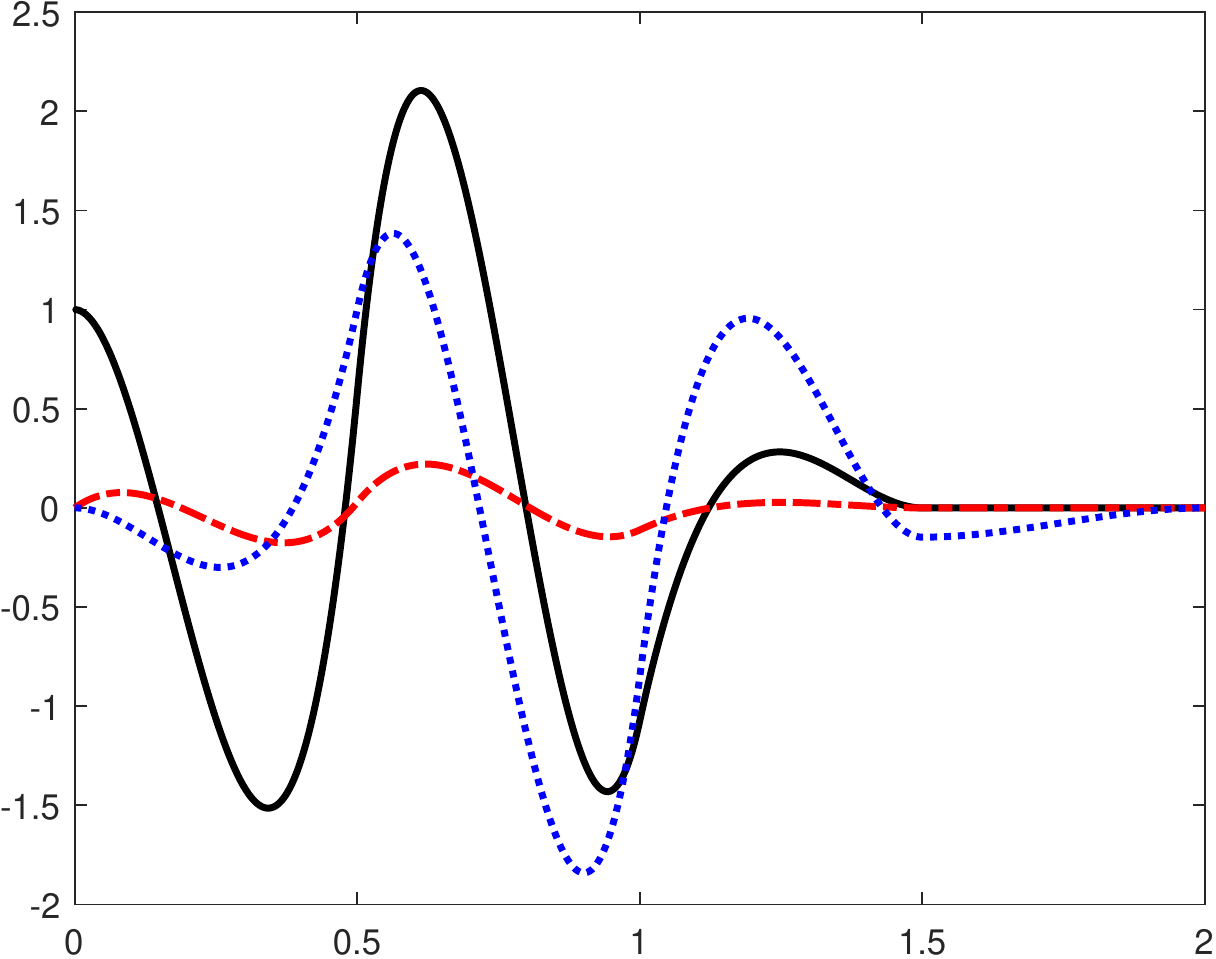}
		\caption{$\psi^{L}$}
	\end{subfigure}
	\begin{subfigure}[b]{0.24\textwidth} \includegraphics[width=\textwidth,height=0.6\textwidth]{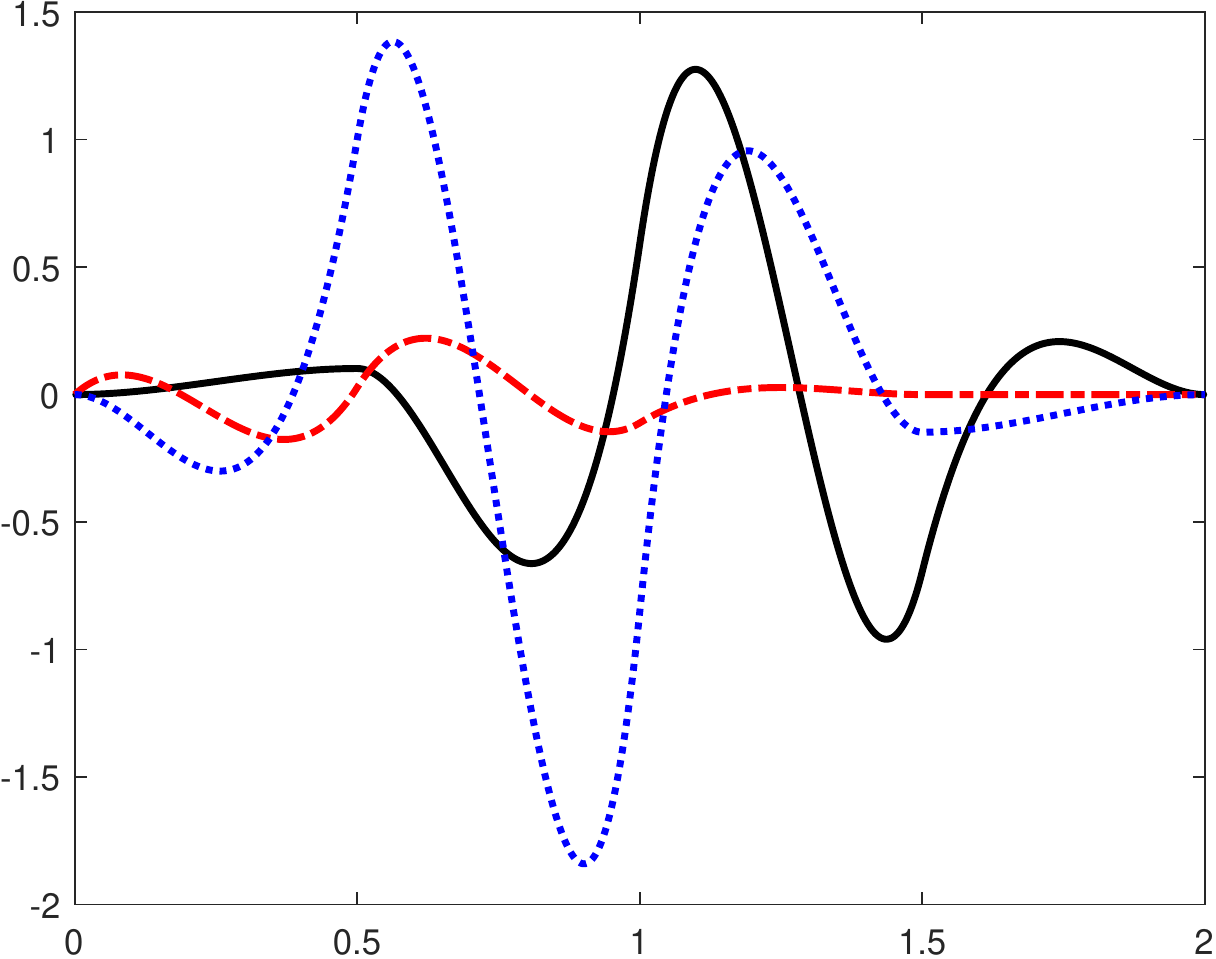}
		\caption{$\psi^{L,bc}$}
	\end{subfigure}
	\begin{subfigure}[b]{0.24\textwidth} \includegraphics[width=\textwidth,height=0.6\textwidth]{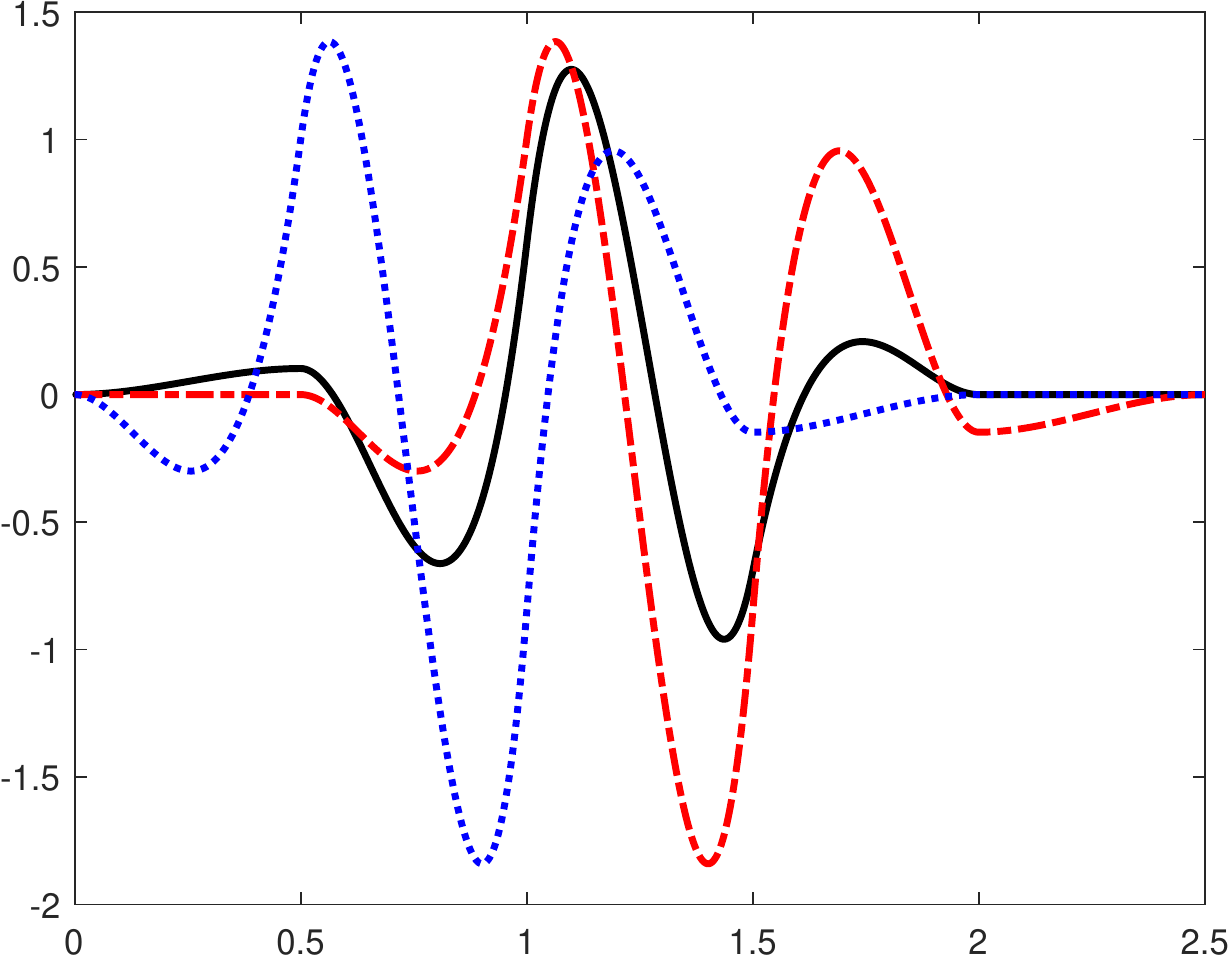}
		\caption{$\psi^{L,bc1}$}
	\end{subfigure}
	\caption{The generators of the Riesz bases $\cB_J$, $\cB_J^{bc}$, $\cB_J^{bc1}$ for $L_{2}([0,1])$ with $J \ge 2$ in \cref{ex:hmtvm4} such that $\eta(0)=\eta(1)$ for all $\eta\in \cB^{bc}_J$ and $h(0)=h'(0)=h(1)=h'(1)=0$ for all $h\in \cB^{bc1}_J$.
		The black, red, and blue lines correspond to the first, second, and third components of a vector function. Note that $\phi^{L,bc}$ in $\cB^{bc}_J$ is the second entry of $\phi^L$, $\phi^{L,bc1}=\emptyset$ in $\cB^{bc1}_J$, and $\vmo(\psi^L)=\vmo(\psi^{L,bc})=\vmo(\psi^{L,bc1})=\vmo(\psi)=4$.}
	\label{fig:hmtvm4}
\end{figure}

\begin{example} \label{ex:legall}
	\normalfont Consider the scalar biorthogonal wavelet $(\{\tilde{\phi};\tilde{\psi}\},\{\phi;\psi\})$ in \cite{cdf92}
	with $\wh{\phi}(0)=\wh{\tilde{\phi}}(0)=1$ and a biorthogonal wavelet filter bank $(\{\tilde{a};\tilde{b}\},\{a;b\})$  given by
	\begin{align*}
	 a=&\left\{\tfrac{1}{4},\tfrac{1}{2},\tfrac{1}{4}\right\}_{[-1,1]}, \quad b=\left\{-\tfrac{1}{8},-\tfrac{1}{4},\tfrac{3}{4},-\tfrac{1}{4},-\tfrac{1}{8}\right\}_{[-1,3]},\\
	\tilde{a}=&\left\{-\tfrac{1}{8}, \tfrac{1}{4}, \tfrac{3}{4}, \tfrac{1}{4}, -\tfrac{1}{8} \right\}_{[-2,2]}, \quad \tilde{b}=\left\{-\tfrac{1}{4}, \tfrac{1}{2}, -\tfrac{1}{4}\right\}_{[0,2]}.
	\end{align*}
	Then, $\sm(a)=1.5$, $\sm(\tilde{a}) \approx 0.440765$, and $\sr(a)=\sr(\tilde{a})=2$. Note that $\phi$ is a piecewise linear function. By item (i) of \cref{prop:phicut} with $n_\phi=1$, we have the left boundary refinable function $\phi^{L}:=\phi\chi_{[0,\infty)}=\phi^{L}(2\cdot) + \tfrac{1}{2} \phi(2\cdot-1)$.
	We use the direct approach in \cref{sec:direct}. Taking $n_{\psi}=1$ and $m_{\phi}=4$ in \cref{thm:direct}, we have $\psi^{L} = \phi^{L}(2\cdot) - \tfrac{5}{6} \phi(2\cdot-1) + \tfrac{1}{3} \phi(2\cdot -2)$, which satisfies items (i) and (ii) of \cref{thm:direct} with $\fs(C)=[1,2]$, $\fs(D)=[1,1]$, and
	{\begin{align*}
		& A_{0}=[\tfrac{2}{3},\tfrac{2}{3},-\tfrac{1}{3},0]^{\tp}, \quad
		 B_{0}=[\tfrac{1}{3},-\tfrac{2}{3},\tfrac{1}{3},0]^{\tp}, \quad
		 C(1)=[-\tfrac{7}{72},\tfrac{7}{36},\tfrac{7}{9},\tfrac{1}{4}]^{\tp},\\
		& C(2)=[\tfrac{1}{72},-\tfrac{1}{36},-\tfrac{1}{9},\tfrac{1}{4}]^{\tp},\quad
		 D(1)=[\tfrac{1}{36},-\tfrac{1}{18},-\tfrac{2}{9},\tfrac{1}{2}]^{\tp}.
	\end{align*}}
	Since \eqref{tAL} is satisfied with $\rho(\tilde{A}_{L})=1/2$, we conclude from \cref{thm:direct,thm:bw:0N} with $N=1$
	that  $\cB_J=\Phi_J \cup \{\Psi_j \setsp j\ge J\}$ is a Riesz basis of $L_{2}([0,1])$ for all $J\ge J_{0}:=1$, where $\Phi_j$ and $\Psi_j$ in \eqref{Phij} and \eqref{Psij} with $n_{\phi}=n_{\mathring{\phi}}=n_{\psi}=1$ and $n_{\mathring{\psi}}=2$ are given by
	\begin{align*}
	\Phi_{j} & := \{\phi^{L}_{j;0},\phi_{j;1},\phi_{j;2}\} \cup \{\phi_{j;k}:3\le k \le 2^{j}-3\} \cup \{\phi^{R}_{j;2^{j}-1},\phi_{j;2^{j}-1},\phi_{j;2^{j}-2}\},\\
	\Psi_{j} & := \{\psi^{L}_{j;0},\psi_{j;1}\} \cup \{\psi_{j;k}:2\le k \le 2^{j}-3\} \cup \{\psi^{R}_{j;2^{j}-1},\psi_{j;2^{j}-2}\},
	\end{align*}
	where $\phi^{R}:=\phi^{L}(1-\cdot)$ and $\psi^{R}:=\psi^{L}(1-\cdot)$ with $\#\phi^{L} = \# \phi^{R}= \#\psi^{L} = \# \psi^{R}=1$, $\#\Phi_j=2^j+1$ and $\#\Psi_j=2^j$. For the cases $j=1$ and $j=2$,
	$\Phi_1=\{\phi^L_{1;0}, \phi_{1;1},\phi^R_{1;1}\}$, $\Psi_{1}=\{\psi^{L}_{1;0},\psi^{R}_{1;1}\}$,
	and $\Phi_2=\{\phi^L_{2,0}, \phi_{2;1},\phi_{2;2},\phi_{2;3},\phi^R_{2;3}\}$
	after
	removing repeated elements.
	Note that $\vmo(\psi^{L})=\vmo(\psi^{R})=\vmo(\psi)=2=\sr(\tilde{a})$. The dual Riesz basis $\tilde{\cB}_j$ of $\cB_{j}$ with $j\ge \tilde{J}_{0} := 3$ is given by \cref{thm:direct} through \eqref{tphi:implicit} and \eqref{tpsi:implicit}. We rewrite $\tilde{\phi}^{L}$ in \eqref{tphi:implicit} as $\{\tilde{\phi}^{L},\tilde{\phi}(\cdot-3)\}$ with true boundary elements $\tilde{\phi}^{L}$ and $\# \tilde{\phi}^{L}=3$, and rewrite $\tilde{\psi}^{L}$ in \eqref{tpsi:implicit} as $\{\tilde{\psi}^{L},\tilde{\psi}(\cdot-2),\tilde{\psi}(\cdot-3)\}$ with true boundary elements $\tilde{\psi}^{L}$ and $\# \tilde{\psi}^{L}=2$. Hence, $\tilde{\cB}_J=\tilde{\Phi}_J \cup \{\tilde{\Psi}_j \setsp j\ge J\}$ for $J\ge 3$ is given by
	\begin{align*}
	\tilde{\Phi}_{j} & := \{\tilde{\phi}^{L}_{j;0}\} \cup \{\tilde{\phi}_{j;k}:3\le k \le 2^{j}-3\} \cup \{\tilde{\phi}^{R}_{j;2^{j}-1}\}, \quad \mbox{with} \quad \tilde{\phi}^{R}:=\tilde{\phi}^{L}(1-\cdot),\\
	\tilde{\Psi}_{j} & := \{\tilde{\psi}^{L}_{j;0}\} \cup \{\tilde{\psi}_{j;k}:2\le k \le 2^{j}-3\} \cup \{\tilde{\psi}^{R}_{j;2^{j}-1}\}, \quad \mbox{with} \quad \tilde{\psi}^{R}:=\tilde{\psi}^{L}(1-\cdot).
	\end{align*}
	Note that $\vmo(\tilde{\psi}^{L})=\vmo(\tilde{\psi}^{R})=\vmo(\tilde{\psi})=2=\sr(a)$ and $\PL_{1} \chi_{[0,1]} \subset\mbox{span}(\Phi_{j})$ for all $j \ge 3$. According to Theorem \ref{thm:bw:0N} with $N=1$, $(\tilde{\cB}_J,\cB_J)$ forms a biorthogonal Riesz basis of $L_{2}([0,1])$ for every $J \ge 3$.
	
	By item (ii) of \cref{prop:phicut} with $n_\phi=1$ and $\textsf{p}(x)=x$, the left boundary refinable vector function is $\phi^{L,bc}:=\emptyset$. Taking $n_{\psi}=1$ and $m_{\phi}=4$ in \cref{thm:direct}, we have $\# \psi^{L,bc}=1$ and
	\[
	 \psi^{L,bc}:=\psi^L-\phi^{L}(2\cdot)+\tfrac{4}{3}\phi(2\cdot -1)-\tfrac{4}{3}\phi(2\cdot -2) + \tfrac{1}{2} \phi(2\cdot-3)
	\]
	satisfying both items (i) and (ii) of \cref{thm:direct}, where $A^{bc}_0$, $B^{bc}_0$, $C^{bc}$, and $D^{bc}$ can be easily derived from $A_{0}$, $B_{0}$, $C$, and $D$. More precisely, $A_0^{bc}$ is obtained from $U^{-1} A_0$ by taking out its first row and first column, and $B_0^{bc}, C^{bc}, D^{bc}$ are obtained from $U^{-1}B_0, U^{-1}C, U^{-1}D$, respectively by removing their first rows, where the invertible matrix $U$ is given by
	\[
	 U:=I_{4}+B_0[-1,\tfrac{4}{3},-\tfrac{4}{3},\tfrac{1}{2}].
	\]
	Since \eqref{tAL} is satisfied with $\rho(\tilde{A}_{L}^{bc})=1/2$, we conclude from \cref{thm:direct,thm:bw:0N} with $N=1$ that $\cB_J^{bc}=\Phi_J^{bc} \cup \{\Psi_j^{bc} \setsp j\ge J\}$ is a Riesz basis of $L_{2}([0,1])$ for every  $J\ge J_{0}:=2$ such that $h(0)=h(1)=0$ for all $h\in \cB_J^{bc}$, where $\Phi_j^{bc}$ and $\Psi_j^{bc}$  in \eqref{Phij} and \eqref{Psij} with $n_{\phi}=n_{\mathring{\phi}}=n_{\psi}=1$ and $n_{\mathring{\psi}}=2$ are given by
	\begin{align*}
	\Phi_{j}^{bc} & := \{\phi_{j;1},\phi_{j;2}\} \cup \{\phi_{j;k}:3\le k \le 2^{j}-3\} \cup \{\phi_{j;2^{j}-1}, \phi_{j;2^{j}-2}\},\\
	\Psi_{j}^{bc} & := \{\psi^{L,bc}_{j;0},\psi_{j;1}\} \cup \{\psi_{j;k}:2\le k \le 2^{j}-3\} \cup \{\psi^{R,bc}_{j;2^{j}-1},\psi_{j;2^j-2}\},
	\end{align*}
	where $\phi^{L,bc}=\phi^{R,bc}=\emptyset$ and $\psi^{R,bc}:=\psi^{L,bc}(1-\cdot)$
	with $\#\psi^{L,bc}=\#\psi^{R,bc}=1$, $\#\Phi^{bc}_j=2^j-1$, and $\#\Psi^{bc}_j=2^j$. For the case $j=2$, $\Phi^{bc}_2=\{\phi_{2;1},\phi_{2;2},\phi_{2;3}\}$ after removing repeated elements.
	Note that $\vmo(\psi^{L,bc})=\vmo(\psi^{R,bc})=\vmo(\psi)=2$ and
	 $\Phi^{bc}_j=\Phi_j\bs\{\phi^L_{j;0},\phi^R_{j;2^j-1}\}$ as in \cref{prop:mod}. We rewrite $\tilde{\phi}^{L,bc}$ in \eqref{tphi:implicit} as $\{\tilde{\phi}^{L,bc},\tilde{\phi}(\cdot-3)\}$ with true boundary elements $\tilde{\phi}^{L,bc}$ and $\# \tilde{\phi}^{L,bc}=2$, and $\tilde{\psi}^{L,bc}$ in \eqref{tpsi:implicit} as $\{\tilde{\psi}^{L,bc},\tilde{\psi}(\cdot-2),\tilde{\psi}(\cdot-3)\}$ with true boundary elements $\tilde{\psi}^{L,bc}$ and $\# \tilde{\psi}^{L,bc}=2$.
	The dual Riesz basis $\tilde{\cB}_J^{bc}:=\tilde{\Phi}_J^{bc} \cup \{\tilde{\Psi}_j^{bc} \setsp j\ge J\}$ of $\cB_{J}^{bc}$ with $J\ge \tilde{J}_{0} := 3$ is given by
	\begin{align*}
	\tilde{\Phi}_{j}^{bc} & := \{\tilde{\phi}^{L,bc}_{j;0}\} \cup \{\tilde{\phi}_{j;k}:3\le k \le 2^{j}-3\} \cup \{\tilde{\phi}^{R,bc}_{j;2^{j}-1}\}, \quad \mbox{with} \quad \tilde{\phi}^{R,bc}:=\tilde{\phi}^{L,bc}(1-\cdot),\\
	\tilde{\Psi}_{j}^{bc} & := \{\tilde{\psi}^{L,bc}_{j;0}\} \cup \{\tilde{\psi}_{j;k}:2\le k \le 2^{j}-3\} \cup \{\tilde{\psi}^{R,bc}_{j;2^{j}-1}\}, \quad \mbox{with} \quad \tilde{\psi}^{R,bc}:=\tilde{\psi}^{L,bc}(1-\cdot).
	\end{align*}
	Note 
	 $x\chi_{[0,1]}\subset\mbox{span}(\Phi_{j}^{bc})$ for all $j \ge 2$. By Theorem \ref{thm:bw:0N} with $N=1$, $(\tilde{\cB}_J^{bc},\cB_J^{bc})$ forms a biorthogonal Riesz basis of $L_{2}([0,1])$ for $J \ge 3$.
	See \cref{fig:legall} for the graphs of $\phi,\psi,\tilde{\phi},\tilde{\psi}$ and all boundary elements.
\end{example}

\begin{figure}
	\begin{subfigure}[b]{0.24\textwidth}		
		 \includegraphics[width=\textwidth,height=0.6\textwidth]{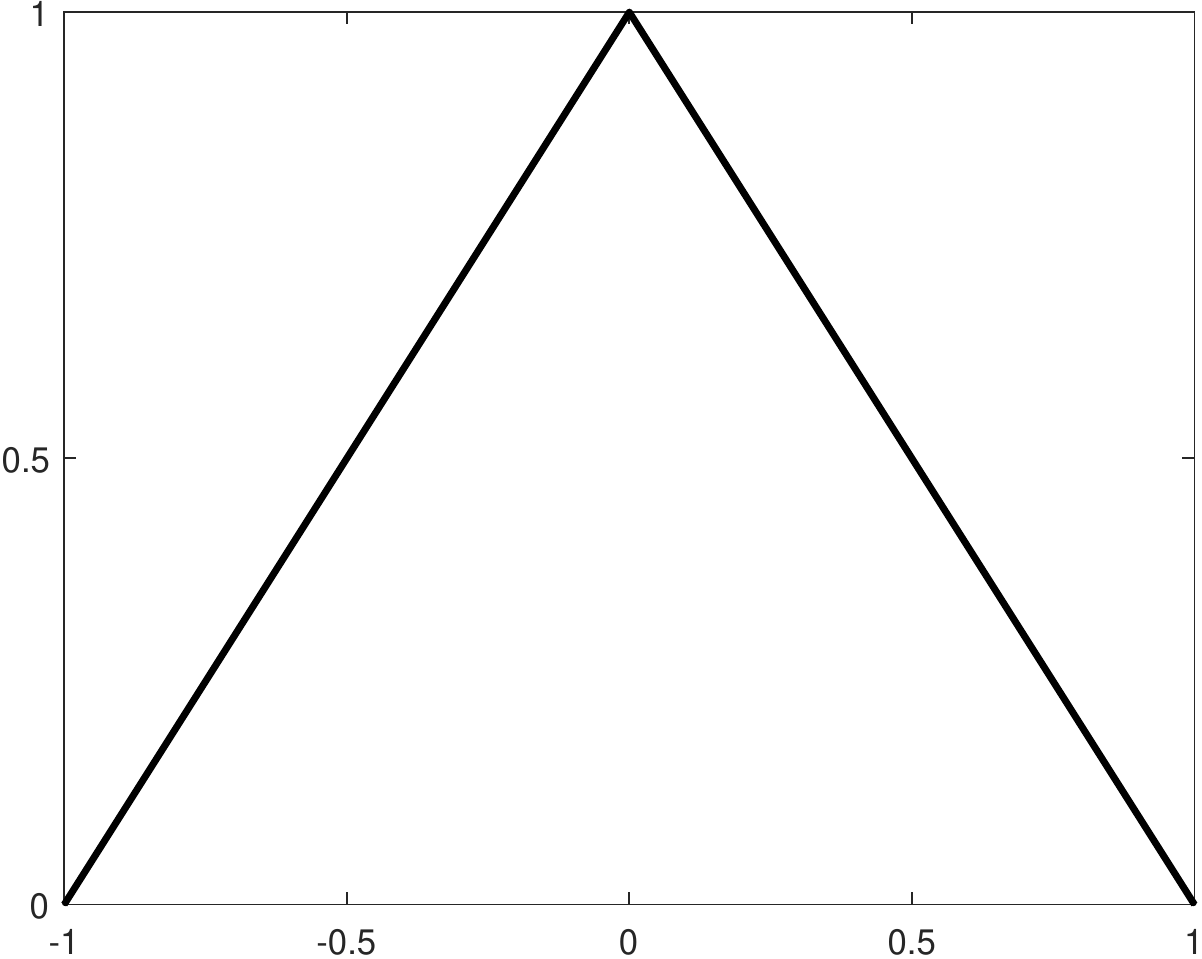}
		\caption{$\phi$}
	\end{subfigure}
	\begin{subfigure}[b]{0.24\textwidth}		
		 \includegraphics[width=\textwidth,height=0.6\textwidth]{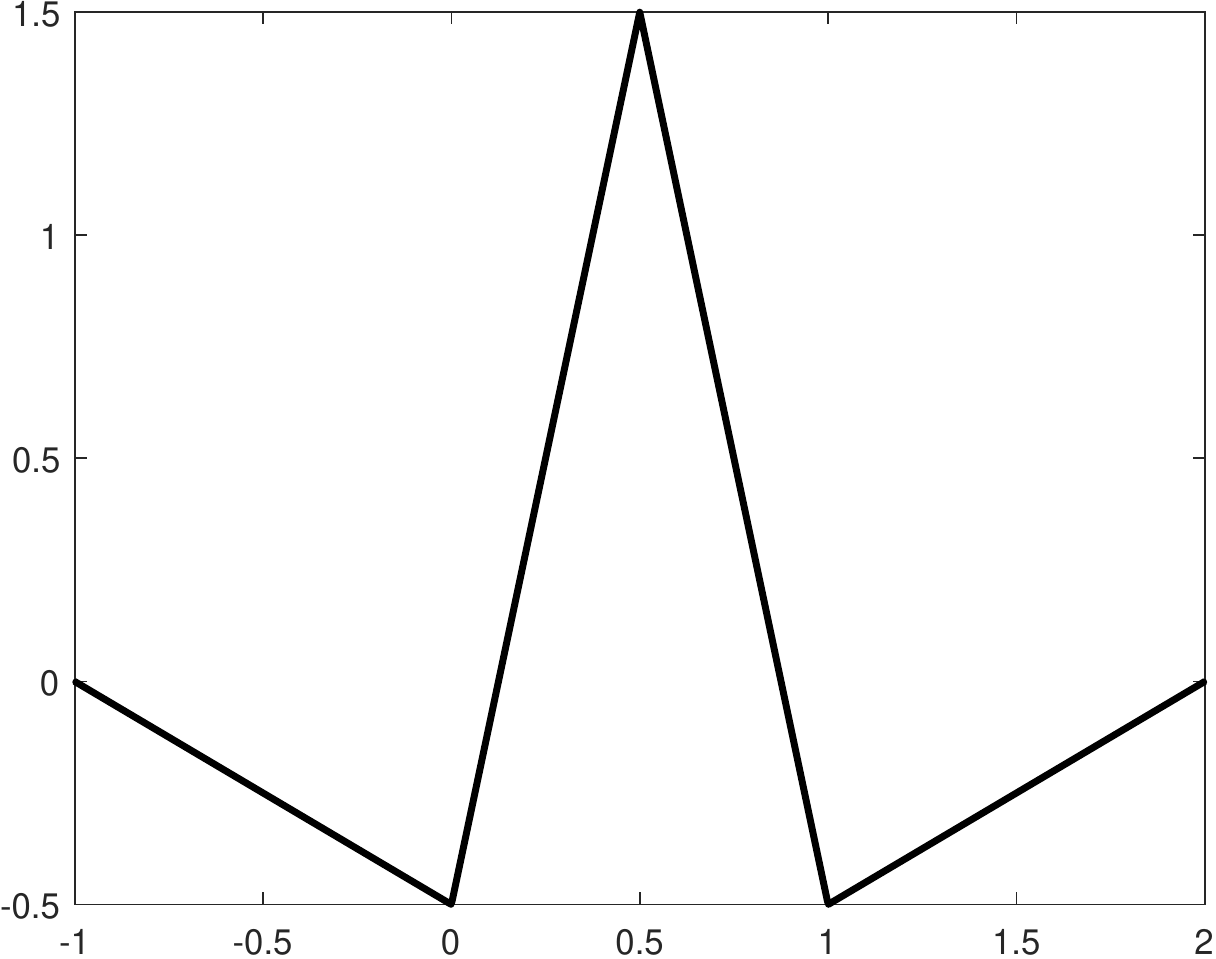}
		\caption{$\psi$}
	\end{subfigure}
	\begin{subfigure}[b]{0.24\textwidth}		
		 \includegraphics[width=\textwidth,height=0.6\textwidth]{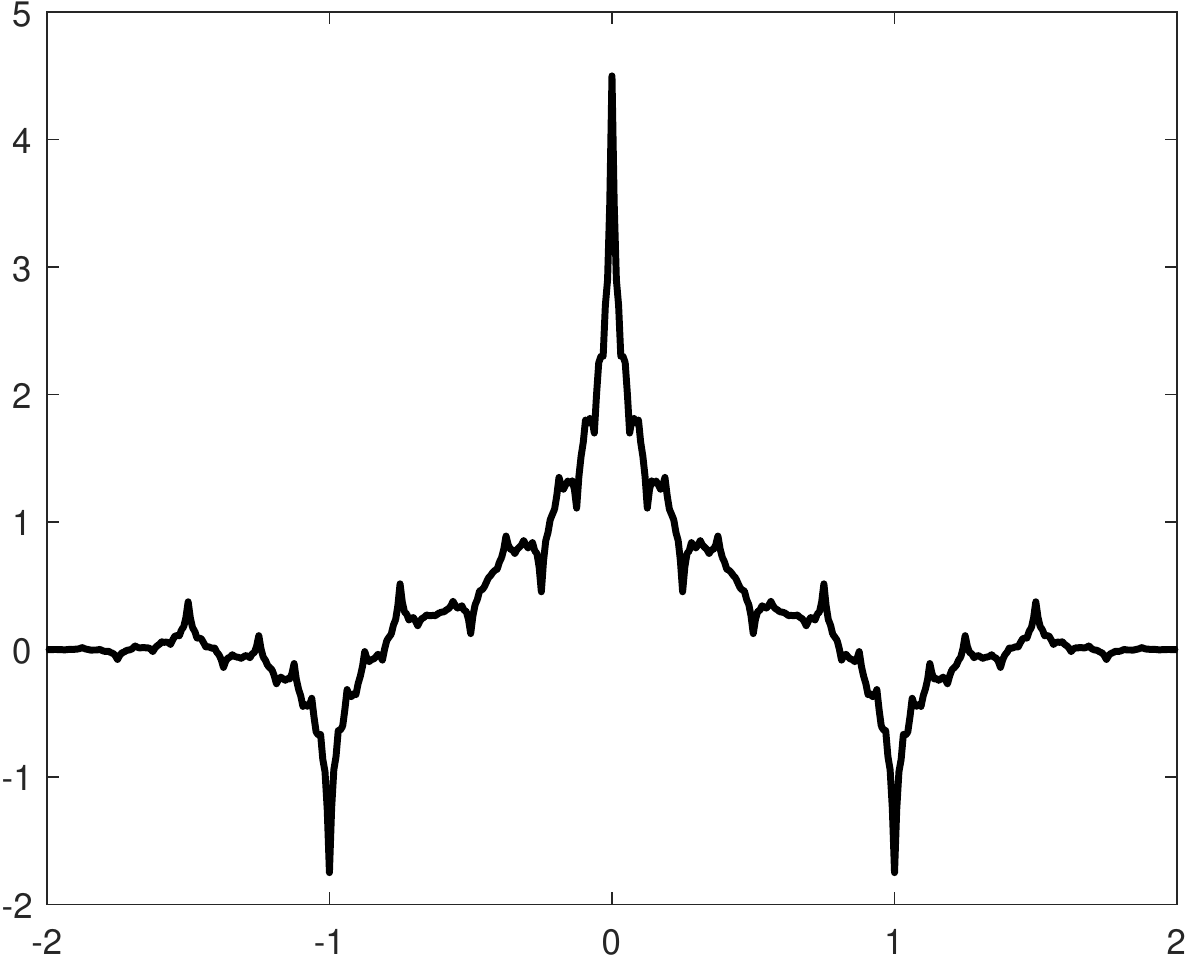}
		\caption{$\tilde{\phi}$}
	\end{subfigure}
	\begin{subfigure}[b]{0.24\textwidth}		
		 \includegraphics[width=\textwidth,height=0.6\textwidth]{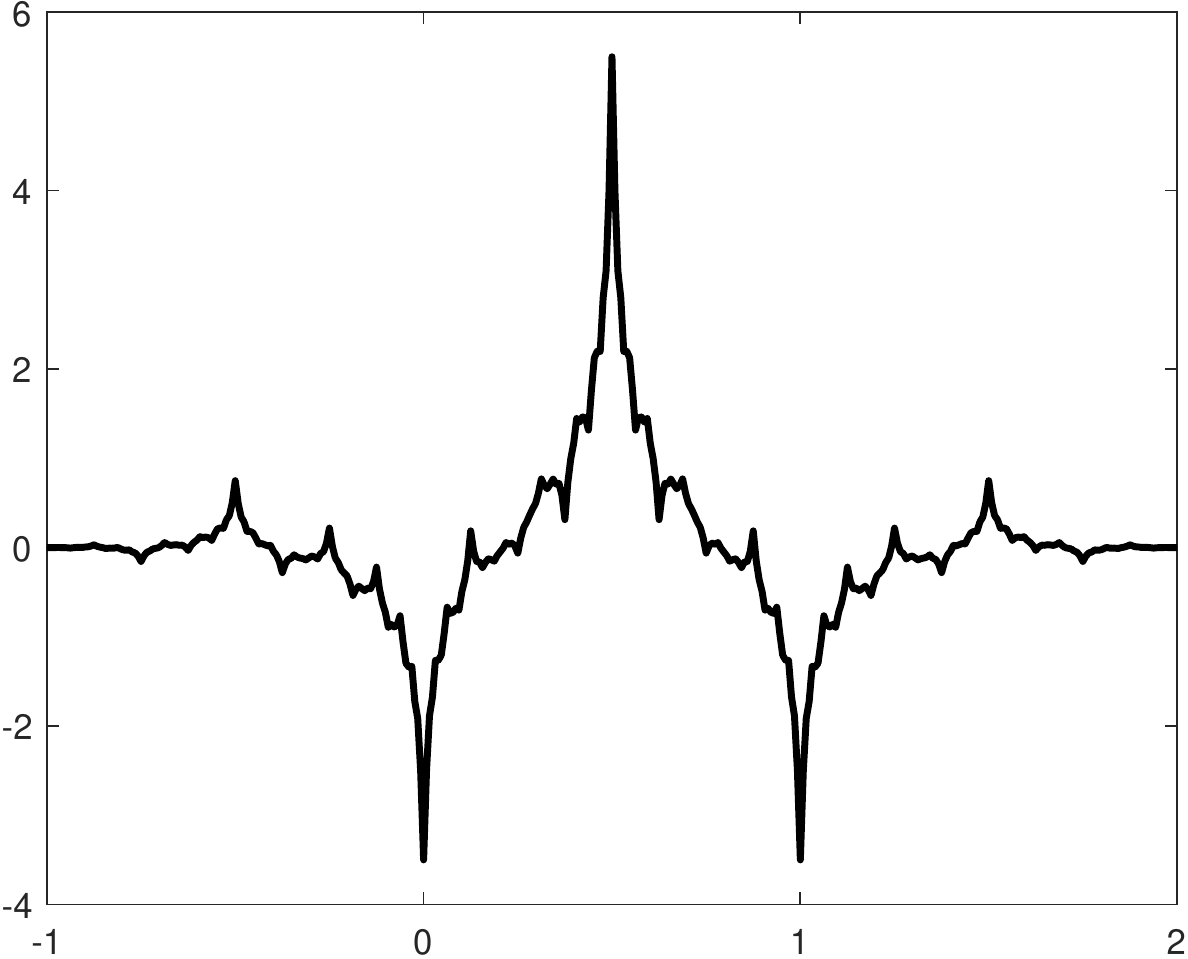}
		\caption{$\tilde{\psi}$}
	\end{subfigure}
	\begin{subfigure}[b]{0.24\textwidth}		
		 \includegraphics[width=\textwidth,height=0.6\textwidth]{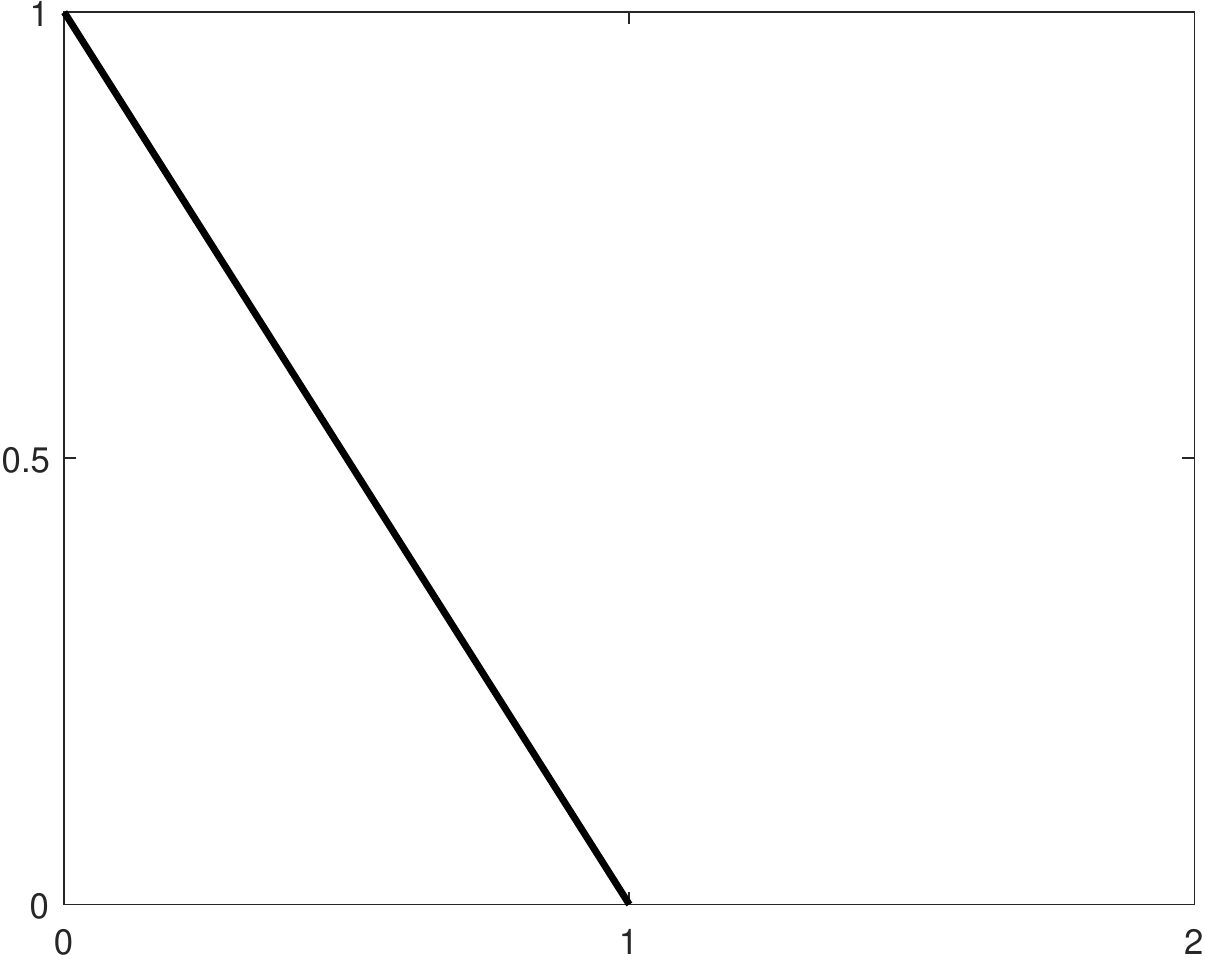}
		\caption{$\phi^{L}$}
	\end{subfigure}
	\begin{subfigure}[b]{0.24\textwidth}	
		 \includegraphics[width=\textwidth,height=0.6\textwidth]{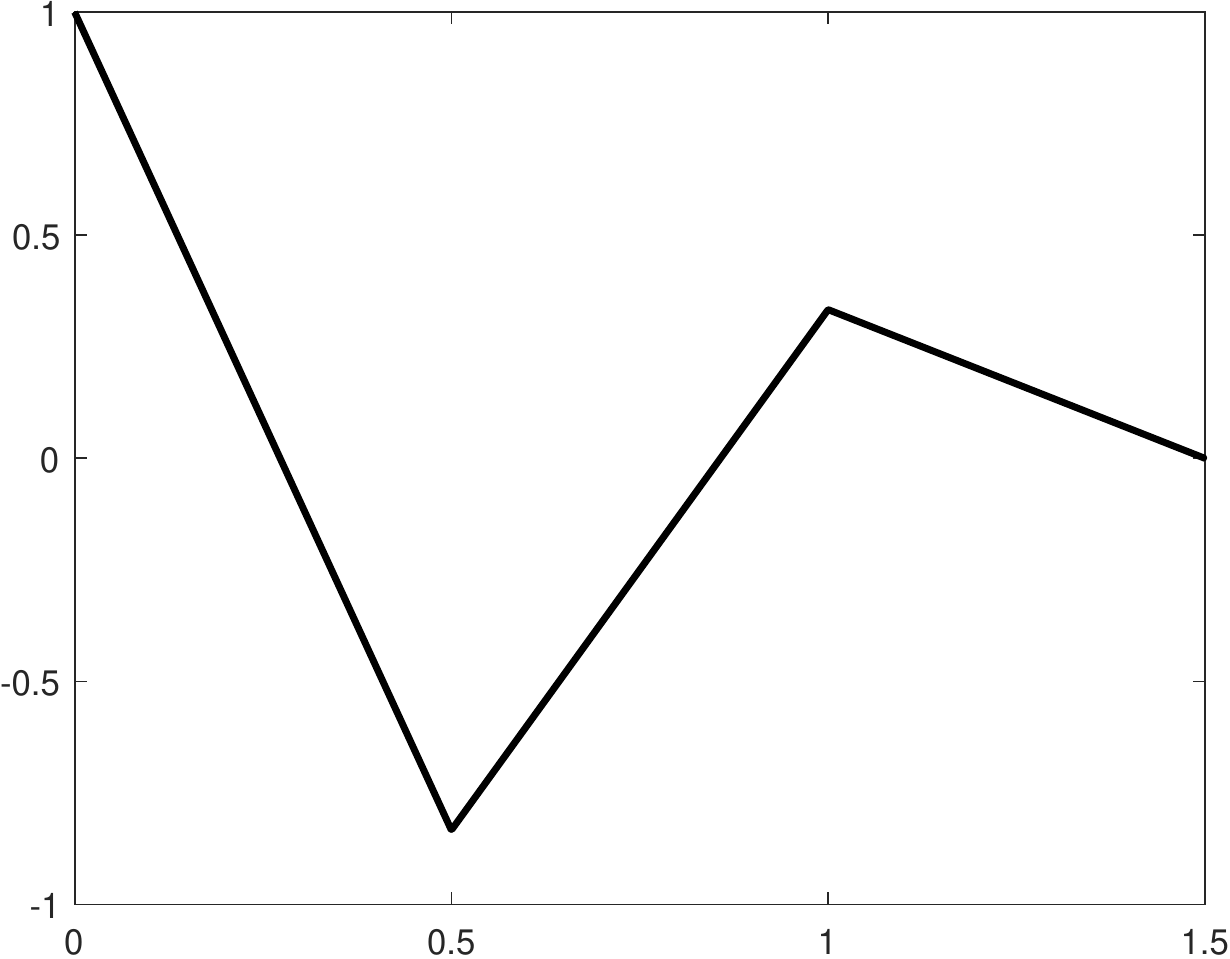}
		\caption{$\psi^{L}$}
	\end{subfigure}
	\begin{subfigure}[b]{0.24\textwidth}	
		 \includegraphics[width=\textwidth,height=0.6\textwidth]{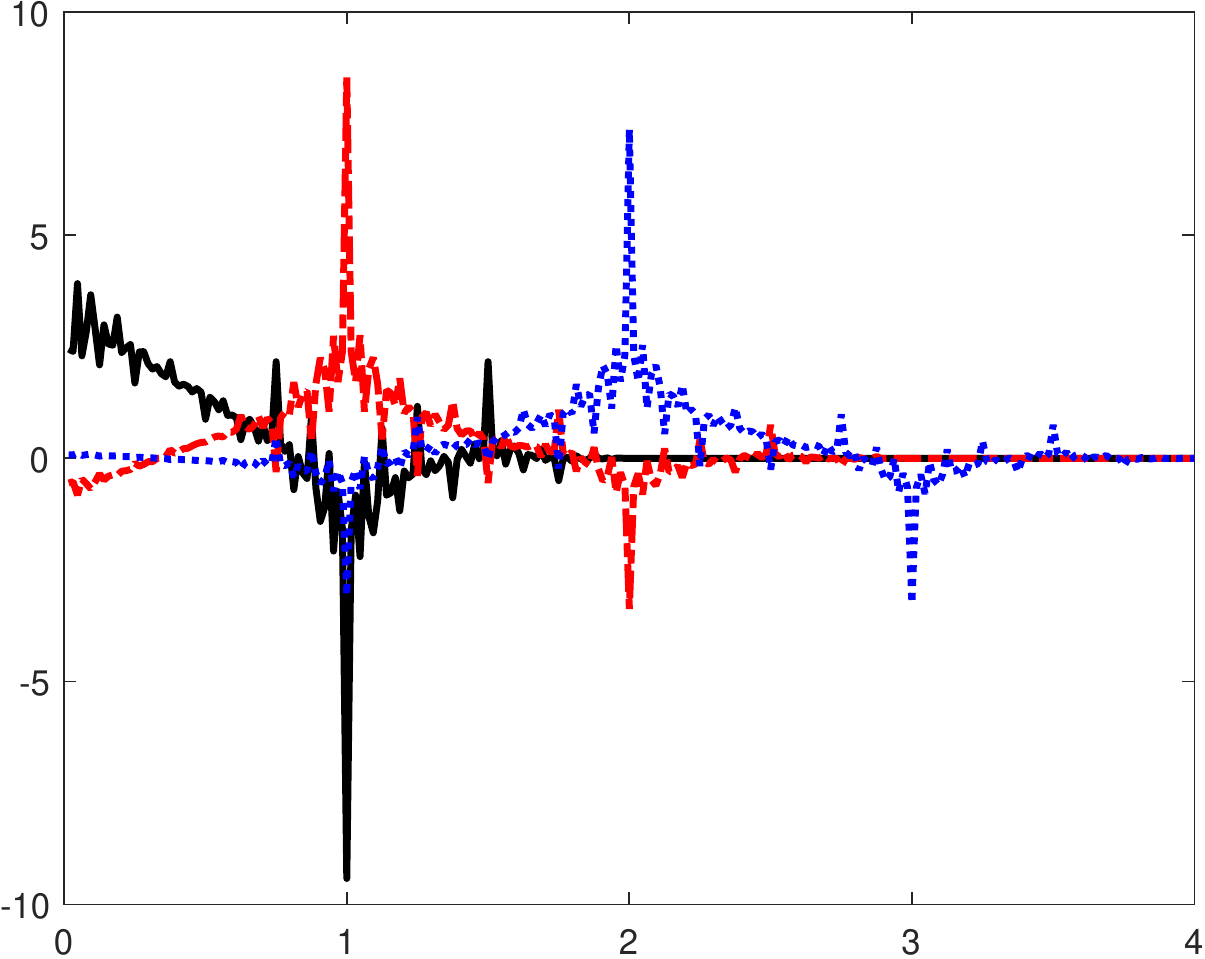}
		\caption{$\tilde{\phi}^{L}$}
	\end{subfigure}
	\begin{subfigure}[b]{0.24\textwidth}
		 \includegraphics[width=\textwidth,height=0.6\textwidth]{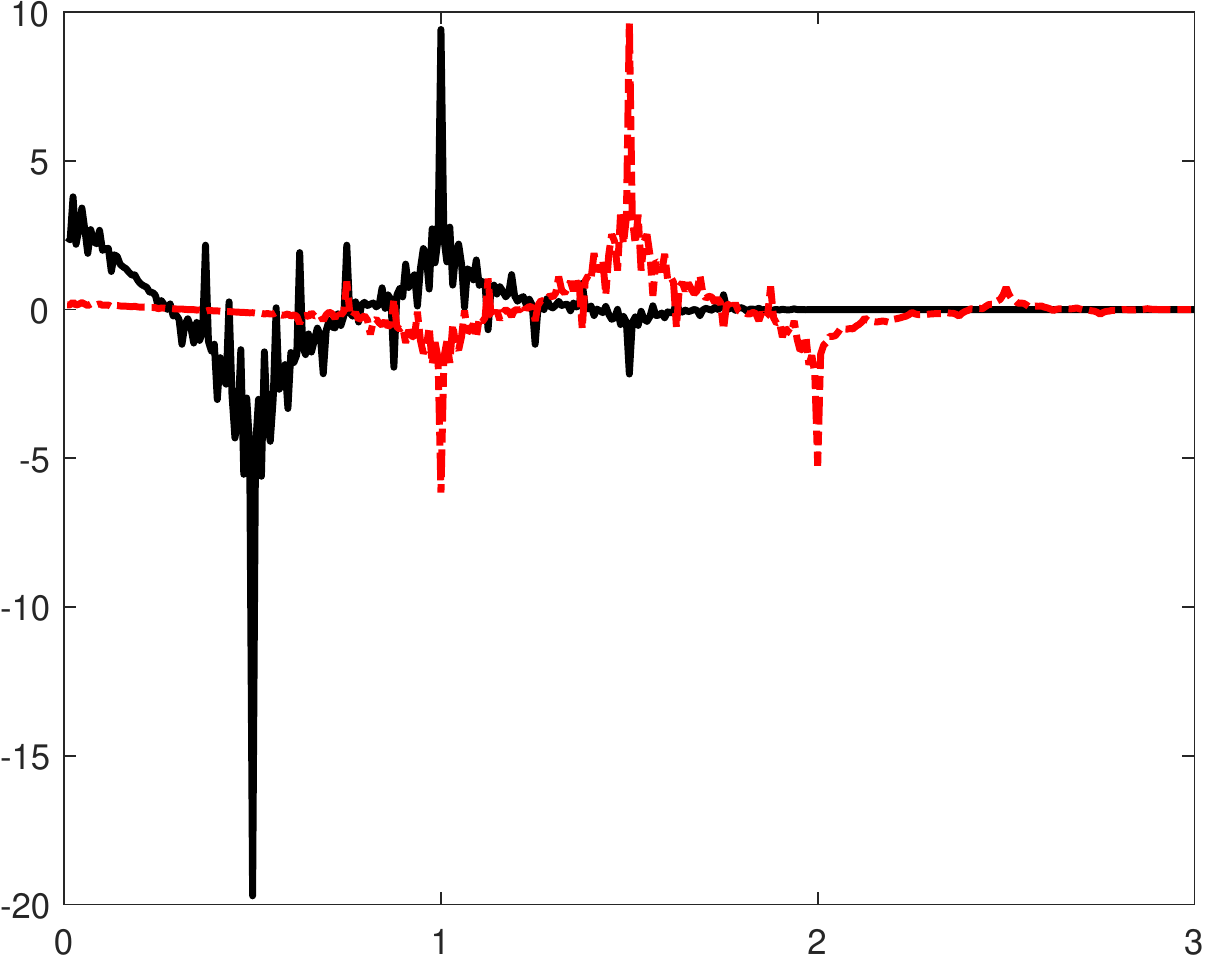}
		\caption{$\tilde{\psi}^{L}$}
	\end{subfigure}
	\begin{subfigure}[b]{0.24\textwidth}	
		 \includegraphics[width=\textwidth,height=0.6\textwidth]{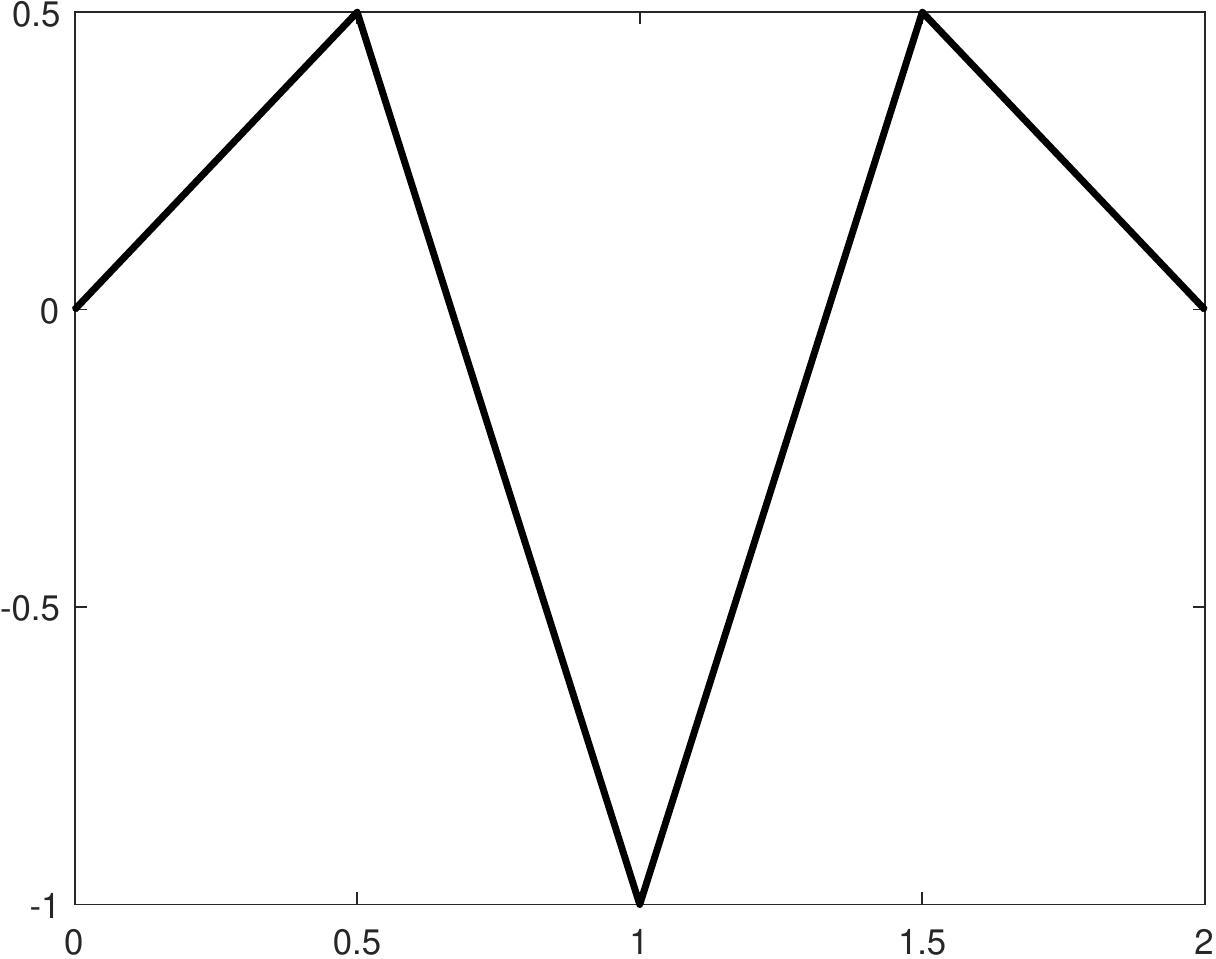}
		\caption{$\psi^{L,bc}$}
	\end{subfigure}
	\begin{subfigure}[b]{0.24\textwidth}	
		 \includegraphics[width=\textwidth,height=0.6\textwidth]{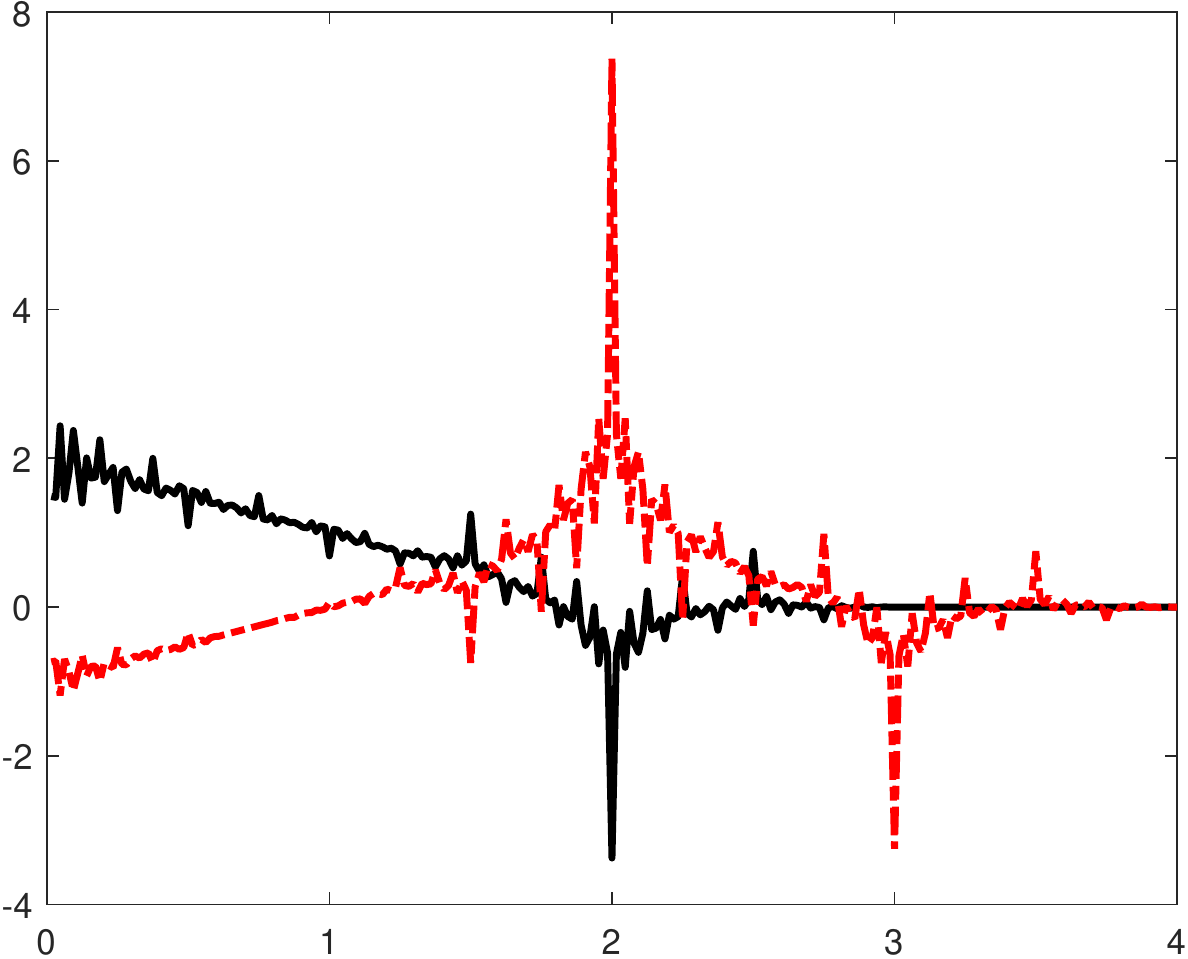}
		\caption{$\tilde{\phi}^{L,bc}$}
	\end{subfigure}
	\begin{subfigure}[b]{0.24\textwidth}
		 \includegraphics[width=\textwidth,height=0.6\textwidth]{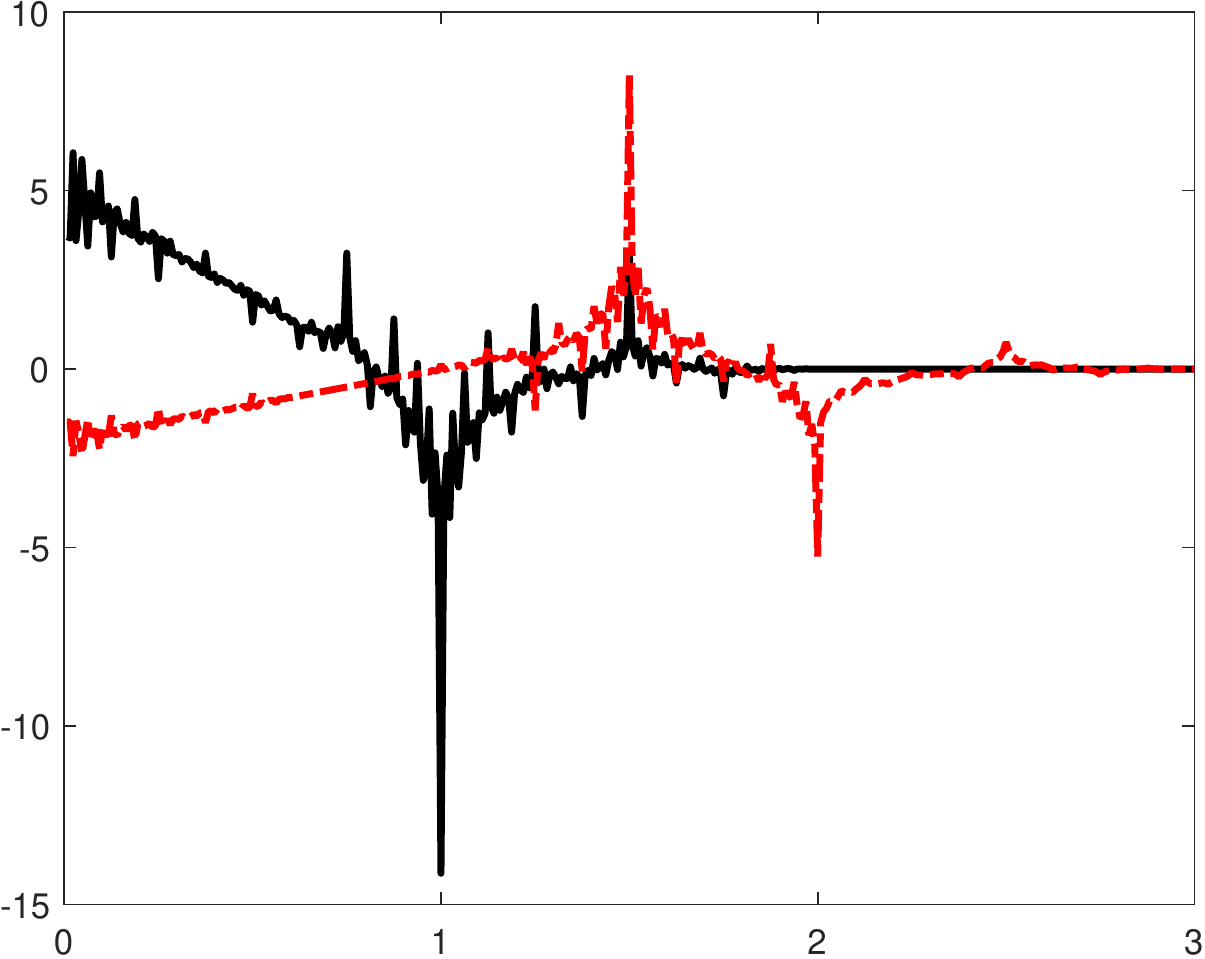}
		\caption{$\tilde{\psi}^{L,bc}$}
	\end{subfigure}
	\begin{subfigure}[b]{0.24\textwidth}
		 \includegraphics[width=\textwidth,height=0.6\textwidth]{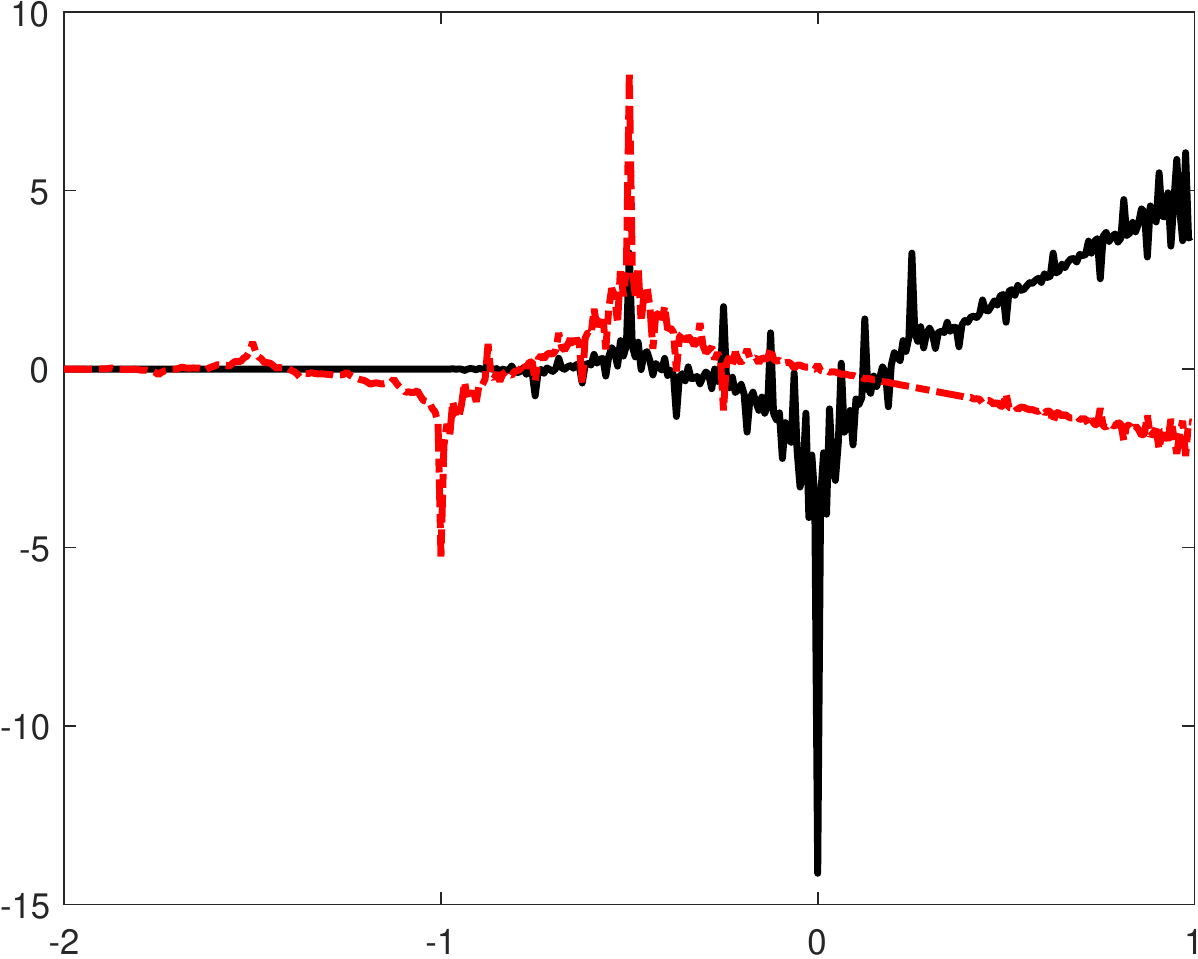}
		\caption{$\tilde{\psi}^{R,bc}$}
	\end{subfigure}
	\caption{The generators of the biorthogonal wavelet bases $(\tilde{\cB}_J,\cB_J)$ and $(\tilde{\cB}_J^{bc},\cB_J^{bc})$ of $L_{2}([0,1])$ for $J \ge 3$ in Example \ref{ex:legall} with $h(0)=h(1)=0$ for all $h \in \cB_J^{bc}$. The black, red, and blue lines correspond to the first, second, and third components of a vector function. Note that $\vmo(\psi^{L})=\vmo(\psi^{L,bc})=\vmo(\psi)=2$.}
	\label{fig:legall}
\end{figure}

\begin{example} \label{ex:interpol}
	\normalfont Consider a scalar biorthogonal wavelet $(\{\tilde{\phi};\tilde{\psi}\},\{\phi;\psi\})$ in \cite[Example~2.6.9]{hanbook} with $\wh{\phi}(0)=\wh{\tilde{\phi}}(0)=1$ and a biorthogonal wavelet filter bank $(\{\tilde{a};\tilde{b}\},\{a;b\})$ given by
	\begin{align*}
	& a=\{-\tfrac{1}{32},0,\tfrac{9}{32},\tfrac{1}{2},\tfrac{9}{32},0,-\tfrac{1}{32}\}_{[-3,3]}, \quad\quad\;\;\,
	 b=\{\tfrac{1}{64},0,-\tfrac{1}{8},-\tfrac{1}{4},\tfrac{23}{32},-\tfrac{1}{4}.-\tfrac{1}{8},0,\tfrac{1}{64}\}_{[-3,5]},\\
	& \tilde{a}=\{\tfrac{1}{64},0,-\tfrac{1}{8},\tfrac{1}{4},\tfrac{23}{32},\tfrac{1}{4},-\tfrac{1}{8},0,\tfrac{1}{64}\}_{[-4,4]}, \quad \tilde{b}=\{\tfrac{1}{32},0,-\tfrac{9}{32},\tfrac{1}{2},-\tfrac{9}{32},0,\tfrac{1}{32}\}_{[-2,4]}.
	\end{align*}
	Then $\sm(a)=2.440765$, $\sm(\tilde{a})=0.593223$, $\sr(a)=4$, and $\sr(\tilde{a})=2$.
	By item (ii) of \cref{prop:phicut} with $n_{\phi}=3$ and $\textsf{p}(x)=(1,x,x^2,x^3)^\tp$, the left boundary refinable function satisfies
	\begin{align} \label{phiL:interpol}
	\phi^{L} & =(\phi^{L}_{1},\phi^{L}_{2},\phi^{L}_{3},\phi^{L}_{4})^{\tp} :=(\phi^{L}_{1}(2\cdot),\tfrac{1}{2}\phi^{L}_{2}(2\cdot),\tfrac{1}{4}\phi^{L}_{3}(2\cdot),\tfrac{1}{8}\phi^{L}_{4}(2\cdot))^{\tp} + [\tfrac{17}{16},\tfrac{27}{16},\tfrac{45}{16},\tfrac{81}{16}]^{\tp}\phi(2\cdot -3)\\ \nonumber
	& \qquad \quad + [1,2,4,8]^{\tp}\phi(2\cdot -4) + [\tfrac{1}{2},\tfrac{17}{16},\tfrac{35}{16},\tfrac{71}{16}]^{\tp}\phi(2\cdot -5) - [\tfrac{1}{16},\tfrac{1}{8},\tfrac{1}{4},\tfrac{1}{2}]^{\tp}\phi(2\cdot-7).
	\end{align}
	We use the direct approach in \cref{sec:direct}.
	Taking $n_{\psi}=3$ and $m_{\phi}=9$ in \cref{thm:direct}, we have
	\[
	\psi^{L}=\begin{bmatrix}
	\psi^{L}_{1}\\
	\psi^{L}_{2}\\
	\psi^{L}_{3}\\
	\end{bmatrix}:=
	\begin{bmatrix}
	1 & 0 & -\tfrac{55}{14} & \tfrac{225}{119}\\
	0 & 1 & -\tfrac{375}{196} & \tfrac{4765}{6664}\\
	0 & 0 & 1 & -\tfrac{181}{306}	
	\end{bmatrix}
	{{\begin{bmatrix}
			\phi^{L}_{1}(2\cdot)\\
			\phi^{L}_{2}(2\cdot)\\
			\phi^{L}_{3}(2\cdot)\\
			\phi^{L}_{4}(2\cdot)\\
			\end{bmatrix}}}
	+ \begin{bmatrix}
	0_{2 \times 1}\\
	\tfrac{49}{135}
	\end{bmatrix}\phi(2\cdot -3),
	\]
	which satisfies items (i) and (ii) of \cref{thm:direct} with $\fs(C)=[3,6]$, $\fs(D)=[3,5]$, and
	{{\begin{align*}
			& A_0 =
			\begin{bmatrix}
			 -\frac{5}{2}&-{\frac{495}{79}}&-{\frac{983}{79}}&-{\frac{1938}{79}}&-{\frac{450}{79}}&{\frac{38189}{632}}&-{\frac{8038}{79}}&{\frac{17973}{316}}&0&-{\frac{4019}{632}}\\
			 {\frac{45}{2}}&{\frac{6755}{158}}&{\frac{6327}{79}}&{\frac{12206}{79}}&{\frac{2460}{79}}&-{\frac{114253}{316}}&{\frac{
					 48344}{79}}&-{\frac{108159}{316}}&0&{\frac{6043}{158}}\\
			 -{\frac{235}{8}}&-{\frac{4200}{79}}&-{\frac{31081}{316}}&-{\frac{29835}{158}}&-{\frac{5805}{158}}&{\frac{1148585}{2528}}&-{\frac{121955}{158}}&{\frac{545895}{1264}}&0&-{\frac{121955}{2528}}\\ {\frac{75}{8}}&{\frac{2635}{158}}&{\frac{9705}{316}}&{\frac{9299}{158}}&{\frac{1785}{158}}&-{\frac{364565}{2528}}&{
				 \frac{38815}{158}}&-{\frac{173775}{1264}}&0&{\frac{38815}{2528}}
			\end{bmatrix}^{\tp},\\
			& B_0 =
			\begin{bmatrix}
			 \frac{7}{2}&{\frac{495}{79}}&{\frac{983}{79}}&{\frac{1938}{79}}&{\frac{450}{79}}&-{\frac{38189}{632}}&{\frac{8038}{79}}&-{\frac{17973}{316}}&0&{\frac{4019}{632}}\\
			 -{\frac{45}{4}}&-{\frac{6439}{316}}&-{\frac{6327}{158}}&-{\frac{6103}{79}}&-{\frac{1230}{79}}&{\frac{114253}{632}}&-{\frac{24172}{79}}&{\frac{108159}{632}}&0&-{\frac{6043}{316}}\\
			 -{\frac{675}{1568}}&-{\frac{66825}{61936}}&-{\frac{132705}{61936}}&-{\frac{130815}{30968}}&{\frac{54945}{30968}}&-{\frac{2523285}{495488}}&{\frac{225315}{30968}}&-{\frac{986445}{247744}}&0&{\frac{225315}{495488}}
			\end{bmatrix}^{\tp},\\
			& C(3) =
			\begin{bmatrix}
			0_{1\times 5}&{\frac{1}{64}}&0&{\frac{55}{64}}&\frac{1}{4}&-{\frac{9}{64}}
			\end{bmatrix}^{\tp}, \quad
			C(4) =
			\begin{bmatrix}
			0_{1\times 5}&{\tfrac{1}{64}}&0&-\tfrac{1}{8}&\tfrac{1}{4}&{\frac{23}{32}}
			\end{bmatrix}^{\tp},\\
			& C(5) =
			\begin{bmatrix}
			0_{1\times 7}&{\tfrac{1}{64}}&0&-\tfrac{1}{8}
			\end{bmatrix}^{\tp},\quad
			C(6) =
			\begin{bmatrix}
			0_{1\times 9}&{\tfrac{1}{64}}
			\end{bmatrix}^{\tp},\quad
			D(3) =
			\begin{bmatrix}
			0_{1\times 5}&\tfrac{1}{32}&0&-{\frac{9}{32}}&\tfrac{1}{2}&-{\frac{9}{32}}
			\end{bmatrix}^{\tp},\\
			& D(4) =
			\begin{bmatrix}
			0_{1\times 7}&\tfrac{1}{32}&0&-{\frac{9}{32}}
			\end{bmatrix}^{\tp}, \quad
			D(5) =
			\begin{bmatrix}
			0_{1\times 9}&\tfrac{1}{32}
			\end{bmatrix}^{\tp}.
			\end{align*}
	}}
	Since  \eqref{tAL} is satisfied with $\rho(\tilde{A}_{L})=1/2$, we conclude from \cref{thm:direct,thm:bw:0N} with $N=1$ that  $\cB_J=\Phi_J \cup \{\Psi_j \setsp j\ge J\}$ is a Riesz basis of $L_{2}([0,1])$ for all $J\ge J_{0}:=3$, where $\Phi_j$ and $\Psi_j$ in \eqref{Phij} and \eqref{Psij} with $n_{\phi}=n_{\mathring{\phi}}=n_{\psi}=3$ and $n_{\mathring{\psi}}=4$ are given by
	\begin{align*}
	\Phi_{j} & := \{\phi^{L}_{j;0},\phi_{j;3}\} \cup \{\phi_{j;k}:4\le k \le 2^{j}-4\} \cup \{\phi^{R}_{j;2^{j}-1},\phi_{j;2^{j}-3}\},\\
	\Psi_{j} & := \{\psi^{L}_{j;0}\} \cup \{\psi_{j;k}:3\le k \le 2^{j}-4\} \cup \{\psi^{R}_{j;2^{j}-1}\},
	\end{align*}
	where $\phi^{R}:=\phi^{L}(1-\cdot)$ and $\psi^{R}:=\psi^{L}(1-\cdot)$ with $\#\phi^{L} = \# \phi^{R}=4$, $\#\psi^{L} = \# \psi^{R}=3$, $\#\Phi_j=2^j+3$ and $\#\Psi_j=2^j$. Note that $\vmo(\psi^{L})=\vmo(\psi^{R})=\vmo(\psi)=2=\sr(\tilde{a})$. The dual Riesz basis $\tilde{\cB}_j$ of $\cB_{j}$ with $j\ge \tilde{J}_{0} := 3$ is given by \cref{thm:direct} through \eqref{tphi:implicit} and \eqref{tpsi:implicit}. We rewrite $\tilde{\phi}^{L}$ in \eqref{tphi:implicit} as $\{\tilde{\phi}^{L},\tilde{\phi}(\cdot-4),\tilde{\phi}(\cdot-5),\tilde{\phi}(\cdot-6),\tilde{\phi}(\cdot-7),\tilde{\phi}(\cdot-8)\}$ with true boundary elements $\tilde{\phi}^{L}$ and $\# \tilde{\phi}^{L}=5$, and rewrite $\tilde{\psi}^{L}$ in \eqref{tpsi:implicit} as $\{\tilde{\psi}^{L},\tilde{\psi}(\cdot-3),\tilde{\psi}(\cdot-4),\tilde{\psi}(\cdot-5)\}$ with true boundary elements $\tilde{\psi}^{L}$ and $\# \tilde{\psi}^{L}=3$. Hence, $\tilde{\cB}_J=\tilde{\Phi}_J \cup \{\tilde{\Psi}_j \setsp j\ge J\}$ for $J\ge 3$ is given by
	\begin{align*}
	\tilde{\Phi}_{j} & := \{\tilde{\phi}^{L}_{j;0}\} \cup \{\tilde{\phi}_{j;k}:4\le k \le 2^{j}-4\} \cup \{\tilde{\phi}^{R}_{j;2^{j}-1}\}, \quad \mbox{with} \quad \tilde{\phi}^{R}:=\tilde{\phi}^{L}(1-\cdot),\\
	\tilde{\Psi}_{j} & := \{\tilde{\psi}^{L}_{j;0}\} \cup \{\tilde{\psi}_{j;k}:3\le k \le 2^{j}-4\} \cup \{\tilde{\psi}^{R}_{j;2^{j}-1}\}, \quad \mbox{with} \quad \tilde{\psi}^{R}:=\tilde{\psi}^{L}(1-\cdot).
	\end{align*}
	Note that $\vmo(\tilde{\psi}^{L})=\vmo(\tilde{\psi}^{R})=\vmo(\tilde{\psi})=4=\sr(a)$ and $\PL_{3} \chi_{[0,1]} \subset\mbox{span}(\Phi_{j})$ for all $j \ge 3$. According to Theorem \ref{thm:bw:0N} with $N=1$, $(\tilde{\cB}_J,\cB_J)$ forms a biorthogonal Riesz basis of $L_{2}([0,1])$ for every $J \ge 3$.
	
	By item (ii) of \cref{prop:phicut} with $n_\phi=3$ and $\textsf{p}(x)=(x,x^{2},x^{3})^\tp$, the left boundary refinable vector function is $\phi^{L,bc}:=(\phi^{L}_{2},\phi^{L}_{3},\phi^{L}_{4})^{\tp}$ and satisfies \eqref{I:phi}, which can be easily deduced from \eqref{phiL:interpol}.
	Taking $n_{\psi}=3$ and $m_{\phi}=9$ in \cref{thm:direct}, we have $\psi^{L,bc}:=(\psi^{L,bc}_1, \psi^L_2,\psi^L_3)^\tp$ with $\# \psi^{L,bc}=3$ and
	\[
	\psi^{L,bc}_1:=\psi^L_1
	 -[1,0,-\tfrac{55}{14},\tfrac{2131}{1190}]\phi^{L}(2\cdot) - \tfrac{289}{150}\phi(2\cdot-3) + \tfrac{51}{50}\phi(2\cdot-4)
	\]
	satisfying both items (i) and (ii) of \cref{thm:direct}, where $A^{bc}_0$, $B^{bc}_0$, $C^{bc}$, and $D^{bc}$ can be easily derived from $A_{0}$, $B_{0}$, $C$, and $D$. More precisely, $A_0^{bc}$ is obtained from $U^{-1} A_0$ by taking out its first row and first column, and $B_0^{bc}, C^{bc}, D^{bc}$ are obtained from $U^{-1}B_0, U^{-1}C, U^{-1}D$, respectively by removing their first rows, where the invertible matrix $U$ is given by
	\[
	U:=I_{10}+B_0
	{{\begin{bmatrix} -1&0&{\frac{55}{14}}&-{\frac{2131}{1190}}&-{\frac{289}{150}}&{\frac{51}{50}}& {0_{1 \times 4}}\\
			\multicolumn{7}{c}{0_{2\times 10}}
			\end{bmatrix}}}.
	\]
	Since \eqref{tAL} is satisfied with $\rho(\tilde{A}_{L}^{bc})=1/2$, we conclude from \cref{thm:direct,thm:bw:0N} with $N=1$ that $\cB_J^{bc}=\Phi_J^{bc} \cup \{\Psi_j^{bc} \setsp j\ge J\}$ is a Riesz basis of $L_{2}([0,1])$ for every  $J\ge J_{0}:=3$ such that $h(0)=h(1)=0$ for all $h\in \cB_J^{bc}$, where $\Phi_j^{bc}$ and $\Psi_j^{bc}$  in \eqref{Phij} and \eqref{Psij} with $n_{\phi}=n_{\mathring{\phi}}=n_{\psi}=3$ and $n_{\mathring{\psi}}=4$ are given by
	\begin{align*}
	\Phi_{j}^{bc} & := \{\phi^{L,bc}_{j;0},\phi_{j;3},\phi_{j;4}\} \cup \{\phi_{j;k}:5\le k \le 2^{j}-5\} \cup \{\phi^{R,bc}_{j;2^{j}-1},\phi_{j;2^{j}-3}, \phi_{j;2^{j}-4}\},\\
	\Psi_{j}^{bc} & := \{\psi^{L,bc}_{j;0},\psi_{j;3}\} \cup \{\psi_{j;k}:4\le k \le 2^{j}-5\} \cup \{\psi^{R,bc}_{j;2^{j}-1},\psi_{j;2^j-4}\},
	\end{align*}
	where $\phi^{R,bc}:=\phi^{L,bc}(1-\cdot)$ and $\psi^{R,bc}:=\psi^{L,bc}(1-\cdot)$ with $\#\phi^{L,bc}=\#\psi^{L,bc}=3$, $\#\Phi^{bc}_j=2^j+1$, and $\#\Psi^{bc}_j=2^j$.
	For $j=3$, $\Phi^{bc}_3=\{\phi^{L;bc}_{3;0}, \phi_{3;3}, \phi_{3;4}, \phi_{3;5}, \phi^{R,bc}_{3;7}\}$ after removing repeated elements.
	Note that $\vmo(\psi^{L,bc})=\vmo(\psi^{R,bc})=\vmo(\psi)=2$ and $\Phi^{bc}_j=\Phi_j\bs\{(\phi^L_1)_{j;0},(\phi^R_1)_{j;2^j-1}\}$ as in \cref{prop:mod}.
	We rewrite $\tilde{\phi}^{L,bc}$ in \eqref{tphi:implicit} as $\{\tilde{\phi}^{L,bc},\tilde{\phi}(\cdot-5),\tilde{\phi}(\cdot-6),\tilde{\phi}(\cdot-7),\tilde{\phi}(\cdot-8)\}$ with true boundary elements $\tilde{\phi}^{L,bc}$ and $\# \tilde{\phi}^{L,bc}=5$, and $\tilde{\psi}^{L,bc}$ in \eqref{tpsi:implicit} as $\{\tilde{\psi}^{L,bc},\tilde{\psi}(\cdot-4),\tilde{\psi}(\cdot-5)\}$ with true boundary elements $\tilde{\psi}^{L,bc}$ and $\# \tilde{\psi}^{L,bc}=4$.
	The dual Riesz basis $\tilde{\cB}_J^{bc}:=\tilde{\Phi}_J^{bc} \cup \{\tilde{\Psi}_j^{bc} \setsp j\ge J\}$ of $\cB_{J}^{bc}$ with $J\ge \tilde{J}_{0} := 4$ is given by
	\begin{align*}
	\tilde{\Phi}_{j}^{bc} & := \{\tilde{\phi}^{L,bc}_{j;0}\} \cup \{\tilde{\phi}_{j;k}:5\le k \le 2^{j}-5\} \cup \{\tilde{\phi}^{R,bc}_{j;2^{j}-1}\}, \quad \mbox{with} \quad \tilde{\phi}^{R,bc}:=\tilde{\phi}^{L,bc}(1-\cdot),\\
	\tilde{\Psi}_{j}^{bc} & := \{\tilde{\psi}^{L,bc}_{j;0}\} \cup \{\tilde{\psi}_{j;k}:4\le k \le 2^{j}-5\} \cup \{\tilde{\psi}^{R,bc}_{j;2^{j}-1}\}, \quad \mbox{with} \quad \tilde{\psi}^{R,bc}:=\tilde{\psi}^{L,bc}(1-\cdot).
	\end{align*}
	Note that $\vmo(\tilde{\psi}^{L,bc})=\vmo(\tilde{\psi}^{R,bc})=0$ and $\{x\chi_{[0,1]},x^2\chi_{[0,1]},x^3\chi_{[0,1]}\} \subset\mbox{span}(\Phi_{j}^{bc})$ for all $j \ge 3$. By Theorem \ref{thm:bw:0N} with $N=1$, $(\tilde{\cB}_J^{bc},\cB_J^{bc})$ forms a biorthogonal Riesz basis of $L_{2}([0,1])$ for every $J \ge 4$.
	
	By item (ii) of \cref{prop:phicut} with $n_\phi=3$ and $\textsf{p}(x)=(x^{2},x^{3})^\tp$, the left boundary refinable vector function is $\phi^{L,bc1}:=(\phi^{L}_{3},\phi^{L}_{4})^{\tp}$ and satisfies \eqref{I:phi}, which can easily deduce from from \eqref{phiL:interpol}.
	Taking $n_{\psi}=3$ and $m_{\phi}=9$ in \cref{thm:direct}, we have $\psi^{L,bc1}:=(\psi^{L,bc}_1,\psi^{L,bc1}_2,\psi^L_3)$ with $\# \psi^{L,bc1}=3$ and
	\[
	\psi^{L,bc1}_2:=\psi^{L}_{2} -
	[0, 1, -\tfrac{375}{196}, \tfrac{4765}{6664}]\phi^{L}(2\cdot)-\tfrac{1}{2}\phi(2\cdot-3)+\phi(2\cdot-4)-\tfrac{1}{2}\phi(2\cdot-5)
	\]
	satisfying both items (i) and (ii) of \cref{thm:direct}, where $A^{bc1}_0$, $B^{bc1}_0$, $C^{bc1}$, and $D^{bc1}$ can be easily derived from $A_{0}$, $B_{0}$, $C$, and $D$. More precisely, $A_0^{bc1}$ is obtained from $V^{-1} A_0$ by taking out its first two rows and the first two columns, and $B_0^{bc1}, C^{bc1}, D^{bc1}$ are obtained from $V^{-1}B_0$, $V^{-1}C$, $V^{-1}D$, respectively by removing their first two rows, where the invertible matrix $V$ is given by
	\[
	V:=I_{10}+B_0
	{{\begin{bmatrix}
			 -1&0&{\frac{55}{14}}&-{\frac{2131}{1190}}&-{\frac{289}{150}}&{\frac{51}{50}}& 0 & {0_{1 \times 3}}\\[0.2em]
			0 & -1 & \tfrac{375}{196} & -\tfrac{4765}{6664} & -\tfrac{1}{2} & 1 & -\tfrac{1}{2} &  {0_{1 \times 3}}\\
			 \multicolumn{8}{c}{0_{1\times10}}
			\end{bmatrix}}}
	\]
	Since \eqref{tAL} is satisfied with $\rho(\tilde{A}_{L}^{bc1})=1/2$, we conclude from \cref{thm:direct,thm:bw:0N} with $N=1$ that $\cB_J^{bc1}=\Phi_J^{bc1} \cup \{\Psi_j^{bc1} \setsp j\ge J\}$ is a Riesz basis of $L_{2}([0,1])$ for every  $J\ge J_{0}:=3$ such that $h(0)=h'(0)=h(1)=h'(1)=0$ for all $h\in \cB^{bc1}_J$, where $\Phi_j^{bc1}$ and $\Psi_j^{bc1}$ in \eqref{Phij} and \eqref{Psij} with $n_{\phi}=n_{\mathring{\phi}}=n_{\psi}=3$ and $n_{\mathring{\psi}}=4$ are given by
	\begin{align*}
	\Phi_{j}^{bc1} & :=  \{\phi^{L,bc1}_{j;0},\phi_{j;3},\phi_{j;4}\} \cup \{\phi_{j;k}:5\le k \le 2^{j}-5\} \cup \{\phi^{R,bc1}_{j;2^{j}-1},\phi_{j;2^{j}-3}, \phi_{j;2^{j}-4}\},\\
	\Psi_{j}^{bc1} & := \{\psi^{L,bc1}_{j;0},\psi_{j;3}\} \cup \{\psi_{j;k}:4\le k \le 2^{j}-5\} \cup \{\psi^{R,bc1}_{j;2^{j}-1},\psi_{j;2^{j}-4}\},
	\end{align*}
	where
	$\phi^{R,bc1}:=\phi^{L,bc1}(1-\cdot)$
	and
	$\psi^{R,bc1}:=\psi^{L,bc1}(1-\cdot)$ with
	$\#\phi^{L,bc1}=\#\phi^{R,bc1}=2$,
	$\#\psi^{L,bc1}=\#\psi^{R,bc1}=3$, $\#\Phi^{bc1}_j=2^j-1$ and $\#\Psi^{bc1}_j=2^j$.
	For the case $j=3$, $\Phi^{bc1}_3=\{\phi^{L,bc1}_{3;0}, \phi_{3;3},\phi_{3;4},\phi_{3;5}, \phi^{R,bc1}_{3;7}\}$ after removing repeated elements.
	Note that $\vmo(\psi^{L,bc1})=\vmo(\psi^{R,bc1})=\vmo(\psi)=2$ and $\Phi^{bc1}_j=\Phi_j\bs \{(\phi^L_1)_{j;0},(\phi^L_2)_{j;0},
	 (\phi^R_1)_{j;2^j-1},(\phi^R_2)_{j;2^j-1}\}$ as in \cref{prop:mod}.
	The dual Riesz basis $\tilde{\cB}_j^{bc1}$ of $\cB_{j}^{bc1}$ with $j\ge \tilde{J}_{0} := 4$ is given by \cref{thm:direct} through \eqref{tphi:implicit} and \eqref{tpsi:implicit}. We rewrite $\tilde{\phi}^{L,bc1}$ in \eqref{tphi:implicit} as $\{\tilde{\phi}^{L,bc1},\tilde{\phi}(\cdot-5),\tilde{\phi}(\cdot-6),\tilde{\phi}(\cdot-7),\tilde{\phi}(\cdot-8)\}$ with true boundary elements $\tilde{\phi}^{L,bc1}$ and $\# \tilde{\phi}^{L,bc1}=4$, and
	$\tilde{\psi}^{L,bc1}$ in \eqref{tpsi:implicit} as $\{\tilde{\psi}^{L,bc1}, \tilde{\psi}(\cdot-4),\tilde{\psi}(\cdot-5)\}$ with true boundary elements $\tilde{\psi}^{L,bc1}$ and
	$\# \tilde{\psi}^{L,bc1}=4$.
	Hence, $\tilde{\cB}_J^{bc1}=\tilde{\Phi}_J^{bc1} \cup \{\tilde{\Psi}_j^{bc1} \setsp j\ge J\}$ for $J \ge 4$ is given by
	\begin{align*}
	\tilde{\Phi}_{j}^{bc1} & := \{\tilde{\phi}^{L,bc1}_{j;0}\} \cup \{\tilde{\phi}_{j;k}:5\le k \le 2^{j}-5\} \cup \{\tilde{\phi}^{R,bc1}_{j;2^{j}-1}\}, \quad \mbox{with} \quad \tilde{\phi}^{R,bc1}:=\tilde{\phi}^{L,bc1}(1-\cdot),\\
	\tilde{\Psi}_{j}^{bc1} & :=\{\tilde{\psi}^{L,bc1}_{j;0}\} \cup \{\tilde{\psi}_{j;k}:4\le k \le 2^{j}-5\} \cup \{\tilde{\psi}^{R,bc1}_{j;2^{j}-1}\}, \quad \mbox{with} \quad \tilde{\psi}^{R,bc1}:=\tilde{\psi}^{L,bc1}(1-\cdot).
	\end{align*}
	Note that $\vmo(\tilde{\psi}^{L,bc1})=\vmo(\tilde{\psi}^{R,bc1})=0$ and $\{x^2\chi_{[0,1]},x^3\chi_{[0,1]}\} \subset\mbox{span}(\Phi_{j}^{bc1})$ for all $j \ge 3$. According to Theorem \ref{thm:bw:0N} with $N=1$, $(\tilde{\cB}_J^{bc1},\cB_J^{bc1})$ forms a biorthogonal Riesz basis of $L_{2}([0,1])$ for every $J \ge 4$.
	See \cref{fig:interpol} for the graphs of $\phi, \psi$ and their associated boundary elements.
\end{example}

\begin{figure}[htbp]
	\centering
	\begin{subfigure}[b]{0.24\textwidth} \includegraphics[width=\textwidth,height=0.6\textwidth]{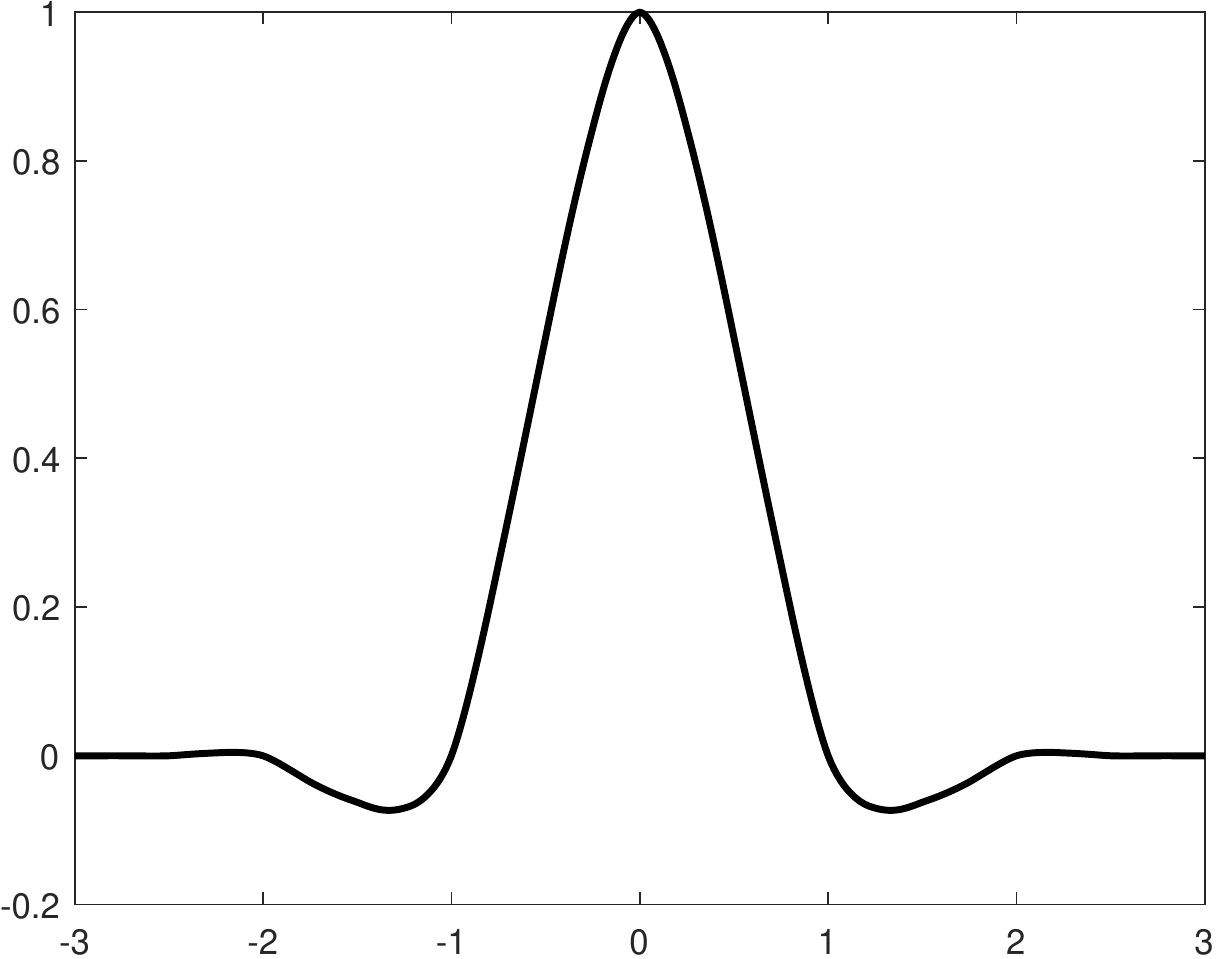}
		\caption{$\phi$}
	\end{subfigure}	
	\begin{subfigure}[b]{0.24\textwidth} \includegraphics[width=\textwidth,height=0.6\textwidth]{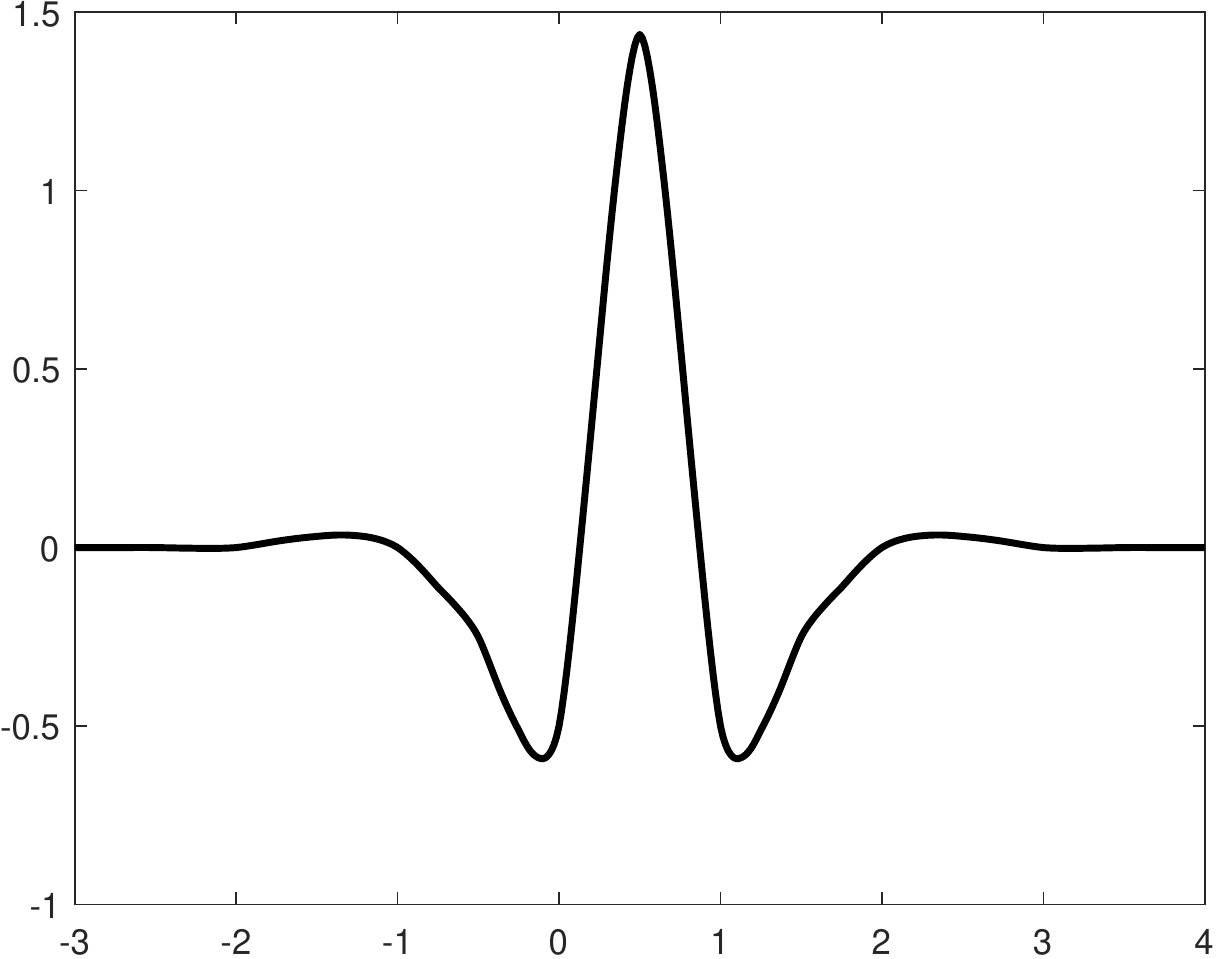}
		\caption{$\psi$}
	\end{subfigure}
	\begin{subfigure}[b]{0.24\textwidth} \includegraphics[width=\textwidth,height=0.6\textwidth]{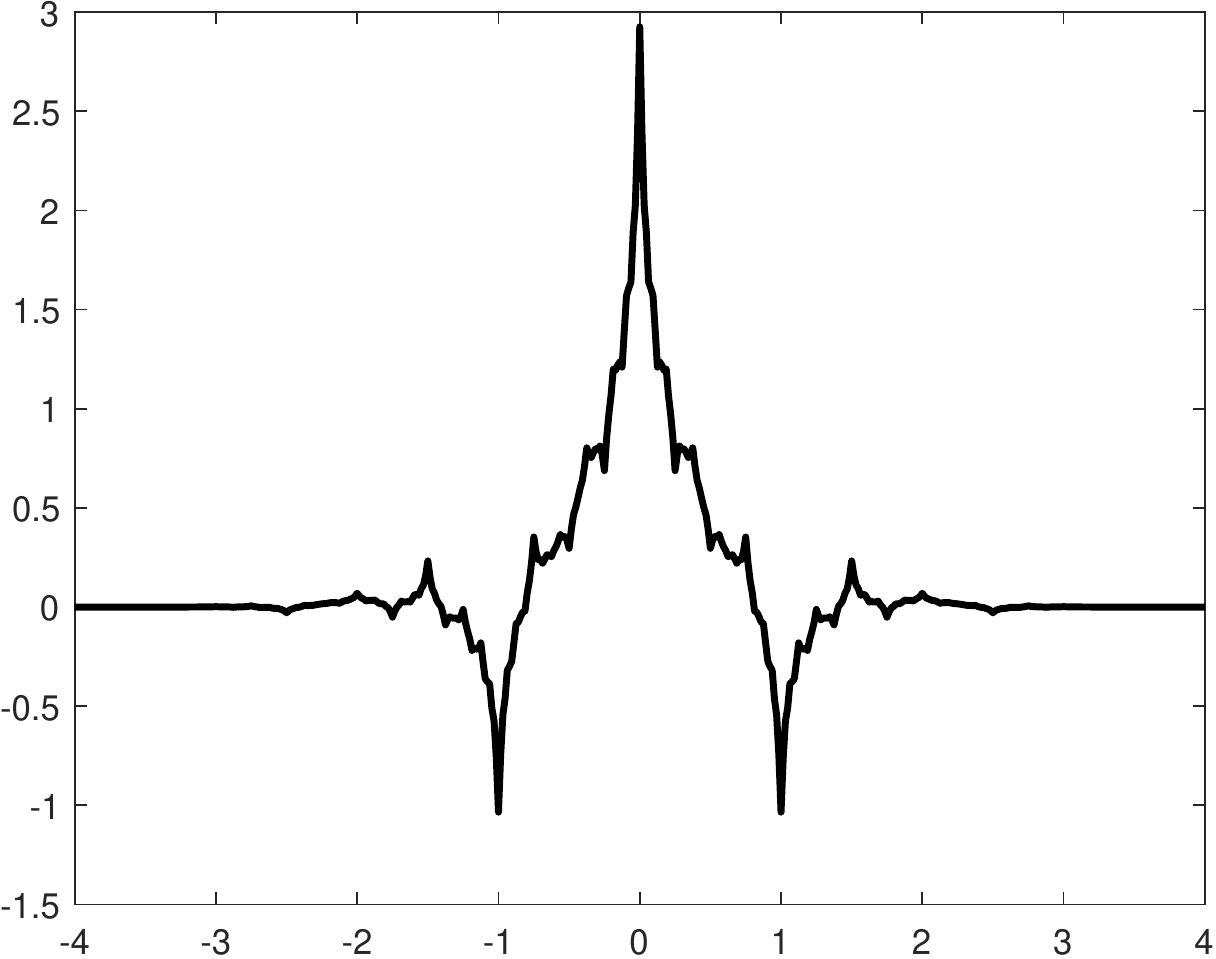}
		\caption{$\tilde{\phi}$}
	\end{subfigure}	
	\begin{subfigure}[b]{0.24\textwidth} \includegraphics[width=\textwidth,height=0.6\textwidth]{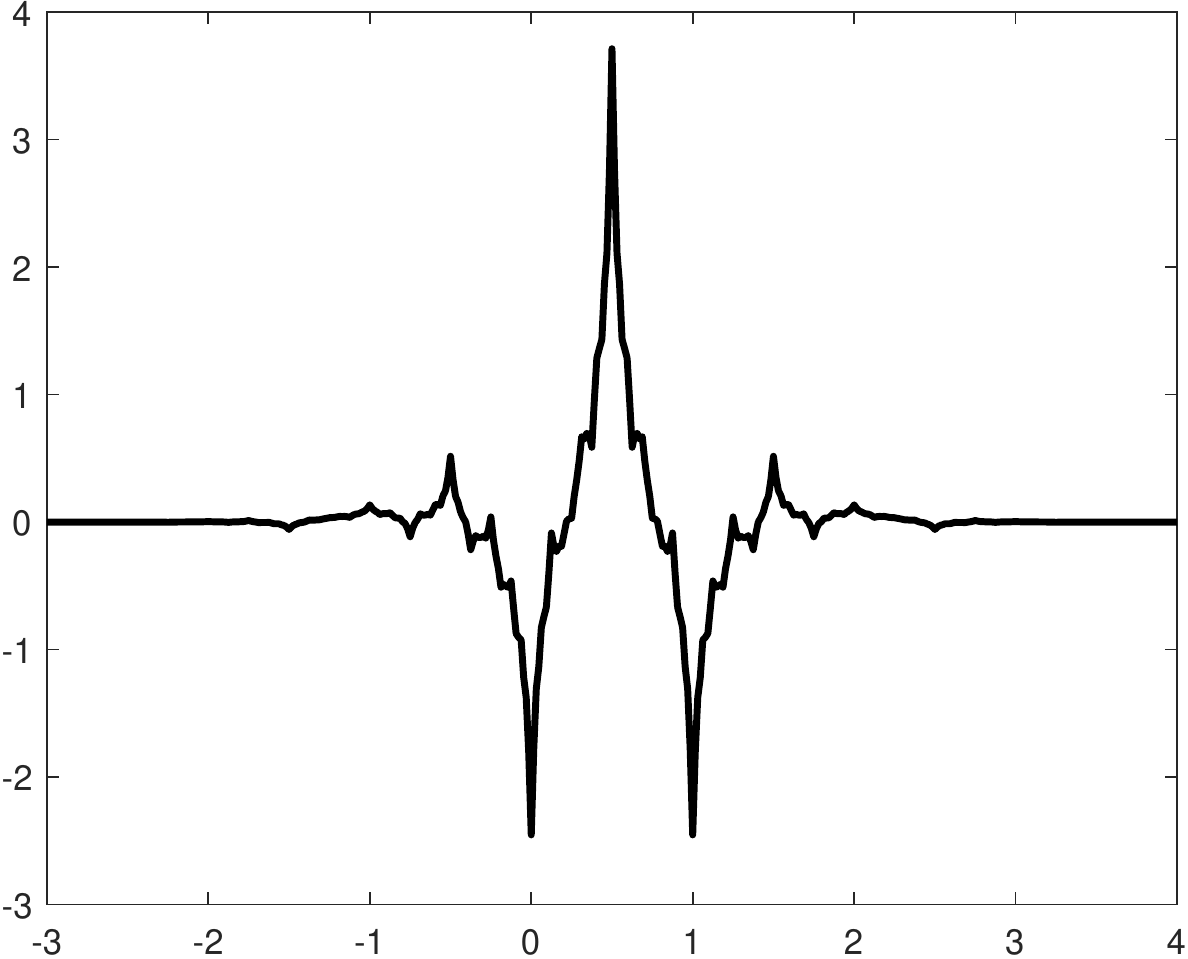}
		\caption{$\tilde{\psi}$}
	\end{subfigure}
	\begin{subfigure}[b]{0.24\textwidth} \includegraphics[width=\textwidth,height=0.6\textwidth]{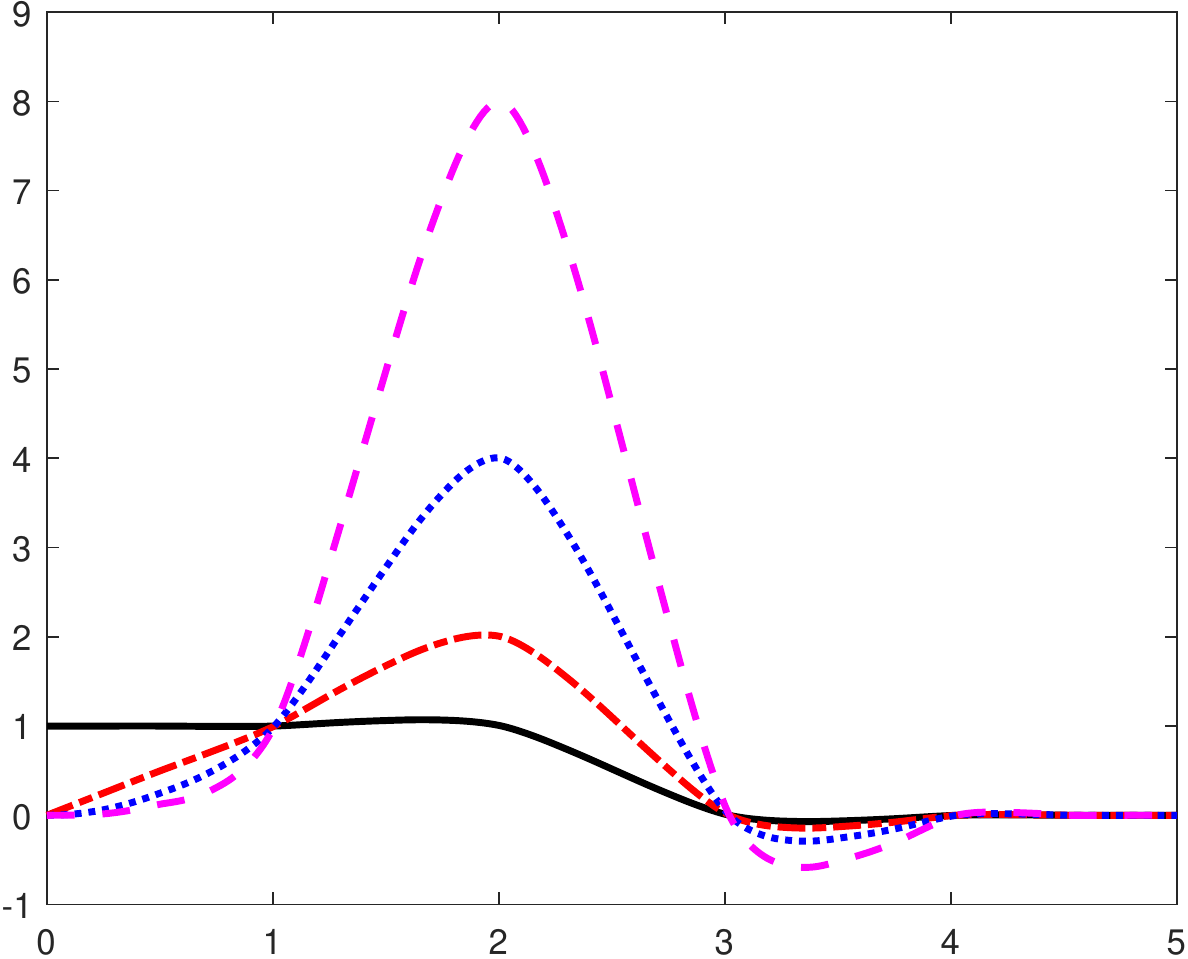}
		\caption{$\phi^{L}$}
	\end{subfigure}	
	\begin{subfigure}[b]{0.24\textwidth} \includegraphics[width=\textwidth,height=0.6\textwidth]{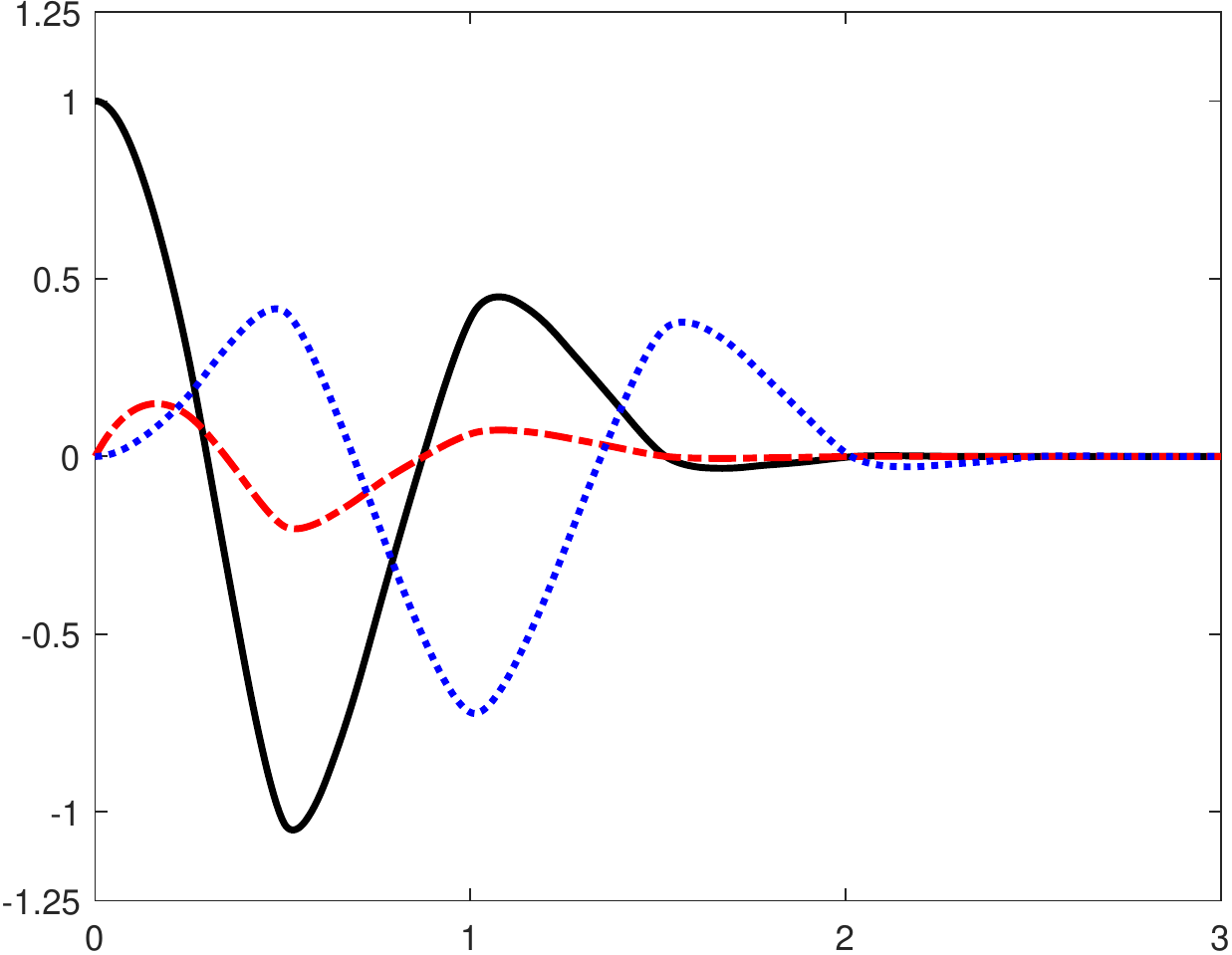}
		\caption{$\psi^{L}$}
	\end{subfigure}
	\begin{subfigure}[b]{0.24\textwidth} \includegraphics[width=\textwidth,height=0.6\textwidth]{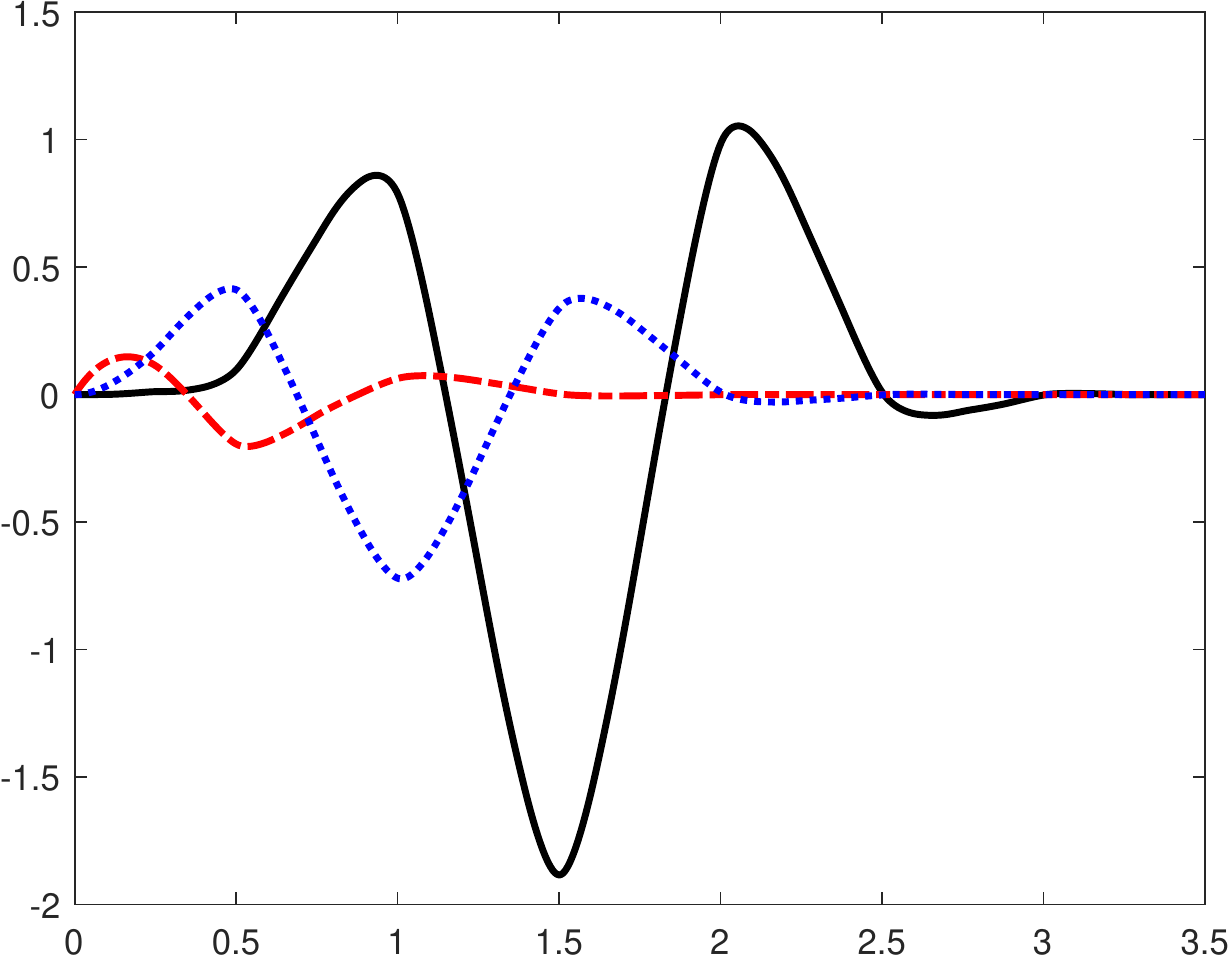}
		\caption{$\psi^{L,bc}$}
	\end{subfigure}
	\begin{subfigure}[b]{0.24\textwidth} \includegraphics[width=\textwidth,height=0.6\textwidth]{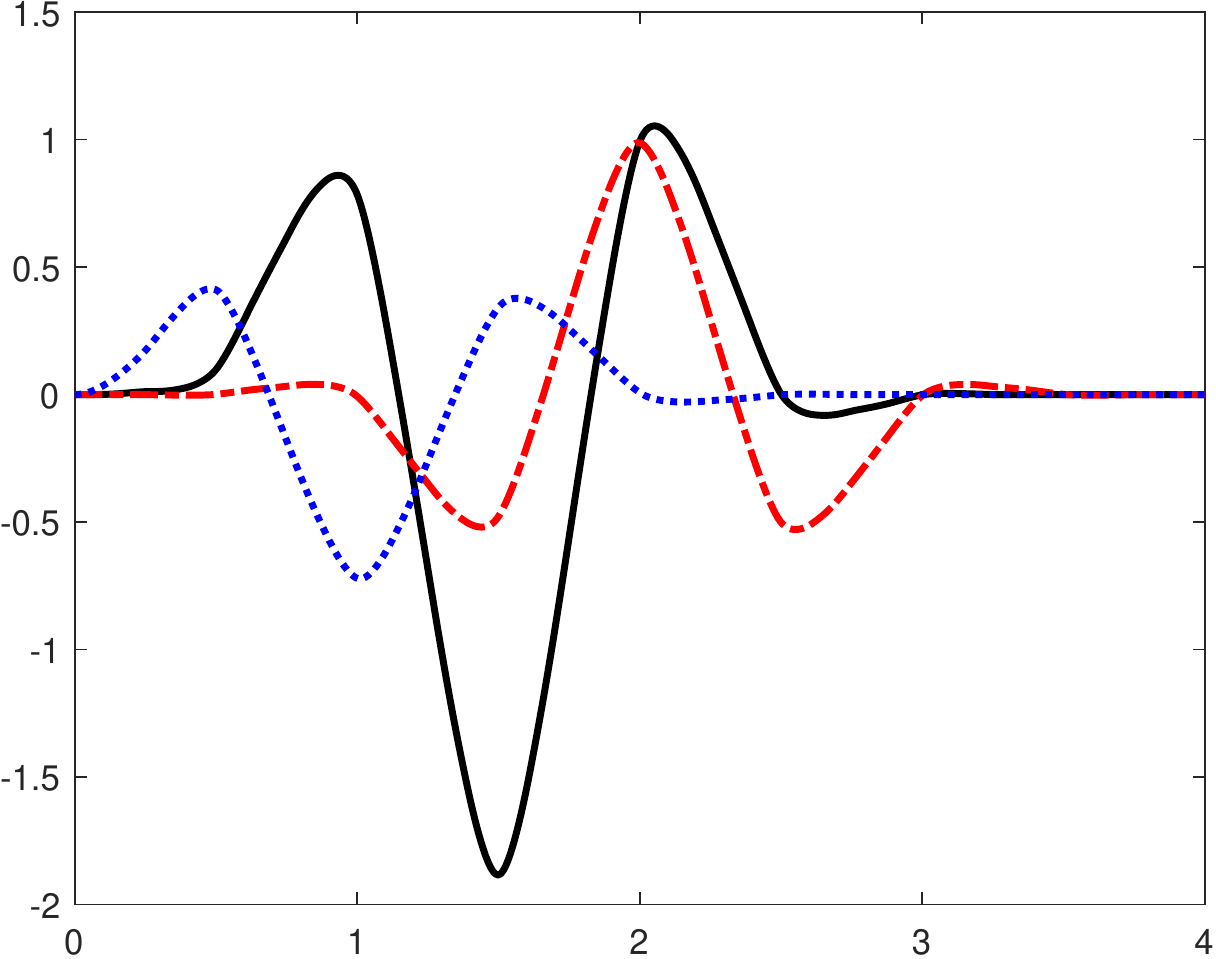}
		\caption{$\psi^{L,bc1}$}
	\end{subfigure}
	\caption{The generators of the Riesz bases $\cB_J$, $\cB_J^{bc}$, $\cB_J^{bc1}$ for $L_{2}([0,1])$ with $J \ge 3$ in \cref{ex:interpol} such that $\eta(0)=\eta(1)$ for all $\eta\in \cB^{bc}_J$ and $h(0)=h'(0)=h(1)=h'(1)=0$ for all $h\in \cB^{bc1}_J$. The black, red, blue, and purple lines correspond to the first, second, third, and fourth components of a vector function. Note that $\phi^{L,bc}=(\phi^L_2,\phi^L_3,\phi^L_4)^\tp$ in $\cB_J^{bc}$,
		 $\phi^{L,bc1}=(\phi^L_3,\phi^L_4)^\tp$ in $\cB^{bc1}_J$, and
		 $\vmo(\psi^L)=\vmo(\psi^{L,bc})=\vmo(\psi^{L,bc1})=\vmo(\psi)=2$.}
	\label{fig:interpol}
\end{figure}

\section{Proofs of
	 \cref{thm:wbd:0,thm:wbd,thm:integral,thm:bw:0N,thm:direct,thm:direct:tPhi}}
\label{sec:proof}

In this section, we provide the detailed proofs for
\cref{thm:wbd:0,thm:wbd,thm:integral,thm:bw:0N,thm:direct,thm:direct:tPhi}.

\begin{proof}[Proof of \cref{thm:wbd:0}]
	Since $n_\phi\ge \max(-l_\phi,-l_a)$ and $n_\psi\ge \max(-l_\psi,\frac{n_\phi-l_b}{2})$, the relations in \eqref{nphi} and \eqref{npsi} must hold, see \Cref{subsec:dual}.
	Hence, $\AS_0(\Phi;\Psi)_{[0,\infty)}\subseteq L_2([0,\infty))$.
	By our assumption that $\AS_0(\Phi;\Psi)_{[0,\infty)}$ is a Riesz basis of $L_2([0,\infty))$, it must have a unique dual Riesz basis, denoted by $\tilde{\cB}$ here, in $L_2([0,\infty))$.
	We first prove that
	 $\tilde{\cB}=\AS_0(\tilde{\Phi};\tilde{\Psi})_{[0,\infty)}$ for some $\tilde{\Phi}, \tilde{\Psi}$ in \eqref{tPhiPsiI}.
	Since both $\phi^L$ and $\psi^L$ have compact support, we have $\fs(\phi^L)\cup \fs(\psi^L)\subseteq [0,N]$ for some $N\in \N$. Consequently, we have
\[
\supp(\phi^L(2^j\cdot))\cup \supp(\psi^L(2^j\cdot))\subseteq
	[0, 2^{-j}N]\subseteq
	[0,N]\quad \mbox{for all}\quad j\in \NN.
\]
	We take $n_{\tilde{\phi}}\ge \max(-l_{\tilde{\phi}}, -l_{\tilde{a}}, n_\phi)$ such that
	the supports of all $\tilde{\phi}(\cdot-k),k\ge n_{\tilde{\phi}}$ do not essentially overlap with $[0,N]$. Similarly, we take
	$n_{\tilde{\psi}}\ge \max(-l_{\tilde{\psi}}, \frac{n_{\tilde{\phi}}-l_{\tilde{b}}}{2},n_\psi)$ such that the supports of all $\tilde{\psi}(\cdot-k),k\ge n_{\tilde{\psi}}$ do not essentially overlap with $[0,N]$.
	Consequently, \eqref{I:tphi} and \eqref{I:tpsi} must hold and we trivially have
	\be \label{dual:phi:I}
	\la \tilde{\phi}(\cdot-k_0), \phi^L(2^j\cdot)\ra=0,
	\quad \la \tilde{\phi}(\cdot-k_0), \psi^L(2^j\cdot)\ra=0, \qquad \forall\, k_0\ge n_{\tilde{\phi}}, j\in \NN,
	\ee
	and
	\be \label{dual:psi:I}
	\la \tilde{\psi}(\cdot-k_0), \phi^L(2^j\cdot)\ra=0,
	\quad \la \tilde{\psi}(\cdot-k_0), \psi^L(2^j\cdot)\ra=0, \qquad \forall\, k_0\ge n_{\tilde{\psi}},j\in \NN.
	\ee

	Let $k_0\ge n_{\tilde{\phi}}$ be arbitrarily fixed.
	Since $(\{\tilde{\phi};\tilde{\psi}\},\{\phi;\psi\})$ is a biorthogonal wavelet in $\Lp{2}$, it is trivial that
	\[
	\la \tilde{\phi}(\cdot-k_0), \phi(\cdot-k_0)\ra=I_r \quad \mbox{and}\quad \la \tilde{\phi}(\cdot-k_0), h\ra=0 \qquad \forall h \in \AS_0(\phi;\psi)\bs\{\phi(\cdot-k_0)\}.
	\]
	In particular, we have $\la \tilde{\phi}(\cdot-k_0), \phi(\cdot-k)\ra=0$ for all $k \ge n_\phi$ but $k\ne k_0$ and $\la \tilde{\phi}(\cdot-k_0), \psi_{j;k}\ra=0$ for all $j\in \NN$ and $k\in \Z$.
	Now it follows from \eqref{dual:phi:I} that $\tilde{\phi}(\cdot-k_0)$ must be the unique biorthogonal element/vector in $L_2([0,\infty))$ corresponding to the element $\phi(\cdot-k_0)\in \AS_0(\Phi;\Psi)_{[0,\infty)}$; more precisely, $\la \tilde{\phi}(\cdot-k_0), \phi(\cdot-k_0)\ra=I_r$ and $\la \tilde{\phi}(\cdot-k_0),h\ra=0$ for all $h\in \AS_0(\Phi;\Psi)_{[0,\infty)}\bs\{\phi(\cdot-k_0)\}$.
	
	Let $j_0\in \NN:=\N\cup\{0\}$ and $k_0\ge n_{\tilde{\psi}}$ be arbitrarily fixed.
	We now show that $\tilde{\psi}_{j_0;k_0}$ is the unique biorthogonal element in $L_2([0,\infty))$ corresponding to the element $\psi_{j_0;k_0}\in \AS_0(\Phi;\Psi)_{[0,\infty)}$.
	Indeed, it follows from the biorthogonality relation between $\AS_0(\tilde{\phi};\tilde{\psi})$ and $\AS_0(\phi;\psi)$ that
	\be \label{bio:phi:psi:I}
	\la \tilde{\psi}_{j_0;k_0}, \phi_{j_0;k}\ra=0\quad \mbox{and}\quad
	\la \tilde{\psi}_{j_0;k_0}, \psi_{j;k}\ra=\td(j-j_0)\td(k-k_0),\qquad \forall\; j\in \NN, k\in \Z.
	\ee
	By \eqref{dual:psi:I},  we have
	\be \label{phi:eq1}
	\la \tilde{\psi}_{j_0;k_0}, \phi^L(2^{j_0}\cdot)\ra=0,\quad
	\la \tilde{\psi}_{j_0;k_0}, \psi^L(2^j\cdot)\ra=
	2^{-j_0/2} \la \tilde{\psi}(\cdot-k_0), \psi^L(2^{j-j_0}\cdot)\ra=0,\qquad \forall\; j\ge j_0.
	\ee
	By the identities in \eqref{I:phi}, \eqref{I:psi}, \eqref{I:tphi} and \eqref{I:tpsi}, we see that every element in
	\be \label{Sj0}
	S_{j_0}:=\Phi \cup\{2^{j/2} \eta(2^j\cdot) \setsp j=0,\ldots, j_0-1, \eta\in \Psi\}
	\ee
	is a finite linear combination of $\{2^{j_0/2}\phi^L(2^{j_0}\cdot)\}\cup\{ \phi_{j_0;k} \setsp k\ge n_\phi\}$.
	Consequently, it follows from
	\eqref{bio:phi:psi:I} and \eqref{phi:eq1} that $\la \tilde{\psi}_{j_0;k_0}, h\ra=0$ for all $h\in S_{j_0}$.
	Hence, by \eqref{bio:phi:psi:I}, we proved that $\tilde{\psi}_{j_0;k_0}$ is the unique biorthogonal element in $L_2([0,\infty))$ corresponding to the element $\psi_{j_0;k_0}\in \AS_0(\Phi;\Psi)_{[0,\infty)}$.
	
	Since $n_{\tilde{\phi}}\ge n_\phi$, we add the elements $\phi(\cdot-k), n_\phi\le k<n_{\tilde{\phi}}$ to the vector function $\phi^L$ to form a new vector function $\mathring{\phi}^L$. We define $\mathring{\psi}^L$ similarly by adding $\psi(\cdot-k), n_\psi\le k<n_{\tilde{\psi}}$ to $\psi^L$.
	Let $\tilde{\phi}^L$ be the unique biorthogonal element/vector in $L_2([0,\infty))$ corresponding to $\mathring{\phi}^L$, and
	$\tilde{\psi}^L$ be the unique biorthogonal element/vector in $L_2([0,\infty))$ corresponding to $\mathring{\psi}^L$ for the Riesz basis $\AS_0(\Phi;\Psi)_{[0,\infty)}$ of $L_2([0,\infty))$.
	Let $j_0\in \NN$ be arbitrarily fixed. We now prove that $2^{j_0/2}\tilde{\psi}^L(2^{j_0}\cdot)$ is the unique biorthogonal element in $L_2([0,\infty))$ corresponding to $2^{j_0/2}\mathring{\psi}^L(2^{j_0}\cdot)\in \AS_0(\Phi;\Psi)_{[0,\infty)}$.
	By the definition of $\tilde{\psi}^L$, we have
	\[
	\la 2^{j_0/2}\tilde{\psi}^L(2^{j_0}\cdot), \psi_{j;k}\ra=\la \tilde{\psi}^L, \psi_{j-j_0;k}\ra=0,\qquad \forall\; j\ge j_0, k\ge n_{\tilde{\phi}}
	\]
	and
	\[
	\la 2^{j_0/2}\tilde{\psi}^L(2^{j_0}\cdot), 2^{j/2}\mathring{\psi}^L(2^j\cdot)\ra=
	\la \tilde{\psi}^L, 2^{(j-j_0)/2}\mathring{\psi}^L(2^{j-j_0}\cdot)\ra=\td(j-j_0)I_{\#\tilde{\psi}^L},\qquad \forall\, j\ge j_0.
	\]
	Define the set $S_{j_0}$ as in \eqref{Sj0}.  As we proved before, every element in $S_{j_0}$ must be a finite linear combination of
	 $\{2^{j_0/2}\mathring{\phi}^L(2^{j_0}\cdot)\}
	\cup\{ \phi_{j_0;k} \setsp k\ge n_{\tilde{\phi}}\}$.
	By the definition of $\tilde{\psi}^L$, we must have
	\[
	\la 2^{j_0/2}\tilde{\psi}^L(2^{j_0}\cdot),
	 2^{j_0/2}\mathring{\phi}^L(2^{j_0}\cdot)\ra
	=\la \tilde{\psi}^L, \mathring{\phi}^L\ra=0
	\]
	and
	\[
	\la 2^{j_0/2}\tilde{\psi}^L(2^{j_0}\cdot),
	\phi_{j_0;k}\ra=
	\la \tilde{\psi}^L, \phi(\cdot-k)\ra=0, \qquad \forall\, k\ge n_{\tilde{\phi}}.
	\]
	Consequently, we proved that $\la 2^{j_0/2}\tilde{\psi}^L(2^{j_0}\cdot),h\ra=0$ for all $h\in S_{j_0}$.
	This shows that
	$2^{j_0/2}\tilde{\psi}^L(2^{j_0}\cdot)$ must be the unique biorthogonal element in $L_2([0,\infty))$ corresponding to $2^{j_0/2}\mathring{\psi}^L(2^{j_0}\cdot)\in \AS_0(\Phi;\Psi)_{[0,\infty)}$.
	In summary, we proved $\tilde{\cB}=\AS_0(\tilde{\Phi};\tilde{\Psi})_{[0,\infty)}$.
	Note that $\#\tilde{\phi}^L=\#\mathring{\phi}^L
	 =\#\phi^L+(n_{\tilde{\phi}}-n_\phi)(\#\phi)$ and $\#\tilde{\psi}^L=\#\psi^L
	 =\#\psi^L+(n_{\tilde{\psi}}-n_\psi)(\#\psi)$. Therefore, \eqref{cardinality} holds. To complete the proof of item (1) in \cref{thm:wbd:0},
	next we prove that both $\tilde{\phi}^L$ and $\tilde{\psi}^L$ must have compact support. Since $(\{\tilde{\phi};\tilde{\psi}\},\{\phi;\psi\})$ is a biorthogonal wavelet in $\Lp{2}$, we can expand $\tilde{\phi}^L$ in \eqref{expr} with $J=0$ as follows:
	\[
	\tilde{\phi}^L=\sum_{k\in \Z}
	\la \tilde{\phi}^L, \phi(\cdot-k)\ra \tilde{\phi}(\cdot-k)+\sum_{j=0}^\infty \sum_{k\in \Z} \la \tilde{\phi}^L,\psi_{j;k}\ra\tilde{\psi}_{j;k}.
	\]
	Since $\tilde{\phi}^L$ is perpendicular to all elements in $\AS_0(\Phi;\Psi)_{[0,\infty)}\bs \{\mathring{\phi}^L\}$, we see from the above identity that
	\be \label{phiL:refstr}
	 \tilde{\phi}^L=\sum_{k=-\infty}^{n_{\tilde{\phi}}-1}
	\la \tilde{\phi}^L, \phi(\cdot-k)\ra \tilde{\phi}(\cdot-k)+\sum_{j=0}^\infty \sum_{k=-\infty}^{n_{\tilde{\psi}}-1} \la \tilde{\phi}^L,\psi_{j;k}\ra\tilde{\psi}_{j;k},
	\ee
	from which we deduce that $\tilde{\phi}^L$ must be supported inside $(-\infty,M]$ with $M:=\max(n_{\tilde{\phi}}+h_{\tilde{\phi}}, n_{\tilde{\psi}}+h_{\tilde{\psi}})$. Because $\tilde{\phi}^L$ lies in $L_2([0,\infty))$ and hence is supported inside $[0,\infty)$, we conclude that $\tilde{\phi}^L$ must have compact support with $\fs(\tilde{\phi}^L)\subseteq [0,M]$. By the same argument, we can prove that \eqref{phiL:refstr} holds with $\tilde{\phi}^L$ being replaced by $\tilde{\psi}^L$ and hence
	$\tilde{\psi}^L$ also has compact support. This proves item (1).
	
	We now prove item (2).
	Since $\AS_0(\Phi;\Psi)_{[0,\infty)}$ and $\AS_0(\tilde{\Phi};\tilde{\Psi})_{[0,\infty)}$ form a pair of biorthogonal Riesz bases in $L_2([0,\infty))$, by a simple scaling argument (e.g., see \cite[Proposition~4 and (2.6)]{han12} and \cite[Theorem 4.3.3]{hanbook}), it is straightforward to verify that $\AS_J(\Phi;\Psi)_{[0,\infty)}$ and $\AS_J(\tilde{\Phi};\tilde{\Psi})_{[0,\infty)}$ form a pair of biorthogonal Riesz bases in $L_2([0,\infty))$ for every $J\in \Z$. Expanding $\tilde{\phi}^L$ under the biorthogonal basis formed by $\AS_1(\tilde{\Phi};\tilde{\Psi})_{[0,\infty)}$ and $\AS_1(\Phi;\Psi)_{[0,\infty)}$,
	we have
	\[
	\tilde{\phi}^L=\sum_{h\in \AS_1(\Phi;\Psi)_{[0,\infty)}} \la \tilde{\phi}^L, h\ra \tilde{h}=
	2 \la \tilde{\phi}^L, \mathring{\phi}^L(2\cdot)\ra  \tilde{\phi}^L(2\cdot)+
	2\sum_{k=n_{\tilde{\phi}}}^\infty \la \tilde{\phi}^L, \phi(2\cdot-k)\ra \tilde{\phi}(2\cdot-k),
	\]
	since $\la \tilde{\phi}^{L}, h\ra=0$ for all $h \in \Psi(2^j\cdot)$ with $j\ge 0$. Hence, \eqref{I:phi:dual} holds with $\tilde{A}_L:=\la \tilde{\phi}^L, \mathring{\phi}^L(2\cdot)\ra$ and $\tilde{A}(k):=\la \tilde{\phi}^L, \phi(2\cdot-k)\ra$ for $k\ge n_{\tilde{\phi}}$. Since both $\tilde{\phi}^L$ and $\phi$ have compact support, the sequence $A$ must be finitely supported.
	The identity
	in \eqref{I:psi:dual} can be proved similarly by expanding $\tilde{\psi}^L$ instead of $\tilde{\phi}^L$, under the biorthogonal basis formed by $\AS_1(\tilde{\Phi};\tilde{\Psi})_{[0,\infty)}$ and $\AS_1(\Phi;\Psi)_{[0,\infty)}$.
	Using the same argument as in the proof of \eqref{nphi},
	we see that
	the identities in \eqref{I:tphi} and \eqref{I:tpsi} follow directly from
	\eqref{refstr:dual} and the assumption that $n_{\tilde{\phi}}\ge \max(-l_{\tilde{\phi}},-l_a)$ and $n_{\tilde{\psi}}\ge \max(-l_{\tilde{\psi}}, \frac{n_{\tilde{\phi}}-l_{\tilde{b}}}{2})$.
	This proves item (2).
	
	To prove item (3), since $(\AS_0(\tilde{\Phi};\tilde{\Psi})_{[0,\infty)},
	\AS_0(\Phi;\Psi)_{[0,\infty)})$ is a pair of biorthogonal Riesz bases in $L_2([0,\infty))$, noting that
	$\Phi(2\cdot) \perp \tilde{\Psi}(2^j\cdot)$ for all $j\ge 1$, for $\eta\in \Phi(2\cdot)$ we have
	\[
	\eta=\la \eta, \tilde{\phi}^L\ra \mathring{\phi}^L+\sum_{k=n_{\tilde{\phi}}}^\infty
	\la \eta, \tilde{\phi}(\cdot-k)\ra \phi(\cdot-k)+
	\la \eta, \tilde{\psi}^L\ra \mathring{\psi}^L+\sum_{k=n_{\tilde{\psi}}}^\infty
	\la \eta, \tilde{\psi}(\cdot-k)\ra \psi(\cdot-k).
	\]
	Since all functions in $\Phi\cup \Psi\cup\tilde{\Phi}\cup\tilde{\Psi}$ have compact support and $\Phi\cup \Psi$ is a Riesz sequence,
	we conclude from the above identity that item~(3) holds.
\end{proof}

\begin{proof}[Proof of Theorem \ref{thm:wbd}]
	By assumption $\phi^L\cup \psi^L\subseteq \HH{\tau}$ for some $\tau>0$, since $\AS_0(\phi;\psi)$ is a Bessel sequence in $\Lp{2}$,  $\psi$ must have at least one vanishing moment and we conclude from \cref{thm:phi:bessel} that $\AS_J(\Phi;\Psi)_{[0,\infty)}$ is a Bessel sequence in $L_2([0,\infty))$. Similarly, by $\tilde{\phi}^L\cup \tilde{\psi}^L\subseteq \HH{\tau}$ for some $\tau>0$, we conclude from \cref{thm:phi:bessel} that $\AS_J(\tilde{\Phi};\tilde{\Psi})_{[0,\infty)}$ is a Bessel sequence in $L_2([0,\infty))$.
	Now the rest of the argument is quite standard for proving that $\AS_J(\tilde{\Phi};\tilde{\Psi})_{[0,\infty)}$ and $\AS_J(\Phi;\Psi)_{[0,\infty)}$
	form a pair of biorthogonal Riesz bases in $L_2([0,\infty))$.
	By scaling, it suffices to prove the claim for $J=0$. Define $S_{j_0}$ as in \eqref{Sj0} for $j_0\in \N$. Using items (i)--(iv) and the same argument as in the proof of \cref{thm:wbd:0}, we see that every element in $S_{j_0}$ is a finite linear combination of $\Phi(2^{j_0}\cdot):=\{ \phi^L(2^{j_0}\cdot)\}\cup \{ \phi(2^{j_0}\cdot-k) \setsp k\ge n_\phi\}$.
	Now by the biorthogonality between $\Phi\cup \Psi$ and $\tilde{\Phi}\cup\tilde{\Psi}$,
	it follows from
	the same argument as in the proof of \cref{thm:wbd:0} that $\AS_J(\Phi;\Psi)_{[0,\infty)}$ and $\AS_J(\tilde{\Phi};\tilde{\Psi})_{[0,\infty)}$ must be biorthogonal to each other in $L_2([0,\infty))$.
	Consequently, by the standard argument (e.g., see the proof of (4)$\imply$(1) in \cref{thm:rz}), we conclude that both $\AS_J(\Phi;\Psi)_{[0,\infty)}$ and $\AS_J(\tilde{\Phi};\tilde{\Psi})_{[0,\infty)}$
	are Riesz sequences in $L_2([0,\infty))$.
	By item (iii), we have $\si_0(\Psi)=\si_1(\Phi)\cap (\si_0(\tilde{\Phi}))^\perp$. For any $f\in \si_1(\Phi)$, define $g:=\sum_{\eta\in \Phi} \la f,\tilde{\eta}\ra \eta$. Then $f-g \perp \tilde{\Phi}$ and we conclude that $f=g+(f-g)$ such that $g\in \si_0(\Phi)$ and $f-g\in \si_1(\Phi)\cap (\si_0(\tilde{\Phi}))^\perp=\si_0(\Psi)$. This proves $\si_1(\Phi)\subseteq \si_0(\Phi\cup \Psi)$. Because $\si_0(\Phi\cup \Psi)\subseteq \si_1(\Phi)$ is trivial, we conclude that $\si_0(\Phi\cup \Psi)=\si_1(\Phi)$.
	By the scaling argument, we must have \[
\si_0(\Phi\cup \Psi\cup\Psi(2\cdot)\cup \cdots \cup \Psi(2^{j-1}\cdot))=
	\si_1(\Phi\cup \Psi\cup \cdots \cup \Psi(2^{j-2}\cdot))=\cdots= \si_j(\Phi).
\]
	Hence, $\si_0(\AS_0(\Phi;\Psi)_{[0,\infty)})$ contains $\cup_{j=1}^\infty \si_j(\Phi)$, which includes the subset $\cup_{j=1}^\infty \{ \phi(2^j\cdot-k) \setsp k\ge n_\phi\}$, whose linear span is dense in $L_2([0,\infty))$ due to $\lim_{j_0\to \infty}
	2^{-j_0} n_\phi=0$. Therefore, the linear span of $\AS_0(\Phi;\Psi)_{[0,\infty)}$ is dense in $L_2([0,\infty))$. This proves that $\AS_0(\Phi;\Psi)_{[0,\infty)}$ is a Riesz basis of $L_2([0,\infty))$. Similarly, $\AS_0(\tilde{\Phi};\tilde{\Psi})_{[0,\infty)}$ is also a Riesz basis of $L_2([0,\infty))$.
	This completes the proof of \cref{thm:wbd}.
\end{proof}

\begin{proof}[Proof of \cref{thm:integral}]
	Define $\phi_j:=\phi(\cdot-j)\chi_{[0,1]}$ for $j=1-h_\phi,\ldots,-l_\phi$ (for other $j\in \Z$, $\phi_j$ is identically zero).
	By $\phi=2\sum_{k\in \Z} a(k) \phi(2\cdot-k)$,
	for $j\in \Z$, we have $\phi(\cdot-j)=
	2\sum_{k\in \Z} a(k-2j) \phi(2\cdot-k)$.
	Multiplying $\chi_{[0,1]}$ on both sides of this identity,
	we particularly have
	\[
	\phi_j(x)=\phi(x-j)\chi_{[0,1]}(x)=
	2\sum_{k\in \Z} a(k-2j) \phi(2x-k)\chi_{[0,1]}(x).
	\]
	Note that
	\begin{align*}
	 &\phi(2x-k)\chi_{[0,1/2]}(x)=\phi(2x-k)\chi_{[0,1]}(2x)=
	 \Big[\phi(\cdot-k)\chi_{[0,1]}(\cdot)\Big](2x)
	=\phi_k(2x),\\
	 &\phi(2x-k)\chi_{[1/2,1]}(x)=\phi(2x-k)\chi_{[0,1]}(2x-1)=
	\Big[ \phi(\cdot+1-k)\chi_{[0,1]}(\cdot)\Big](2x-1)
	=\phi_{k-1}(2x-1).
	\end{align*}
	Hence, we have
	\begin{align*}
	\phi_j(x)&=
	2\sum_{k\in \Z} a(k-2j) \phi(2x-k)\chi_{[0,1]}(x)
	=2\sum_{k\in \Z} a(k-2j) [\phi_k(2x)+\phi_{k-1}(2x-1)]\\
	&=2\sum_{k=1-h_\phi}^{-l_\phi} a(k-2j)\phi_k (2x)
	+2\sum_{k=1-h_\phi}^{-l_\phi} a(k+1-2j)\phi_k(2x-1),
	\end{align*}
	since all $\phi_k$ for $k\in \Z \bs [1-h_\phi,-l_\phi]$ are identically zero.
	By $\vec{\phi}=(\phi_{1-h_\phi},\ldots,\phi_{-l_\phi})^\tp$, this proves the first identity in \eqref{vecphi:refeq}. The proof of the second identity in \eqref{vecphi:refeq} is similar.
	
	We now prove \eqref{Mphi}. Noting that $\vec{\phi}(2x)$ is supported inside $[0,1/2]$ and $\vec{\phi}(2x-1)$ is supported inside $[1/2,1]$, we deduce from \eqref{vecphi:refeq} that
	\begin{align*}
	M&=\int_0^1 \vec{\tilde{\phi}}(x)\ol{ \vec{\phi}(x)}^\tp dx\\
	&=4 \int_0^1 \Big( \tilde{A}_0 \vec{\tilde{\phi}}(2x)+\tilde{A}_1 \vec{\tilde{\phi}}(2x-1)\Big)
	\Big( \ol{\vec{\phi}(2x)}^\tp \ol{A_0}^\tp
	+\ol{\vec{\tilde{\phi}}(2x-1)}^\tp \ol{A_1}^\tp\Big) dx\\
	&=4 \tilde{A}_0 \int_0^1 \vec{\tilde{\phi}}(2x) \ol{\vec{\phi}(2x)}^\tp dx \ol{A_0}^\tp+4\tilde{A}_1 \int_0^1 \vec{\tilde{\phi}}(2x-1) \ol{\vec{\phi}(2x-1)}^\tp dx \ol{A_1}^\tp\\
	&=2\tilde{A}_0 \int_0^1 \vec{\tilde{\phi}}(x) \ol{\vec{\phi}(x)}^\tp dx \ol{A_0}^\tp+2\tilde{A}_1 \int_0^1 \vec{\tilde{\phi}}(x) \ol{\vec{\phi}(x)}^\tp dx \ol{A_1}^\tp\\
	&=2\tilde{A}_0 M \ol{A_0}^\tp+
	2\tilde{A}_1 M \ol{A_1}^\tp.
	\end{align*}
	This proves \eqref{Mphi}. We now prove that up to a multiplicative constant \eqref{Mphi} has a unique solution. By $\mbox{vec}(M)$ we denote the column vector by arranging the columns of $M$ one by one.
	Then \eqref{Mphi} is equivalent to
	\be \label{Mphi:2}
	T\mbox{vec}(M)=\mbox{vec}(M)\quad \mbox{and}\quad
	T:=2(\ol{A_0}\otimes \tilde{A}_0+\ol{A_1}\otimes \tilde{A}_1),
	\ee
	where $\otimes$ stands for the right Kronecker product of matrices.
	To prove (S3), it suffices to prove that $1$ is a simple eigenvalue of the matrix $T$ in \eqref{Mphi:2}.
	Define $s:=\#\vec{\phi}$ and $\tilde{s}:=\#\vec{\tilde{\phi}}$.
	Note that $\vec{\phi}$ is a compactly supported refinable vector function in $\Lp{2}$ and the integer shifts of $\vec{\phi}$ are linearly independent. By \cite[Corollary~5.6.12 and Proposition~5.6.2]{hanbook}, we conclude that $1$ must be a simple eigenvalue of $A_0+A_1$ and the mask/filter associated with $\vec{\phi}$ must have at least order one sum rule, that is, by \eqref{sr},
	there must exist a nontrivial row vector $\vec{v}\in \C^{1\times s}$ such that
	 $\vec{v}A_0=\vec{v}A_1=\frac{1}{2}\vec{v}$ and
	$\vec{v}\wh{\vec{\phi}}(0)=1$.
	Since $1$ is a simple eigenvalue of $A_0+A_1$, such row vector $\vec{v}$ must be unique by $\vec{v}(A_0+A_1)=\vec{v}$.
	Similarly, there exists a unique row vector $\vec{\tilde{v}}\in \C^{1\times \tilde{s}}$ such that
	 $\vec{\tilde{v}}\tilde{A}_0=\vec{\tilde{v}}\tilde{A}_1=\frac{1}{2}\vec{\tilde{v}}$ and $\vec{\tilde{v}}\wh{\vec{\tilde{\phi}}}(0)=1$.
	By $\vec{v}A_0=\vec{v}A_1=\frac{1}{2}\vec{v}$ and $\vec{\tilde{v}}\tilde{A}_0=\vec{\tilde{v}}\tilde{A}_1=\frac{1}{2}\vec{\tilde{v}}$, we trivially have
	\begin{align*}
	(\ol{\vec{v}} \otimes \vec{\tilde{v}})T
	&=(\ol{\vec{v}} \otimes \vec{\tilde{v}}) 2(\ol{A_0}\otimes \tilde{A}_0+\ol{A_1}\otimes \tilde{A}_1)
	=2(\ol{\vec{v}A_0})\otimes (\vec{\tilde{v}}\tilde{A}_0)+
	2(\ol{\vec{v}A_1})\otimes (\vec{\tilde{v}} \tilde{A}_1)\\
	&=\frac{1}{2}(\ol{\vec{v}} \otimes \vec{\tilde{v}})+\frac{1}{2}(\ol{\vec{v}} \otimes \vec{\tilde{v}})
	=\ol{\vec{v}} \otimes \vec{\tilde{v}}.
	\end{align*}
	Hence $1$ must be an eigenvalue of $T$.
	Next we prove that the eigenvalue $1$ of $T$ has multiplicity one by employing the joint spectral radius technique in \cite{hj98}.
	Let $U$ be the space of all column vectors $u\in \C^s$ such that $\vec{v} u=0$. By $\vec{v}A_0=\vec{v}A_1=\frac{1}{2}\vec{v}$, it is trivial to observe that $A_0 U\subseteq U$ and $A_1 U\subseteq U$. Since all the entries in $\vec{\phi}$ are compactly supported functions in $\Lp{2}$ and the integer shifts of $\vec{\phi}$ are linearly independent, we must have (e.g., see \cite[Theorems~5.6.11 and~5.7.4]{hanbook}. Also c.f. \cite[Theorem~3.3]{hj98}) that
	\be \label{sm:tz}
	\lim_{n\to \infty} 2^{n/2} \|\{A_0,A_1\}^n u\|_{l_2}=0, \qquad \forall \, u\in U
	\ee
	and for every $w\in \C^s$, there exists a positive constant $C_w$ such that
	\be \label{sm:tz:2}
	2^{n/2}\|\{A_0,A_1\}^n w\|_{l_2}\le C_w,\qquad \forall\, n\in \N,
	\ee
	where as in \cite[Section~2]{hj98} we define
	\[
	\| \{A_0,A_1\}^n u\|^2_{l_2}:=\sum_{\gamma_1=0}^1\cdots\sum_{\gamma_n=0}^1
	\| A_{\gamma_1}\cdots A_{\gamma_n} u\|^2.
	\]
	Similar conclusions hold for $\tilde{A}_0$ and $\tilde{A}_1$.
	Take particular vectors $w:=\wh{\vec{\phi}}(0)$ and
	 $\tilde{w}:=\wh{\vec{\tilde{\phi}}}(0)$.
	Hence, we must have $\vec{v} w=1$ and $\vec{\tilde{v}}\tilde{w}=1$.
	Now considering $T^n (\ol{u}\otimes \tilde{w})$ with $u\in U$ and using the Cauchy-Schwarz inequality, we conclude that
	\[
	\|T^n (\ol{u}\otimes \tilde{w})\|
	\le
	2^n \| \{A_0, A_1\} u \|_{l_2}
	\| \{\tilde{A}_0, \tilde{A_1}\} \tilde{w}\|_{l_2}
	\le C_{\tilde{w}} 2^{n/2} \| \{A_0, A_1\} u \|_{l_2} \to 0
	\]
	as $n\to \infty$. Similarly, for $\tilde{u}\in \tilde{U}$, as $n\to \infty$, we have
	\[
	\|T^n (\ol{w}\otimes \tilde{u})\|
	\le 2^n \| \{A_0, A_1\} w \|_{l_2} \| \{\tilde{A}_0, \tilde{A_1}\} \tilde{u}\|_{l_2}
	\le C_w 2^{n/2} \| \{\tilde{A}_0,\tilde{A}_1\} \tilde{u} \|_{l_2} \to 0.
	\]
	Also, for all $u\in U$ and $\tilde{u}\in \tilde{U}$, we similarly have $\lim_{n\to \infty}
	\|T^n (\ol{u}\otimes \tilde{u})\|=0$.
	Note that $w$ and $U$ span the whole space $\C^s$ while $\tilde{w}$ and $\tilde{U}$ span $\C^{\tilde{s}}$, where $s:=\#\vec{\phi}$ and $\tilde{s}:=\#\vec{\tilde{\phi}}$.
	The above three identities prove that all the other eigenvalues of $T$ must be less than one in modulus. Hence, $1$ is a simple eigenvalue of $T$. Hence, up to a multiplicative constant, $M$ is the unique solution to \eqref{Mphi}.
	
	Note that $\sum_{k\in \Z} \vec{v}\vec{\phi}(\cdot-k)=1$ and
	$\sum_{k\in \Z} \vec{\tilde{v}}\vec{\tilde{\phi}}(\cdot-k)=1$.
	Since both $\vec{\phi}$ and $\vec{\tilde{\phi}}$ are supported inside $[0,1]$, we must have $\vec{v}\vec{\phi}(x)=1$ and $\vec{\tilde{v}}\vec{\tilde{\phi}}(x)=1$ for almost every $x\in [0,1]$ and hence
	\[
	\vec{\tilde{v}} M \ol{\vec{v}}^\tp=\int_0^1 \vec{\tilde{v}} \vec{\tilde{\phi}}(x)
	\ol{ \vec{v}\vec{\phi}(x)}^\tp dx=
	\int_0^1 1 dx
	=1.
	\]
	This proves \eqref{Mphi:normalize} and completes the proof.
\end{proof}

\begin{proof}[Proof of \cref{thm:direct}]
	By $m_\phi=\max(2n_\phi+h_{\tilde{a}},2n_\psi+h_{\tilde{b}})$, we have $m_\phi\ge 2n_\phi+h_{\tilde{a}}\ge n_\phi$ since $n_\phi\ge -l_a\ge -h_{\tilde{a}}$.
	By the definition of $n_{\tilde{\phi}}$ and $n_{\tilde{\psi}}$, we trivially have $n_{\tilde{\phi}}\ge \max(-l_{\tilde{\phi}}, -l_{\tilde{a}})$ and $n_{\tilde{\psi}}\ge \max(-l_{\tilde{\psi}}, \frac{n_{\tilde{\phi}}-l_{\tilde{b}}}{2})$. Therefore, \eqref{I:tphi} and \eqref{I:tpsi} must hold.
	We now prove \eqref{phi2k0:2}.
	By the definition of $n_{\tilde{\phi}}$ and $n_{\tilde{\psi}}$,
	we also have $n_{\tilde{\phi}}\ge 1-l_{\tilde{a}}$ and $n_{\tilde{\psi}}\ge \lceil \frac{n_{\tilde{\phi}}
		-l_{\tilde{b}}+1}{2}\rceil$.
	Hence, we have $\frac{k_0-l_{\tilde{a}}}{2}\le n_{\tilde{\phi}}-1$ and $\frac{k_0-l_{\tilde{b}}}{2}\le n_{\tilde{\psi}}-1$ for all $k_0<n_{\tilde{\phi}}$. Therefore, for all $k_0\in \Z$ satisfying $m_\phi\le k_0<n_{\tilde{\phi}}$, it follows from \eqref{phi2k0} and \cref{lem:W} that
	\eqref{phi2k0:2} must hold.
	
	Define (infinite) column vector functions by
	\[
	\vec{\phi}:=\{ \phi(\cdot-k) \setsp k\ge n_{\tilde{\phi}}\}
	\quad \mbox{and}\quad
	\vec{\psi}:=\{\psi(\cdot-k) \setsp k\ge n_{\tilde{\psi}}\}.
	\]
	Abusing notations a little bit by using the same notations for augmented $A_L, B_L$ and $A, B$ with $\phi^L,\psi^L$ being replaced by $\mathring{\phi}^L,\mathring{\psi}^L$, respectively, we can equivalently rewrite \eqref{nphi}, \eqref{npsi}, \eqref{I:phi} and \eqref{I:psi} as
	\be \label{M}
	{{\begin{bmatrix}
				\mathring{\phi}^L\\
				\vec{\phi}\\
				\mathring{\psi}^L\\
				\vec{\psi}
	\end{bmatrix}}}
	=2 \mathcal{M} \begin{bmatrix} \mathring{\phi}^L(2\cdot)\\
		\vec{\phi}(2\cdot)\end{bmatrix}
	\quad \mbox{with}\quad
	\mathcal{M}:=
	{{\begin{bmatrix}
				A_L &M_A\\
				0 &M_a\\
				B_L &M_B\\
				0 &M_b
	\end{bmatrix}}},
	\ee
	where $M_A, M_B, M_{a}, M_b$ are matrices associated with filters $A, B, a, b$, respectively.
	More precisely, using \eqref{I:phi}, \eqref{I:psi}, \eqref{nphi} and \eqref{npsi},
	we have
	$M_A:=[A(k)]_{n_{\tilde{\phi}}\le k<\infty}$ (which is equivalent to $M_A \vec{\phi}=\sum_{k=n_{\tilde{\phi}}}^\infty A(k) \phi(\cdot-k)$),
	$M_B:=[B(k)]_{n_{\tilde{\phi}}\le k<\infty}$, and
	\be \label{Mab}
	M_a:=[a(k-2k_0)]_{n_{\tilde{\phi}}\le k_0<\infty, n_{\tilde{\phi}}\le k<\infty}
	\quad\mbox{and}\quad M_b:=[b(k-2k_0)]_{n_{\tilde{\psi}}\le k_0<\infty, n_{\tilde{\phi}}\le k<\infty},
	\ee
	where $k_0$ is row index and $k$ is column index.
	By $n_{\tilde{\phi}}\ge m_\phi$,
	we see from \cref{lem:W} and \eqref{phi2k0} that
	\be \label{vec:phi}
	\vec{\phi}(2\cdot)=
	\ol{M_{\tilde{A}}}^\tp \mathring{\phi}^L+\ol{M_{\tilde{a}}}^\tp \vec{\phi}+\ol{M_{\tilde{B}}}^\tp \mathring{\psi}^L+\ol{M_{\tilde{b}}}^\tp \vec{\psi},
	\ee
	where $M_{\tilde{A}}, M_{\tilde{B}}, M_{\tilde{a}}, M_{\tilde{b}}$ are matrices uniquely determined by the filters $\tilde{a}$ and $\tilde{b}$.
	More precisely,
	\be \label{tMAB}
	M_{\tilde{A}}:=
	\begin{bmatrix}
		0_{(\#\phi^L)\times \infty}\\
		(\tilde{a}(k-2k_0))_{n_\phi\le k_0<n_{\tilde{\phi}},
			n_{\tilde{\phi}}\le k<\infty}
	\end{bmatrix},
	\qquad
	M_{\tilde{B}}:=
	\begin{bmatrix}
		0_{(\#\psi^L)\times \infty}\\
		(\tilde{b}(k-2k_0))_{n_\psi\le k_0<n_{\tilde{\psi}},
			n_{\tilde{\phi}}\le k<\infty}
	\end{bmatrix},
	\ee
	and $M_{\tilde{a}}$, $M_{\tilde{b}}$ are defined similarly as in
	\eqref{Mab} using $\tilde{a}$ and $\tilde{b}$ instead of $a$ and $b$, where $k_0$ is row index and $k$ is column index.
	Therefore, we deduce from \eqref{inv:phiL:2} and the above identity in \eqref{vec:phi} that
	\be \label{tM}
	\begin{bmatrix}
		\mathring{\phi}^L(2\cdot)\\
		\vec{\phi}(2\cdot)\end{bmatrix}
	=\ol{\tilde{\mathcal{M}}}^\tp
	{{ \begin{bmatrix}
				\mathring{\phi}^L\\
				\vec{\phi}\\
				\mathring{\psi}^L\\
				\vec{\psi}
	\end{bmatrix}}}
	\quad \mbox{with}\quad
	\tilde{\mathcal{M}}:=
	{{\begin{bmatrix}
				\tilde{A}_L &M_{\tilde{A}}\\
				0 &M_{\tilde{a}}\\
				\tilde{B}_L &M_{\tilde{B}}\\
				0 &M_{\tilde{b}}
	\end{bmatrix}}}.
	\ee
	By assumption, $\Phi$ is a Riesz sequence in $L_2([0,\infty))$ and hence linearly independent.
	By item (i), the elements in $\Phi\cup \Psi$ must be linearly independent.
	Consequently, we deduce from \eqref{M} and \eqref{tM}
	that $\mathcal{\tilde{M}} \ol{\mathcal{M}}^\tp=2^{-1}I$ and $\ol{\mathcal{M}}^\tp\tilde{\mathcal{M}} =2^{-1}I$, where $I$ here stands for the infinite identity matrix.
	
	We now prove that $\tilde{A}$ and $\tilde{B}$ in \eqref{tAB} are finitely supported.
	For all $k\ge 2n_{\tilde{\phi}}+h_{\tilde{a}}$, we have $k-2k_0>h_{\tilde{a}}$ for all $k_0=n_\phi,\ldots,n_{\tilde{\phi}}-1$ and hence $\tilde{a}(k-2k_0)=0$.
	So, $\{\tilde{A}(k)\}_{k= n_{\tilde{\phi}}}^\infty$ in \eqref{tAB} is finitely supported.
	Similarly, for all $k\ge 2n_{\tilde{\psi}}+h_{\tilde{b}}$, we have $k-2k_0>h_{\tilde{b}}$ for all $k_0=n_{\psi},\ldots,n_{\tilde{\psi}}-1$ and hence $\tilde{b}(k-2k_0)=0$. So, $\{\tilde{B}(k)\}_{k= n_{\tilde{\phi}}}^\infty$ in \eqref{tAB} is finitely supported.
	Since $\rho(\tilde{A}_L)<2^{-1/2}$ in \eqref{tAL}, we conclude from \cref{thm:Phi:direct} that $\tilde{\phi}^L$ in \eqref{tphi:implicit} is a well-defined compactly supported vector function in $L_2([0,\infty))\cap \HH{\tau}$ for some $\tau>0$ and satisfies \eqref{I:phi:dual}.
	Since $\tilde{B}$ is finitely supported, $\tilde{\psi}^L$ in \eqref{tpsi:implicit} is a well-defined compactly supported vector function in
	$L_2([0,\infty))\cap \HH{\tau}$ and satisfies \eqref{I:psi:dual}.
	
	We now prove that $\tilde{\Phi}$ must be biorthogonal to $\Phi$.
	Define $\vec{\tilde{\phi}}:=\{\tilde{\phi}(\cdot-k) \setsp k\ge n_{\tilde{\phi}}\}$.
	By \eqref{tMAB},
	we have $M_{\tilde{A}} \vec{\tilde{\phi}}
	=\sum_{k=n_{\tilde{\phi}}}^\infty \tilde{A}(k) \tilde{\phi}(\cdot-k)$.
	Since we assumed that $\Phi$ is a Riesz sequence,
	we conclude from \cref{thm:rz}  that item (4) of \cref{thm:rz} holds. If necessary, enlarging $n_{\tilde{\phi}}$, then we can assume that $\tilde{H}=\tilde{\eta}^L \cup \vec{\tilde{\phi}}$ in item (4) of \cref{thm:rz} is biorthogonal to $\Phi$ with $\#\tilde{\eta}^L=\#\mathring{\phi}^L$.
	Define $f_0:=\mathring{\phi}^L$ and $\tilde{f}_0:=\tilde{\eta}^L$.
	For $n\in \N$, we can recursively define
	\be \label{fn}
	f_n:=2A_L f_{n-1}(2\cdot)+g(2\cdot) \quad \mbox{and}\quad \tilde{f}_n:=2\tilde{A}_L \tilde{f}_{n-1}(2\cdot)+\tilde{g}(2\cdot),\qquad n\in \N,
	\ee
	where $g:=2\sum_{k=n_{\tilde{\phi}}}^\infty A(k) \phi(\cdot-k)$, $\tilde{g}$ is given in \eqref{tphi:implicit}, and $A_L$ and $A$ are augmented version in \eqref{I:phi}.
	Let $F_n:=f_n\cup \vec{\phi}$ and $\tilde{F}_n:=\tilde{f}_n\cup \vec{\tilde{\phi}}$. By the choice of $\tilde{f}_0=\tilde{\eta}^L$  and $f_0=\mathring{\phi}^L$, we have
	$\tilde{F}_0=\tilde{H}$ and
	$F_0=\Phi$. Therefore,
	$\tilde{F}_0$ is biorthogonal to $F_0$ by \cref{thm:rz}. Suppose that $\tilde{F}_{n-1}$ is biorthogonal to $F_{n-1}$ (induction hypothesis), i.e., $\la \tilde{F}_{n-1},F_{n-1}\ra=I$.
	We now prove the claim for $n$. Note that
	\be \label{NtN}
	F_n=2\mathcal{N} F_{n-1}(2\cdot),\quad
	\tilde{F}_n=2\tilde{\mathcal{N}} \tilde{F}_{n-1}(2\cdot)
	\quad \mbox{with}\quad
	\mathcal{N}:={{ \begin{bmatrix} A_L &M_A\\
				0 &M_a\end{bmatrix}}},
	\quad
	\tilde{\mathcal{N}}:={{ \begin{bmatrix} \tilde{A}_L &M_{\tilde{A}}\\
				0 &M_{\tilde{a}}\end{bmatrix}}}.
	\ee
	It follows trivially from the identity
	 $\tilde{\mathcal{M}}\ol{\mathcal{M}}^\tp=2^{-1}I$
	that $\tilde{\mathcal{N}}\ol{\mathcal{N}}^\tp=2^{-1}I$.
	Therefore, by induction hypothesis $\la \tilde{F}_{n-1},F_{n-1}\ra=I$, we have
	\[
	\la \tilde{F}_n,F_n\ra
	=4 \tilde{\mathcal{N}} \la \tilde{F}_{n-1}(2\cdot), F_{n-1}(2\cdot) \ra \ol{\mathcal{N}}^\tp
	=2 \tilde{\mathcal{N}}\ol{\mathcal{N}}^\tp=I.
	\]
	This proves the claim for $n$.
	By mathematical induction,
	we proved that $\tilde{F}_n$ is biorthogonal to $F_n$ for all $n\in \N$.
	By \eqref{I:phi:dual} and \eqref{fn} with $f_0=\mathring{\phi}^L$,
	we trivially have $f_n=\mathring{\phi}^L$ and hence, $F_n=\Phi$ for all $n\in \N$.
	So, $\tilde{F}_n$ is biorthogonal to $\Phi$ for all $n\in \N$. We deduce from the definition of $\tilde{f}_n$ in \eqref{fn} that
	\[
	\tilde{f}_n=2^{n} \tilde{A}_L^n \tilde{f}_0(2^{n}\cdot)+
	\sum_{j=1}^{n} 2^{j-1} \tilde{A}_L^{j-1} g(2^{j}\cdot).
	\]
	Since $\rho(\tilde{A}_L)<2^{-1/2}$ and $\|2^{n} \tilde{A}_L^n \tilde{f}_0(2^{n}\cdot)\|_{\Lp{2}}
	\le \|\tilde{f}_0\|_{\Lp{2}} 2^{n/2} \|\tilde{A}_L^n\|$, we conclude that
\[
\lim_{n\to \infty} \|2^{n} \tilde{A}_L^n \tilde{f}_0(2^{n}\cdot)\|_{\Lp{2}}=0.
\]
This proves $\lim_{n\to \infty} \|\tilde{f}_n-\tilde{\phi}^L\|_{\Lp{2}}=0$.
	Since $\tilde{\Phi}=\lim_{n\to \infty} \tilde{F_n}$ in $\Lp{2}$ and $\tilde{F}_n$ is biorthogonal to $\Phi$, we conclude that $\tilde{\Phi}$ must be biorthogonal to $\Phi$.
	
	By \eqref{tMAB},
	we have $\tilde{\psi}^L=2\tilde{B}_L\tilde{\phi}^L(2\cdot)+
	2M_{\tilde{B}} \vec{\tilde{\phi}}(2\cdot)$.
	Also note that
	$\tilde{\phi}^L=2\tilde{A}_L \tilde{\phi}^L(2\cdot)+2M_{\tilde{A}}\vec{\tilde{\phi}}(2\cdot)$, $\vec{\tilde{\phi}}=2M_{\tilde{a}} \vec{\tilde{\phi}}(2\cdot)$, and
	 $\vec{\tilde{\psi}}=2M_{\tilde{b}}\vec{\tilde{\phi}}(2\cdot)$ with $\vec{\tilde{\psi}}:=\{\tilde{\psi}(\cdot-k) \setsp k\ge n_{\tilde{\psi}}\}$.
	Using the identities $\tilde{\mathcal{M}} \ol{\mathcal{M}}^\tp=2^{-1}I$ and $\ol{\mathcal{M}}^\tp \tilde{\mathcal{M}}=2^{-1}I$, we can now check that all items (i)--(iv) of \cref{thm:wbd} are satisfied.
	Because we assumed $\phi^L\subseteq \HH{\tau}$ for some $\tau>0$, by the choice of $\psi^L$ in item (i), we must have $\psi^L\subseteq \HH{\tau}$. Note that we already proved $\tilde{\phi}^L\cup \tilde{\psi}^L \subseteq \HH{\tau}$ for some $\tau>0$.
	Since all the conditions in \cref{thm:wbd} are satisfied,
	we conclude from \cref{thm:wbd} that
	 $\AS_J(\tilde{\Phi};\tilde{\Psi})_{[0,\infty)}$ and $\AS_J(\Phi;\Psi)_{[0,\infty)}$ form a pair of biorthogonal Riesz bases in $L_2([0,\infty))$ for every $J\in \Z$.
\end{proof}

\begin{proof}[Proof of \cref{thm:direct:tPhi}]
	By the choice of $n_{\tilde{\phi}}$ in item (S1) of \cref{alg:tPhi}, for all $k\ge n_{\tilde{\phi}}$, we have $\la \tilde{\phi}(\cdot-k),\phi(\cdot-k)\ra=I_r$ and $\la \tilde{\phi}(\cdot-k), \eta\ra=0$ for all $\eta\in \Phi \bs \{\phi(\cdot-k)\}$. Define $\vec{\phi}:=\{\phi(\cdot-k) \setsp k\ge n_{\tilde{\phi}}\}$ and $\vec{\tilde{\phi}}:=\{\tilde{\phi}(\cdot-k) \setsp k\ge n_{\tilde{\phi}}\}$. Then \eqref{I:phi} and \eqref{I:phi:dual} become
	\[
	\mathring{\phi}^L=2A_L \mathring{\phi}^L(2\cdot)+2M_A\vec{\phi}(2\cdot),\qquad
	\tilde{\phi}^L=2\tilde{A}_L \tilde{\phi}^L(2\cdot)+2M_{\tilde{A}}
	\vec{\tilde{\phi}}(2\cdot),
	\]
	where $M_A:=[A(k)]_{n_{\tilde{\phi}} \le k<\infty}$ and $M_{\tilde{A}}:=[\tilde{A}(k)]_{n_{\tilde{\phi}}\le k<\infty}$ with $k$ being the column index. Since $n_{\tilde{\phi}}\ge \max(-l_{\tilde{\phi}},-l_{\tilde{a}})$ and $n_{\tilde{\phi}}\ge n_\phi\ge \max(-l_\phi,-l_a)$,
	by \eqref{nphi} and \eqref{I:tphi}, we must have $\vec{\phi}=2M_a \vec{\phi}(2\cdot)$ and $\vec{\tilde{\phi}}=2 M_{\tilde{a}} \vec{\tilde{\phi}}(2\cdot)$, where $M_a$ is defined in \eqref{Mab} and $M_{\tilde{a}}$ is defined similarly.
	Define $\mathcal{N}$ and $\tilde{\mathcal{N}}$ as in \eqref{NtN}.
	Using \eqref{bwfb} and \eqref{filter:biorth}, we must have
	 $\tilde{\mathcal{N}}\ol{\mathcal{N}}^\tp=2^{-1}I$.
	Since $\Phi$ is a Riesz sequence,
	by the same argument as in the proof of \cref{thm:direct}, we conclude that $\tilde{\Phi}$ is biorthogonal to $\Phi$ and $\tilde{\Phi}$ satisfies item (ii) of \cref{thm:wbd}.
\end{proof}

\begin{proof}[Proof of \cref{thm:bw:0N}]
	Let $j\ge J_0$. By assumption in (S3), all the boundary elements in $\Phi_j\cup \Psi_j$ belong to $L_2([0,N])$.
	Since $\fs(\phi(\cdot-k))\subseteq [0,\infty)$ for all $k\ge n_\phi$,
	to show $\Phi_j\subseteq L_2([0,N])$, it suffices to prove that
	$\fs(\phi(\cdot-k))\subseteq (-\infty,2^j N]$ for all $k\le 2^jN-n_{\mathring{\phi}}$, which is equivalent to $\phi(2^j N-\cdot-k)
	\in L_2([0,\infty))$.
	For all $k\le 2^jN-n_{\mathring{\phi}}$, we note that $\phi(2^j N-\cdot-k)
	=\mathring{\phi}(\cdot-(2^j N-k))$
	and $2^j N-k \ge n_{\mathring{\phi}}$.
	By the definition of $n_{\mathring{\phi}}$, we must have
\[
\phi(2^j N-\cdot-k)=\mathring{\phi}(\cdot-(2^j N-k))\subseteq
	L_2([0,\infty)).
\]
This proves $\phi_{j;k}\in L_2([0,N])$ for all $n_\phi\le k\le 2^j N-n_{\mathring{\phi}}$. Consequently, we proved $\Phi_j\subseteq L_2([0,N])$ for all $j\ge J_0$. Similarly, we have $\Psi_j\subseteq L_2([0,N])$ for all $j\ge J_0$.
	
	We now prove item (1).
	By \eqref{I:psi} and $h_B+n_{\mathring{\phi}}\le 2^{j+1}N$ in \eqref{J0}, we have $h_B\le 2^{j+1}N-n_{\mathring{\phi}}$ and
	\[
	\psi^L_{j;0}=\sqrt{2}B_L \phi^L_{j+1;0}+\sqrt{2}\sum_{k=n_\phi}^{h_B} B(k)\phi_{j+1;k}=\sqrt{2}B_L \phi^L_{j+1;0}+\sqrt{2}\sum_{k=n_\phi}^{2^{j+1}N-n_{\mathring{\phi}}} B(k)\phi_{j+1;k}
	\]
	with $B(h_B+1)=\cdots=B(2^{j+1}N-n_{\mathring{\phi}})=0$ due to $[l_B,h_B]=\fs(B)$ and $2^{j+1}N\ge h_B+n_{\mathring{\phi}}$ in \eqref{J0} for all $j\ge J_0$.
	By $n_{\mathring{\psi}}\ge \max(-l_{\mathring{\psi}}, \frac{n_{\mathring{\phi}}-l_{\mathring{b}}}{2})$, we have $l_{\mathring{b}}+2n_{\mathring{\psi}}\ge n_{\mathring{\phi}}$. Since $\mathring{b}=b(-\cdot)$, we have $l_{\mathring{b}}=-h_b$ and hence, we proved $h_b-2n_{\mathring{\psi}}\le -n_{\mathring{\phi}}$, from which we get $h_b+2(2^jN-n_{\mathring{\psi}})\le 2^{j+1}N-n_{\mathring{\phi}}$.
	By $n_\psi\ge \max(-l_\psi, \frac{n_\phi-l_b}{2})$, we have $l_b+2n_\psi\ge n_\phi$.
	So,
	for every $k=n_\psi,\ldots, 2^jN-n_{\mathring{\psi}}$,
	we have $n_\phi\le l_b+2k\le h_b+2k\le 2^{j+1}N-n_{\mathring{\phi}}$ and thus we deduce from $\psi=2\sum_{n=l_b}^{h_b} b(n) \phi(2\cdot-n)$ that
	\[
	\psi_{j;k}=\sqrt{2} \sum_{n=l_b+2k}^{h_b+2k}
	b(n-2k) \phi_{j+1;n}=\sqrt{2} \sum_{n=n_\phi}^{2^{j+1}N-n_{\mathring{\phi}}}
	b(n-2k) \phi_{j+1;n},\qquad k=n_\psi,\ldots, 2^jN-n_{\mathring{\psi}}.
	\]
	By \eqref{I:psi} for $\mathring{\psi}$ and $h_{\mathring{B}}+n_{\phi}\le 2^{j+1}N$ in \eqref{J0}, noting that $\psi^R=\psi^L(N-\cdot)$ and
	$\mathring{\phi}=\phi(-\cdot)$,
	we have $2^{j+1}N-h_{\mathring{B}}\ge
	n_{\phi}$ and
	\begin{align*}
	\psi^R_{j;2^jN-N}
	&=\mathring{\psi}^L_{j;0}(N-\cdot)\\
	&=\sqrt{2} \mathring{B}_L \mathring{\phi}^L_{j+1;0}(
	 N-\cdot)+\sqrt{2}\sum_{k=n_{\mathring{\phi}}}
	^{h_{\mathring{B}}} \mathring{B}(k) \mathring{\phi}_{j+1;k}(N-\cdot)\\
	&=\sqrt{2} \mathring{B}_L \phi^R_{j+1;2^{j+1}N-N}+\sqrt{2}\sum_{k=n_{\mathring{\phi}}}
	^{h_{\mathring{B}}} \mathring{B}(k) \phi_{j+1;2^{j+1}N-k}\\
	&=\sqrt{2} \mathring{B}_L \phi^R_{j+1;2^{j+1}N-N}+\sqrt{2}\sum^{2^{j+1}N-n_{\mathring{\phi}}}
	_{k=n_\phi} \mathring{B}(2^{j+1}N-k) \phi_{j+1;k},
	\end{align*}
	where we used $2^{j+1}N-h_{\mathring{B}}\ge
	n_{\phi}$ due to our assumption
	$h_{\mathring{B}}+n_\phi\le 2^{j+1}N$ in \eqref{J0} for $j\ge J_0$.
	Hence, we proved the existence of a matrix $B_j$ such that $\Psi_j=B_j \Phi_{j+1}$. The existence of a matrix $A_j$ can be proved similarly by the same argument.
	Similarly we can prove the first part of item (2).
	
	Due to item (S1), by item (ii) of \cref{thm:wbd} or \eqref{cardinality} in \cref{thm:wbd:0}, we must have
	 $\#\tilde{\phi}^L-\#\phi^L=(n_{\tilde{\phi}}-n_\phi)(\#\phi)$ and
	 $\#\tilde{\mathring{\phi}}^L-\#\mathring{\phi}^L=
	 (n_{\tilde{\mathring{\phi}}}-n_{\mathring{\phi}})(\#\phi)$, from which we have
	\[
	 \#\tilde{\phi}^L+\#\tilde{\mathring{\phi}}^L
	 -(n_{\tilde{\mathring{\phi}}}+n_{\tilde{\phi}})
	(\#\phi)
	=(\#\phi^L+\#\mathring{\phi}^L)
	-(n_{\mathring{\phi}}+n_{\phi})
	(\#\phi).
	\]
	By \eqref{reflection}, we have $\#\tilde{\phi}^R=\#\tilde{\mathring{\phi}}^L$ and
	$\#\phi^R=\#\mathring{\phi}^L$.
	Note that  $2^j N\ge n_\phi+n_{\mathring{\phi}}-1$ for all $j\ge J_0$ by \eqref{J0} and $2^j N\ge n_{\tilde{\phi}}+n_{\tilde{\mathring{\phi}}}-1$ for all $j\ge \tilde{J}_0$ by \eqref{tJ0}. Consequently,
	\be \label{Phij:count}
	\#\Phi_j=\#\phi^L+\#\mathring{\phi}^L
	 +(2^jN-n_{\mathring{\phi}}-n_{\phi}+1)(\#\phi),\qquad j\ge J_0.
	\ee
	Using the above two identities and $\tilde{J}_0\ge J_0$, for $j\ge \tilde{J}_0$, we deduce that
	\[ \#\tilde{\Phi}_j=\#\tilde{\phi}^L+\#\tilde{\mathring{\phi}}^L
	 +(2^jN-n_{\tilde{\mathring{\phi}}}-n_{\tilde{\phi}}+1)(\#\phi)
	=\#\phi^L+\#\mathring{\phi}^L
	 +(2^jN-n_{\mathring{\phi}}-n_{\phi}+1)(\#\phi)=\#\Phi_j.
	\]
	By \cref{thm:wbd}, we must have
	 $\#\tilde{\psi}^L-\#\psi^L=(n_{\tilde{\psi}}-n_\psi)(\#\psi)$ and
	 $\#\tilde{\mathring{\psi}}^L-\#\mathring{\psi}^L=
	 (n_{\tilde{\mathring{\psi}}}-n_{\mathring{\psi}})(\#\psi)$. By the same argument, we must have $\#\tilde{\Psi}_j=\#\Psi_j$.
	By the proved identities $\#\tilde{\Phi}_j=\#\Phi_j$ and $\#\tilde{\Psi}_j=\#\Psi_j$ for $j\ge \tilde{J}_0$, we observe from \eqref{disjointPhiPsi} that $\tilde{\Phi}_j\cup \tilde{\Psi}_j$ is biorthogonal to $\Phi_j\cup \Psi_j$ for all $j\ge \tilde{J}_0$.
	
	Since $\tilde{\Phi}_j\cup \tilde{\Psi}_j$ is biorthogonal to $\Phi_j\cup \Psi_j$, the proved identities $\Phi_j=A_j\Phi_{j+1}$ and $\Psi_j=B_j\Phi_{j+1}$ imply $\#\Phi_j+\#\Psi_j\le \#\Phi_{j+1}$ and $\Phi_j\cup \Psi_j$ is a Riesz sequence. To prove the other direction,
	by \eqref{phi2k0},
	\[
	\phi_{j+1;m}=\sqrt{2}\sum_{k=\lceil \frac{m-h_{\tilde{a}}}{2}\rceil}
	^{\lfloor \frac{m-l_{\tilde{a}}}{2}\rfloor}
	\ol{\tilde{a}(m-2k)}^\tp\phi_{j;k}+
	\sqrt{2}\sum_{k=\lceil \frac{m-h_{\tilde{b}}}{2}\rceil}
	^{\lfloor \frac{m-l_{\tilde{b}}}{2}\rfloor}
	\ol{\tilde{b}(m-2k)}^\tp \psi_{j;k}.
	\]
	Define $m_1:=\max(2n_\phi+h_{\tilde{a}}, 2n_\psi+h_{\tilde{b}})$ and $m_2:=\max(2n_{\mathring{\phi}}-l_{\tilde{a}},
	2n_{\mathring{\psi}}-l_{\tilde{b}})$. Using \eqref{lem:W}, we conclude from the above identity that $\phi_{j+1;m}\in \mbox{span}(\Phi_j\cup\Psi_j)$ for all $m_1\le m\le 2^{j+1}N-m_2$.
	On the other hand, for sufficiently large $j$, it follows directly from item (3) of \cref{thm:wbd:0} that
\[
\phi^L_{j+1;0} \cup\{\phi_{j+1;k}\}_{k=n_\phi}^{m_1-1}\subseteq \mbox{span}(\Phi_j\cup\Psi_j)
\]
and
\[
\phi^R_{j+1;2^jN-N}\cup \{\phi_{j+1;k}\}_{k=2^{j+1}N-m_2+1}^{2^{j+1}N-n_{\mathring{\phi}}}
	\subseteq \mbox{span}(\Phi_j\cup\Psi_j).
\]
This proves that $\Phi_{j+1}\subseteq \mbox{span}(\Phi_j\cup \Psi_j)$ for sufficiently large $j$. Since $\Phi_j\cup \Psi_j$ is a Riesz sequence, we conclude from $\Phi_{j+1}\subseteq \mbox{span}(\Phi_j\cup \Psi_j)$ that $\#\Phi_{j+1}\le \#\Phi_j+\#\Psi_j$. Hence, we proved $\#\Phi_{j+1}= \#\Phi_j+\#\Psi_j$ and we deduce from \eqref{Phij:count} that
	$\#\Psi_j=\#\Phi_{j+1}-\#\Phi_j=
	2^jN(\#\phi)$ for sufficiently large $j$.
	Note that $2^j N\ge n_\psi+n_{\mathring{\psi}}-1$ for all $j\ge J_0$ by \eqref{J0} and
	$2^j N\ge n_{\tilde{\psi}}+n_{\tilde{\mathring{\psi}}}-1$ for all $j\ge \tilde{J}_0$ by \eqref{tJ0}.
	From the definition of $\Psi_j$, noting that $\#\psi=\#\phi$ (see item (3) of \cref{thm:bw}) and $\#\Psi_j=2^j N(\#\phi)$, we have
	\be \label{cardmatchPsi}
	\#\psi^L+\#\psi^R=\#\Psi_j-(2^j N-n_\psi-n_{\mathring{\psi}}+1)(\#\psi)=
	(n_\psi+n_{\mathring{\psi}}-1)(\#\phi).
	\ee
	Now for any arbitrary $j\ge J_0$, by definition  of $\Psi_j$ in \eqref{Psij} and the above identity, 
	we have
	\[
	 \#\Psi_j=\#\psi^L+\#\psi^R+(2^jN-n_{\mathring{\psi}}-n_\psi+1)(\#\phi)
	=(n_\psi+n_{\mathring{\psi}}-1)(\#\phi)
	 +(2^jN-n_{\mathring{\psi}}-n_\psi+1)(\#\phi)
	=2^j N(\#\phi).
	\]
	Consequently, $\#\Psi_j=2^j N(\#\phi)=\#\Phi_{j+1}-\#\Phi_j$ and hence $\#\Phi_{j+1}= \#\Phi_j+\#\Psi_j$ for all $j\ge J_0$.
	This proves both items (1) and (2).
	
	Next we prove item (3).
	By proved items (1) and (2),
	$[\ol{A_j}^\tp, \ol{B_j}^\tp]$
	must be a square matrix for all $j\ge J_0$ and $[\tilde{A}_j^\tp, \tilde{B}_j^\tp]$ must be a square matrix for all $j\ge \tilde{J}_0$. Since $\tilde{\Phi}_j\cup \tilde{\Psi}_j$ is biorthogonal to $\Phi_j\cup \Psi_j$, we must have \eqref{refstr:inv}.
	Now by items (1) and (2),
	we can directly check that $\cB_J$ and $\tilde{\cB}_J$ are biorthogonal to each other for all $J\ge \tilde{J}_0$.
	Note that $\cB_{J}$ is a Bessel sequence, since $\cB_{J} \subseteq \AS_{J}(\Phi;\Psi)_{[0,\infty)} \cup \{\eta(N-\cdot): \eta \in\AS_{J}(\mathring{\Phi};\mathring{\Psi})_{[0,\infty)}\}$. By a similar reasoning, $\tilde{\cB}_{J}$ is also a Bessel sequence.
	By the same standard argument as in the proof of (4)$\imply$(1) in \cref{thm:rz}, both $\cB_{J}$ and $\tilde{\cB}_J$ are Riesz sequences in $L_2([0,N])$ for all $J\ge \tilde{J}_0$.
	By the proved item (1), we have $\mbox{span}(\cB_J)\supset \cup_{j=J}^\infty \{\phi_{j;k} \setsp n_\phi\le k\le 2^j N-n_{\mathring{\phi}}\}$, which spans  a dense subset of $L_2([0,N])$. That is, $\mbox{span}(\cB_J)$ is dense in $L_2([0,N])$ and hence,
	$\cB_J$ must be a Riesz basis of $L_2([0,N])$. Similarly, we can prove that $\tilde{\cB}_J$ is a Riesz basis of $L_2([0,N])$ for all $J\ge \tilde{J}_0$.
	Since $\cB_J$ and $\tilde{\cB}_J$ are biorthogonal to each other,
	this proves that $\tilde{\cB}_J$ and $\cB_J$ form a pair of biorthogonal Riesz bases for $L_2([0,N])$ for all $J\ge \tilde{J}_0$. This proves item (3).
	
	Using item (3), item (4) can be easily proved by the same argument as in \cref{lem:vm}.
	
	By item (3),
	for $J_0\le J<\tilde{J}_0$ such that $J$ decreases from $\tilde{J}_0-1$ to $J_0$, using \eqref{refstr:inv} and the biorthogonality relation between $\cB_J$ and $\tilde{\cB}_J$, we can recursively prove that $\tilde{\cB}_J$ and $\cB_J$ form a pair of biorthogonal Riesz bases for $L_2([0,N])$. This proves item (5).
\end{proof}


\begin{thebibliography}{10}
\bibitem{ak12}
A.~Alt\"urk and F.~Keinert, Regularity of boundary wavelets. \emph{Appl. Comput. Harmon. Anal.} \textbf{32} (2012), 65--85.


\bibitem{ahjp94}
L.~Andersson, N.~Hall, B.~Jawerth, and G.~Peters, Wavelets on closed subsets of the real line. Recent advances in wavelet analysis, 1--61, \emph{Wavelet Anal. Appl.}, 3, Academic Press, Boston, MA, 1994.

\bibitem{ahl17}
E.~Ashpazzadeh, B.~Han, and M.~Lakestani, Biorthogonal multiwavelets on the interval for numerical solutions of Burgers' equation. \emph{J. Comput. Appl. Math.} \textbf{317} (2017), 510--534.

\bibitem{a93}
P.~Auscher, Ondelletes \'a support compact et conditions aux limites. \emph{J. Funct. Anal.} \textbf{111} (1993), 29--43.


\bibitem{cer19}
D.~\v{C}ern\'a, Wavelets on the interval and their applications, Habilitation thesis at Masaryk University, (2019).

\bibitem{cf11}
D.~\v{C}ern\'a and V.~Fin\v{e}k,
Construction of optimally conditioned cubic spline wavelets on the interval. \emph{Adv. Comput. Math.} \textbf{34} (2011),219--252.

\bibitem{cf12}
D.~\v{C}ern\'a and V.~Fin\v{e}k,
Cubic spline wavelets with complementary boundary conditions. \emph{Appl. Math. Comput.} \textbf{219} (2012), 1853--1865.




\bibitem{cq04}
C.~K.~Chui and E.~Quak, Wavelets on a bounded interval. {Numerical methods in approximation theory}, Vol. 9, 53--75, \emph{Internat. Ser. Numer. Math.}, 105, Birkh\"auser, Basel, 1992.

\bibitem{cw92}
C.~K.~Chui and J.~Z.~Wang, On compactly supported wavelets and
a duality principle, \emph{Trans. Amer. Math. Soc.} \textbf{330}
(1992) 903--916.

\bibitem{cdf92}
A.~Cohen, I.~Daubechies, and J.~C.~Feauveau, Biorthogonal bases of compactly supported wavelets. \emph{Comm. Pure Appl. Math.}
\textbf{45} (1992), 485--560.

\bibitem{cdp97}
A.~Cohen, I.~Daubechies, and G.~Plonka, Regularity of refinable function vectors. \emph{J. Fourier Anal. Appl.} \textbf{3} (1997), 295--324.

\bibitem{cdv93}
A.~Cohen, I.~Daubechies, and P.~Vial, Wavelets on the interval and fast wavelet transforms. \emph{Appl. Comput. Harmon. Anal.} \textbf{1} (1993), 54--81.

\bibitem{dgh96}
G.~Donovan, J.~Geronimo, and D.~Hardin, Intertwining multiresolution analyses and the construction of piecewise-polynomial wavelets. \emph{SIAM J. Math. Anal.} \textbf{27} (1996), 1791--1815.


\bibitem{dhjk00}
W.~Dahmen, B.~Han, R.-Q.~Jia, and A.~Kunoth, Biorthogonal multiwavelets on the interval: cubic Hermite splines. \emph{Constr. Approx.} \textbf{16} (2000), 221--259.

\bibitem{dku99}
W.~Dahmen, A.~Kunoth and K.~Urban, Biorthogonal spline wavelets on the interval---stability and moment conditions. \emph{Appl. Comput. Harmon. Anal.} \textbf{6} (1999), 132--196.

\bibitem{ds98}
W.~Dahmen and R.~Schneider, Wavelets with complementary boundary conditions--function spaces on the cube. \emph{Results Math.} \textbf{34} (1998), 255--293.

\bibitem{dau88}
I.~Daubechies, Orthonormal bases of compactly supported wavelets. \emph{Comm. Pure Appl. Math.} \textbf{41} (1988), 909--996.

\bibitem{ds10}
T.~J.~Dijkema and R.~Stevenson. A sparse Laplacian in tensor product wavelet corrdinates. \emph{Numer. Math.} \textbf{115} (2010), 433--449.

\bibitem{ghm94}
J.~S.~Geronimo, D.~P.~Hardin, and P.~R.~Massopust, Fractal functions and wavelet expansions based on several scaling functions. \emph{J. Approx. Theory} \textbf{78} (1994), 373--401.

\bibitem{gjx00}
S.~S.~Goh, Q.~T.~Jiang, and T.~Xia, Construction of biorthogonal multiwavelets using the lifting scheme. \emph{Appl. Comput. Harmon. Anal.} \textbf{9} (2000), 336--352.



\bibitem{han01}
B.~Han, Approximation properties and construction of Hermite interpolants and biorthogonal multiwavelets. \emph{J. Approx. Theory} \textbf{110} (2001), 18--53.

\bibitem{han03jat}
B.~Han, Vector cascade algorithms and refinable function vectors in Sobolev space, \emph{J. Approx. Theory.} \textbf{124} (2003), 44--88.

\bibitem{han03}
B.~Han, Compactly supported tight wavelet frames and orthonormal wavelets of exponential decay with a general dilation matrix. \emph{J. Comput. Appl. Math.} \textbf{155} (2003), 43--67.

\bibitem{han06}
B.~Han, Solutions in Sobolev spaces of vector refinement equations with a general dilation matrix. \emph{Adv. Comput. Math.} \textbf{24} (2006), 375--403.


\bibitem{han10}
B.~Han, Pairs of frequency-based nonhomogeneous dual wavelet frames in the distribution space. \emph{Appl. Comput. Harmon. Anal.} \textbf{29} (2010), 330--353.

\bibitem{han12}
B.~Han, Nonhomogeneous wavelet systems in high dimensions. \emph{Appl. Comput. Harmon. Anal.} \textbf{32} (2012), 169--196.

\bibitem{hanbook}
B.~Han, Framelets and wavelets: Algorithms, analysis, and applications. \emph{Applied and Numerical Harmonic Analysis}. Birkh\"auser/Springer, Cham, 2017. xxxiii + 724 pp.

\bibitem{hkp06}
B.~Han, S.-G.~Kwon, and S. S.~Park, Riesz multiwavelet bases. \emph{Appl. Comput. Harmon. Anal.} \textbf{20} (2006), 161--183.

\bibitem{hj98}
B.~Han and R.-Q.~Jia, Multivariate refinement equations and convergence of subdivision schemes. \emph{SIAM J. Math. Anal.} \textbf{29} (1998), 1177--1199.


\bibitem{hj02}
B.~Han and Q.~T.~Jiang, Multiwavelets on the interval. \emph{Appl. Comput. Harmon. Anal.} \textbf{12} (2002), 100--127.


\bibitem{hm18}
B.~Han and M.~Michelle, Construction of wavelets and framelets on a bounded interval. \emph{Anal. Appl.} \textbf{16} (2018), 807--849.

\bibitem{hm07}
B.~Han and Q.~Mo, Analysis of optimal bivariate symmetric refinable Hermite interpolants. \emph{Commun. Pure Appl. Anal.} \textbf{6} (2007), 689--718.

\bibitem{hz10}
B.~Han and X.~Zhuang, Matrix extension with symmetry and its application to symmetric orthonormal multiwavelets. \emph{SIAM J. Math. Anal.} \textbf{42} (2010), 2297--2317.


\bibitem{hm99}
D.~P.~Hardin and S.~A.~Marasovich, Biorthogonal multiwavelets on $[-1,1]$.
\emph{Appl. Comput. Harmon. Anal.} \textbf{7} (1999), 34--53.



\bibitem{jia09}
R.-Q.~Jia, Spline wavelets on the interval with homogeneous boundary conditions. \emph{Adv. Comput. Math.} \textbf{30} (2009), 177--200.

\bibitem{jj03}
R.-Q.~Jia and Q.~T.~Jiang, Spectral analysis of the transition operator and its applications to smoothness analysis of wavelets. \emph{SIAM J. Matrix Anal. Appl.} \textbf{24} (2003), 1071--1109.

\bibitem{jrz99}
R.-Q.~Jia, S.~D.~Riemenschneider, and D.-X.~Zhou, Smoothness of multiple refinable functions and multiple wavelets. \emph{SIAM J. Matrix Anal. Appl.} \textbf{21} (1999), 1--28.

\bibitem{jiang99}
Q.~T.~Jiang, Multivariate matrix refinable functions with arbitrary matrix dilation. \emph{Trans. Amer. Math. Soc.} \textbf{351} (1999), 2407--2438.

\bibitem{jl93}
A.~Jouini and P.~G.~Lemarie-Rieusset, Analyse multi-r\'esolution bi-orthogonale sur l'intervalle et applications.  \emph{Ann. Inst. H. Poincar\'e Anal. Non Lin\'eaire} \textbf{10} (1993), 453--476.

\bibitem{keibook}
F.~Keinert, Wavelets and multiwavelets. \emph{Studies in Advanced Mathematics}. Chapman \& Hall/CRC, Boca Raton, FL, 2004. xii+275 pp.

\bibitem{kei15}
F.~Keinert, Regularity and construction of boundary multiwavelets. \emph{Poincar\'e J. Anal. Appl.} \textbf{2}, (2015), 1--12.

\bibitem{mad97}
W.~R.~Madych, Finite orthogonal transforms and multiresolution analyses on intervals. \emph{J. Fourier Anal. Appl.} \textbf{3} (1997), 257--294.

\bibitem{mas96}
R.~Masson, Biorthogonal spline wavelets on the interval for the resolution of boundary problems. \emph{Math. Models Methods Appl. Sci.} \textbf{6} (1996), 749--791.

\bibitem{mey91}
Y.~Meyer, Wavelets on the interval. \emph{Rev. Mat. Iberoamericana} \textbf{7} (1991), 115--133.

\bibitem{pst95}
G.~Plonka, K.~Selig, and M.~Tasche, On the construction of wavelets on a bounded interval. \emph{Adv. Comput. Math.} \textbf{4} (1995), 357--388.

\bibitem{pri10}
M.~Primbs, New stable biorthogonal spline-wavelets on the interval. \emph{Results Math.} \textbf{57} (2010), 121--162.
\end{thebibliography}
\end{document}